\documentclass[11pt,a4paper]{article}
\usepackage[utf8]{inputenc}

\usepackage[
    backend=biber,
    style=numeric,
    sorting=none,
    giveninits=true,
    maxbibnames=5
]{biblatex}
\DeclareSourcemap{
  \maps[datatype=bibtex]{
    \map{
      \step[fieldsource=doi, final] 
      \step[fieldset=url, null]    
      \step[fieldset=urldate, null] 
      \step[fieldset=eprint, null] 
    }
  }
}
\DeclareSourcemap{
  \maps[datatype=bibtex]{
    \map{
      \step[fieldsource=eprint, final] 
      \step[fieldset=url, null]    
      \step[fieldset=urldate, null] 
    }
  }
}
\renewbibmacro*{in:}{%
  \iffieldundef{eprint}
    {\printtext{\bibstring{in}\intitlepunct}}
    {}%
}

\usepackage{amsmath,amssymb, amsthm, amsfonts}
\usepackage[english]{babel}
\usepackage{xcolor}
\usepackage{enumitem}
\usepackage{graphicx}
\usepackage{hyperref}
\usepackage{mathtools}
\usepackage{bbm}
\usepackage{csquotes}
\usepackage{comment}
\usepackage{hyperref}
\usepackage{caption}
\usepackage{subcaption}
\usepackage{placeins}
\usepackage{orcidlink}
\usepackage{tikz}
\usetikzlibrary{patterns,positioning}
\usepackage{pgfplots}
\usepackage{url}
\usetikzlibrary{positioning, fit, shapes.geometric, backgrounds}
\usetikzlibrary{shadows.blur,arrows.meta}

\pgfplotsset{compat=newest}
\usepgfplotslibrary{fillbetween,groupplots}
\usetikzlibrary{positioning,shapes.geometric}

\allowdisplaybreaks

\usetikzlibrary{fit}

\newtheorem{proposition}{Proposition}
\newtheorem{corollary}{Corollary}

\newtheorem{remark}{Remark}

\newtheorem{assumption}{Assumption}
\graphicspath{{graphics/}}

\usepackage{tcolorbox}
\newtcolorbox{boxmras}{colback=black!2!white,colframe=black!77!black}

\title{Data assimilation  via model reference adaptation for linear and nonlinear dynamical systems}

\author{Benedikt Kaltenbach\footnote{University of G\"ottingen, Germany, \href{mailto:b.kaltenbach02@stud.uni-goettingen.de}{b.kaltenbach02@stud.uni-goettingen.de} The author declares no relationship with the first author of \cite{Tram_Kaltenbacher_2021}}\orcidlink{0009-0004-1387-1339}, Christian Aarset\footnote{Aix-Marseille University, France, \href{mailto:christian.AARSET@univ-amu.fr}{christian.aarset@univ-amu.fr}}  \orcidlink{0000-0001-8163-9305}
, Tram Thi Ngoc Nguyen\footnote{University of Cambridge, United Kingdom --  Department of Applied Mathematics and Theoretical Physics \href{mailto:ttn35@cam.ac.uk{ttn35@cam.ac.uk}}  \emph{(corresponding author)}\\
Accepted for publication in \emph{Applied Mathematics for Modern Challenges}, June 2026} \orcidlink{0000-0002-7245-7611}}
\date{}

\usepackage{soul}

\usepackage{diagbox}

\def\deltil{\tilde{\delta}}
\def\ti{{\rm ti}}
\def\sp{{\rm sp}}
\def\cR{\mathcal{R}}

\def\N{\mathbb{N}}
\def\R{\mathbb{R}}

\def\Uc{\mathcal{U}}

\def\Qc{\mathcal{Q}}

\def\Cc{\mathcal{C}}
\def\cC{\mathcal{C}}
\def\Lc{\mathcal{L}}

\def\nvec{\mathbf{n}}
\def\dom{\Omega}
\def\bou{{\partial\dom}}

\def\ctil{\tilde{c}}

\renewcommand{\d}{\,\mathrm{d}}
\def\I{[0,\infty)}
\def\embed{\hookrightarrow} 
\def\dom{\Omega}
\def\Ntil{\widetilde{N}}
\def\Mtil{\widetilde{M}}
\def\Ccoe{C_\text{coe}}

\def\errpar{\varepsilon_q}
\def\errst{\varepsilon_u}

\def\lan{\langle}
\def\ran{\rangle}

\def\utrue{u^\dagger}
\def\uobs{z}  
\def\qtrue{q^\dagger}

\DeclareMathOperator{\tr}{Tr} 
\begin{document}

\maketitle
\paragraph{Abstract} We address data assimilation for linear and nonlinear dynamical systems via the so-called \emph{model reference adaptive system}. Continuing our theoretical developments in \cite{Tram_Kaltenbacher_2021}, we deliver the first practical implementation of this approach for online parameter identification with time series data. Our semi‑implicit scheme couples a modified state equation with a parameter evolution law that is driven by model-data residuals. We demonstrate four benchmark problems of increasing complexity: the Darcy flow, the Fisher-KPP equation, a nonlinear potential equation and finally, an Allen–Cahn type equation. Across all cases, explicit model reference adaptive system construction, verified assumptions and numerically stable \linebreak reconstructions underline our proposed method as a reliable, versatile tool for data assimilation and real-time inversion. \\[1.5ex]
\textbf{MSC classes.}  65M32, 35R30, 35K57, 37N30, 65J22, 65M60.

\section{Introduction}
Data assimilation is the process of estimating the evolving state of a dynamical system by optimally 
combining observational data with a model prediction in an \emph{online} fashion, such as in weather forecast. 
Mathematically, one considers a dynamical model governed by a time-dependent partial differential equation (PDE) together with an observation operator, 
and seeks an optimal state or model parameter that balances  data fidelity with the model dynamics.

A wide variety of data assimilation methods have been developed for probabilistic finite dimensions; these may broadly be classified into \linebreak sequential and variational approaches. Sequential schemes, exemplified by the Kalman filter and its nonlinear extensions such as extended Kalman filter and ensemble Kalman filter, update the system state whenever new observations become available \cite{Albers:2019,SchillingsStuart2017}.
These approaches, due to the cost of explicitly evolving  
the full error covariance, prompted the development of 
reduced-rank  \cite{DihlmannHaasdonk2015} and learning variants \cite{frerix21a}. Variational methods \cite{rabier2003variational} instead
minimize a cost functional that penalizes deviations between estimated states and observations, typically under Gaussian error assumptions. Three-dimensional variational assimilation (3D-Var) \cite{lorenc81,lorenc86} solves a static analysis problem at a single time, while four-dimensional variational assimilation (4D-Var) \cite{ledimet86variational,lewis85adjoint} generalizes this to time-distributed observations. 
We refer to the seminal book \cite{Reich15} for a general framework on probabilistic forecasting and Bayesian data assimilation.

Although data assimilation is predominantly used for state estimation, it also provides a natural framework for \emph{estimating uncertain parameters} in the models. 
From the perspective of inverse problems \cite{BarbaraNeubauerScherzer,Kirsch}, this corresponds to an iterative regularization procedure for \emph{parameter identification} in parabolic PDEs \cite{KNO}. 
However, the distinguishing feature of data assimilation-based parameter identification is that the inversion is carried out \emph{simultaneously} with data acquisition, rather than in a purely off-line setting as in the classical inversion framework. This perspective connects data assimilation with online or \emph{on‑the‑fly estimation} methods 
that continuously refine parameter estimates as new data arrives. Such online parameter estimation is of particular importance in model predictive control and related applications \cite{Narendra05, Ioannou, Sastry}, where decisions must be updated in real time.

\emph{Model reference adaptive systems} (MRAS) form a particular class of adaptive control schemes, designing  
{dynamic update laws} for both parameter and state that drive the estimated output toward the reference trajectory.  
More precisely, the state equation is modified by feedback terms 
and the parameter evolution 
is driven by model-observation mismatch. 
We refer to the literature review in \cite{Baumeister-etal} as well as the work \cite{Kugler1,Kugler, BoigerKaltenbacher} for MRAS-based approaches in the context of PDEs. 
These earlier MRAS approaches have been developed mainly under the assumption that the PDE models depends linearly on the unknown parameters.

Our early work \cite{Tram_Kaltenbacher_2021} lifted this restriction by allowing the PDE model to be \emph{nonlinear} both in the state and in the parameters.  
There, we introduced a nonlinear MRAS in which the state dynamics are modified by an observation-driven feedback operator and the parameter evolution involves a suitable linearization and additional stabilization. Through the MRAS, we developed an online parameter identification method that can be interpreted as a data assimilation strategy for infinite dimension. It may also be seen as a deterministic alternative to statistical approaches, and is particularly suited to real-time inversion.

While \cite{Tram_Kaltenbacher_2021} provided an abstract well‑posedness and convergence analysis for the nonlinear MRAS,
we did not address numerical realization or demonstrate the approach on concrete PDE applications. Building on this theoretical framework, the present work implements, for the first time, the MRAS proposed in \cite{Tram_Kaltenbacher_2021} for PDEs with strong nonlinearity in both the state and the parameter. We propose a semi‑implicit time‑stepping scheme that assimilates time series data in an online fashion and allows for stable recovery of physical parameters in several benchmark scenarios. The four examples summarized in Table \ref{table:overview}, ordered by increasing complexity, illustrate in detail how to derive the individual components of the MRAS. In each, we verify the analytical conditions required for convergence, and realize an efficient numerical implementation. This is followed by systematic presentations of the resulting reconstructions. Taken together, these examples demonstrate the wide-reaching applicability of the MRAS framework to online identification of physical parameters, as highlighted in Figure \ref{fig:overview}.

\paragraph{Outline}
The structure of the paper is as follows. Section \ref{sec::adptive_sys} introduces the MRAS structure used throughout this work, state the standing assumptions, and briefly reviews the relevant convergence results. Section \ref{sec::femdiscr} describes the numerical realization, including the spatial finite element discretization and the semi‑implicit time‑stepping scheme. The subsequent sections are devoted to applications: in Section \ref{sec::darcy} we consider a Darcy–type problem, in Section \ref{sec::nonlinear_state} a Fisher–KPP‑type equation with nonlinear state dynamics, and in Sections \ref{sec::nonlinear_cprob} and \ref{sec::allen-cahn} two nonlinear coefficient identification problems, including an Allen–Cahn–type equation. Each example follows a common structure: explicit derivation of the MRAS components, verification of the analytical conditions, and presentation and discussion of the numerical results.

\begingroup
\renewcommand{\arraystretch}{1.2}
\begin{table}[t]
\centering
\small
\begin{tabular}{p{2cm}
                p{4.5cm}
                p{2.5cm}
                p{2cm}
                }
\hline
\textbf{Test case} &
\textbf{PDE} &
{\bf Type}&
{\bf Unknown parameter} \\%
\hline\hline

Darcy flow&
$D_t u - \nabla\cdot(a\nabla u) = g$ &
linear& \qquad$a$ \\
\hline

Fisher-KPP &
$D_t u - \nabla\cdot(a\nabla u) + u - u^2 = g$ &
nonlinear in $u$ & \qquad$a$ \\
\hline

Nonlinear potential &
$D_t u - \Delta u + cu + c|c|^{\frac{2}{3}}u = g$ &
nonlinear in $c$& \qquad$c$ \\
\hline

Modified Allen-Cahn&
$D_t u - \Delta u + cu^3 + c|c|^{\frac{2}{3}}u = g$ &
nonlinear in $u$, $c$& \qquad$c$ \\
\hline
\end{tabular}
\caption{Benchmark test cases for MRAS-based online parameter recovery.}
\label{table:overview}
\end{table}

\begin{figure}
\centering
\begin{tabular}{cccc}
\begin{tikzpicture}
  \node[anchor=south west, inner sep=0] (image) at (0,0) {
    \includegraphics[width=0.24\linewidth]{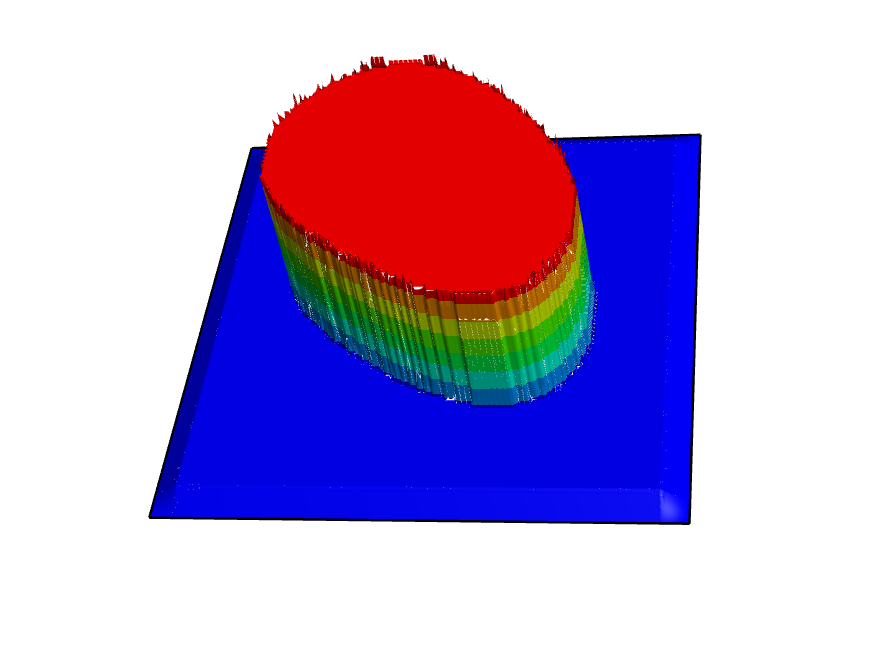}
  };
  \begin{scope}[x={(image.south east)}, y={(image.north west)}]
    \draw[-stealth, thick, red] (0.17, 0.23) -- (0.17, 0.545) node[midway, left] {z};
    \draw[-stealth, thick, red] (0.17, 0.25) -- (0.23, 0.52) node[midway, right] {y};
    \draw[-stealth, thick, red] (0.17, 0.23) -- (0.4, 0.222) node[midway, below] {x};
  \end{scope}
\end{tikzpicture}
&
\begin{tikzpicture}
  \node[anchor=south west, inner sep=0] (image) at (-0.045,-0.095) {
    \includegraphics[width=0.24\linewidth]{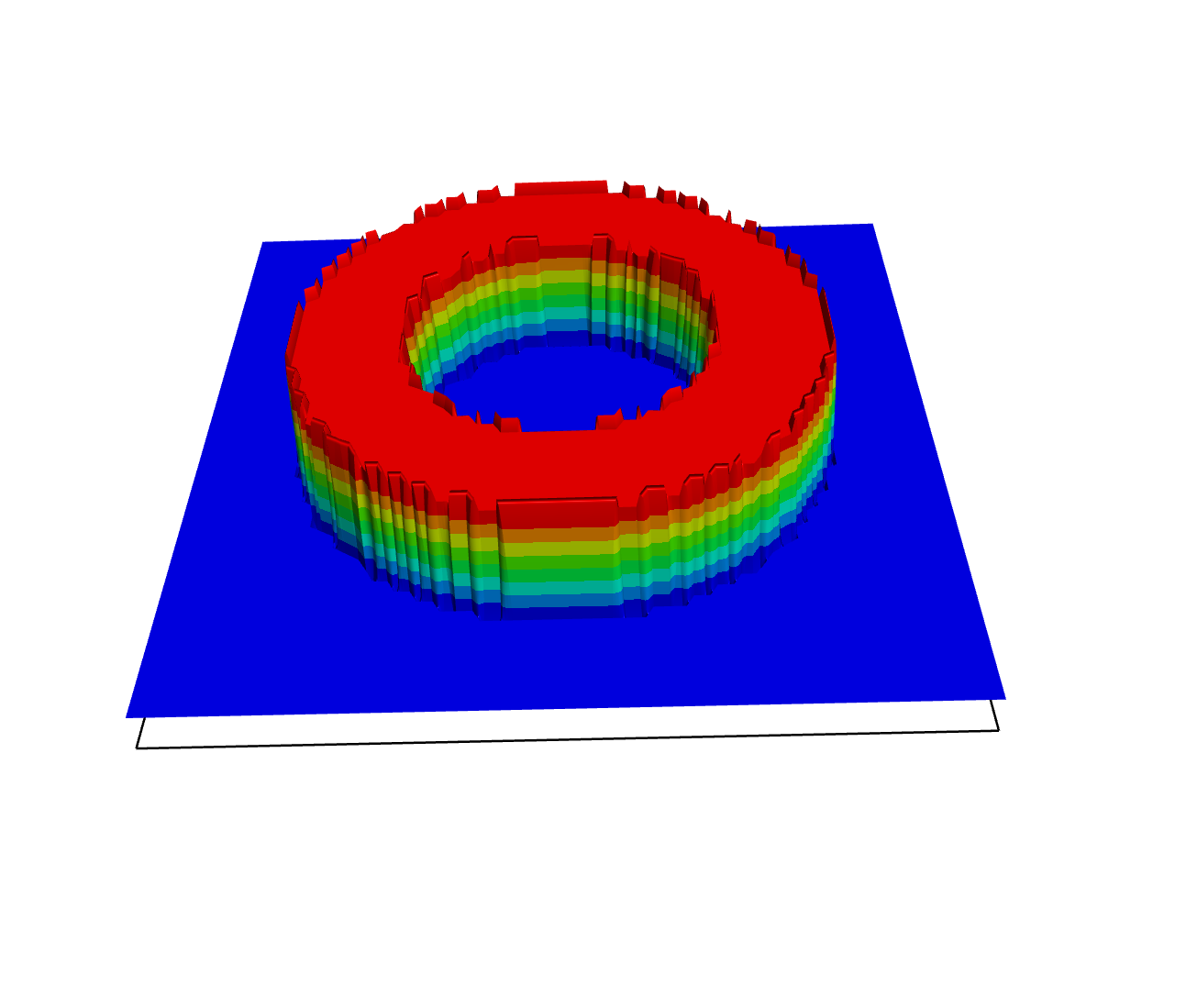}
  };
  \begin{scope}[x={(image.south east)}, y={(image.north west)}]
    \draw[-stealth, thick, red] (0.1, 0.271) -- (0.1, 0.54) node[midway, left] {z};
    \draw[-stealth, thick, red] (0.1, 0.271) -- (0.155, 0.5) node[midway, right] {y};
    \draw[-stealth, thick, red] (0.1, 0.271) -- (0.35, 0.2835) node[midway, below] {x};
  \end{scope}
\end{tikzpicture} &
\begin{tikzpicture}
  \node[anchor=south west, inner sep=0] (image) at (0,0) {    \includegraphics[width=0.24\linewidth]{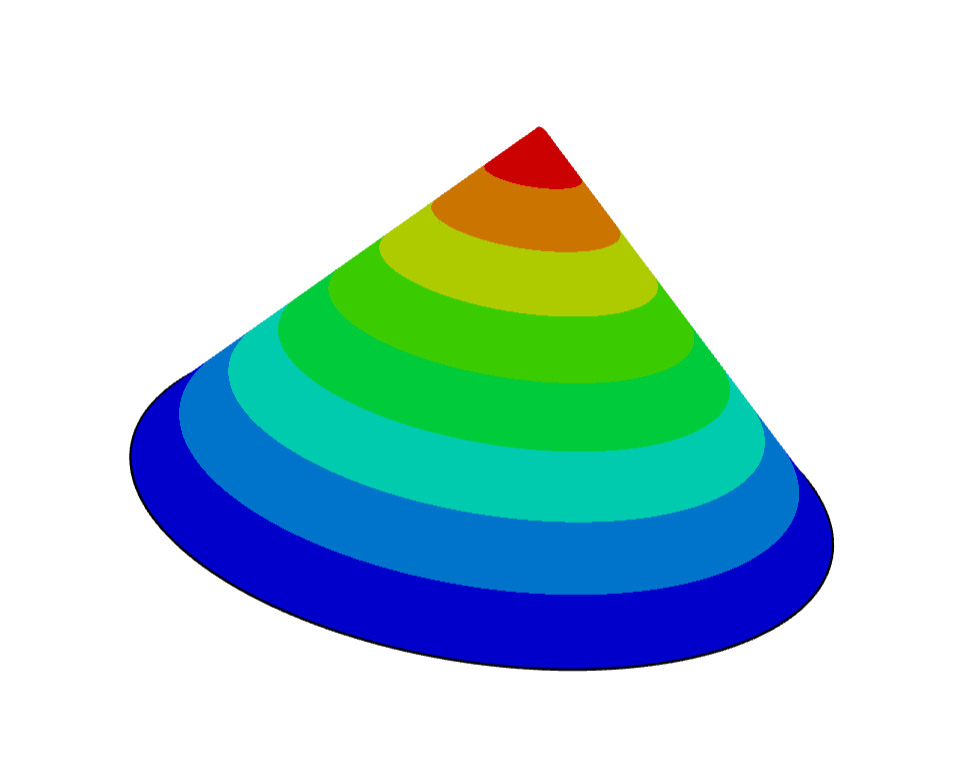}
  };
  \begin{scope}[x={(image.south east)}, y={(image.north west)}]
    \draw[-stealth, thick, red] (0.555, 0.827) -- (0.60, 1.072) node[midway, right] {z};
  \end{scope}
\end{tikzpicture} &
\begin{tikzpicture}
  \node[anchor=south west, inner sep=0] (image) at (0.4,0.4) {
    \includegraphics[width=0.17\linewidth]{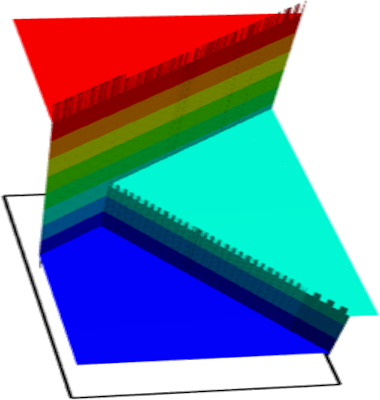}
  };
  \begin{scope}[x={(image.south east)}, y={(image.north west)}]
    \draw[-stealth, thick, red] (0.3, 0.18) -- (0.55, 0.16) node[midway, below] {x};
    \draw[-stealth, thick, red] (0.3, 0.18) -- (0.188, 0.4) node[midway, left] {y};
    \draw[-stealth, thick, red] (0.3, 0.18) -- (0.3, 0.42) node[midway, right] {z};
  \end{scope}
\end{tikzpicture}
\\[-0.1cm]
\pgfplotscolorbardrawstandalone[
    colormap/jet,
    colorbar horizontal,
    point meta min=0,
    point meta max=1,
    colorbar style={
        height=0.25cm,
        width=1.8cm,
        xtick={0,0.5,1},
        xticklabels={0,0.5,1},
        tick label style={font=\tiny},
    }]
&
\pgfplotscolorbardrawstandalone[
    colormap/jet,
    colorbar horizontal,
    point meta min=0,
    point meta max=1,
    colorbar style={
        height=0.25cm,
        width=1.8cm,
        xtick={0,0.5,1},
        xticklabels={0,0.5,1},
        tick label style={font=\tiny},
    }]
&
\pgfplotscolorbardrawstandalone[
    colormap/jet,
    colorbar horizontal,
    point meta min=0,
    point meta max=3,
    colorbar style={
        height=0.25cm,
        width=1.8cm,
        xtick={0,1,2,3},
        xticklabels={0,1,2,3},
        tick label style={font=\tiny},
    }]
&
\pgfplotscolorbardrawstandalone[
    colormap/jet,
    colorbar horizontal,
    point meta min=0,
    point meta max=5,
    colorbar style={
        height=0.25cm,
        width=1.8cm,
        xtick={0,2.5,5},
        xticklabels={0,2.5,5},
        tick label style={font=\tiny},
    }]
\end{tabular}
\caption{Parameters of four benchmark examples in Table \ref{table:overview}. From left to right Darcy flow, Fisher--KPP, nonlinear potential and modified Allen--Cahn equation.}
\label{fig:overview}
\end{figure}
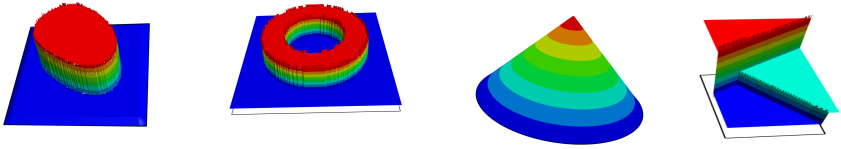

\section{Model reference adaptive system for data assimilation}\label{sec::adptive_sys}

In this section, we lay out the theoretical basis for the MRAS. In the present context, the MRAS will always be used to solve the parabolic equation \begin{equation}\label{model}
\begin{aligned}
D_t\utrue(t) + f(\qtrue,\utrue(t)) & = g(t), \qquad t>0, \\
\utrue(0) & = u_0,
\end{aligned}
\end{equation}
with time-dependent state $\utrue$ and ground truth physical parameter $\qtrue$ -- which is generally spatially dependent -- to be determined, given a source $g$ and an initial condition $u_0$. Boundary conditions are assumed to be captured by the function space setting. The model term $f$, which is supposed known and in general non-linear, gives \eqref{model} sufficient generality to cover a wide family of interesting parabolic PDEs. The details of the function space setting will be specified in the upcoming section. 

The main goal of this work is the \emph{reconstruction of the spatially dependent parameter $\qtrue$} from observation of the state $\utrue(t)$ on the domain $\dom$ as
\begin{equation}\label{data}
    \uobs(t) \approx \utrue(t) \qquad \text{for $t>0$ until current time.}
\end{equation}
It is apparent that if the whole data $z$ were available at once, this problem could be formulated as an inverse parameter identification problem, estimating the spatially dependent parameter $\qtrue$ from the {full time data} $\uobs$ \cite{nguyen24, nguyen25}.

However, in data assimilation, $z(t)$ is available only up to the current time. Building on the work and perspective in \cite{Tram_Kaltenbacher_2021}, we will therefore instead approximate $\qtrue$ in an \emph{online} fashion, gradually updating our estimate of $\qtrue$ as time progresses and more data $z(t)$ becomes available. 
The main contribution of \cite{Tram_Kaltenbacher_2021} was the development and convergence analysis of the so-called \emph{model reference adaptive system} (MRAS) \eqref{mras}, which will lay the foundation for the upcoming numerical study.

\begin{center}
\hspace*{-33pt}
\begin{tikzpicture}[scale=0.65, yscale=0.9, transform shape]
\centering
\begin{axis}[
    width=18cm,
    height=10cm,
    grid=none,
    grid style={dashed, gray!30},
    xlabel={\textbf{Time}},
    ylabel={\textbf{State}},
    xlabel style={font=\large, yshift=-5pt},
    ylabel style={font=\large},
    xmin=-0.5, xmax=10.8,
    ymin=0.1, ymax=1.1,
    xtick={0,2,4,6,8,10}, 
    xticklabels={$t_0$,$t_1$,$t_2$,$t_3$,$t_4$, $t_5$},
    legend style={
        at={(0.99,0.99)},
        anchor=north east,
        font=\small,
        cells={anchor=west},
        draw=black!20
    },
    legend cell align=left,
    tick align=outside,
    tick pos=left
]

\addplot[
    blue,
    dashed,
    line width=2pt,
    smooth,
    tension=0.9
] coordinates {
    (0, 0.63) (0.5, 0.63) 
    (1, 0.68) (1.5, 0.8) 
    (2, 0.87) (2.5, 0.87)
    (3, 0.83) 
    (4, 0.55) 
    (5, 0.35) 
    (6, 0.22) 
    (7, 0.28) (7.5, 0.28)
    (8, 0.47) 
    (9, 0.65) 
    (10, 0.73)
};
\addlegendentry{Exact state}

\addplot[
    red!70,
    line width=2.5pt,
] coordinates {(-1,-1) (-0.5,-1)}; 
\addlegendentry{\begin{minipage}{0.3\textwidth}{\vspace{-2pt}State forecast from \\[-1ex] parameter prediction}\end{minipage}}

\addplot[
    only marks,
    mark=square*,
    mark size=3.5pt,
    red,
    mark options={draw=black!60, line width=0.5pt}
] coordinates {(-1,-1)};
\addlegendentry{Corrected state}

\addplot[red!70, line width=2.5pt, smooth, tension=0.5, forget plot] coordinates {(0, 0.83) (2, 1)};
\addplot[red!70, line width=2.5pt, smooth, tension=0.5, forget plot] coordinates {(2, 0.80) (3, 0.63) (4, 0.42)};
\addplot[red!70, line width=2.5pt, smooth, tension=0.5, forget plot] coordinates {(4, 0.59) (5, 0.48) (6, 0.37)};
\addplot[red!70, line width=2.5pt, smooth, tension=0.5, forget plot] coordinates {(6, 0.25) (7, 0.3) (8, 0.4)};
\addplot[red!70, line width=2.5pt, smooth, tension=0.5, forget plot] coordinates {(8, 0.47) (9, 0.56) (10, 0.72)};

\addplot[
    only marks,
    mark=square*,
    mark size=3.5pt,
    red,
    mark options={draw=black!60, line width=0.5pt},
    forget plot
] coordinates {
    (2, 0.80) (4, 0.59) (6, 0.25) (8, 0.47) (10, 0.72)
};

\addplot[
    only marks,
    mark=star,
    mark size=6pt,
    blue,
    mark options={fill=blue, line width=1pt}
] coordinates {
    (2, 0.72) (4, 0.66) (6, 0.17) (8, 0.54) (10, 0.74)
};3
\addlegendentry{Observations}

\draw[green!50!black, dotted, line width=2.5pt] (axis cs:2, 1.0) -- (axis cs:2, 0.80);
\draw[green!50!black, dotted, line width=2.5pt] (axis cs:4, 0.42) -- (axis cs:4, 0.59);
\draw[green!50!black, dotted, line width=2.5pt] (axis cs:6, 0.37) -- (axis cs:6, 0.25);
\draw[green!50!black, dotted, line width=2.5pt] (axis cs:8, 0.4) -- (axis cs:8, 0.47);

\node[
    draw=green!50!black,
    fill=green!10,
    rounded corners=3pt,
    align=left,
    font=\small,
    anchor=west,
    inner sep=3pt
] (note) at (axis cs:3, 0.92) {State correction};

\draw[->, >=stealth, green!50!black, thick] (note.west) -- (axis cs:2.05, 0.93);

\end{axis}
\end{tikzpicture}

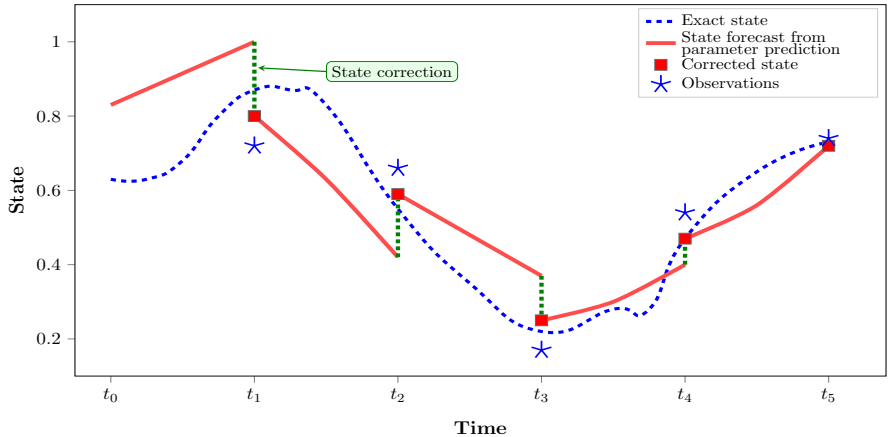
\captionof{figure}{
    Schematic of data assimilation for the state $u$ over assimilation windows $t_{i+1}-t_i$. MRAS iteratively reconstructs the unknown parameters driven by the assimilation of the state.}
\end{center}

\subsection{Underlying parametric equation and time series data}
Given a spatial domain $\dom$ and the infinite time domain $[0,\infty)$, let $H$ be a Hilbert space of functions on $\Omega$. We consider the evolution equation \eqref{model}, given initial data $u_0\in H$ and a time-dependent source $g\in L^2(\I;U^*)$. Here, $U$ is a smooth Sobolev space on $\Omega$, belonging to the Gelfand triple
\[
    U \hookrightarrow H \hookrightarrow U^*,
\]
where $(\cdot)^*$ denotes the dual space. Similarly, the true parameter $\qtrue$ lies in a smooth Sobolev space $Q\subseteq U$, satisfying the Gelfand triple
\[
    Q \hookrightarrow H \hookrightarrow Q^*.
\]
The time-dependent state $\utrue$ is viewed as an element of a Sobolev-Bochner space $\Uc$, which is assumed to be contained in $L^2(\I;U)\cap H^1(\I;U^*)$, a suitable function space setting for parabolic equations \cite{TCC21, Roubicek}. In particular, $\utrue(t)\in U$ and $D_t\utrue(t)\in U^*$ are both meaningful for a.e.~$t\in\I$, with $L^2$-regularity in time. Notation-wise, calligraphic notation indicates Sobolev-Bochner spaces with time dependence, while standard notation indicates Sobolev spaces on the spatial domain $\dom$.

As previously introduced, the nonlinear \emph{model} $f: H\times U\to U^*$ captures all terms that are zeroth order in time, including advection, diffusion and, importantly, nonlinearities with respect to $u$ and $q$. Moreover, it induces (by abuse of notation) a Nemytskii operator $f:L^2 ([0,\infty); H) \times L^2 ([0, \infty); U )\to L^2 ([0, \infty); U^* )$ by pointwise (in time) transformations on its function inputs. Throughout this work, we assume that the PDE \eqref{model} is uniquely solvable with respect to $\utrue\in\Uc$ at the true parameter $\qtrue\in Q$.

\subsection{Data assimilation via model reference adaptive system}
To derive an update law for the unknown parameter, we first extend the originally stationary parameter $\qtrue$ to be a 
constant function of time, i.e.~with zero time derivative.  This yields the \emph{equivalent model system} to \eqref{model}, that is,
\begin{equation}\label{model-q}
\begin{aligned}
D_t q(t) &=0\\
D_t \utrue(t) + f(q(t), \utrue(t)) & = g(t)   \qquad t>0\\
(q,\utrue)(0) & = (\qtrue,u_0).
\end{aligned}
\end{equation}
We then find an approximation $q$ in some time-smooth Bochner space $\Qc$, updated in an online-in-time fashion, such that its asymptotics approximates the true parameter. That is, 
\[
    \lim_{t\to\infty}q(t) - \qtrue = 0
\]
in the sense of Proposition \ref{prop-noisefree}.

\emph{Online identification} signifies that the parameter identification, the data collection process and the system operation are taking place simultaneously. During this joint process, the \emph{model reference system} employs the data $\uobs$  as in \eqref{data}
in order to obtain a prediction for $q$, which it uses to adapt the state $u$. It follows that the model reference system modifies the original model \eqref{model-q} in such a way that its dynamics are driven by the observation mismatch $u-z$ overtime, as:

\begin{itemize}
    \item the parameter equation is driven by the combination of model residuals and observation residuals with suitable stabilization, while
    \item the state equation is modified by feedback terms depending on the state residual and a suitable parameter-dependent operator.
\end{itemize}

This is possible because the state residual and parameter error satisfy a closed error system, and, under suitable structural conditions, can be shown to converge to the ground truths as $t\to\infty$. To this end, we proposed in \cite{Tram_Kaltenbacher_2021} the \emph{reference model adaptive} system
\begin{boxmras}\vspace{-5pt}
\begin{equation}\label{mras}
\begin{aligned}
D_tq + \sigma[D_t \uobs+f(q,\uobs)-g]-f'_q(\tilde{q},\uobs)^*(u-\uobs) & = 0, \quad \sigma=\{0,1\}, \\
D_tu + f(q,\uobs) + \cC(\|q\|_H)(u-\uobs) & = g,  \\
(q,u)(0) & = (q_0,u_0).
\end{aligned}
\end{equation}
\end{boxmras}

In the MRAS \eqref{mras}, at each time point, $\cC(\|q\|_H)\in \Lc(U, U^*)$ is chosen to be a linear operator that depends on the norm of the parameter $q$; see Assumption \ref{ass-1}. We note that in \eqref{mras} and in much of what follows, we avoid explicitly stating various time dependences, e.g.~$\|q\|_H$ is henceforth understood to mean the a.e.~defined map $t\in\I\mapsto \|q(t)\|_H\in\R^+$. Clearly, the MRAS \eqref{mras} is equivalent to the system \eqref{model-q} if the data $\uobs$ is identical to $u$ and if the approximate parameter $q$ is equal to the exact parameter $\qtrue$.

The initial condition $q_0$, which is known, acts as an initial guess for $\qtrue$, while the scalar $\sigma\in\{0,1\}$ is a switching parameter, which is set to zero if $f$ is linear with respect to $q$ and is set to one otherwise. The reference parameter $\tilde{q}\in\Qc$ is any fixed point in which the model $f$ is Gâteaux differentiable, in the sense that
$
    f'_q(\tilde{q},\uobs)\in L^\infty(\I;\Lc(H,U^*))
$.

A key aspect of the MRAS \eqref{mras} is that $q$ no longer needs to be time-constant, but is rather elevated to a fully time-dependent variable in a Sobolev-Bochner space $\Qc$, similarly to how the time-dependent state $u$ is viewed as an element of $\Uc$. Naturally, the choice of function space settings with sufficient regularity in time and space is essential for well-posedness of the MRAS \eqref{mras}.

\begin{remark}\label{rem:mras}
We remark that although originally, the model $f$ in \eqref{model} was in general nonlinear in $u$, the MRAS \eqref{mras} exchanges $f(q,u)$ with $f(q,z)$. Thus, the nonlinearity now acts on the data $z$, rather than on the state $u$. As $\cC(\|q\|_H)(u-\uobs)$ is (affine) linear in $u$, the MRAS \eqref{mras} is significantly easier to solve for $u$ in practice. This highlights the benefits of the MRAS when handling equations that are nonlinear in the state, which are common in practice, e.g.~reaction-diffusion PDEs.

On the other hand, the parameter dependence of the MRAS remains nonlinear if the original $f$ is nonlinear in $q$. Handling equations that are nonlinear with respect to the parameter can be highly challenging, and is notably less explored than for linear parameter laws. This capacity represents the novelty of our MRAS compared to existing methods of online parameter identification.
\end{remark}

 \subsection{Asymptotic convergence of the MRAS}

The MRAS \eqref{mras} is phrased as an update law, coupling the state equation with a parameter equation, driven by data $z$ that is fed to the system over time, i.e.
\[
    (q(t), u(t)) = \mathrm{MRAS}(q_0 , z(t)),
\qquad t>0
\]
obtaining regularized estimates $u(t)$ of $\utrue(t)$  at each time step $t\in\I$, and simultaneously updating the time-dependent estimate $q(t)$ of $\qtrue$. Indeed, in \cite{Tram_Kaltenbacher_2021}, we have proven unique solvability of the MRAS \eqref{mras}, and moreover demonstrated the convergence
\[
\|q(t)-\qtrue\| \stackrel{t\to\infty}{\to} 0 
\quad \text{and} \quad
\|u(t)-\utrue(t)\| \stackrel{t\to\infty}{\to} 0
\]
under suitable assumptions \cite[Assumption 1]{Tram_Kaltenbacher_2021}. We here present an adaptation of these assumptions and the accompanying convergence result, tailored to the numerical settings studied in this work.  These assumptions will be verified for each of the examples studied in the upcoming sections. Denote by $\langle\cdot,\cdot\rangle$ the dual paring between dual spaces and  by $C_{X\to Y}$ the norm of the continuous Sobolev embedding $X\embed Y$.

\begin{assumption}\quad\label{ass-1} 
\begin{enumerate}[label=(A\arabic*)]
\item \label{A-lip}
 The Gâteaux derivative   $f'_q(\tilde{q},\uobs)\in L^\infty(\I;\Lc(H,U^*))$ at the reference parameter $\tilde{q}$ satisfies
\begin{equation}
\|f(q,\uobs)-f(\qtrue,\uobs)- f'_q(\tilde{q},\uobs)(q-\qtrue)\|_{U^*} \leq
L(\|q\|_H)\|q-\qtrue\|_H
\end{equation}
a.e.~in $\I$ for all $q\in H$, with some monotonically increasing function $L:\I\to\I$. 

\item \label{A-coe}
There exists a constant $\Ccoe>0$ such that for all $q\in Q$ and a.e.~in $\I$, 
\[\lan f(q,\uobs)-f(\qtrue,\uobs),q-\qtrue\ran_{Q^*,Q} \geq \Ccoe\|q-\qtrue\|_H^2.\] 
\item \label{A-funcC}
The map $\cC:[0,\infty)\to \Lc(U,U^*)$ is chosen such that \label{A-C} for given $q\in H$ and all $v, w\in U$,
\begin{align}
\lan \cC(\|q\|_H)v,v \ran_{U^*,U}&\geq\left( \frac{L(\|q\|_H)^2}{2\Ccoe}+M\right)\|v\|_U^2 =:\Mtil(\|q\|_H)\|v\|_U^2
\label{cCCcoeM}\\
\lan \cC(\|q\|_H)v,w \ran_{U^*,U}&\leq\Ntil(\|q\|_H)\|v\|_U\|w\|_U
\end{align}
for some constant $M>0$ and some monotonically increasing function $\Ntil:\I\to\I$.
\end{enumerate}
\end{assumption}

\begin{proposition}[Convergence -- clean data]\label{prop-noisefree}
Let Assumption \ref{ass-1} be fulfilled. Then the following statements on the reconstructed parameter $q$, the state $u$ and the corresponding errors $\errpar:=q-\qtrue$, $\errst:=u-\utrue$ hold true:

\begin{enumerate}[label=(\roman*)]
\item
$u\in \Uc=L^2(\I;U)\cap H^1(\I;U^*)\cap L^\infty(\I;H),$ \\
$q \in \Qc=L^2(\I;H)\cap H^1(\I;Q^*)\cap L^\infty(\I;H).$

\item
\begin{equation}\label{prop-noisefree:eq:clean-conv-t_sup}
\begin{split}
\sup_{t\geq 0} & \left[ \|\errst(t)\|_H^2 + \|\errpar(t)\|_H^2 \right]
+ \Ccoe\int_0^\infty\|\errpar(s)\|^2_H\d s \\
& + 2M\int_0^\infty\|\errst(s)\|_U^2\d s
\leq \left[ \|\errst(0)\|_H^2+\|\errpar(0)\|_H^2 \right]. 
\end{split}
\end{equation}

\item For all $t\geq0$,
\begin{equation}\label{prop-noisefree:eq:clean-conv-t}
\begin{split}
\left[ \|\errst(t)\|_H^2+\|\errpar(t)\|_H^2 \right]\leq \exp\big(-C't\big)\left[ \|\errst(0)\|_H^2+\|\errpar(0)\|_H^2 \right]
\end{split}
\end{equation}
with $C':=\min\left\{\Ccoe;2MC_{U\to H}\right\}>0$.
\end{enumerate}
\end{proposition}

\begin{proof}
    \cite[Proposition 2.1]{Tram_Kaltenbacher_2021} 
\end{proof}

In essence, \cite{Tram_Kaltenbacher_2021} established existence, uniqueness and regularity of solutions $(q,u)$ to the MRAS \eqref{mras} using pseudomonotonicity techniques. We moreover derived error equations that allow us to prove convergence of both the state and parameter estimates as time tends to infinity, with explicit exponential decay rates under structural coercivity and Lipschitz assumptions. 

\subsection{MRAS for noisy data} 
We now turn our attention to the practically relevant case of noisy data $z^\delta\in L^p(\I;Z)$, $U\subset Z$ in place of $z=u^\dagger$. The noise level $\delta$ is measured in the data space norm
$
\|z^\delta-z\|_{L^p(\I;Z)}\leq \delta\,.
$
Since $z^\delta$ does not satisfy the regularity requirements of the MRAS \eqref{mras}, especially with regards to time and space derivatives, we smooth $z^\delta$ by filtering or local averaging.  More precisely, we introduce regularizing operators \[\cR^\sp:Z\to U \text{ (pointwise in time)}, \quad\cR^\ti:L^p(\I;Z)\to W^{1,p}(\I;H)\]  such that by an appropriate choice of the regularization parameters contained in the definition of $\cR^\sp$, $\cR^\ti$, the smoothing error estimates
\begin{equation}\label{deltatil}
\|\cR^\sp(z^\delta(t))-z(t)\|_{U}\leq \deltil^\sp(t) \,, \quad 
\|D_t(\cR^\ti(z^\delta)-z)\|_{L^p(\I;H)}\leq \deltil^\ti
\end{equation}
hold for some $p\geq2$. Inserting the smoothed data in place of $z$ into \eqref{mras}, we obtain the MRAS
for $z^\delta$ as
\begin{equation}
\label{mras-noise}
\begin{aligned}
 D_tq + \sigma[D_t \cR^\ti(z^\delta)+f(q,\cR^\sp(z^\delta))-g]-f'_q(\tilde{q},\cR^\sp(z^\delta))^*(u-\cR^\sp(z^\delta)) = 0&  
\\
 D_tu + f(q,\cR^\sp(z^\delta)) + \cC(\|q\|_H)(u-\cR^\sp(z^\delta)) = g&  
\\
 (q,u)(0) = (q_0,u_0)&
\,,
\end{aligned}
\end{equation}
recalling that  $\tilde{q}$ is a reference parameter at which $f$ is differentiable.
We refer to  \cite[Proposition 2.3]{Tram_Kaltenbacher_2021} for convergence of the regularized reconstruction.

In this result, in place of the estimate \eqref{prop-noisefree:eq:clean-conv-t}, one obtains, for all $t\geq0$,
\begin{equation}\label{prop:noise-conv-t}
\begin{split}
\left[ \|\errst(t)\|_H^2+\|\errpar(t)\|_H^2 \right] &\leq \exp\big(-\omega t\big)\left[ \|\errst(0)\|_H^2+\|\errpar(0)\|_H^2 \right]\\
&\qquad + C \Bigl(\|\deltil^\sp\|_{L^p(0,t)}^2 + (\deltil^\ti)^2\Bigr).
\end{split}
\end{equation}
for any $\omega < \min\{\Ccoe,2MC_{U\to H}\}$ and with some $C>0$. Notably, this estimate depends on the smoothing error bounds.

\section{MRAS discretization and update rule} \label{sec::femdiscr}

To advance the work carried out in \cite{Tram_Kaltenbacher_2021}, we now explore the numerical techniques required to carry out the MRAS \eqref{mras} in a variety of settings, including linear and nonlinear PDEs. This detailed numerical study, supported by explicit, equation-specific expressions and weak forms of the MRAS, will form the main contribution of this work. In the interest of brevity, we present the MRAS for the noise free case, i.e., \eqref{mras}. The noisy case \eqref{mras-noise} is obtained by simply inserting smoothed data in place of clean data into \eqref{mras}.

In order to carry out these numerical studies for the MRAS \eqref{mras}, where a PDE state $u$ is solved jointly with the parameter $q$, we first decide on an overall discretization scheme. Our strategy will be to employ a continuous Galerkin finite-element-method (FEM) in space combined with semi-implicit Euler stepping in time. 

\paragraph{Finite-element discretization} 
For space discretization, we use the FEM mesh-generation capabilities provided by the \texttt{NGSolve} \cite{ngsolve} Python package. We decompose $\dom$ into disjoint triangular elements with maximal diameter $h_\mathrm{max}$ to be specified later.
The pair $(u(t), q(t))$ will be reconstructed in finite element subspaces of $H^1(\dom)\times L^2(\dom)$; details are presented in each example. We thus define the discrete state space
$$U_h = \{u \in H^1(\dom) \mid u_{|E_i} \in P^3(E_i) \,\forall i\in I\} \subset U=H^1(\dom)$$ and parameter space $$H_h = \{q \in L^2(\dom) \mid q_{|E_i} \in P^0(E_i) \, \forall i\in I\} \subset H=L^2(\dom).$$
Here, $I$ is an index set, and for each $k\in\N_0$, $i\in I$, we let $P^k(E_i)$ denote the polynomial space of degree $k$ on the element $E_i$, where $\dom=\dot\cup_{i\in I} E_i$. In general, the MRAS does not prescribe explicit choices of polynomial degree $k$ or the size or nature of the decomposition $\{E_i\}_{i\in I}$. Instead, these should be chosen in accordance with the expected smoothness of PDE state, data and ground truth parameter. Throughout, we have chosen $k=3$ for the state discretization, being sufficiently high to describe the moderately low-frequency states encountered. Similarly, we have consistently chosen $k=0$ for the parameter discretization, as this reasonably approximates our chosen ground truth parameters and is computationally efficient.

\paragraph{Semi-implicit Euler time stepping}
We consider  a semi-implicit approach to discretize the state  equation \eqref{mras}  in each finite time domain  $[0, T ]$, $T>0$, with equidistant time points $\Delta t$ apart. More precisely, the state equation $D_t u + f(q,z) + \cC(\|q\|_H)(u-\uobs)=g$ in the MRAS \eqref{mras}, written in discretized form, reads as
\begin{equation}\label{generaltimestepping-u}
     \frac{u_{n+1}-u_n}{\Delta t} + f(q_{n+1};q_n, z_{n+1})+ \cC(\|q_{n}\|_H)(u_{n+1}-\uobs_{n+1})=g_{n+1},
\end{equation}
where $u_{n+1}$ is the state at the $(n+1)$-th time step driven by data $z_{n+1}$ and the right hand side $g_{n+1}$. 

Above, the parameter $q$ is treated semi-implicitly. As $\cC(\|q_{n}\|_H)$  is a linear bounded operator with dependence in $q$ via the parameter norm, it was found significantly more feasible to treat this dependence explicitly. Recalling that the state equation in the MRAS is linear in $u$, one similarly  notes that \eqref{generaltimestepping-u} is a linear equation in $u_{n+1}$. 

Regarding nonlinearity in $f(q,z)$, numerical FEM solvers would in general require an internal iterative solver (e.g., fixed-point or Newton-type methods) to handle nonlinearities. Another approach is reformulating the nonlinear term $f(q,z)$ into a term that is linear in the current step  and nonlinear in the previous step, analogous to the treatment of nonlinear convection in the Navier-Stokes equations \cite{pei2023semiimplicitdlnalgorithmnavier} or nonlinear rotation in the Landau-Lifshitz-Gilbert equation \cite{bartelsprohl, alouges} (see also \cite{NguyenWald:2022}).  Strongly inspired by these well-established solvers, we introduce a semi-implicit strategy; that is,
the nonlinearity in $q$ is counteracted by reformulating the $f$-term into $f(q_{n+1};q_{n,} z_{n+1})$, which will be designed to be linear in $q_{n+1}$ and nonlinear in $q_n$. Explicit examples of these reformulations will be given in later sections. This splitting facilitates the use of a linear solver, and, in general, semi-explicit schemes are preferable for nonlinear equations such as Navier-Stokes and reaction-diffusion equations \cite{zhao2024efficientboundpreservingasymptotic}.

Moving forward, the parameter equation in the MRAS \eqref{mras} may be nonlinear in $q$, as is the case for the nonlinear potential problem and the modified Allen-Cahn equation; see Table \ref{table:overview}. 
Thus, we also employ a semi-implicit strategy in the
parameter update step. That is,
\begin{align}\label{generaltimestepping-c}
 \frac{q_{n+1}-q_n}{\Delta t} + \sigma\Big(D_t \uobs_{n+1}+f(q_{n+1};q_{n},\uobs_{n+1})-g_{n+1}\Big)-f'_q({\tilde{q}},\uobs_n)^*(u_n-\uobs_n) =0
\end{align}
with $f'_q({\tilde{q}},\uobs_n)^*(u_n-\uobs_n) $ treated explicitly in $u$.  Recall that for equations that are linear in $q$, one has $\sigma=0$. Thus, \eqref{generaltimestepping-c} simplifies to an explicit scheme for $q$. 

As \eqref{generaltimestepping-c} does not depend on $u_{n+1}$, the above scheme could be realized as an alternating scheme. In this manner, one could view the MRAS as a prediction-correction procedure, first predicting the parameter $q_{n+1}$ based on the current approximate state $u_n$, then correcting the state $u_{n+1}$ based on the prediction $q_{n+1}$. This cycle is perpetuated until the discrepancy between the estimated $u_{n+1}$ and data $z_{n+1}$ is sufficiently small. We illustrate this notion in Figure \ref{dia:mras}.

\paragraph{Assimilation window} It is worth emphasizing that in this manner, with the semi-implicit Euler scheme, the data assimilation window is $\Delta t$, meaning one time block prior to the current time. A one-step window is a natural choice for data assimilation. Higher-order time-stepping schemes, such as Runge–Kutta methods, would correspond to longer assimilation windows and would result in a more complex discretized MRAS formulation. The effect of a longer assimilation window is an important topic for future work.

\paragraph{Update and preconditioner}  
Combining \eqref{generaltimestepping-u}-\eqref{generaltimestepping-c}, we construct the discretized MRAS system
\begin{equation}\label{generaltimestepping}
\begin{gathered}
\begin{pmatrix}
    q_{n+1}\\u_{n+1}
\end{pmatrix}
+\Delta t
\begin{pmatrix}
\sigma (D_tz_{n+1}+f(q_{n+1};q_n, z_{n+1}) - g_{n+1}) - f'_q(\tilde{q},z_n)^*(u_n-z_n)\\
f(q_{n+1;q_n},z_{n+1}) + C\bigl(\|q_{n}\|_H\bigr)\,(u_{n+1}-z_{n+1})
\end{pmatrix}\\
= 
\begin{pmatrix}
q_n\\u_n +\Delta t g_{n+1}
\end{pmatrix}.
\end{gathered}
\end{equation}
Writing the weak form of \eqref{generaltimestepping} lays the foundation for our FEM solver, allowing us to assemble the joint system matrix and linear form in \texttt{NGSolve}.

The system matrix is now solved for in each time step with the standard CG-solver by means of \texttt{NGSolve}'s \texttt{bvp} functionality \cite{bvp}. At each time step, this functionality applies a prescribed preconditioner to improve the condition number of the system matrix before solving it by an inner Conjugate-Gradient (CG) solver. More specifically, we employ a Jacobi preconditioner to accelerate the default CG solver \cite{BadriPreconditioner}
 to the left to the system matrix. The Jacobi preconditioner modifies a general linear system $Ax=f$ to the form $\mathrm{diag}(A)^{-1}Ax=\mathrm{diag}(A)^{-1}y$, where $\mathrm{diag}(A)$ is the matrix containing only the diagonal entries of $A$. The Jacobi preconditioner is a natural choice; while generally not the \enquote{optimal} preconditioner in most senses, its ease of implementation and ability to improve conditioning is well known in the context of local finite elements \cite{Wathen2015}. It counteracts uneven diagonal weighting and reasonably approximating the inverse of diagonal-dominated sparse matrices. This is then followed by the CG solver yielding the update $(q_{n+1},u_{n+1})$. All code is made public on \cite{Code}.

Before closing this section, we remark that the idea of jointly reconstructing state and parameter can be also found in inverse problems that are formulated in an \emph{all-at-once} setting \cite{BurgerMuehlhuberIP,BurgerMuehlhuberSINUM,HaAs01,KKV14b,LeHe16}; for time-dependent inverse problems, we refer to \cite{nguyen19,aao16,KNSW,KaNg2022,holler22learning_parameter_id}. 

With discretization scheme, update scheme and solution methods now established, we are ready to investigate several physical examples; these are summarized in Table \ref{table:overview}. For each example, a complete picture will be presented: the explicit form of the MRAS \eqref{mras} for the given problem, verification of necessary conditions, discretized forms of the MRAS and detailed numerical experiments.

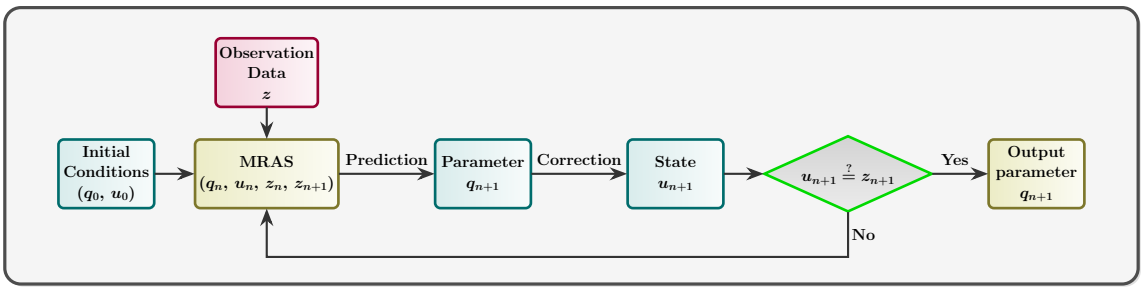
\begin{figure}[htbp]
    \centering \hspace*{-1.2cm} 
    \begin{tikzpicture}[
    scale=0.35, transform shape,
    node distance=1.2cm and 2.1cm,
    every node/.style={font=\LARGE\bfseries, text=black},
    box/.style={
        draw, rectangle, minimum width=3.6cm, minimum height=2.6cm,
        align=center, rounded corners=2pt, line width=1pt
    },
        blue/.style={box, draw=teal!85!black, fill=teal!8, shading=axis, shading angle=225, left color=teal!15, right color=teal!3},
        green/.style={box, inner sep=4pt, draw=olive!75!black, fill=olive!8, shading=axis, shading angle=135, left color=olive!15, right color=olive!3},
        orange/.style={box, draw=purple!80!black, fill=purple!10, shading=axis, shading angle=45, left color=purple!18, right color=purple!4},
        decision/.style={draw, diamond, aspect=2.2, inner sep=4pt, fill=gray!25, draw=green!85!black, line width=1pt, shading=axis, shading angle=90, top color=gray!35, bottom color=gray!15},
        framebox/.style={rectangle, rounded corners=6pt, line width=1.2pt, draw=gray!60!black, fill=gray!8, drop shadow={opacity=0.25, shadow xshift=1pt, shadow yshift=-1pt}},
        arrow/.style={thick, black!80, -Stealth}
    ]
\node[blue] (init) {Initial\\Conditions\\$\boldsymbol{(q_0,\,u_0)}$};
\node[green, right=1.5cm of init] (mras) {MRAS\\$\boldsymbol{(q_n,\,u_n,\,z_n,\, z_{n+1})}$};
\node[blue, right=3.6cm of mras] (param) {Parameter\\$\boldsymbol{q_{n+1}}$};
\node[blue, right=3.6cm of param] (state) {State\\$\boldsymbol{u_{n+1}}$};
\node[orange, above=of mras] (obs) {Observation\\Data\\$\boldsymbol{z}$};
\node[decision, right=1.5cm of state] (dec) {$\boldsymbol{u_{n+1}\stackrel{?}{=}z_{n+1}}$};
\node[green, right=of dec] (final) {Output\\ parameter\\$\boldsymbol{q_{n+1}}$};
\begin{scope}[on background layer]
    \node[framebox, fit=(init.west) (final.east) (obs.north) (state.south), 
          inner sep=7mm, xshift=0pt, yshift=-3mm] {};
\end{scope}
    \draw[arrow] (init) -- (mras);
    \draw[arrow] (mras) -- node[above, yshift=2mm]{Prediction} (param);
    \draw[arrow] (param) -- node[above, yshift=2mm]{Correction} (state);
    \draw[arrow] (obs) -- (mras);
    \draw[arrow] (state) -- (dec);
    \draw[arrow] (dec) -- node[above, yshift=2mm, xshift=-2mm]{Yes} (final);
    \draw[arrow] (dec.south) -- node[right]{No} ++(0,-1.7cm) -| (mras.south);
    \end{tikzpicture}
    \caption{MRAS workflow for dynamic update laws}
    \label{dia:mras}
\end{figure}

\section{Darcy flow: the linear $a$-problem}\label{sec::darcy}

As our first case study, we consider the Darcy flow  with homogeneous Dirichlet boundary and unknown spatially dependent diffusion $a$ defined over the unit square. That is,
\begin{equation}\label{darcy}
\begin{aligned}
D_t u -\nabla\cdot(a\nabla u) &= g \qquad \text{in }I\times\dom:=\I\times(0,1)^2, \\
u|_{\partial\dom} &= 0 \qquad \text{in }I, \\
u(t=0)&=u_0 \hspace{15pt}\text{in }\dom.
\end{aligned}
\end{equation}
This equation and its variations have been widely used for modeling processes such as elasticity \cite{landau1986}, subsurface pressure and water filtration \cite{Hubbert1957}.

\subsection{MRAS analysis}

We begin our analysis by deriving all the components of the MRAS \eqref{mras} for the equation \eqref{darcy}, while verifying Assumption \ref{ass-1} in the appropriate function space setting.

\begin{proposition}\label{prop:darcy}
For the Darcy problem \eqref{darcy} with unknown diffusion $a$ and data $\uobs$, the MRAS \eqref{mras} takes the form
\begin{equation}\label{mras-darcy}
\begin{aligned}
D_t a & =\nabla \uobs\cdot\nabla(u-\uobs), \\
D_tu - \Delta(u-\uobs) & = g + \nabla\cdot(a\nabla \uobs), \\
(a,u)(0) & = (a_0,u_0),
\end{aligned}
\end{equation}
with state space $U:=H^1_0(\dom)$ and parameter space $H:=L^2(\dom)$. Assumptions \ref{A-lip}, \ref{A-funcC} hold.
\end{proposition}

\begin{proof} First of all, as the equation \eqref{darcy} is linear in the parameter $a$, the MRAS \eqref{mras} has $\sigma=0$. We now detail all the non-vanishing components. The derivative of the model $f$ with respect to $a$ and the corresponding Banach space adjoint, respectively, are
\begin{align*}
    &f(a,\uobs):=-\nabla\cdot(a\nabla \uobs),\qquad f'_a(\tilde{a},\uobs)h = -\nabla\cdot(h\nabla \uobs),\\
    &\langle f'_a(\tilde{a},\uobs)h, u-\uobs\rangle = -\int_0^\infty\int_\dom \nabla\cdot(h\nabla \uobs)(u-\uobs)\d x\d t \\
    &=\int_0^\infty\int_\dom h\nabla \uobs\cdot\nabla(u-\uobs)\d x\d t
    = \langle h, \nabla \uobs\cdot\nabla(u-\uobs)\rangle =: \langle h, f'_a(\tilde{a},\uobs)^*(u-\uobs)\rangle
\end{align*}
for any $h\in \Qc$ and any $u,\uobs\in\Uc$, employing integration-by-parts and the homogeneous Dirichlet boundary of $u-\uobs$. This yields the first equation of \eqref{mras-darcy}, describing the update rule for the parameter.

Regarding the state equation, we have in Assumption \ref{A-lip} the Lipschitz constant $L^{\tilde{a},\uobs}=0$, as the model is linear in $a$. This implies in \ref{A-funcC} that the linear operator $\Cc(\|a\|_H)\in\Lc(U,U^*)$ takes the form
\[
 \Cc(\|a\|_H)v:=-\Delta v \,\,\,\Rightarrow\,\,\,   \lan \Cc(\|a\|_H) v,v\ran= \|v\|^2_U, \,\,\, \lan \Cc(\|a\|_H) v,w\ran\leq \|v\|_U\|w\|_U 
\]
for all $v, w\in U=H_0^1(\dom)$. We thereby have coercivity with $\Mtil=M=1$ and boundedness with $\Ntil=1$. Here, we use the equivalent norm $\|u\|_U:=\|\nabla u\|_{L^2}$ due to Poincar\'e-Friedrichs's inequality.
\end{proof}

\begin{remark}[Coercivity]\label{rem:coercive}
We remark that coercivity \ref{A-coe} will require to perturb the underlying  equation \eqref{darcy} by adding $u\Delta a$, as discussed in \cite[Section 3.2]{Tram_Kaltenbacher_2021}.  
However, numerical results empirically prove good convergence even without this modification given small initial data error; see Section \ref{sec::darcy}.
A complete verification of Assumption \ref{ass-1} for nonlinear problems, our main focus, will be presented in Section \ref{sec::nonlinear_cprob}. In Section \ref{sec::convergence_studies}, we shall see that coercivity guarantees stability of the reconstruction with respect to different initial guesses.

An alternative to coercivity is the more general condition of \emph{persistence of excitation} \cite{Baumeister-etal} (c.f~\cite[Remark 1.6.]{Tram_Kaltenbacher_2021}), whose consideration has so far been limited to linear problems.
\end{remark}

We now proceed with the weak formulation of the MRAS.

\begin{corollary}

The weak form of the MRAS \eqref{mras-darcy} in a semi-implicit Euler scheme for the unknowns $a$, $u$ takes the form
\begin{align}
        \int_\Omega a_{n+1}\,s\d x & = \int_\Omega a_{n}\,s\d x + \Delta t\int_\Omega \nabla \uobs_{n}\cdot\nabla(u_{n}-\uobs_{n})s \d x, \label{darcy-discrete-a}\\
\int_\Omega u_{n+1}\, v \d x
&+ \Delta t \int_\Omega\,\nabla u_{n+1}\cdot \nabla v\d x  + \Delta t\int_\Omega a_{n+1}\nabla  \uobs_{n+1}\cdot\nabla v \d x \nonumber\\
&= \int_\Omega u_n\, v\d x+\Delta t\int_\Omega g_{n+1}\, v\d x + \Delta t\int_\Omega \nabla \uobs_{n+1}\cdot \nabla v\d x, \label{darcy-discrete-u}\\
     (a,u)(0) & = (a_0,u_0).
\end{align}
for all $v\in H^1_0(\dom)$ and all $s\in L^2(\dom)$.

\end{corollary}
\begin{proof}
We begin by rearranging the MRAS \eqref{mras-darcy} as
\begin{equation*}
\begin{aligned}
D_t a  & = \nabla \uobs\cdot\nabla(u-\uobs)\\
D_tu -\Delta u-\nabla\cdot(a\nabla \uobs) & = g-\Delta \uobs\,.
\end{aligned}
\end{equation*}
We now follow the discretized MRAS outlined in \eqref{generaltimestepping}. For the parameter equation \eqref{darcy-discrete-a} at $a_{n+1}$, we treat the $u$-term explicitly via $u_n$, yielding the $\nabla \uobs_{n}\cdot\nabla(u_{n}-\uobs_{n})s$ part of the source term. For the state equation \eqref{darcy-discrete-u}, we treat $a$ implicitly, yielding the term $\nabla\cdot(a_{n+1}\nabla \uobs_{n+1})$, while on the right hand side, one has $g_{n+1}$, $z_{n+1}$.

Testing these two equations with, respectively, test functions $s\in H$ and $v\in U$, then using integration by parts while taking into account the zero boundary condition yields the claimed discrete form of the MRAS \eqref{mras-darcy}.
\end{proof}

\subsection{Numerical results}\label{numerical_darcy}

\paragraph{Implementation setup}

To discretize the 2D spatial domain $\Omega:=[0,1]^2$, we employ a FEM mesh with a maximal coarseness $h_\mathrm{max} = 0.04$. As established in Section \ref{sec::femdiscr}, the polynomial degree for the discretized state space $U_h\subset H^1_0(\Omega)$ is $k=3$ and for the discretized parameter space  $H_h \subset L^2(\Omega)$ is $k=0$, resulting in 6703 and 1456 degrees of freedom, respectively. We shall observe the evolution of the MRAS until the final time $T=5$, employing the semi-implicit scheme outlined in Section \ref{sec::femdiscr} with constant time step $\Delta t=0.001$. These parameters in summarised in Table \ref{table-Darcy}.

\begin{table}    
    \centering
    \begin{tabular}{c|c}
         Domain, mesh size & $\Omega=[0,1]^2$, $h_\mathrm{max}= 0.04$\\
         $\#$dofs for $U_h$, $\#$dofs for $Q_h$  &6703, 1456\\
        Max time, time step, \# step &$T=5$, $\Delta t=0.001$, 5000 steps\\
         Source term&  random field\\
    \end{tabular}
    \caption{Setup for the Darcy problem.}
    \label{table-Darcy}
\end{table}

\paragraph{Data preparation}

As ground truth parameter $a^\dagger$, we employ the PDE benchmark database \cite{pdebench}, available in hdf5 format, which contains a domain-wise constant function corresponding to materials with two distinct physical values; see Figure \ref{fig::darcy_evolution_big}.

At each time $t$, we fixed a random source $g(\cdot,t)$ by interpolating a function that takes i.i.d.~values distributed as $\mathcal{N}(0,1)$ on an equidistant $128\times 128$ grid in $[0,1]^2$ via \texttt{NGSolve}'s \texttt{VoxelCoefficient} functionality.

The exact time-dependent solution $u^\dagger$ was computed through implicit Euler time-stepping. To avoid \emph{inverse crime}, this was carried out on a finer mesh with $h_\mathrm{max}=0.03$, as opposed to the coarser mesh with $h_\mathrm{max}=0.04$ used for the reconstruction, and with the higher polynomial degree $k=4$. In contrast, the data $z$ used by the MRAS was the result of interpolating $u^\dagger$ to the coarser state space $U_h$ discussed above.

We set $a(0) = a_0$, which we understand as an initial guess of the exact parameter $a^\dagger$, as the indicator function on the ball around $(0.5, 0.5)$ with radius $0.42$
\[
    a_0(x) := \begin{cases}
        1, & x\in B_{0.42}((0.5,0.5)), \\
        0, & \text{else.}
    \end{cases}
\]

\paragraph{Numerical results}
Figure \ref{fig::darcy_evolution_big} displays the evolution of the state and parameter output by the MRAS. Figure \ref{fig::darcy_evolution_big} clearly displays the convergence of the MRAS approximate parameter $a$ towards the truth $a^\dagger$. It is, however, worth to note that the sharp discontinuity in the ground truth $a^\dagger$ is not perfectly reconstructed. This can be thought to be caused by the setting of $H=L^2(\Omega)$ for the reconstruction, as required by the MRAS analysis. With penalization in e.g.~TV-norm \cite{BrediesHoller20}, which is known for its ability to preserve sharp discontinuities, it is credible that reconstruction could be improved. However, this would require significant changes to the MRAS analysis, which are out of scope for the current work.

The state reconstruction is a by-product of the MRAS, rather than a separate objective. Nevertheless, one observes that despite the fact that the initial state $u_0=u^\dagger_{|t=0}$ is a random field, the state $u(t)$ evolves to the true state $u^\dagger(T)$ at final time. In terms of data assimilation, correction to the predicted parameter is no longer required when the computed state $u(T)$ matches the data $z(T)$ up to a tolerance. Hence, this justifies $T$ as the termination time for the MRAS. 

\begin{figure}
    \centering
    \begin{minipage}[t]{0.1\textwidth}
    \vspace{9ex}t=0\\[11ex]t=0.01\\[12ex]t=0.075\\[11ex] t=0.1
    \end{minipage}
    \begin{minipage}[t]{0.16\textwidth}
        \centering
        State\\[2ex]
        \includegraphics[width=1.2\linewidth,trim={10cm 0cm 10cm 10cm},clip]{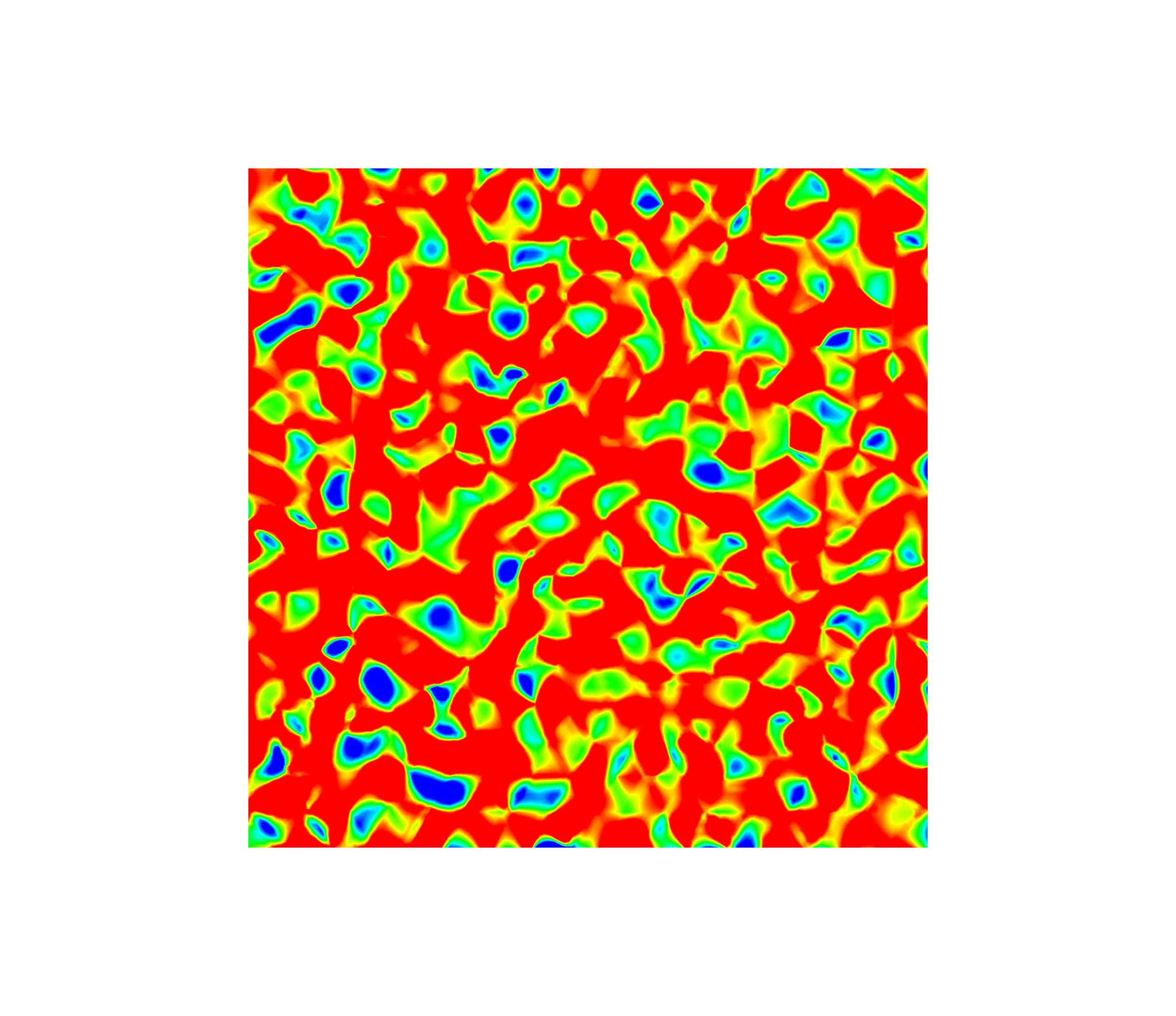}\\[0cm]
        \includegraphics[width=1.2\linewidth,trim={10cm 0cm 10cm 10cm},clip]{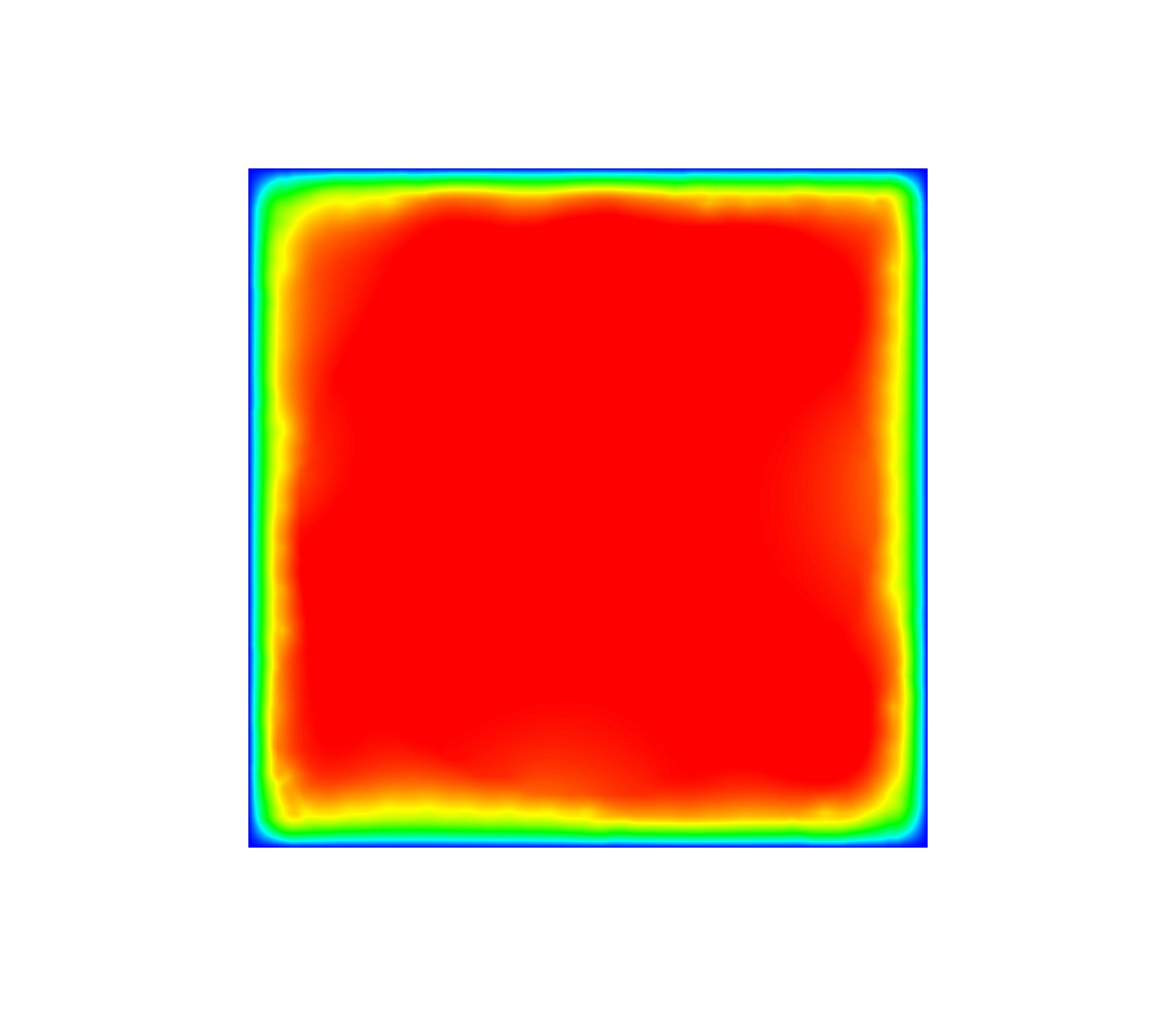}\\[0cm]
        \includegraphics[width=1.2\linewidth,trim={10cm 0cm 10cm 10cm},clip]{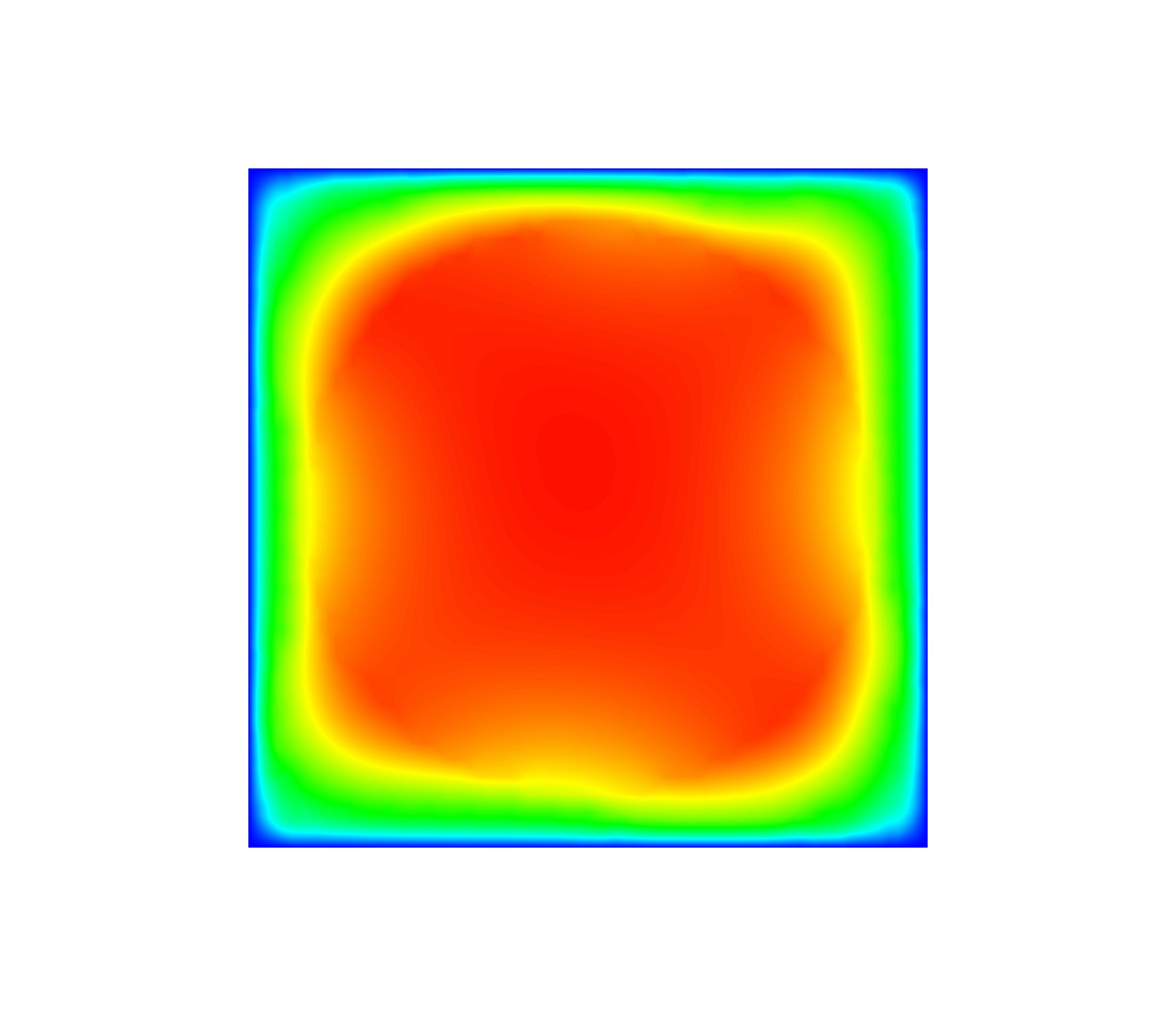}\\[0cm]
        \includegraphics[width=1.2\linewidth,trim={10cm 0cm 10cm 10cm},clip]{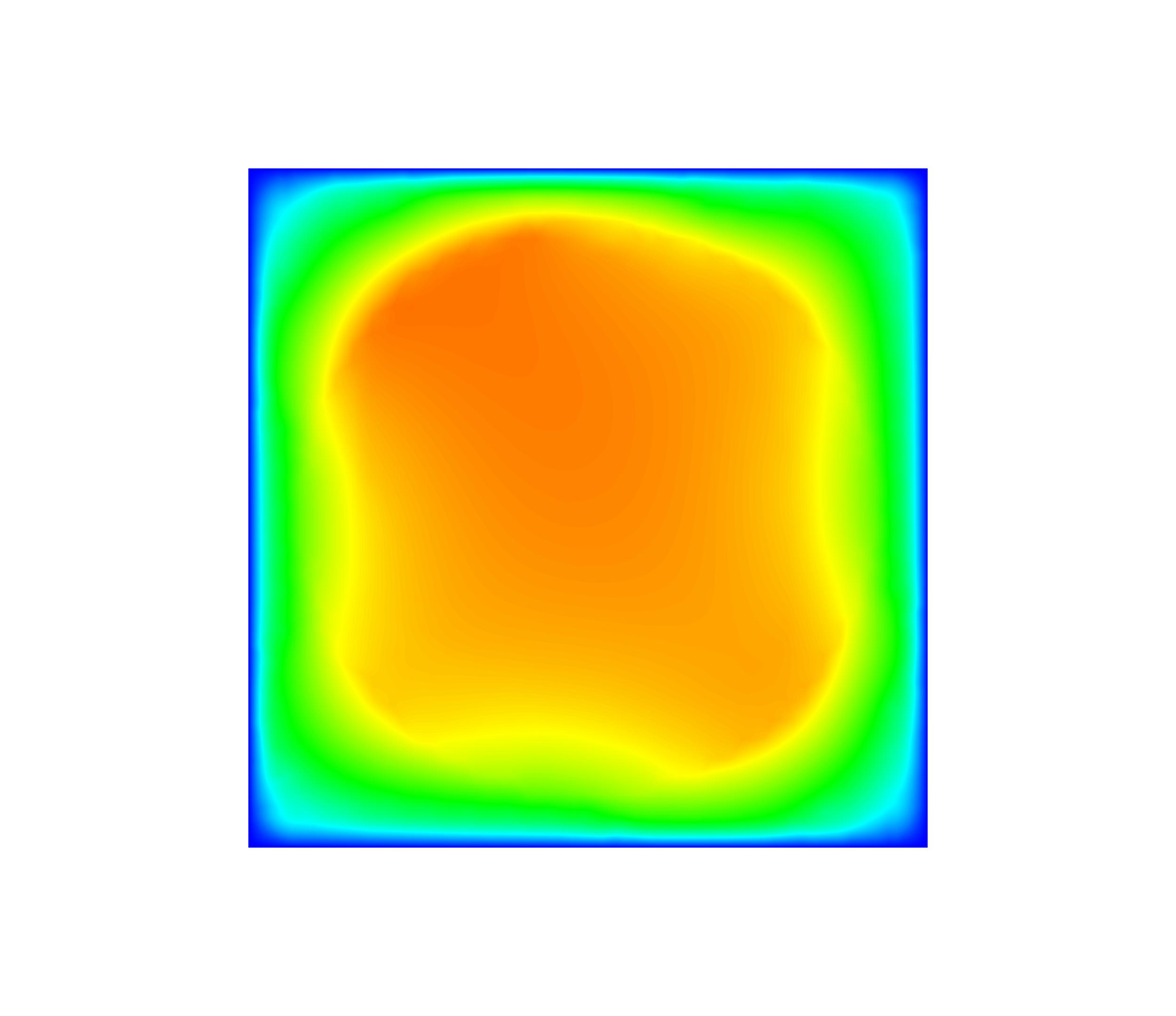}
    \end{minipage}%
    \hfill
    \begin{minipage}[t]{0.16\textwidth}
        \centering
       Parameter\\[2ex]
       \includegraphics[width=1.2\linewidth,trim={10cm 0cm 10cm 10cm},clip]{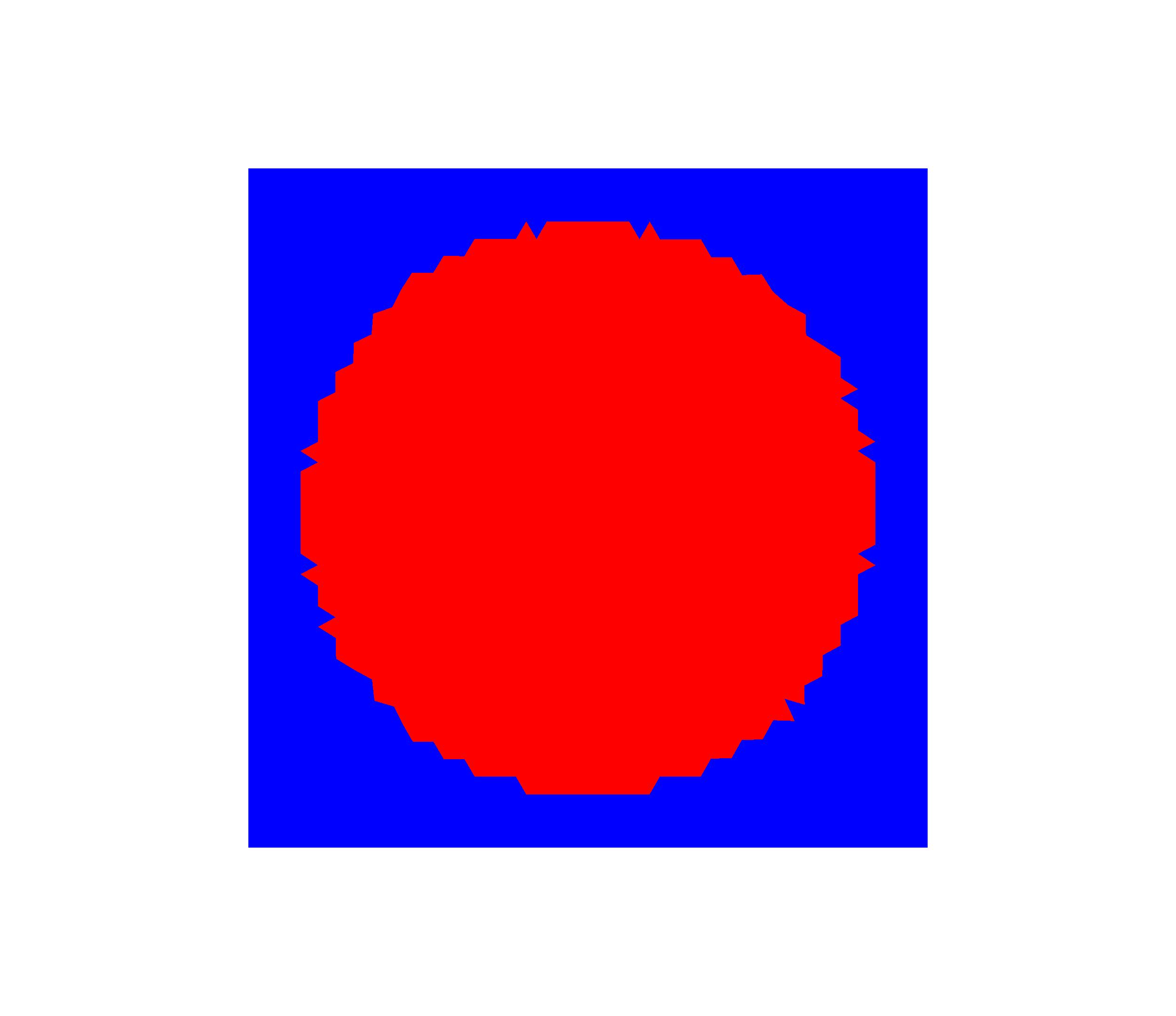}\\[0cm]
        \includegraphics[width=1.2\linewidth,trim={10cm 0cm 10cm 10cm},clip]{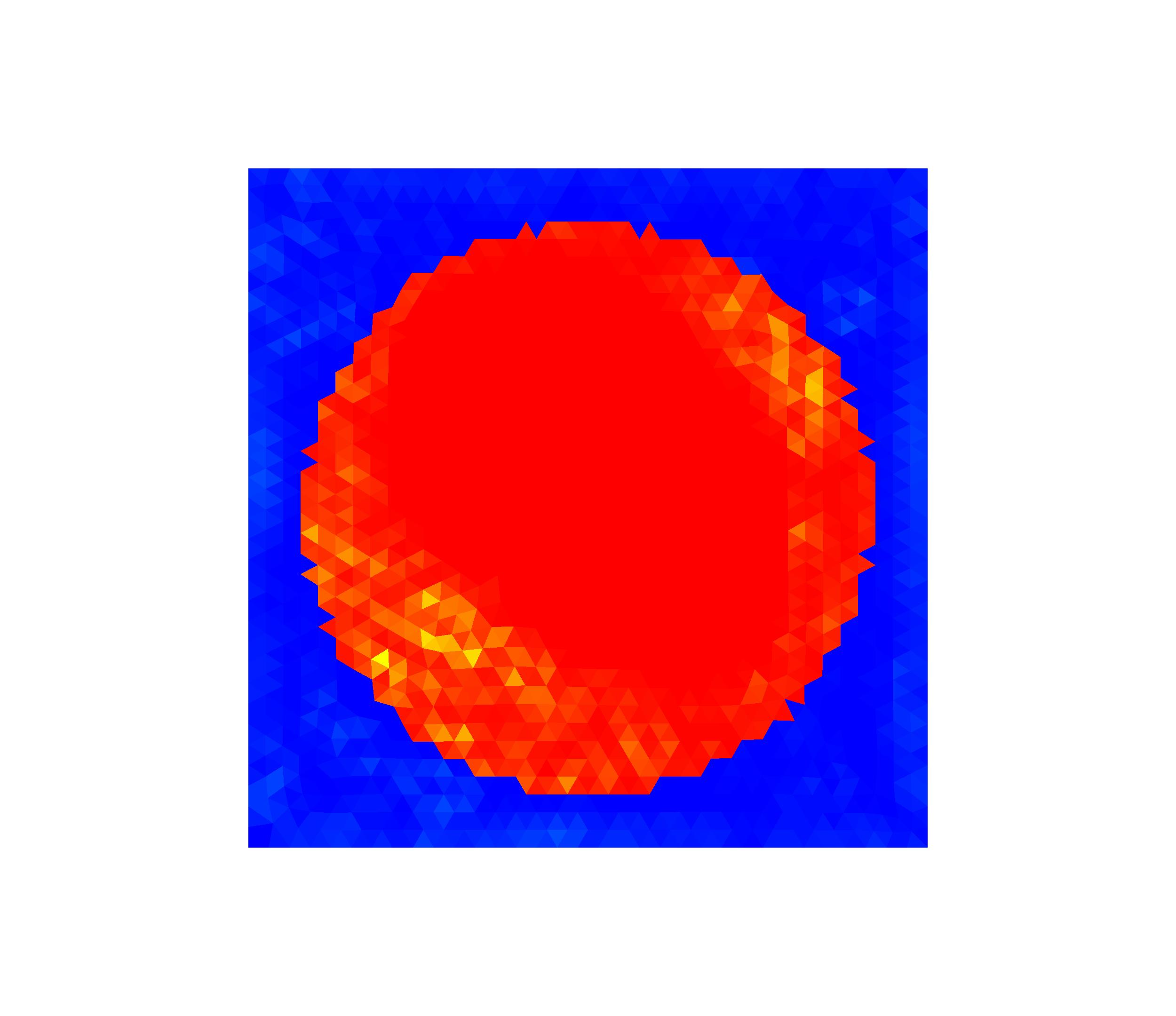}\\[0cm]
        \includegraphics[width=1.2\linewidth,trim={10cm 0cm 10cm 10cm},clip]{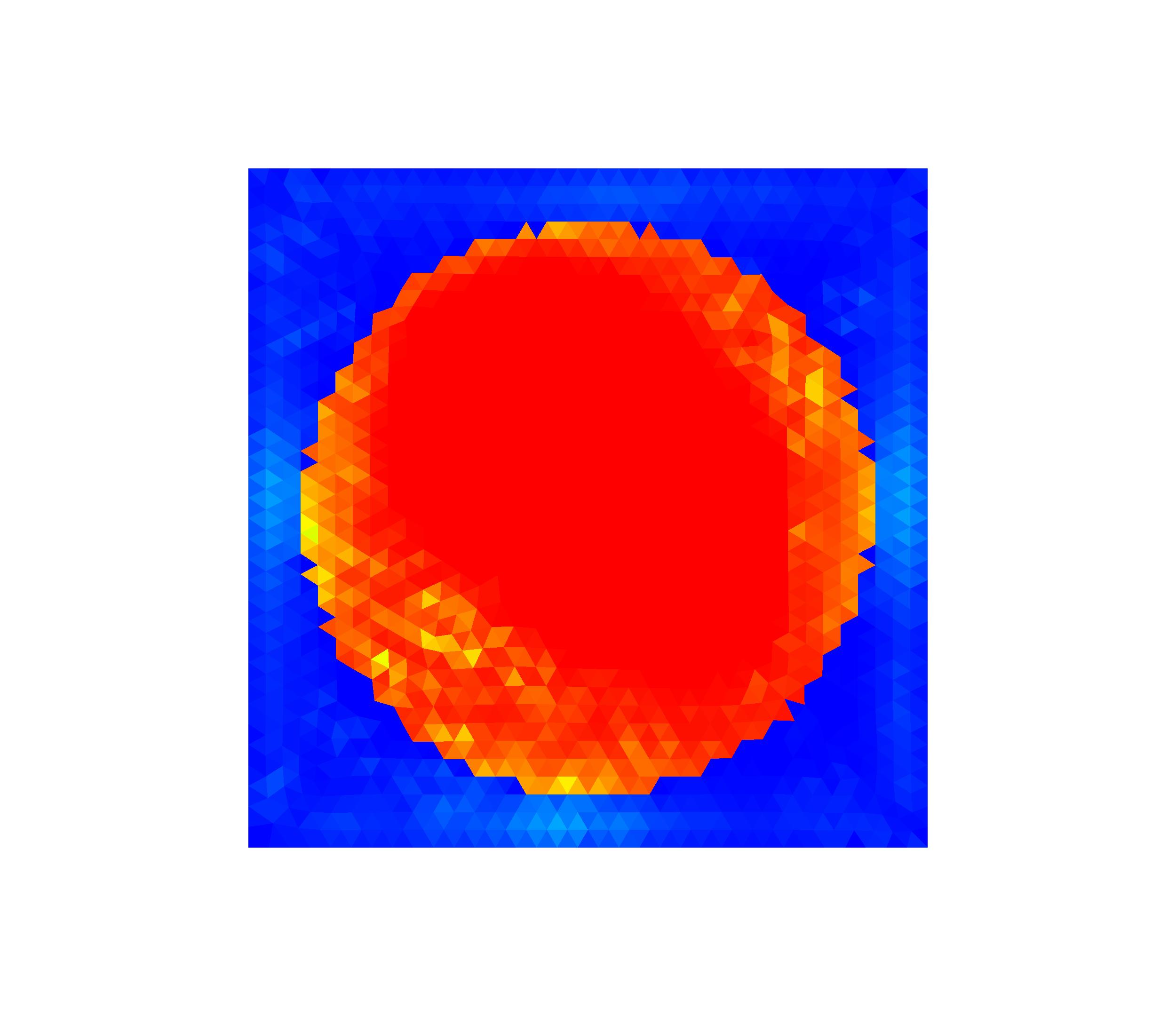}\\[0cm]
        \includegraphics[width=1.2\linewidth,trim={10cm 0cm 10cm 10cm},clip]{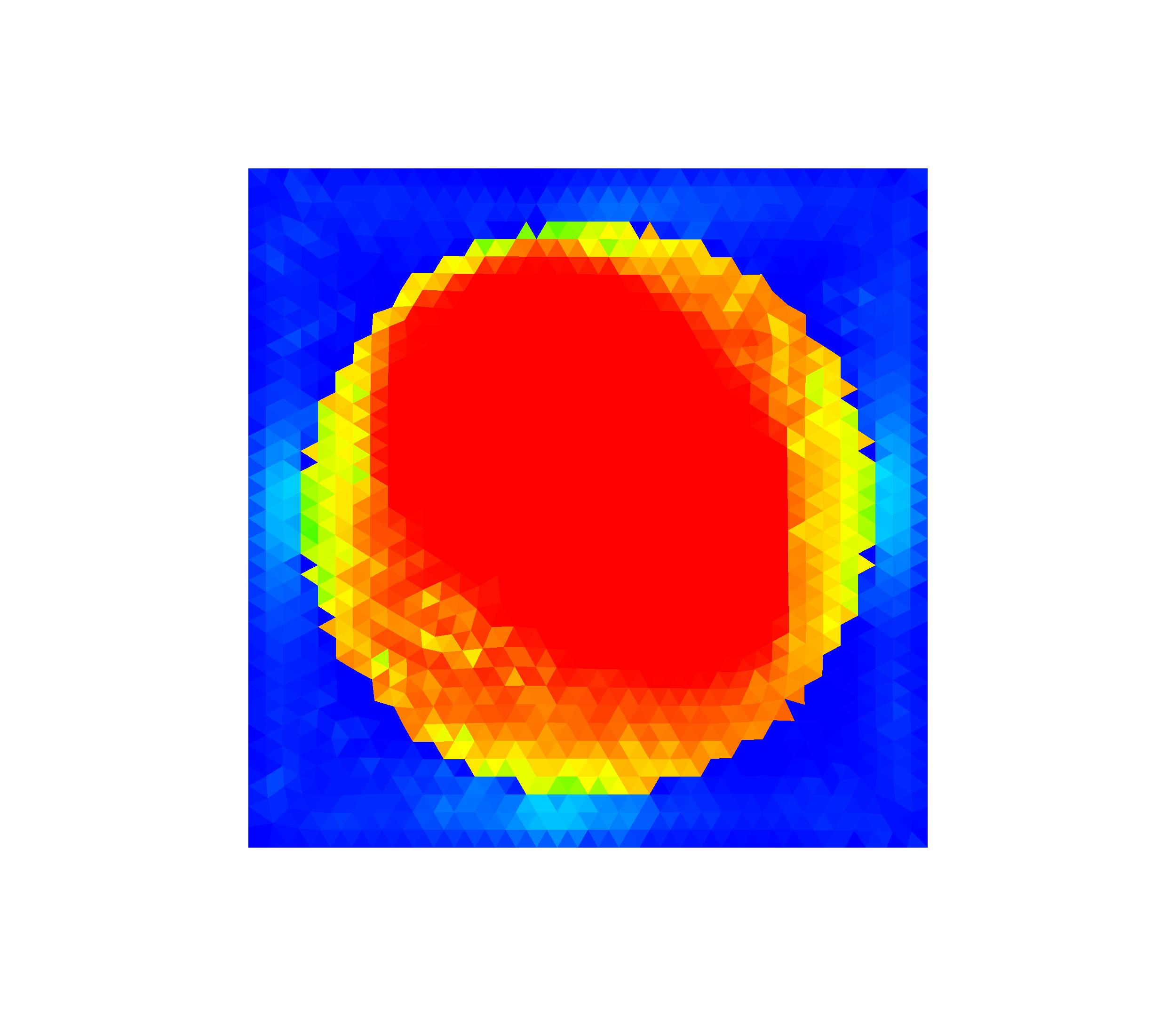}
    \end{minipage}%
    \hfill
     \begin{minipage}[t]{0.08\textwidth}
    \vspace{9ex}t=0.2\\[11ex]t=0.5\\[12ex]t=1.5\\[11ex] t=5
    \end{minipage}
    \begin{minipage}[t]{0.16\textwidth}
        \centering
        \phantom{Evolution of $u$}\\[2ex]
        \includegraphics[width=1.2\linewidth,trim={10cm 0cm 10cm 10cm},clip]{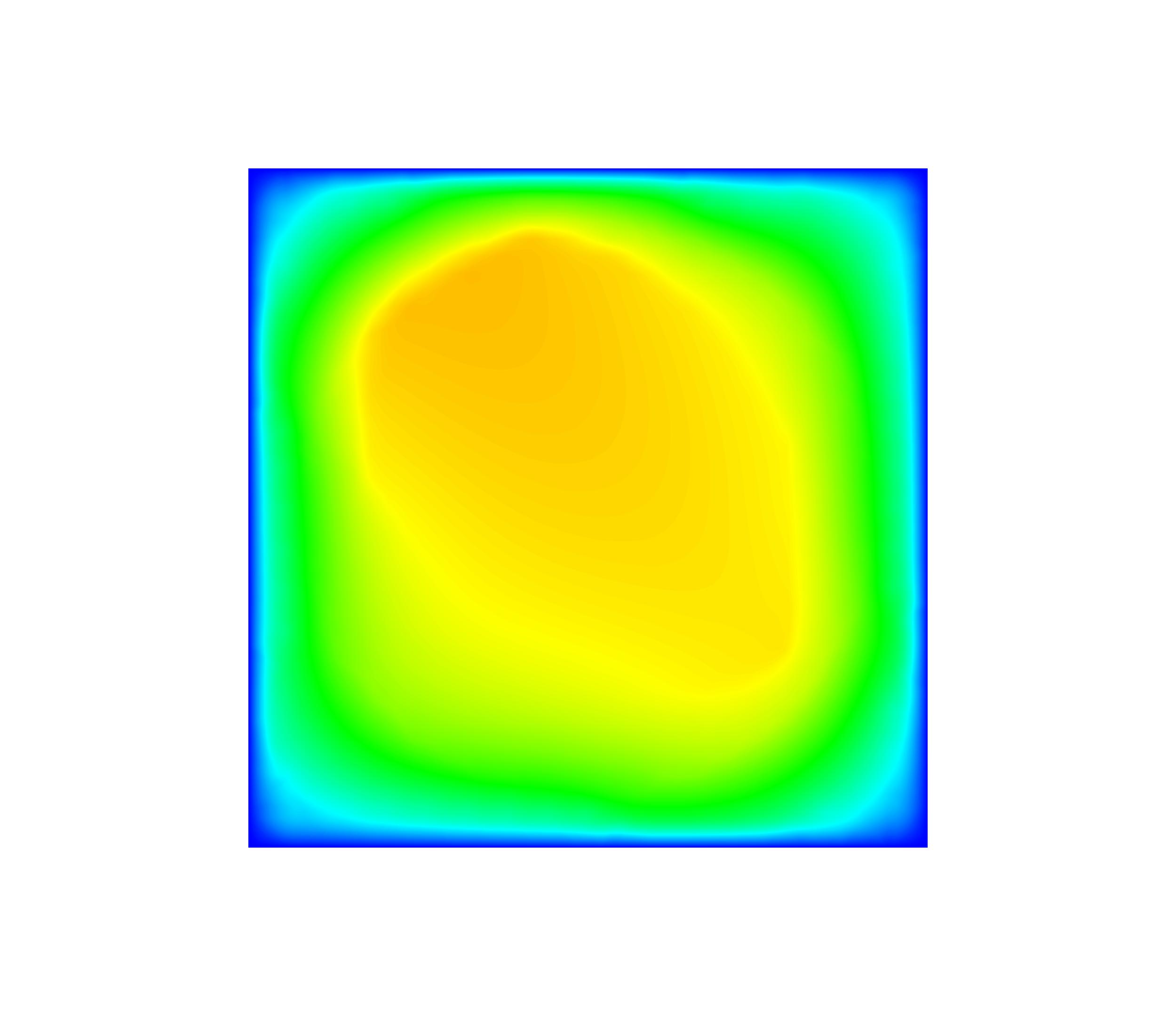}\\[0cm]
        \includegraphics[width=1.2\linewidth,trim={10cm 0cm 10cm 10cm},clip]{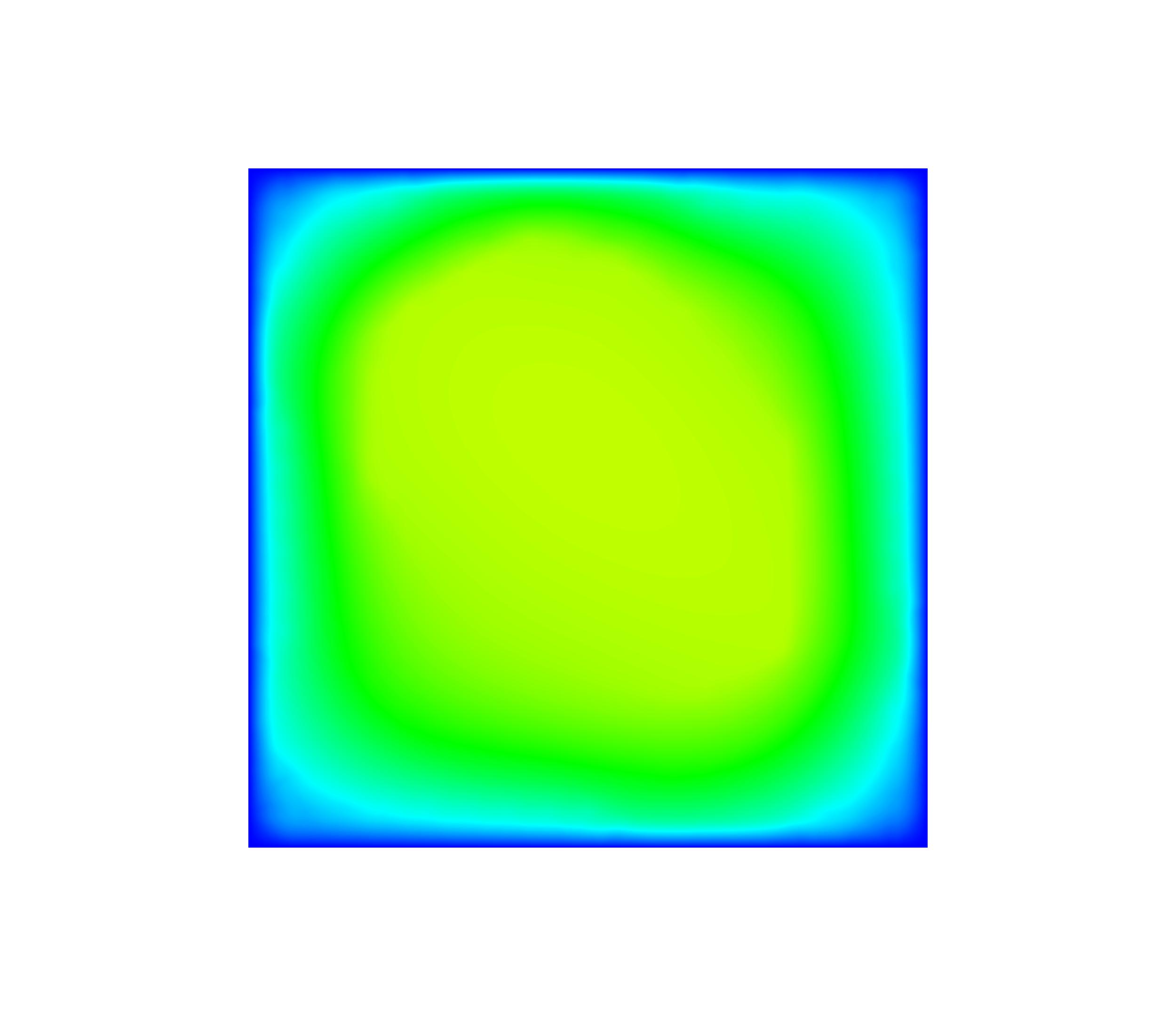}\\[0cm]
        \includegraphics[width=1.2\linewidth,trim={10cm 0cm 10cm 10cm},clip]{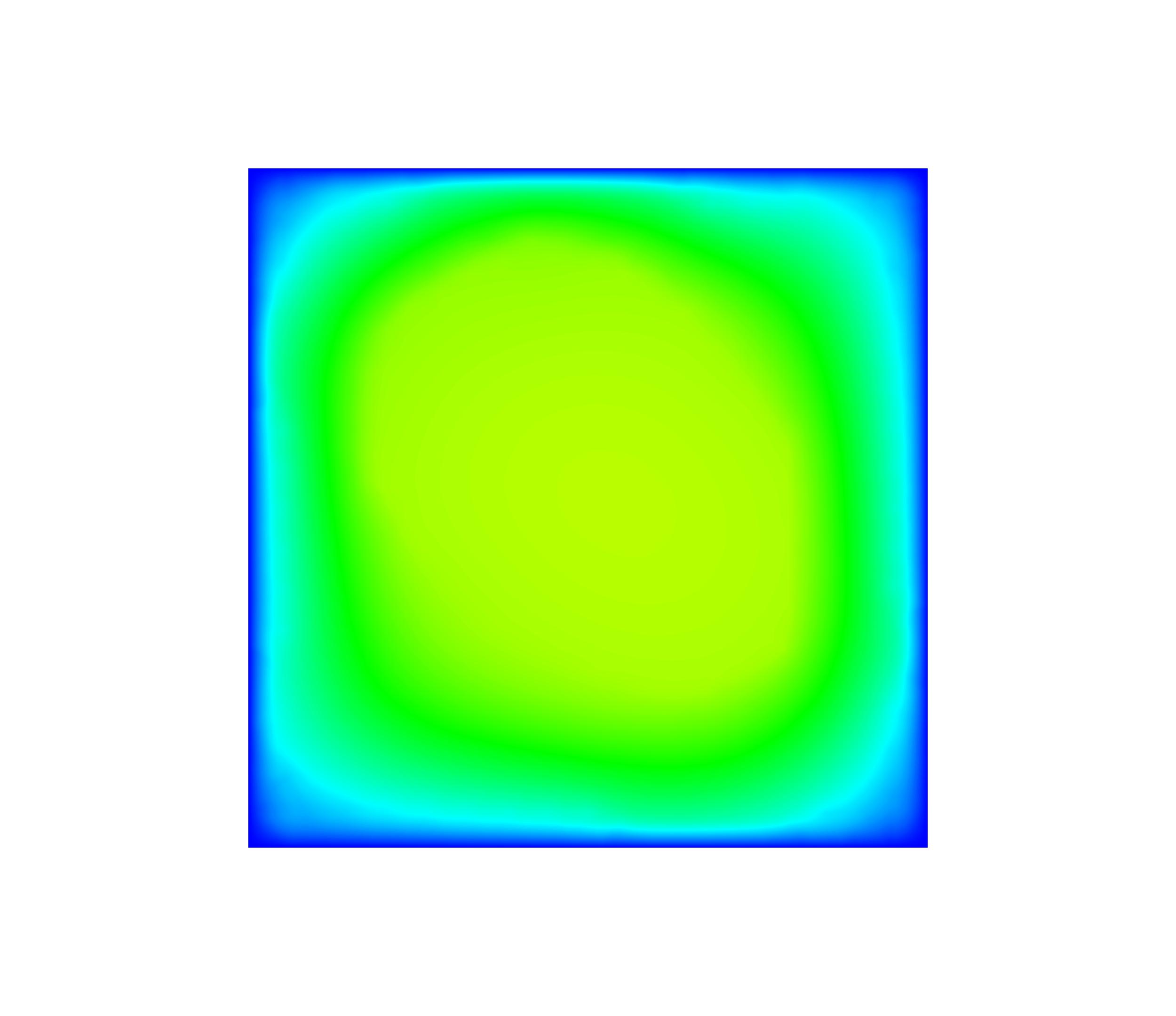}\\[0cm]
        \includegraphics[width=1.2\linewidth,trim={10cm 0cm 10cm 10cm},clip]{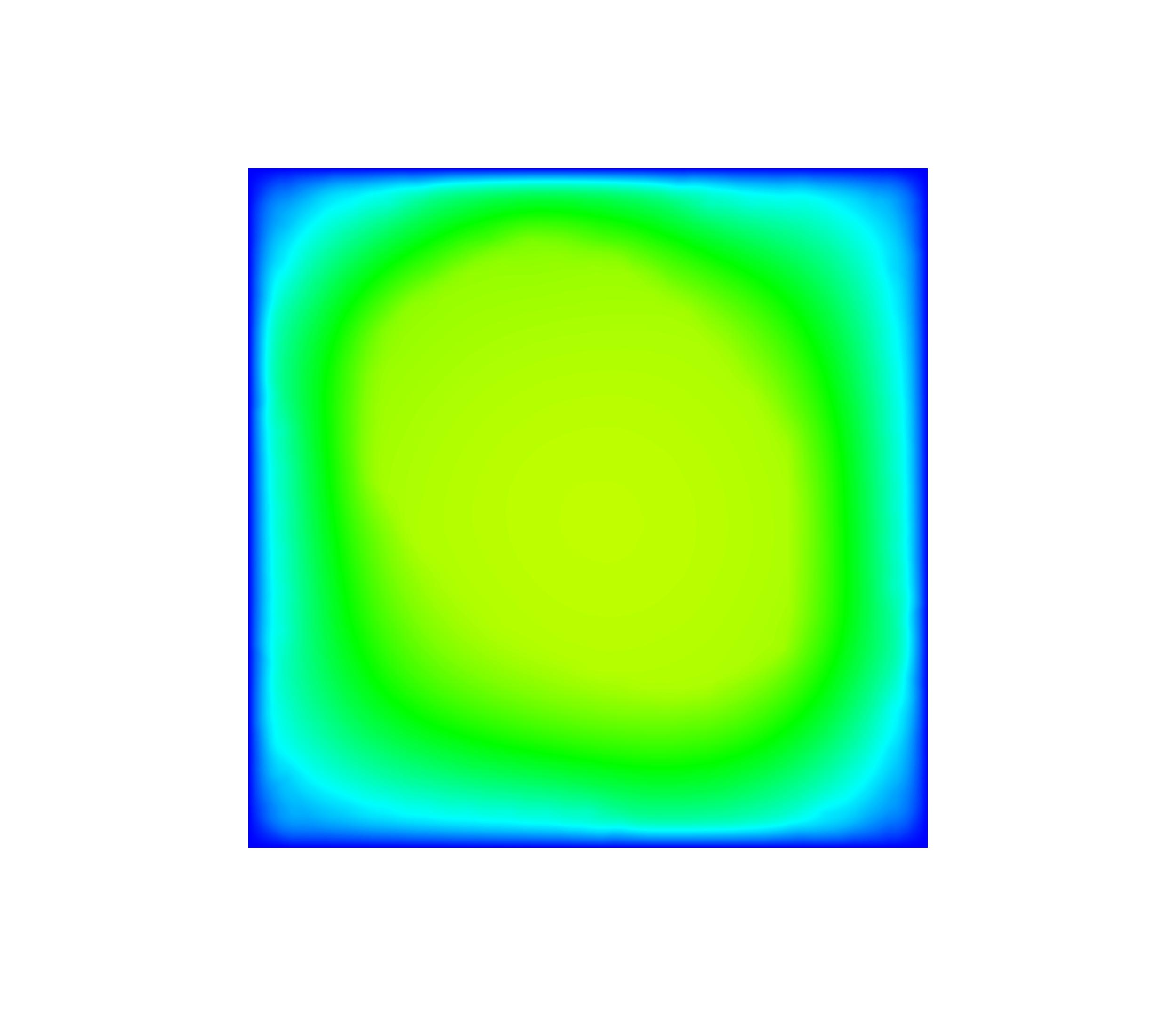}
    \end{minipage}%
    \hfill
    \begin{minipage}[t]{0.16\textwidth}
        \centering
        \phantom{Evolution of $q$}\\[2ex]
        \includegraphics[width=1.2\linewidth,trim={10cm 0cm 10cm 10cm},clip]{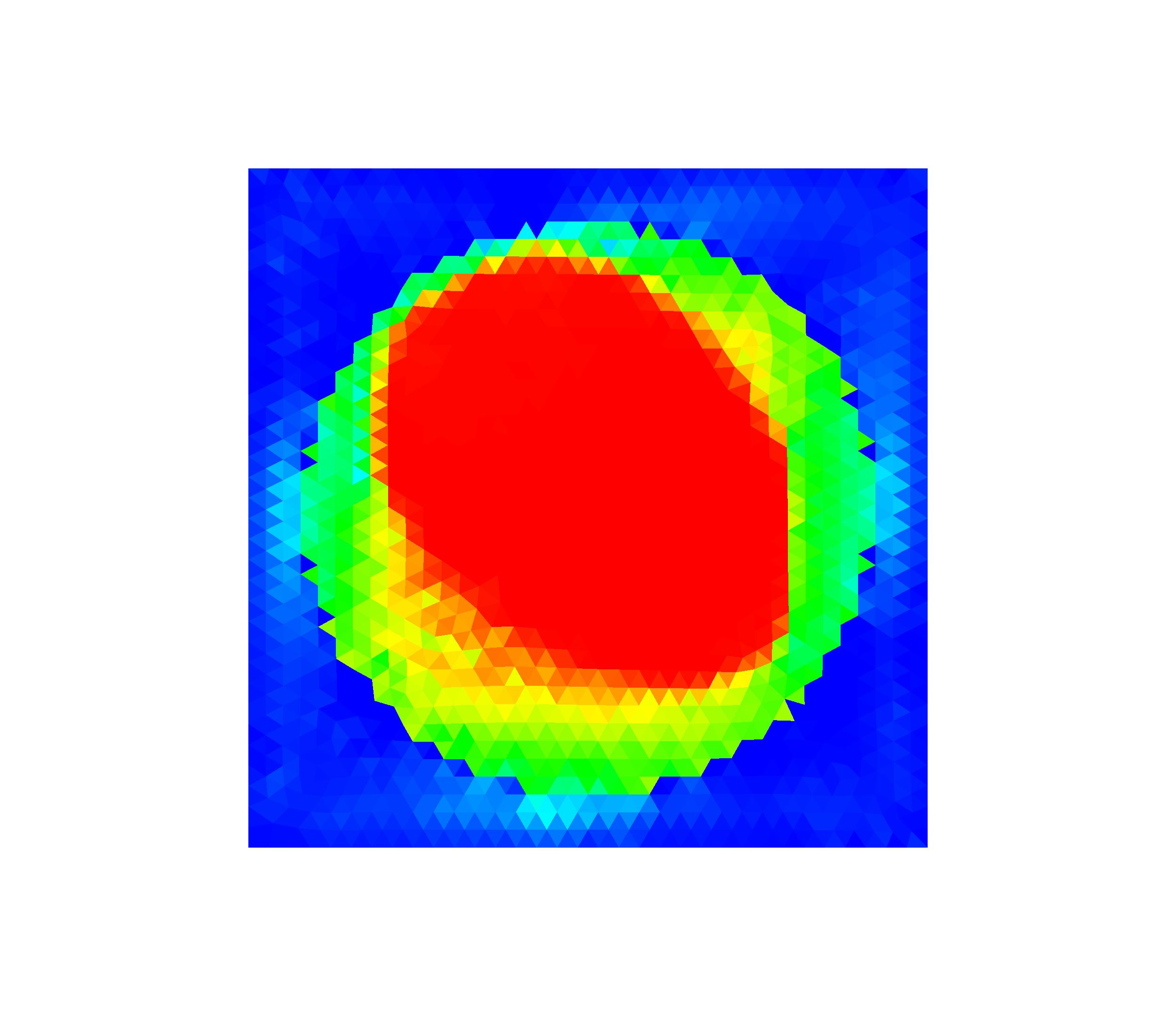}\\[0cm]
        \includegraphics[width=1.2\linewidth,trim={10cm 0cm 10cm 10cm},clip]{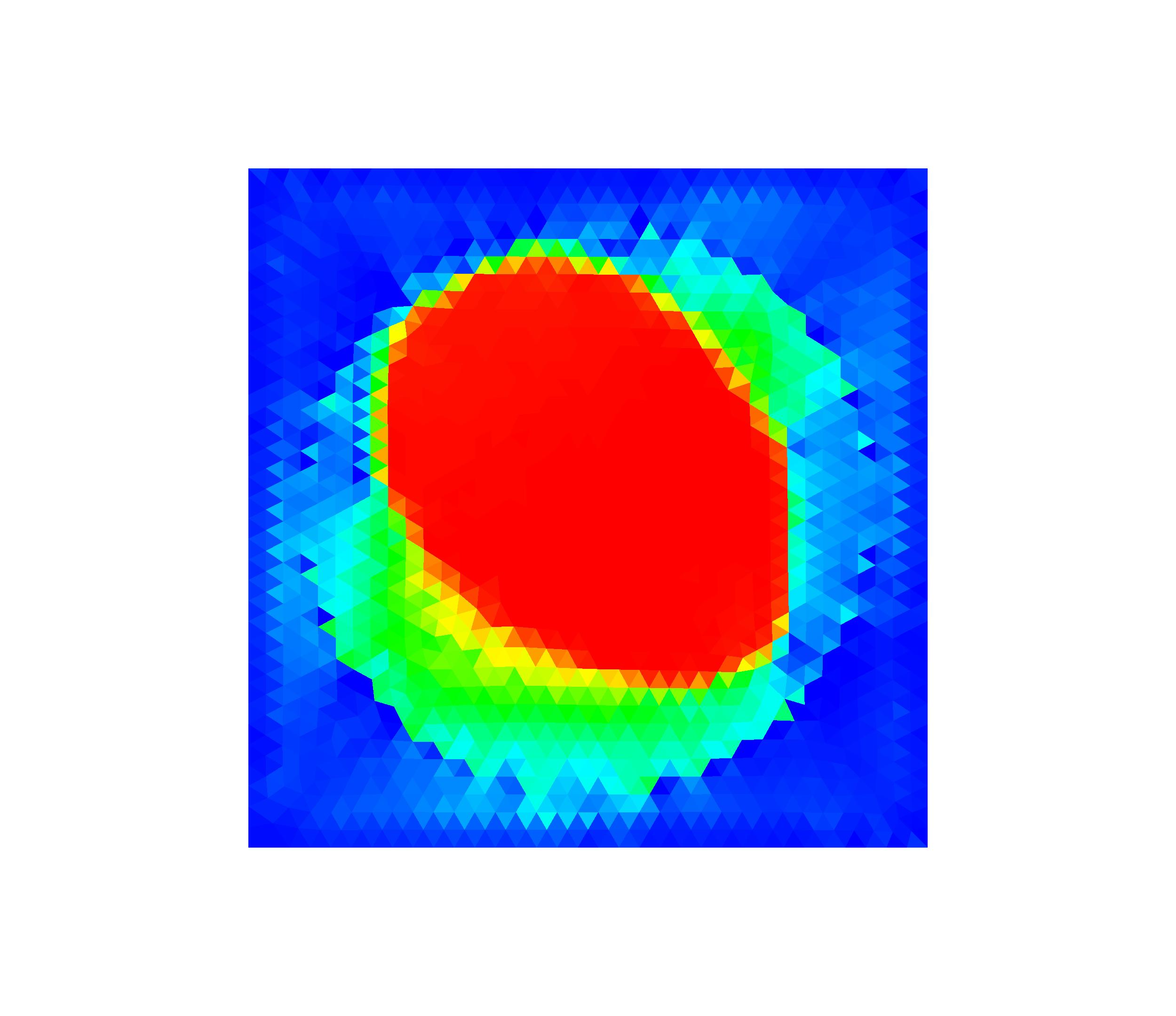}\\[0cm]
        \includegraphics[width=1.2\linewidth,trim={10cm 0cm 10cm 10cm},clip]{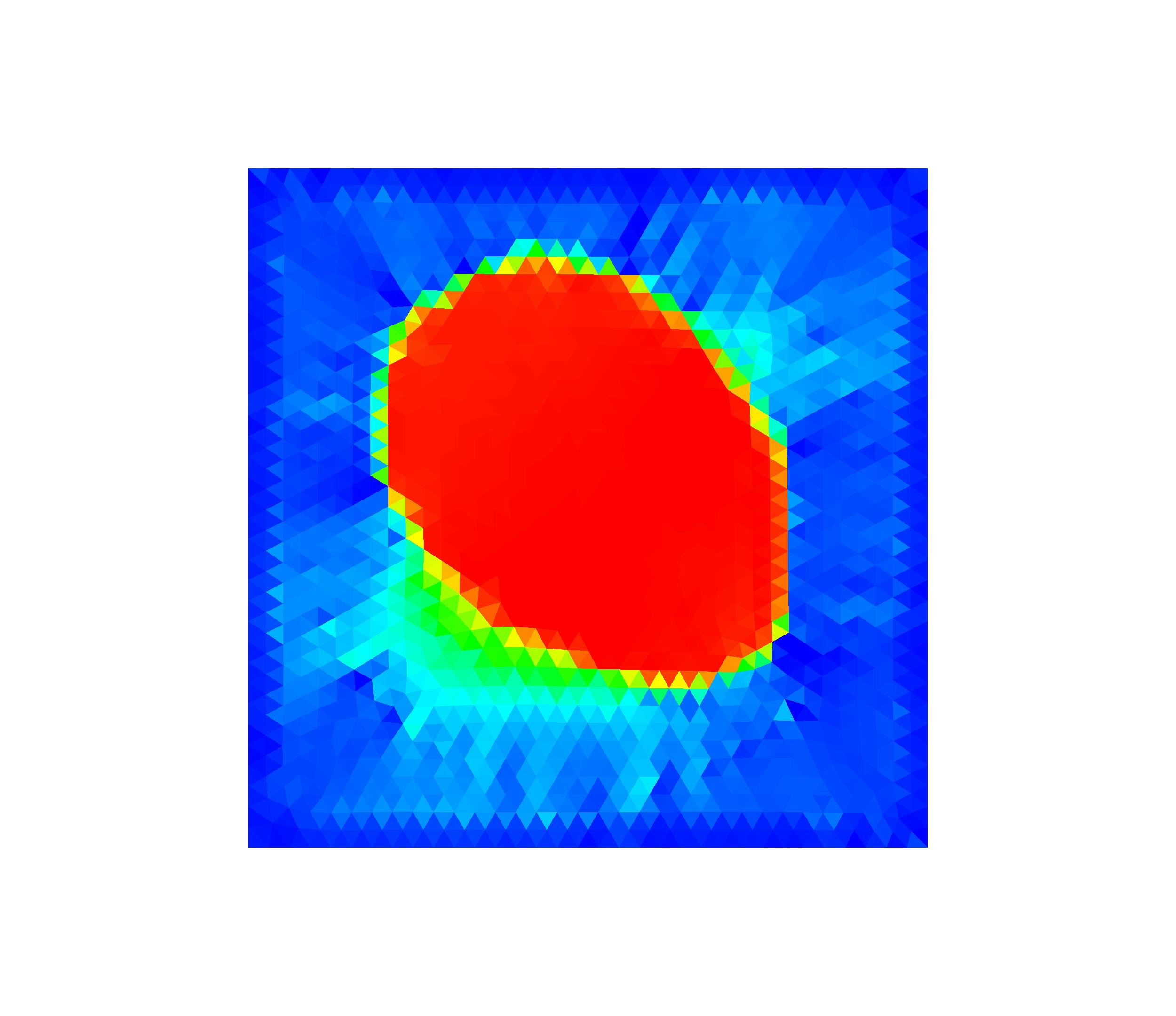}\\[0cm]
        \includegraphics[width=1.2\linewidth,trim={10cm 0cm 10cm 10cm},clip]{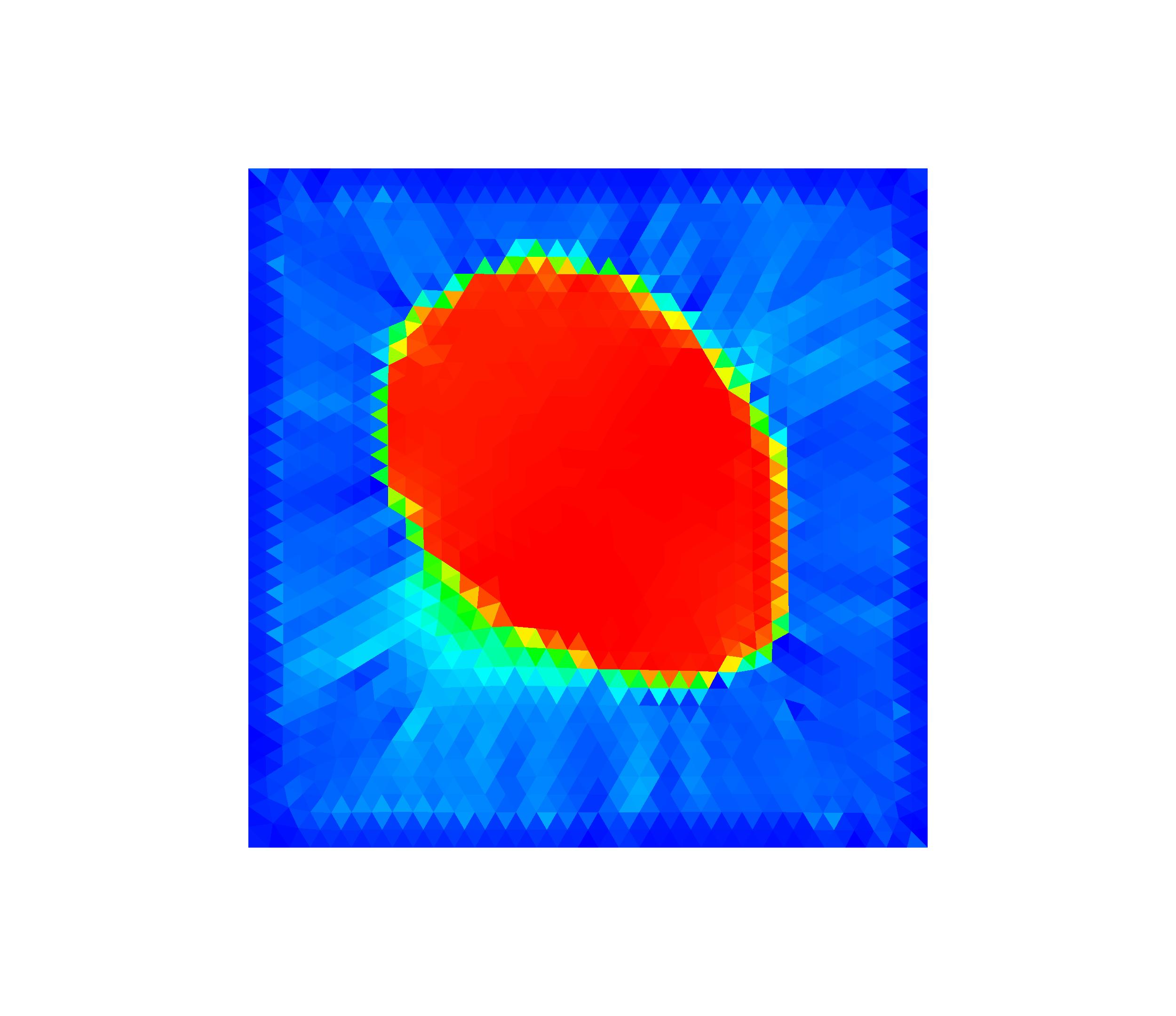}
    \end{minipage}
    
    \par\nointerlineskip
    \vspace{2pt}
    \hrule
    \vspace{3pt}
    
    \noindent%
    \hspace{29pt}
    \begin{minipage}[t]{0.16\textwidth}
        \centering
        \vspace{8pt}
        \pgfplotscolorbardrawstandalone[
            colormap/jet,
            point meta min=0,
            point meta max=0.5,
            colorbar style={
                height = 1cm,
                width=0.08\textwidth,
                /pgf/number format/fixed,
                /pgf/number format/precision=1,
                tick style={font=\tiny},
                ytick={0, 0.25, 0.5},
                yticklabels={0, 0.25, 0.5},                    
            }
        ]
    \end{minipage}%
    \hfill
    \begin{minipage}[t]{0.16\textwidth}
        \centering
        \vspace{8pt}
        \pgfplotscolorbardrawstandalone[
            colormap/jet,
            point meta min=0,
            point meta max=1,
            colorbar style={
                height = 1cm,
                width=0.08\textwidth,
                /pgf/number format/fixed,
                /pgf/number format/precision=1,
                tick style={font=\tiny},
                ytick={0, 0.5, 1},
                yticklabels={0, 0.5, 1},                    
            }
        ]
    \end{minipage}%
    \hspace{30pt}%
    \hspace{27pt}%
    \begin{minipage}[t]{0.16\textwidth}
        \centering
        $u^\dagger(T)$\\
        \includegraphics[width=1.2\linewidth,trim={10cm 0cm 10cm 10cm},clip]{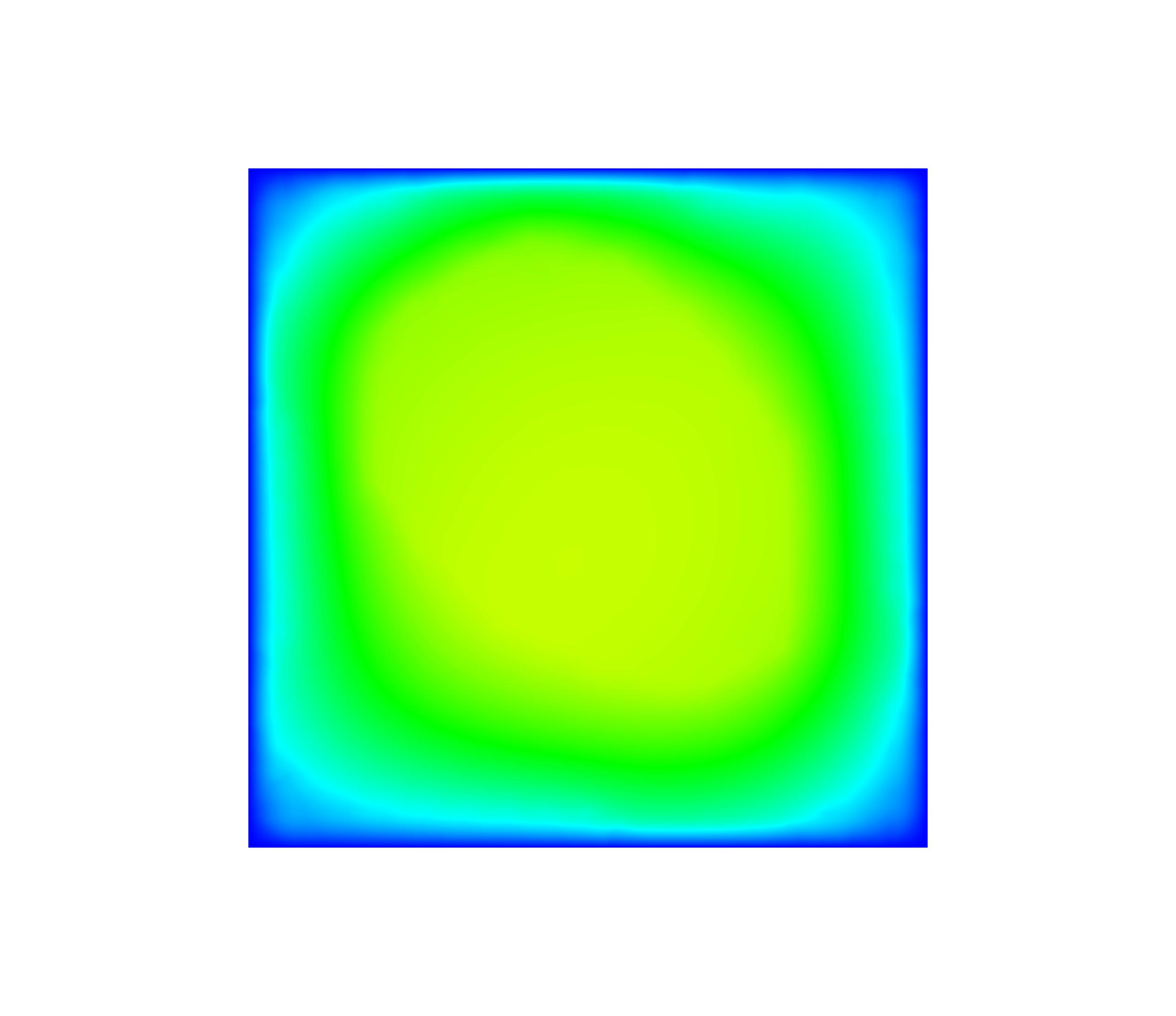}
    \end{minipage}%
    \hspace{14pt}
    \begin{minipage}[t]{0.16\textwidth}
        \centering
        $a^\dagger$\\
        \includegraphics[width=1.2\linewidth,trim={10cm 0cm 10cm 10cm},clip]{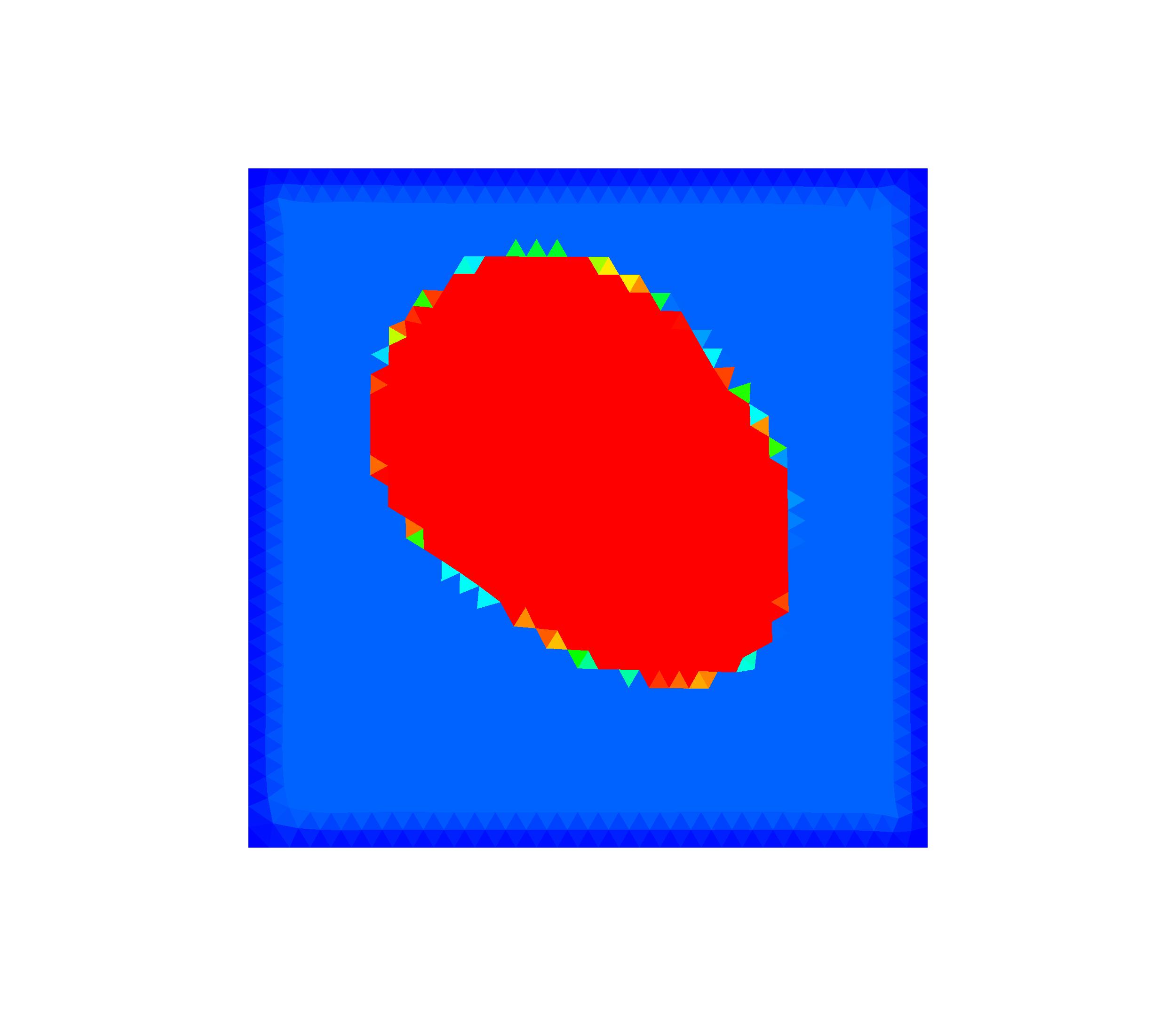}
    \end{minipage}
    
    \caption{Darcy flow. Visualization of evolution of the state $u$ and diffusion parameter $a$ computed from MRAS (top) and exact quantities $u^\dagger, a^\dagger$ (bottom).} 
    \label{fig::darcy_evolution_big}
\end{figure}


\section{Fisher-KPP equation: state nonlinearity}\label{sec::nonlinear_state}
Our second example concerns the Fisher-KPP equation, which is vastly applicable within biology, including for population modeling \cite{Simpson} and chemotaxis \cite{Bortz}.
Concretely, we investigate the quasi-linear PDE
\begin{equation}
\begin{aligned}\label{eq:fisher_KPP}
D_t u -\nabla\cdot(a\nabla u) +u-u^2 & = g \quad \text{ in }I\times\dom:=\I\times[-1.25,1.25]^2\\
u|_{\partial\dom} &= 0 \quad \text{ in }I\\
u(t=0)&=u_0 \quad \text{in }\dom
\end{aligned}
\end{equation}
with unknown spatially dependent diffusion $a$. Equation \eqref{eq:fisher_KPP} can be summarized as a reaction-diffusion equation with nonlinearity in the state $u$.

\subsection{MRAS analysis}
As was the case for the Darcy problem, we first derive an explicit form for the MRAS \eqref{mras}, then express its weak form to prepare for implementation.

\begin{proposition}
    For the Fisher-KPP reaction-diffusion equation \eqref{eq:fisher_KPP} with unknown diffusion $a$ and data $\uobs$, the MRAS \eqref{mras} takes the form
    \begin{equation}\label{eq:MRAS_fisher_KPP}
    \begin{aligned}
        D_t a & = \nabla \uobs\cdot\nabla(u-\uobs) \\
        D_tu - \Delta(u-\uobs) +\uobs-\uobs^2 & = g + \nabla\cdot(a\nabla \uobs) \\
        (a,u)(0) & = (a_0,u_0)
    \end{aligned}
    \end{equation}
with the state space $U:=H^1_0(\dom)$ and parameter space $H:=L^2(\dom)$. Assumptions \ref{A-lip}, \ref{A-funcC} hold.
\end{proposition}
\begin{proof}

The derivation is analogous to  Proposition \eqref{prop:darcy}, with only minor changes. The only notable difference is the nonlinear reaction term $u-u^2$ that appears in the model $f(a,u)$. As a result, the MRAS now includes a nonlinearity; however, only with respect to the data $\uobs$, meaning that in the state equation \eqref{eq:MRAS_fisher_KPP}, one has $f(a,\uobs)= - \nabla\cdot (a\nabla \uobs)+\uobs-\uobs^2$. We emphasize that while the original equation \eqref{eq:fisher_KPP} is nonlinear in $u$, the constructed MRAS is (affine) linear in $u$; see Remark \ref{rem:mras}.

As the equation \eqref{eq:fisher_KPP}  is linear in $a$, the Lipschitz constant in Assumption \ref{A-lip} remains $L^{\tilde{a},\uobs}=0$. Similarly as for the Darcy equation, Assumption \ref{A-funcC} holds with the linear bounded and coercive operator $\Cc(\|a\|_H):=  -\Delta $,
with the same function space setting. 
\end{proof}

\begin{corollary}
    The weak form of the MRAS \eqref{eq:MRAS_fisher_KPP} in a semi-implicit Euler scheme for the unknown $(a,u)$ reads as
\begin{align}
\label{eq:fisher_discret_a}
    &\int_\Omega a_{n+1}s\d x  = \int_\Omega a_{n}\,s\d x + \Delta t\int_\Omega \nabla \uobs_{n}\cdot\nabla(u_{n}-\uobs_{n})s \d x,\\
    \label{eq:fisher_discret_u}
    &\int_\Omega u_{n+1}\, v  
    + \Delta t \,\nabla u_{n+1}\cdot \nabla v \d x+\Delta t \int_\Omega \,a_{n+1}\nabla \uobs_{n+1}\nabla v\d x \\
    &\qquad = \int_\Omega u_n\, v\d x+\Delta t\int_\Omega (g_{n+1}-\uobs_{n+1}+\uobs_{n+1}^2)\, v\d x + \Delta t\int_\Omega \nabla \uobs_{n+1}\, \nabla v\d x\notag\\
     &(a,u)(0)  = (a_0,u_0)
\end{align}    

for all $v\in H^1_0(\dom)$ and all $s\in L^2(\dom)$.

\end{corollary}
\begin{proof}
    We proceed in a similar manner to the Darcy problem, using semi-implicit Euler time stepping. The only new element is $\uobs_{n+1} -\uobs_{n+1}^2$ which now appears as a source in the state equation \eqref{eq:fisher_discret_u}. 
\end{proof}
\FloatBarrier

\subsection{Numerical results}\label{numerical_nonlinearstate}
\paragraph{Data preparation}
For this example, we use the ring-shaped parameter
\[
a^\dagger(x) :=
\begin{cases}
1 & \text{if}\quad 0.5^2 < x_1^2+x_2^2 < 0.9^2, \\
0.25 & \text{else}
\end{cases}
\]
and a time harmonic source $g(x,t) := 3\exp(-\tfrac{x_1^2}{0.4}-\tfrac{x_2^2}{0.4})\cos(2t)$. From these, the exact solution $\utrue$ is simulated using implicit time scheming stepping, with initial state $u_0(x):=3\exp(-\tfrac{x_1^2}{0.4}-\tfrac{x_2^2}{0.4})$.
We examine two measurement scenarios, namely perfect measurement and measurement that is contaminated by  $3\%$ relative noise. 
To describe the noisy data $z^\delta$, discrete Gaussian noise $\epsilon\sim N(0,1)$ is added to the exact data $u^\dagger$, then multiplicatively scaled such that $\|z^\delta-u^\dagger\|_{L^2} = \delta\|u^\dagger\|_{L^2}$ for noise levels $\delta$ to be specified.

\paragraph{Discretization}
\begin{table}
    \centering
    \begin{tabular}{c|c}
         Spatial domain, mesh size & $\Omega=[-1.25,1.25]^2$, $h_\mathrm{max} = 0.1$\\
         $\#$dofs for $U_h$, $\#$dofs for $Q_h$  &6667, 1448\\
         Max time, time step, \#steps & $T=10$, $\Delta t=0.001$, 10000 steps\\[0.5ex]
         Source term & $g = 3\exp(-\tfrac{x^2}{0.4}-\tfrac{y^2}{0.4})\cos(2t)$ \\[0.5ex]
         relative noise levels & $0\%$, $3\%$\\
    \end{tabular}
    \caption{Setup for the Fisher-KPP problem.}
    \label{table-Fisher}
\end{table}

We employ a coarse grid with $h_\mathrm{max} = 0.1$ and a temporal resolution $\Delta t=0.001$, and observe the evolution from $t=0$ until the final time $T=10$. The FEM space setup is otherwise as in the Darcy problem. To avoid inverse crimes, we simulate the data using a finer mesh with $h_\mathrm{max} = 0.09$ and a higher polynomial degree $k=4$, while the reconstruction uses $h_\mathrm{max} = 0.1$, $k=3$. The initial state $u(0)$ is set equal to the initial condition $u_0$, and the initial parameter $q(0)$ is a piecewise constant function taking the value $1$ on the square $[-1.15,1.15]\times[-1.15,1.15]$ and $0$ otherwise; see Figure \ref{fig::fisher_evolution_big}. The setting is summarized in Table \ref{table-Fisher}.

\paragraph{Numerical results}

In the two left columns of Figure \ref{fig::fisher_evolution_big}, we observe that by inputting the time series data $z(t)$ into the MRAS, the output state changes over time, eventually matching the exact state $u^\dagger(T)$. 

More importantly, convergence of the approximate parameter $a$ towards the exact parameter $a^\dagger$ is encouraging, especially for noise-free data ($\delta=0$, left column). The convergence is also evident for noisy data ($\delta=0.03$, right column), despite slight degeneration compared to the noise-free case. The error fields in Figure \ref{fig::fisher_diff} confirm the high quality of the MRAS reconstruction.

\begin{figure}
    \centering
    \begin{minipage}[t]{0.09\textwidth}
        \vspace{8ex}
        \raggedright
        t=0\\[7ex]
        t=0.001\\[9ex]
        t=0.05\\[9ex]
        t=0.1\\[9ex]
        t=0.5\\[9ex]
        t=1.5\\[9ex]
        t=5\\[9ex]
        t=10
    \end{minipage}
    \begin{minipage}[t]{0.16\textwidth}
        \centering
        \hspace{15pt}State\\
        \includegraphics[width=1.3\linewidth,trim={13cm 5cm 13cm 10cm},clip]{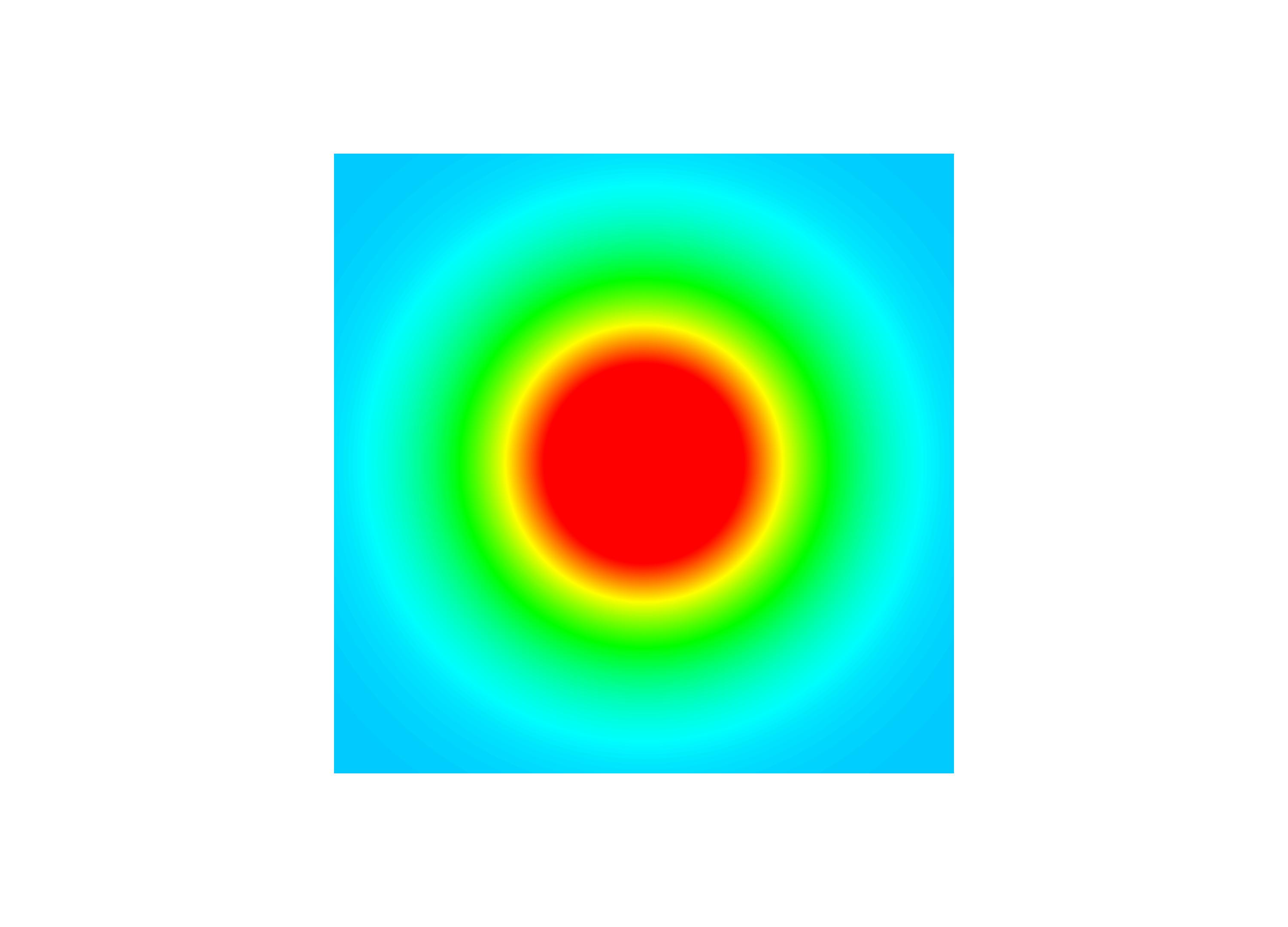}\\
        \includegraphics[width=1.3\linewidth,trim={13cm 5cm 13cm 13cm},clip]{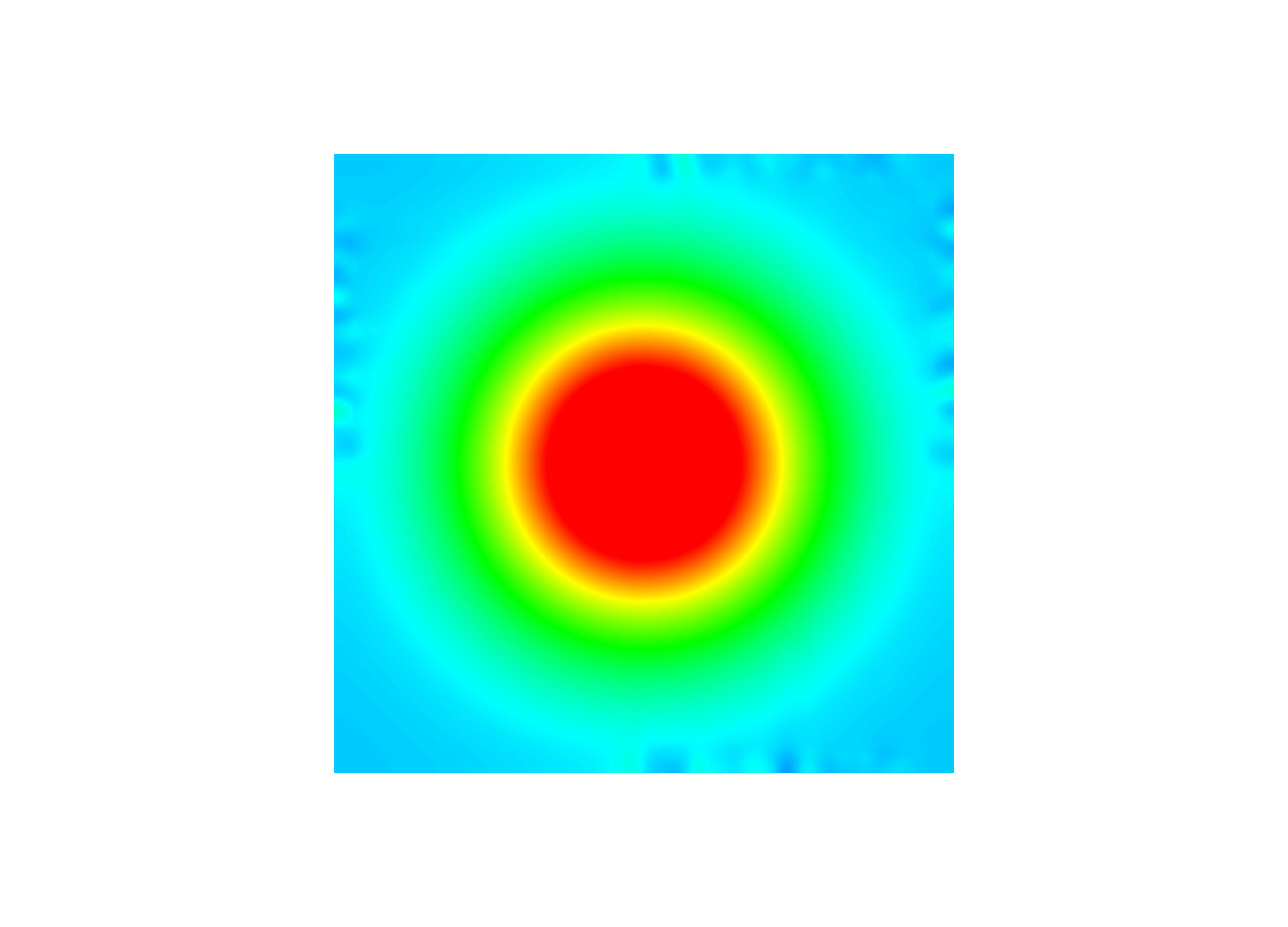}\\
        \includegraphics[width=1.3\linewidth,trim={13cm 5cm 13cm 13cm},clip]{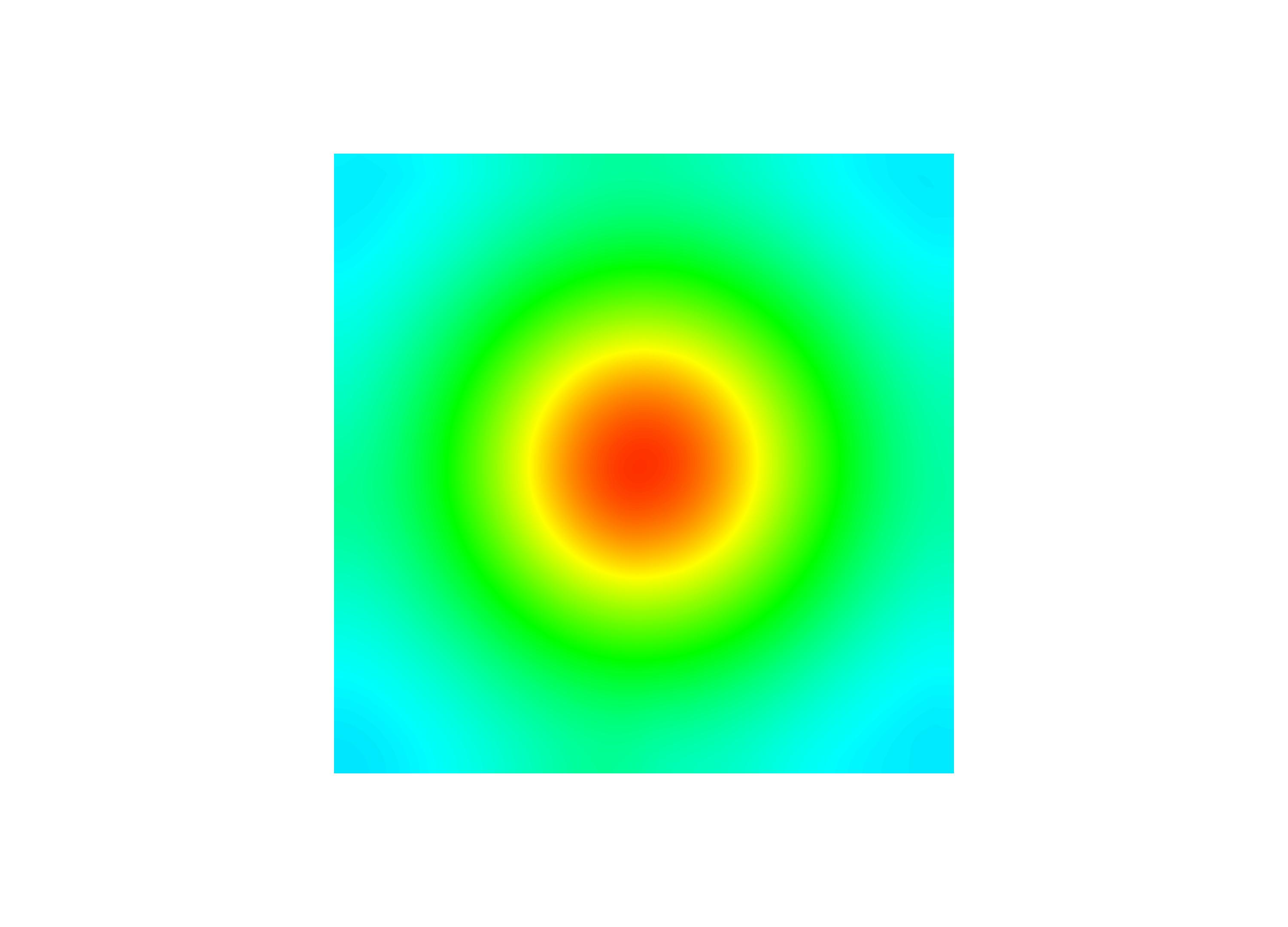}\\
        \includegraphics[width=1.3\linewidth,trim={13cm 5cm 13cm 13cm},clip]{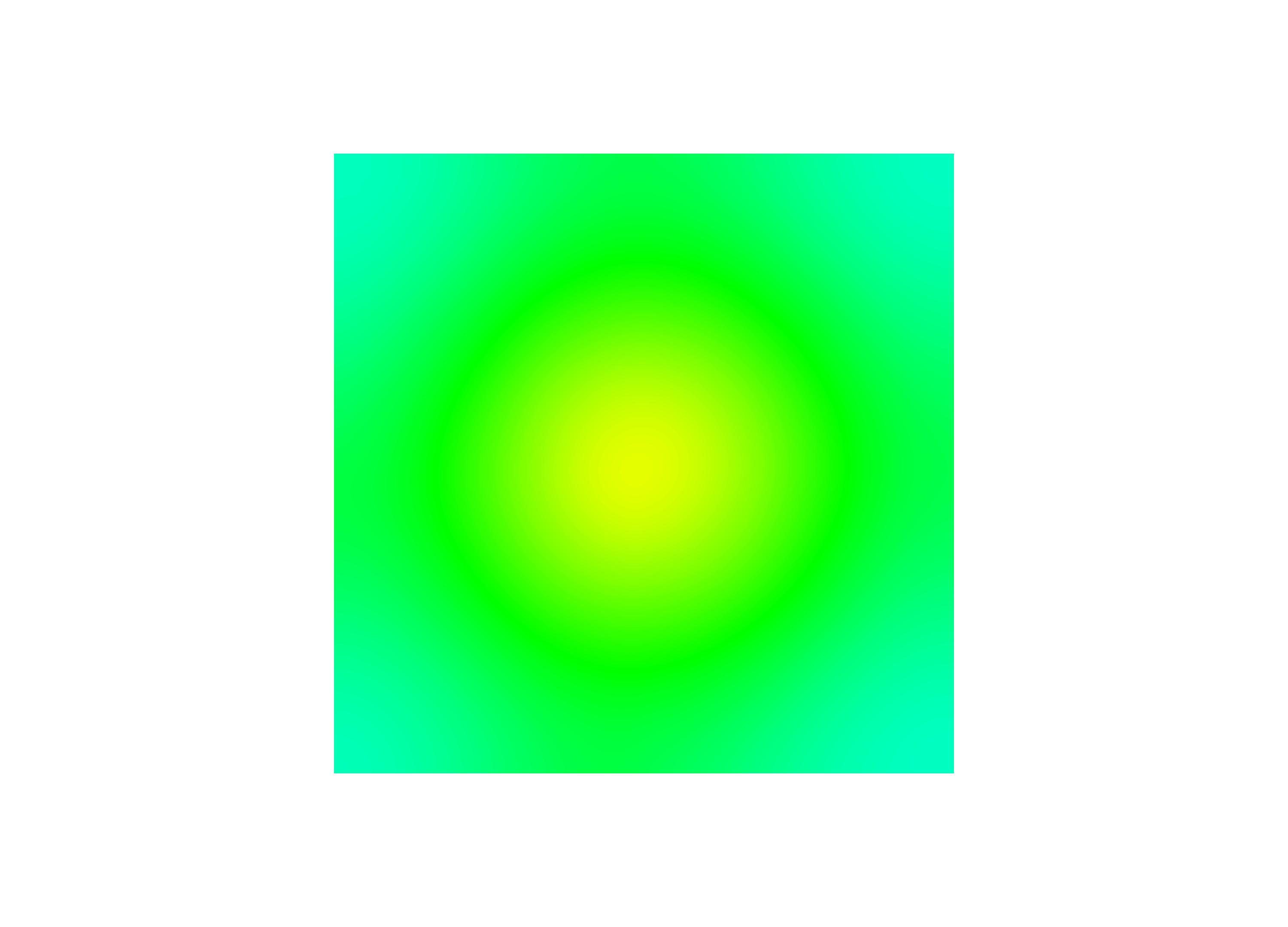}\\
        \includegraphics[width=1.3\linewidth,trim={13cm 5cm 13cm 13cm},clip]{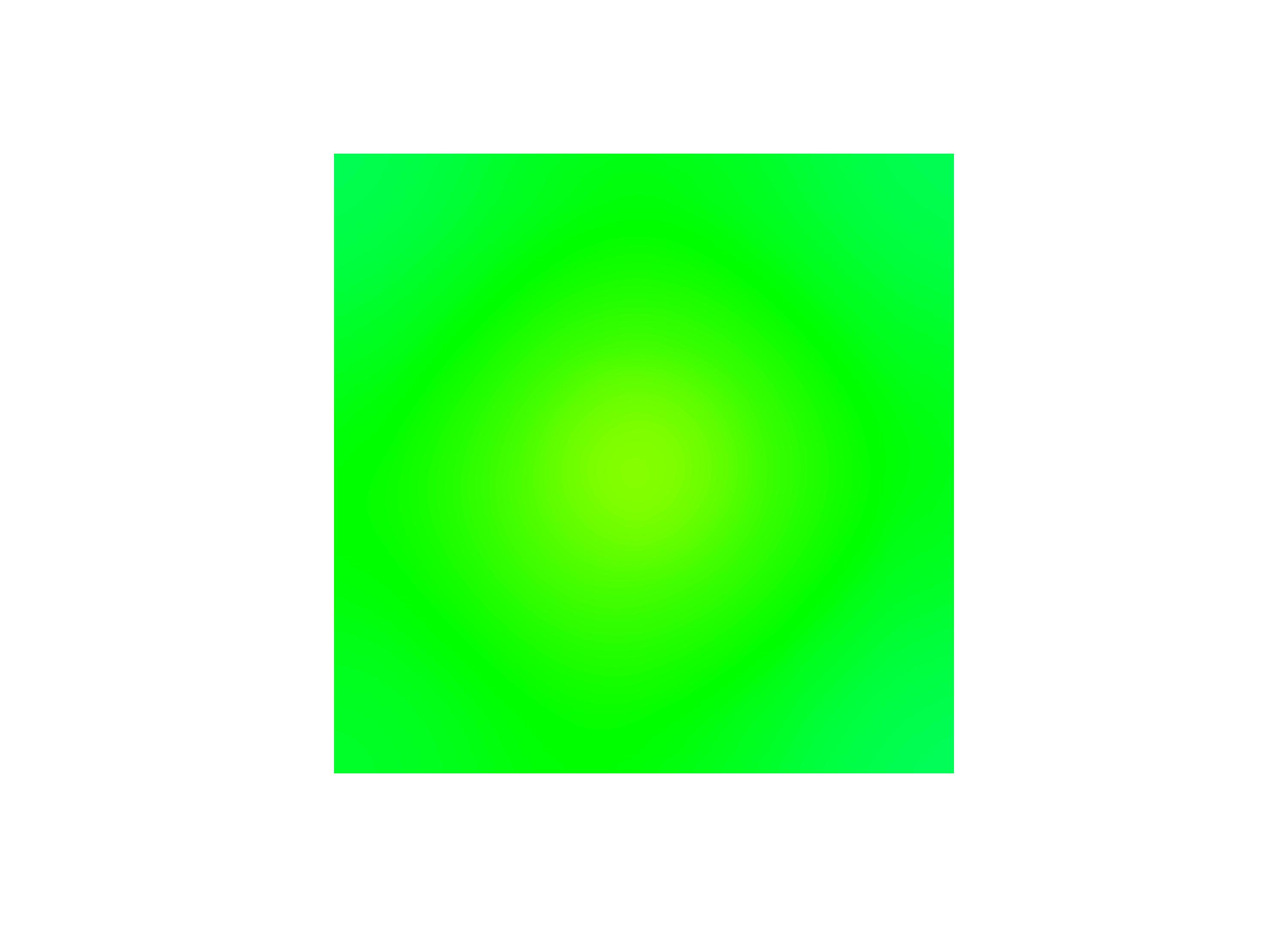}\\
        \includegraphics[width=1.3\linewidth,trim={13cm 5cm 13cm 13cm},clip]{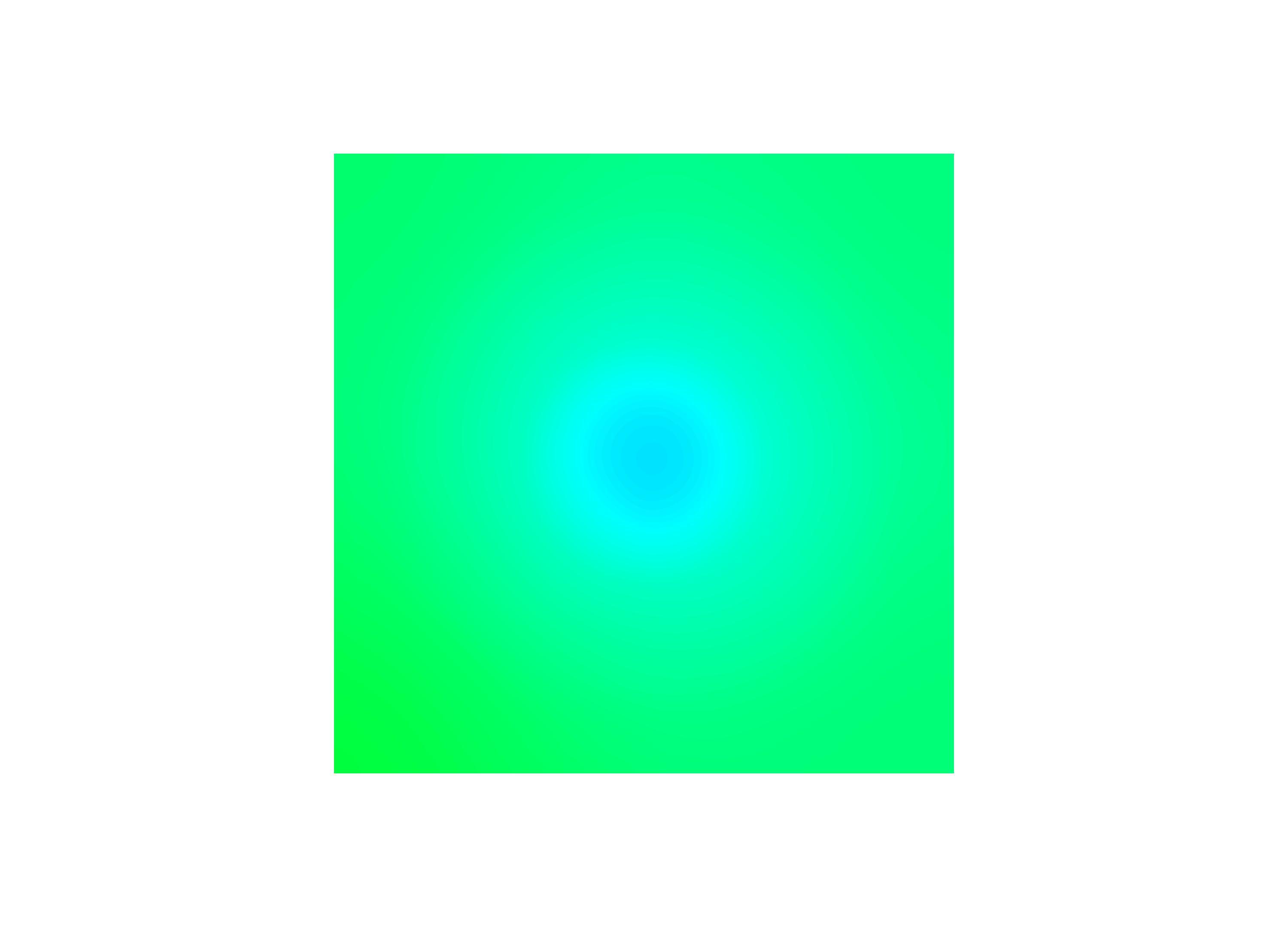}\\
        \includegraphics[width=1.3\linewidth,trim={13cm 5cm 13cm 13cm},clip]{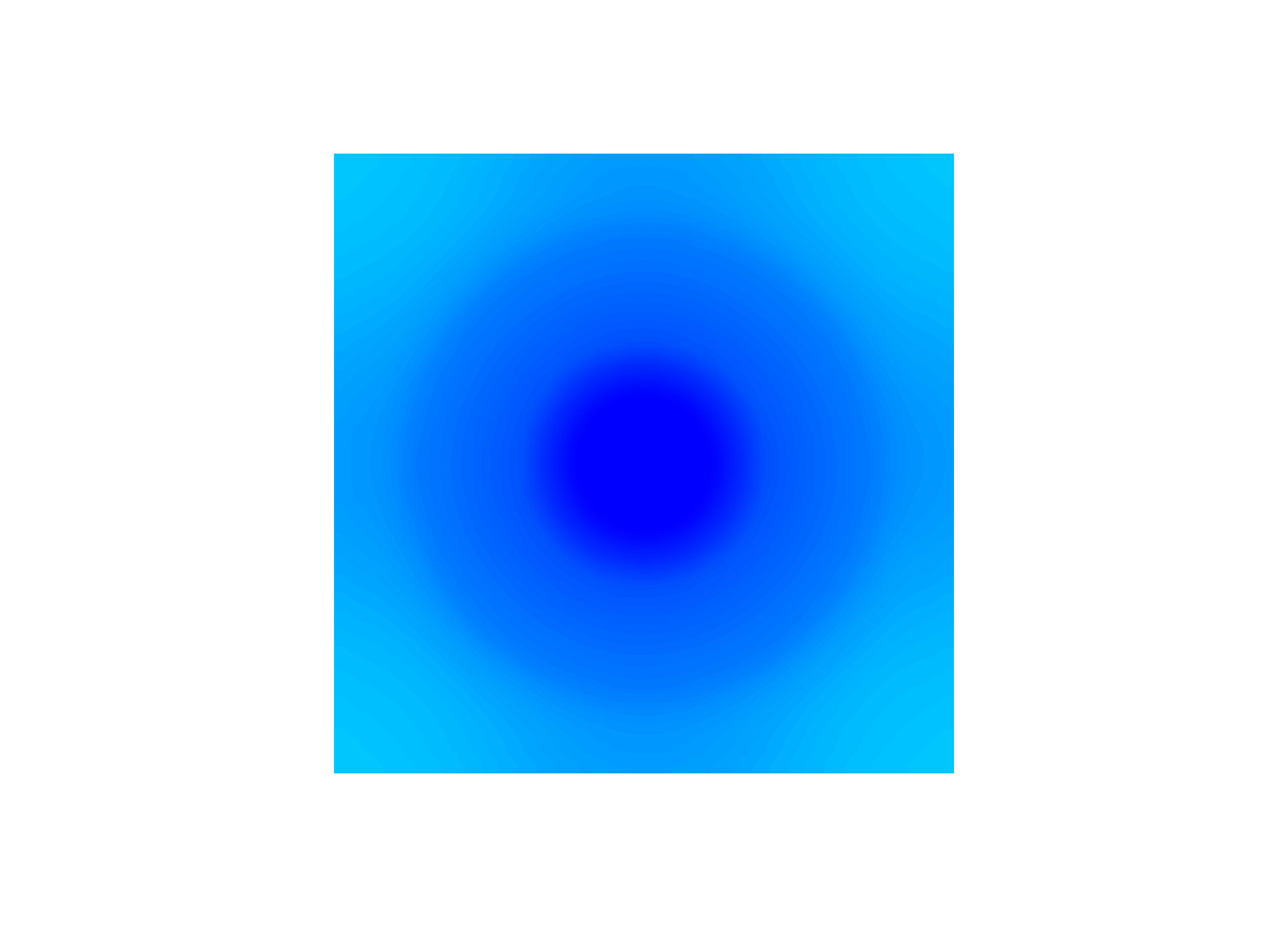}\\
        \includegraphics[width=1.3\linewidth,trim={13cm 5cm 13cm 13cm},clip]{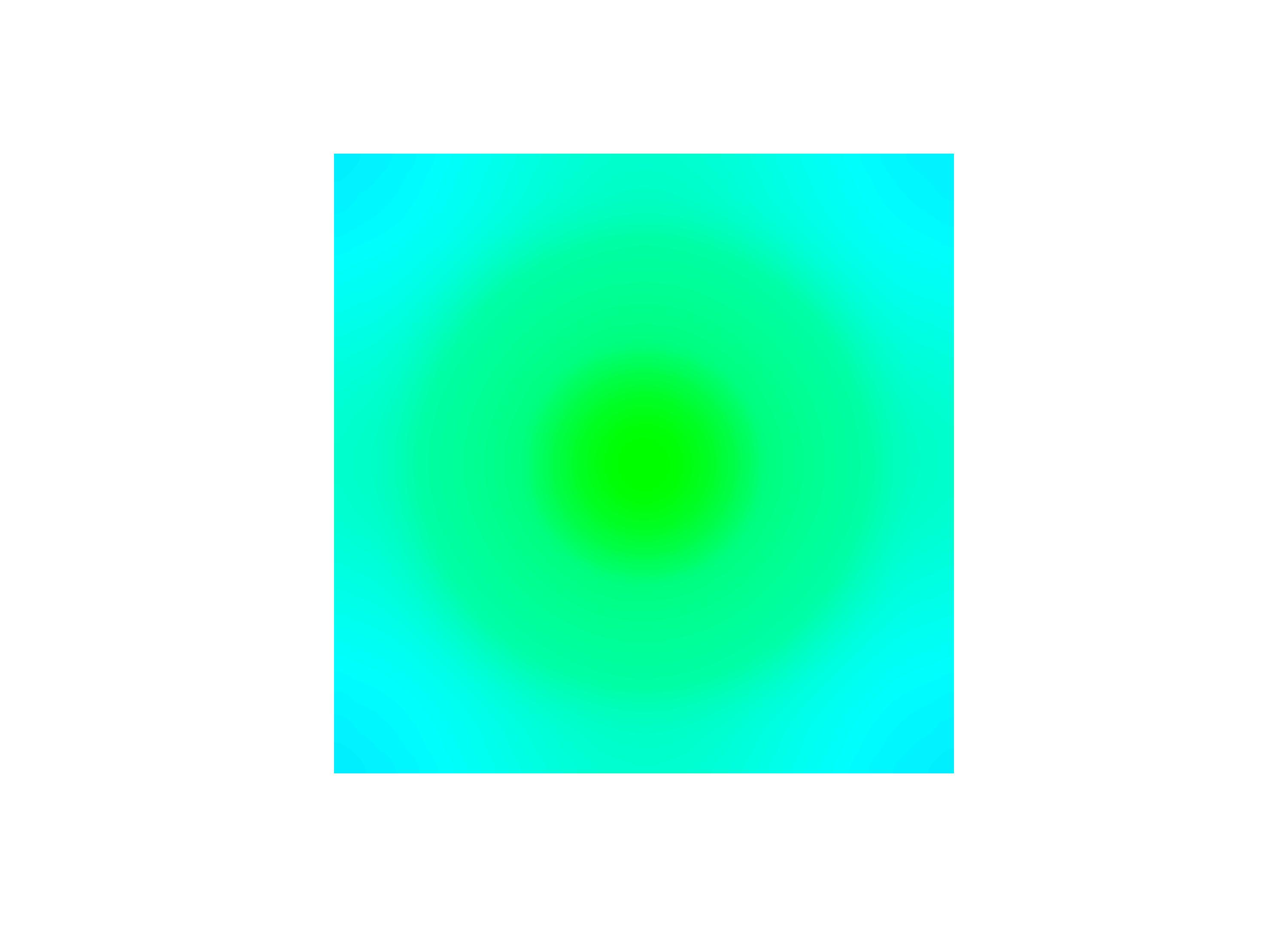}\\
    \end{minipage}
    \hspace{10pt}
    \begin{minipage}[t]{0.16\textwidth}
        \centering
        \hspace{7pt}Parameter\\
        \includegraphics[width=1.3\linewidth,trim={13cm 5cm 13cm 10cm},clip]{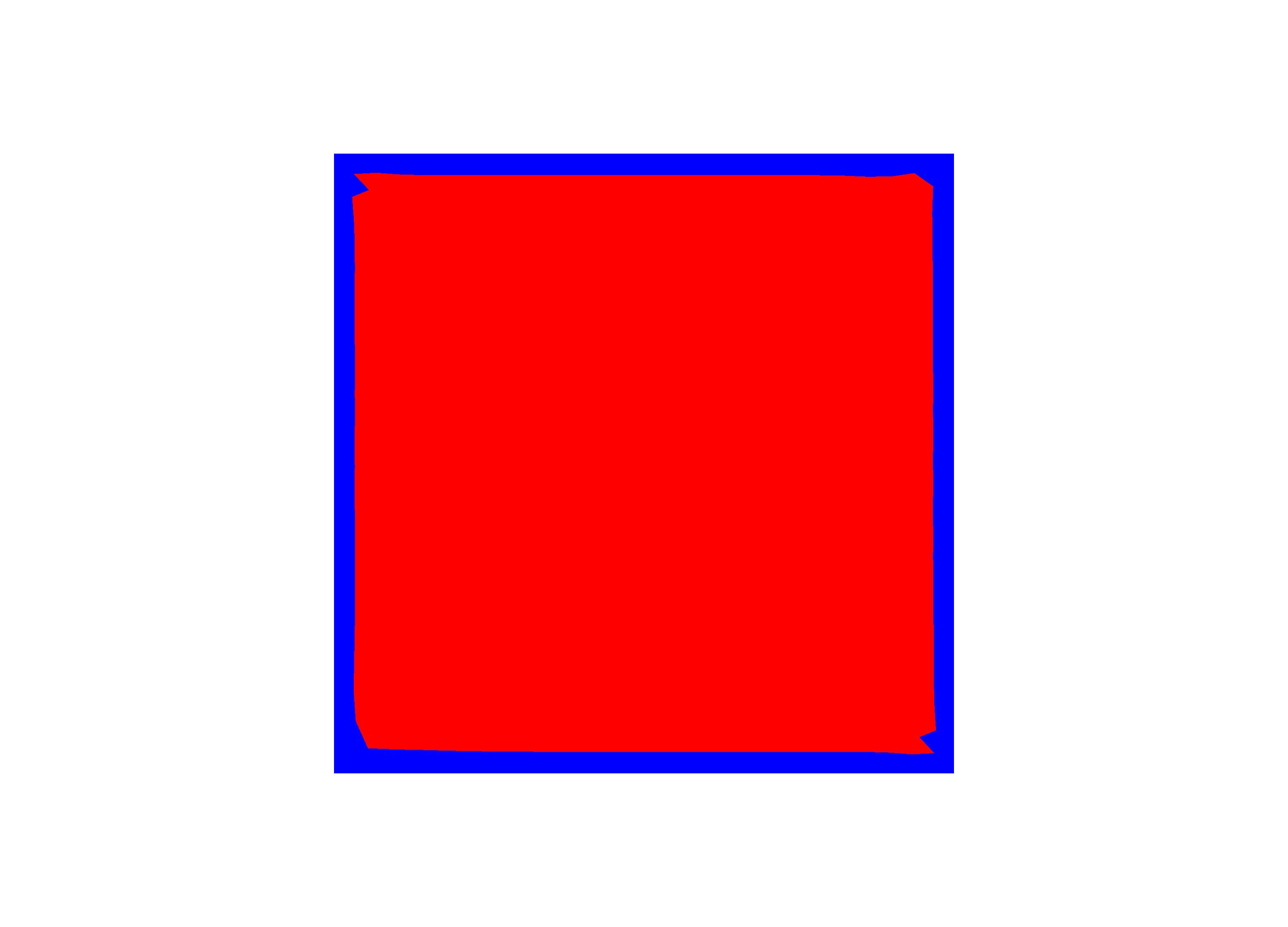}\\[0cm]
        \includegraphics[width=1.3\linewidth,trim={13cm 5cm 13cm 13cm},clip]{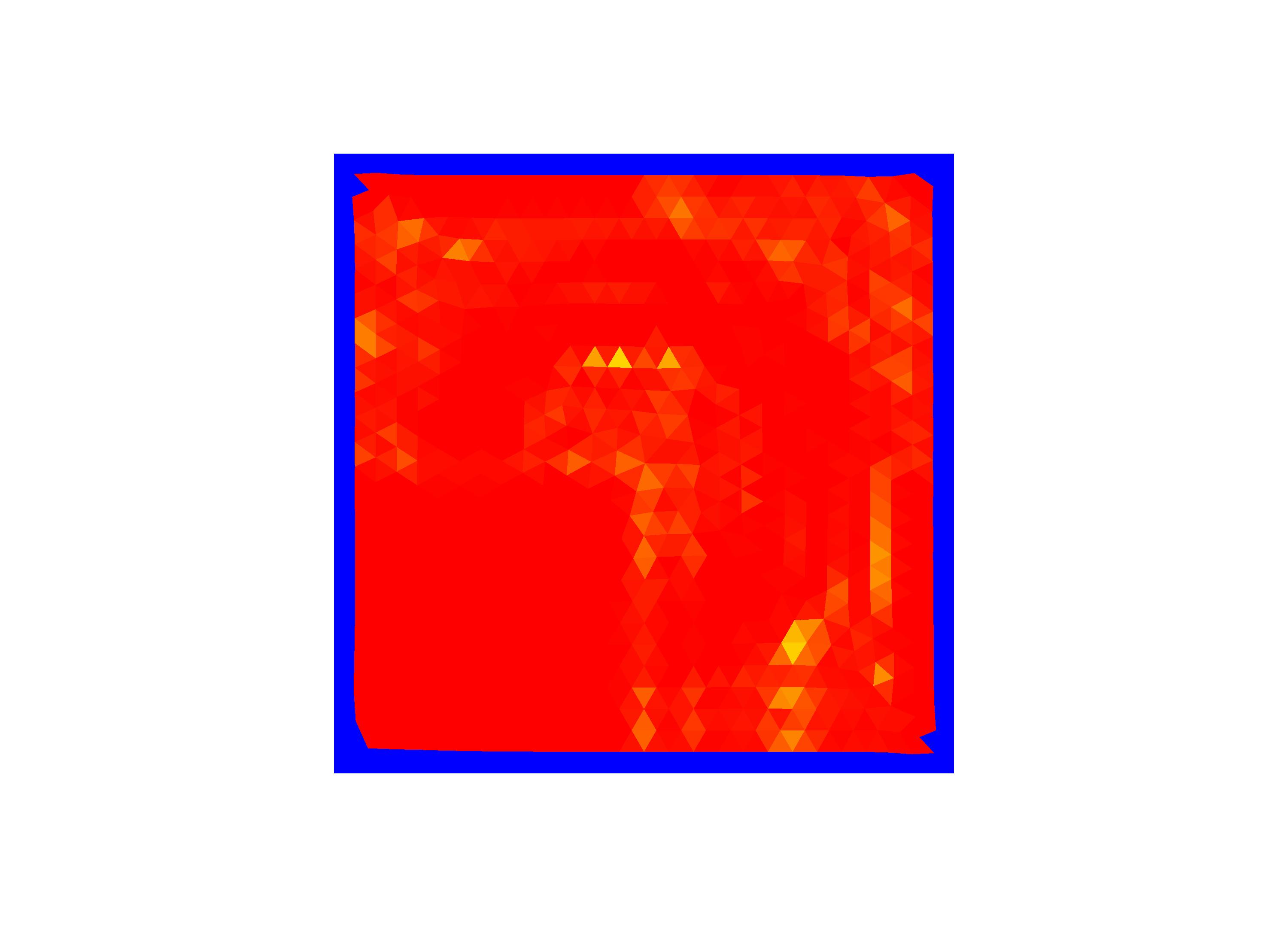}\\[0cm]
        \includegraphics[width=1.3\linewidth,trim={13cm 5cm 13cm 13cm},clip]{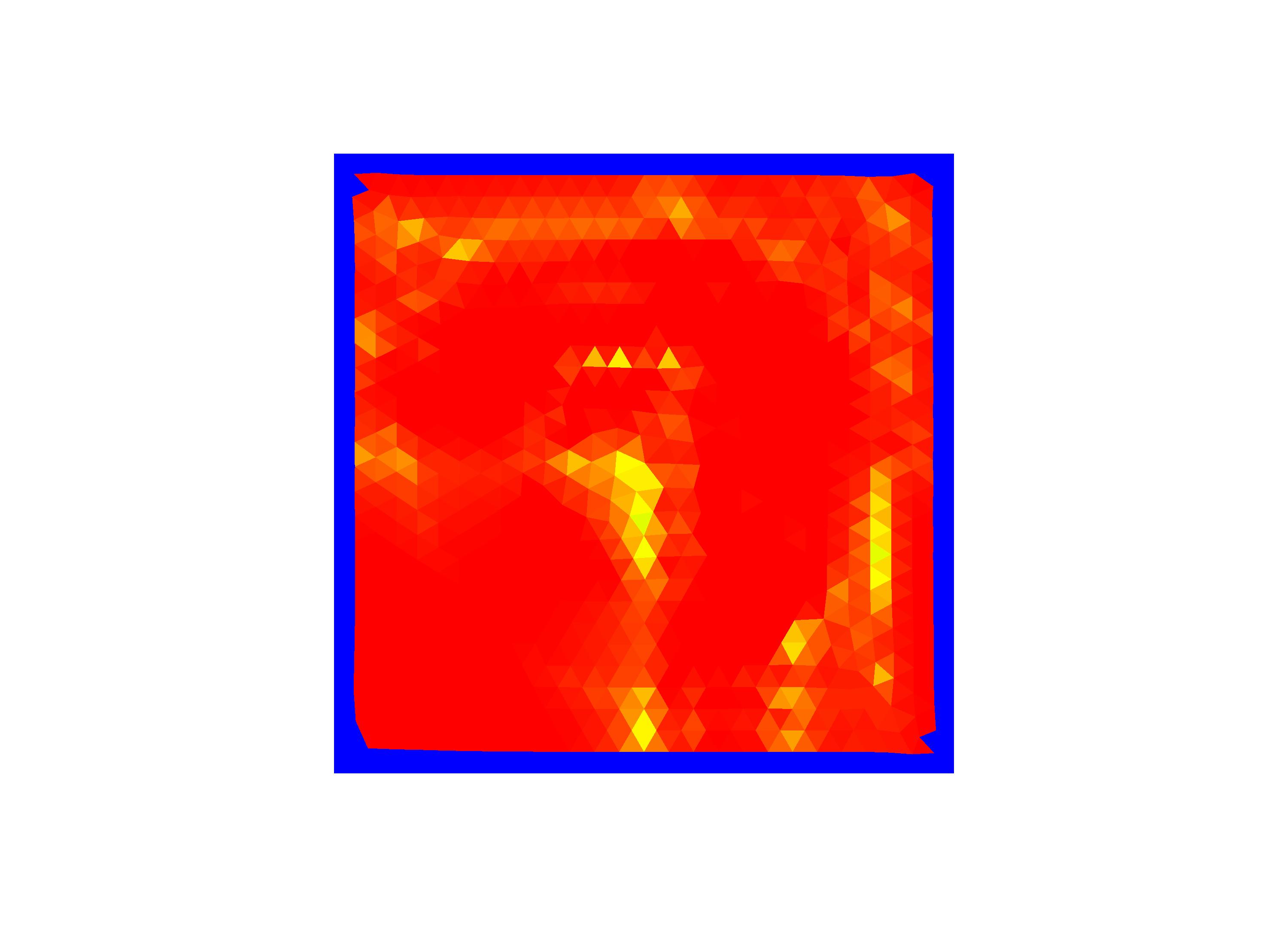}\\
        \includegraphics[width=1.3\linewidth,trim={13cm 5cm 13cm 13cm},clip]{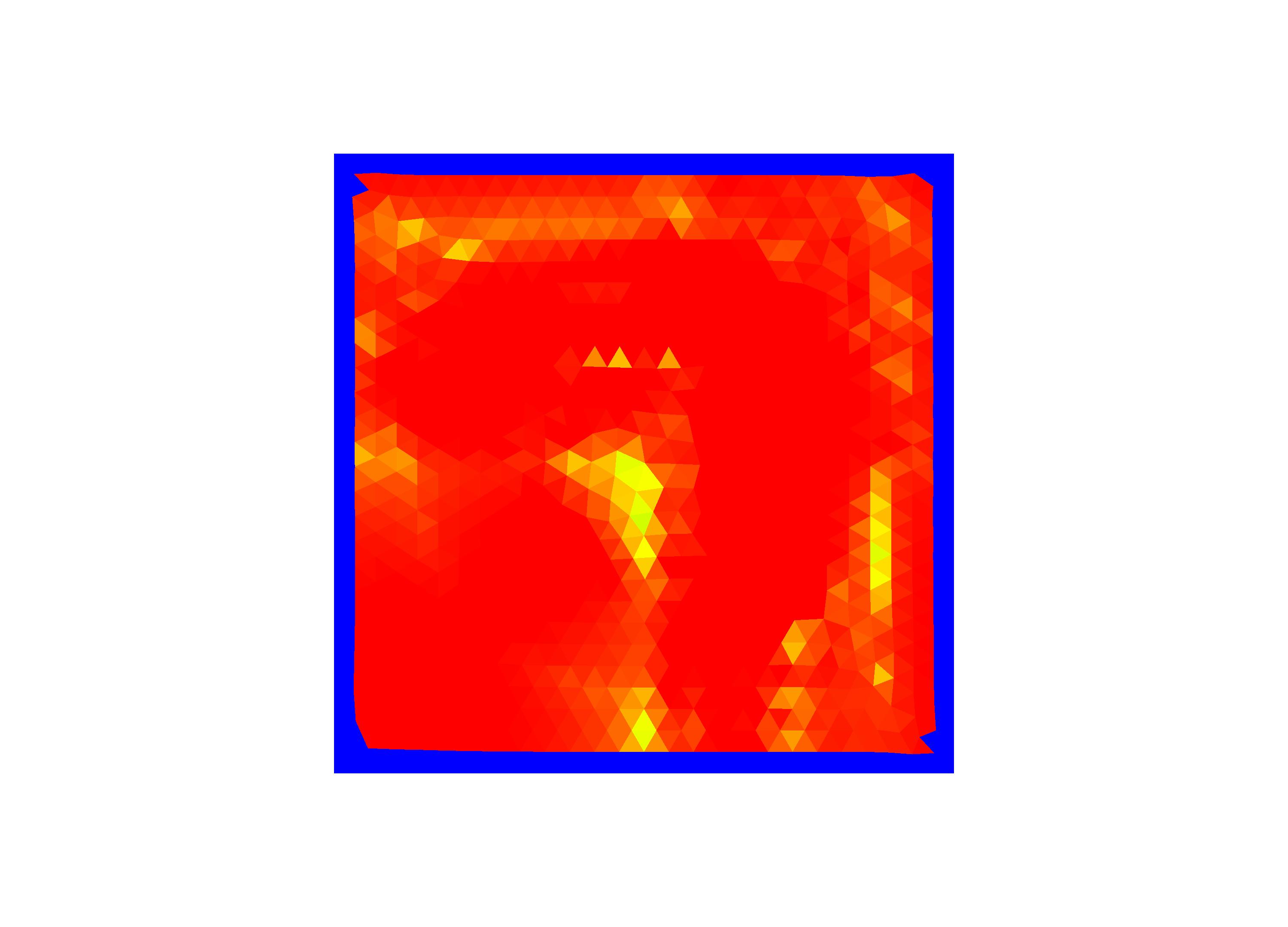}\\
        \includegraphics[width=1.3\linewidth,trim={13cm 5cm 13cm 13cm},clip]{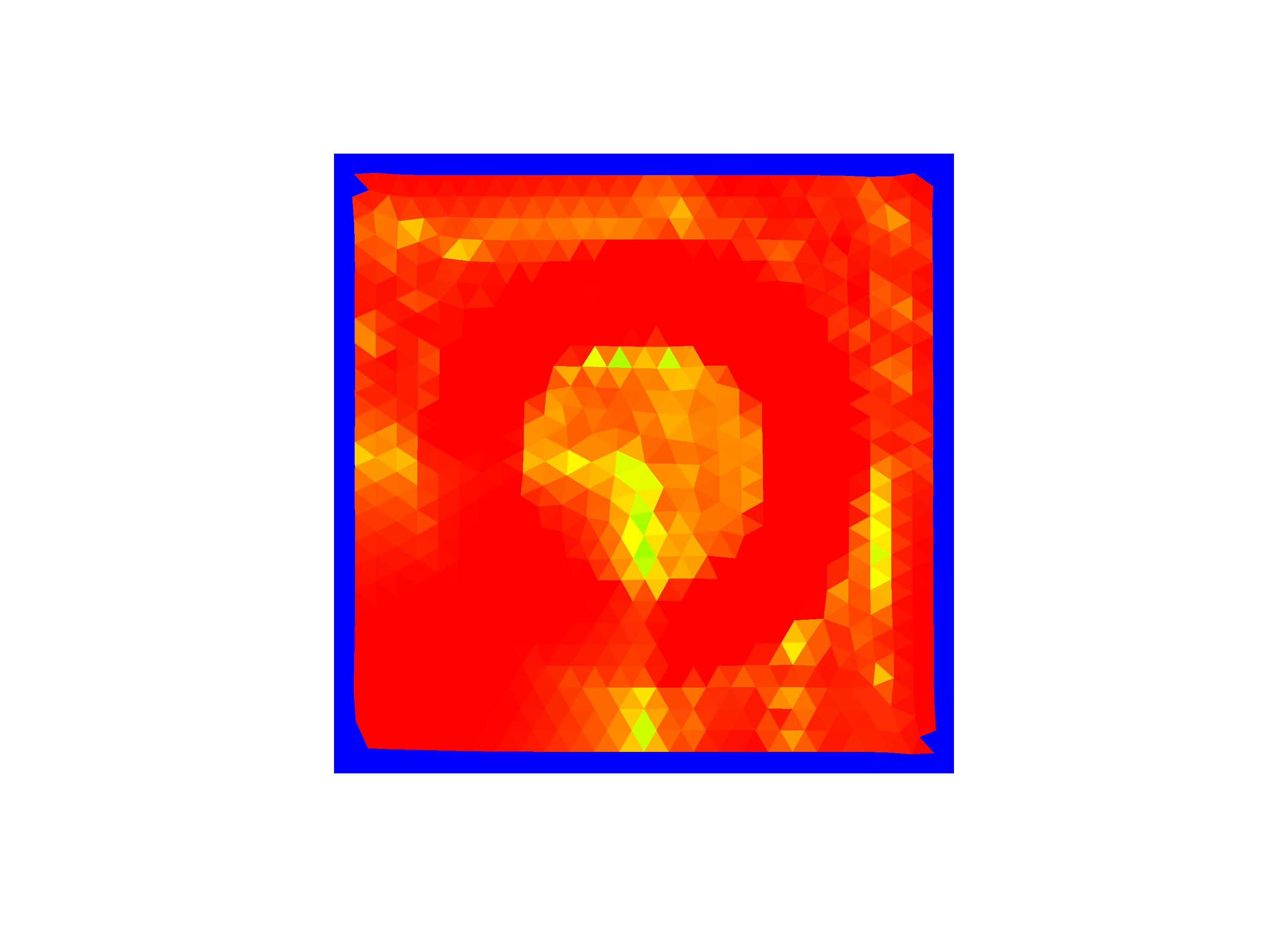}\\
        \includegraphics[width=1.3\linewidth,trim={13cm 5cm 13cm 13cm},clip]{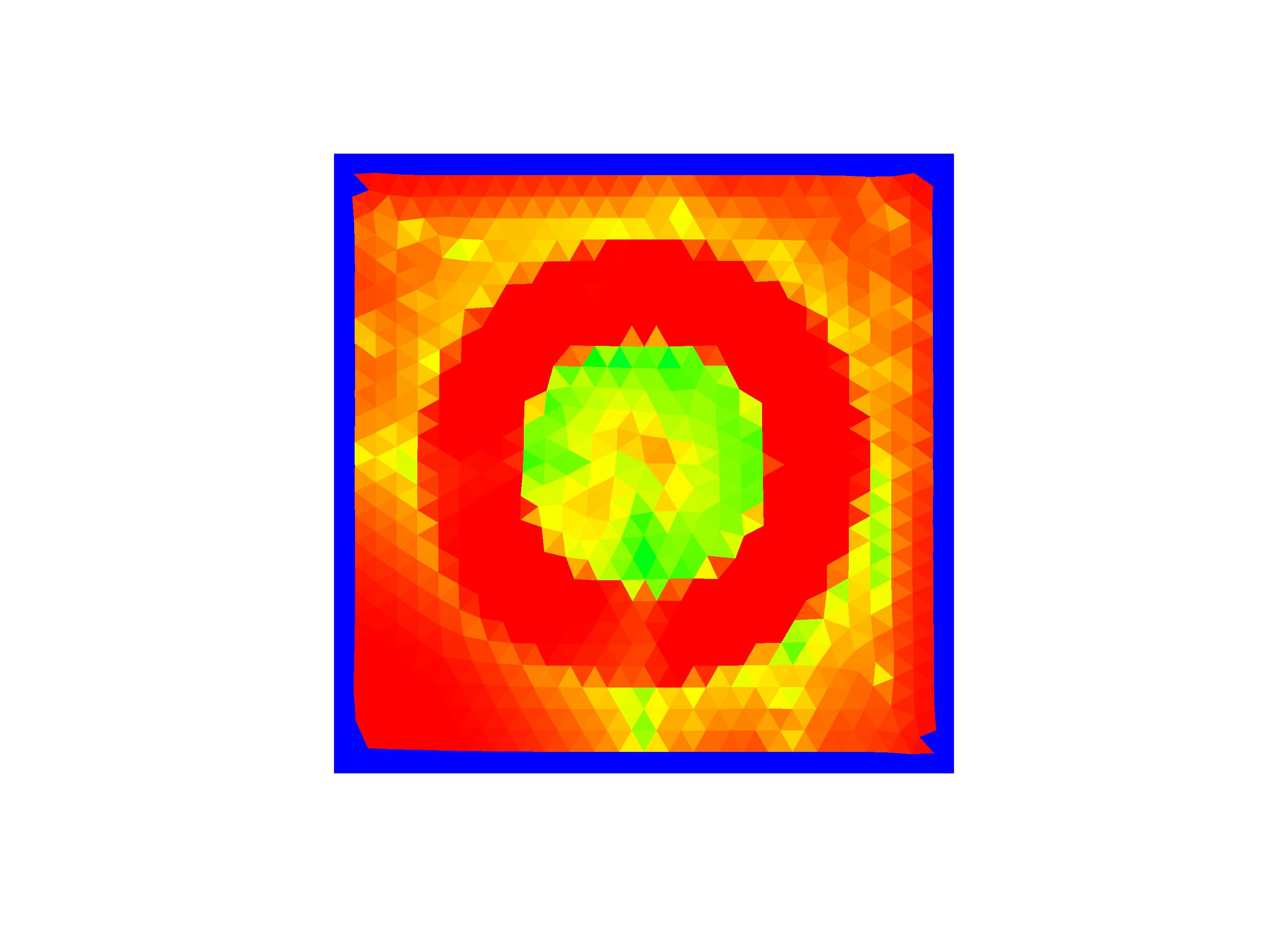}\\
        \includegraphics[width=1.3\linewidth,trim={13cm 5cm 13cm 13cm},clip]{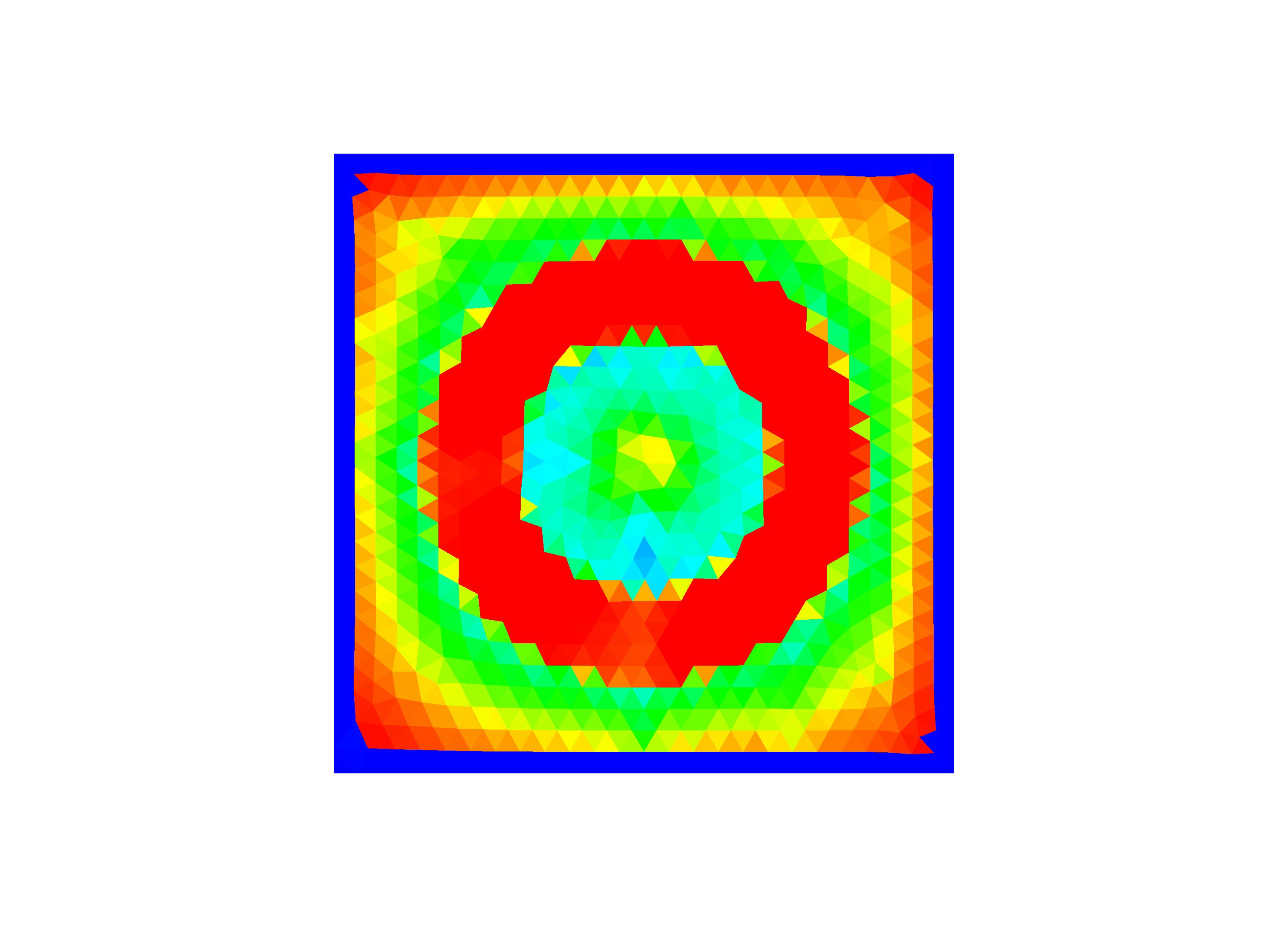}\\
        \includegraphics[width=1.3\linewidth,trim={13cm 5cm 13cm 13cm},clip]{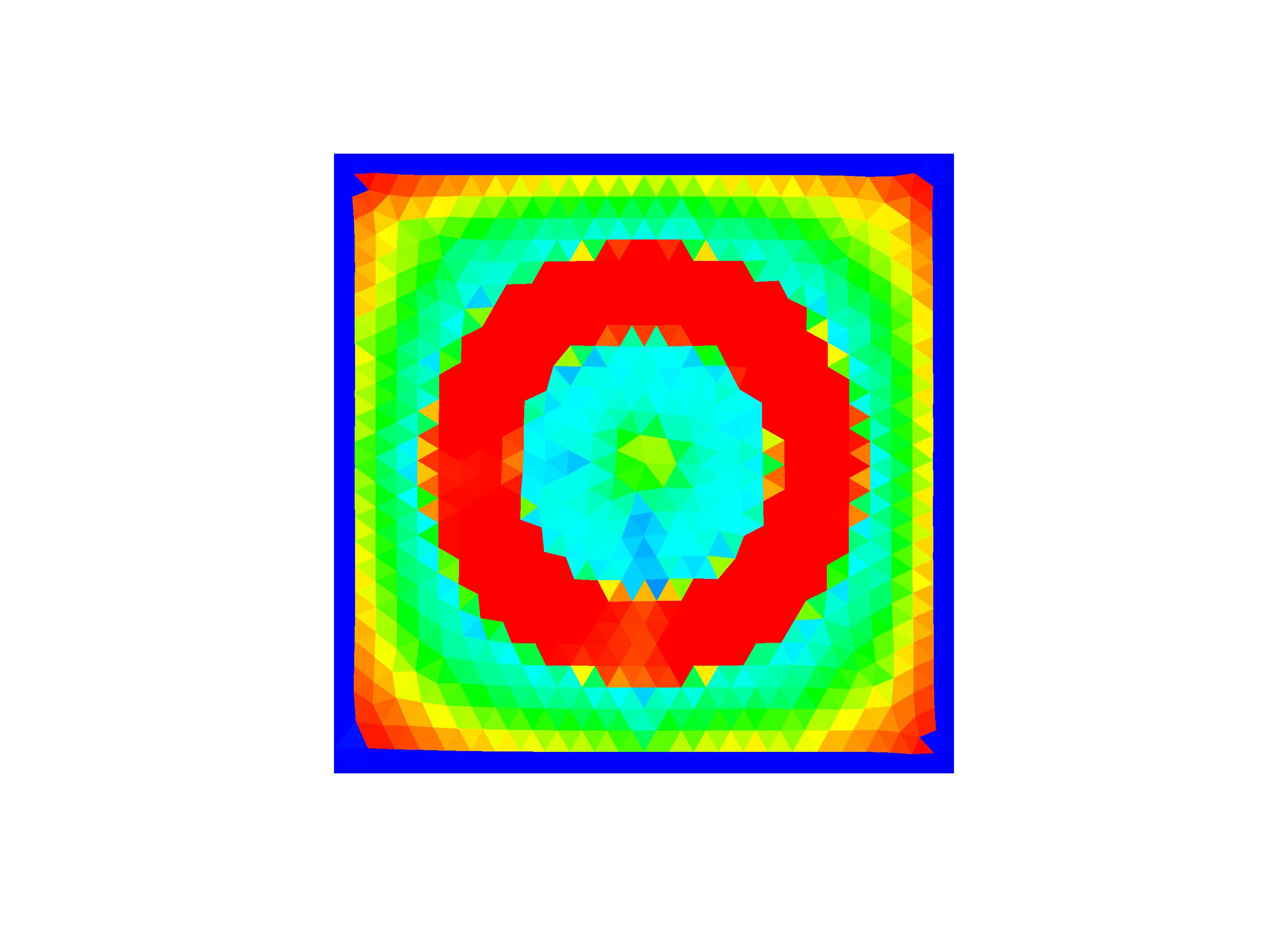}\\
    \end{minipage}
    \hspace{18pt}
    \textcolor{gray}{\vline}
    \hspace{8pt}
    \begin{minipage}[t]{0.16\textwidth}
        \centering
        \hspace{15pt}State\\
        \includegraphics[width=1.3\linewidth,trim={13cm 5cm 13cm 10cm},clip]{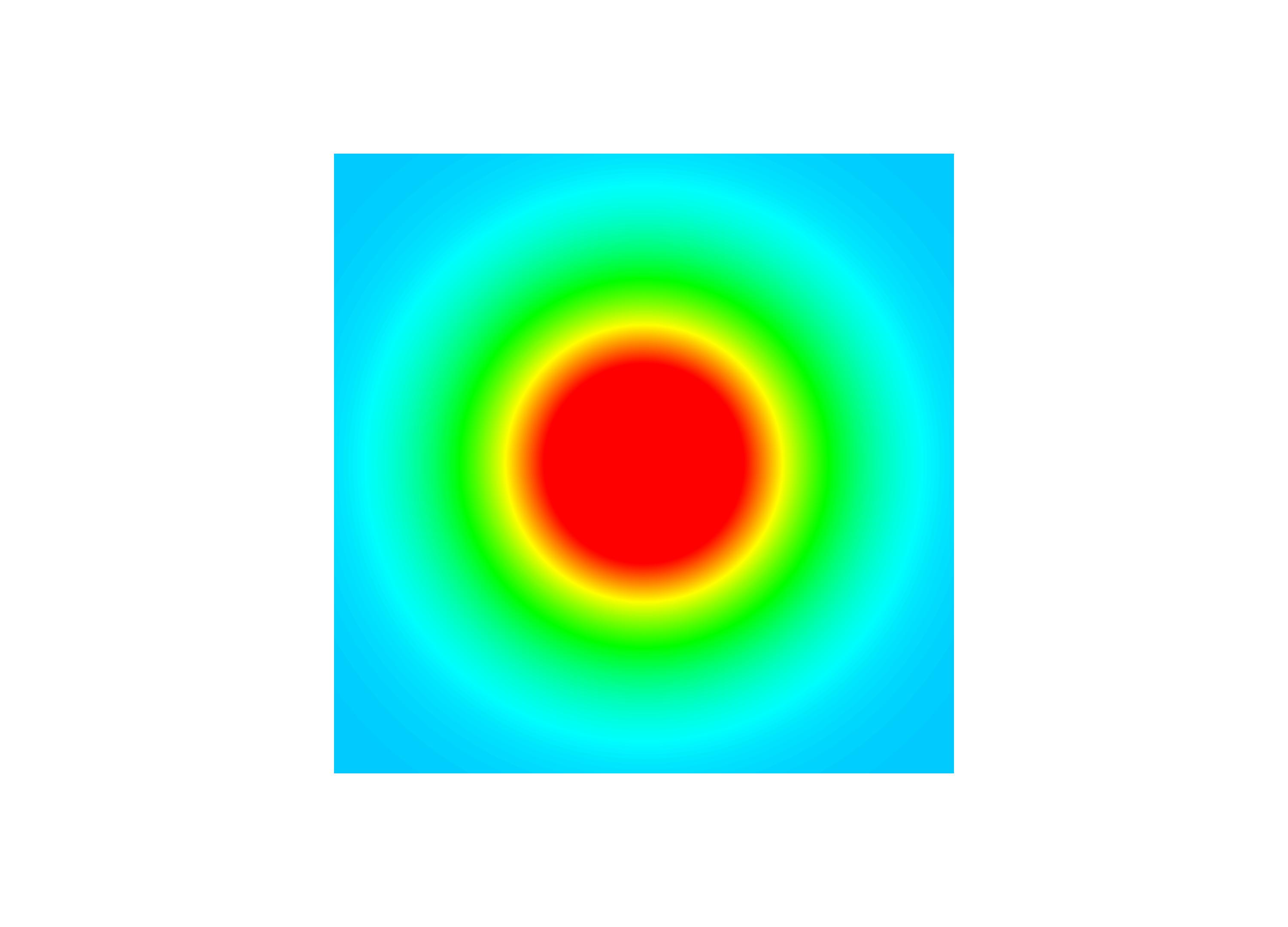}\\
        \includegraphics[width=1.3\linewidth,trim={13cm 5cm 13cm 13cm},clip]{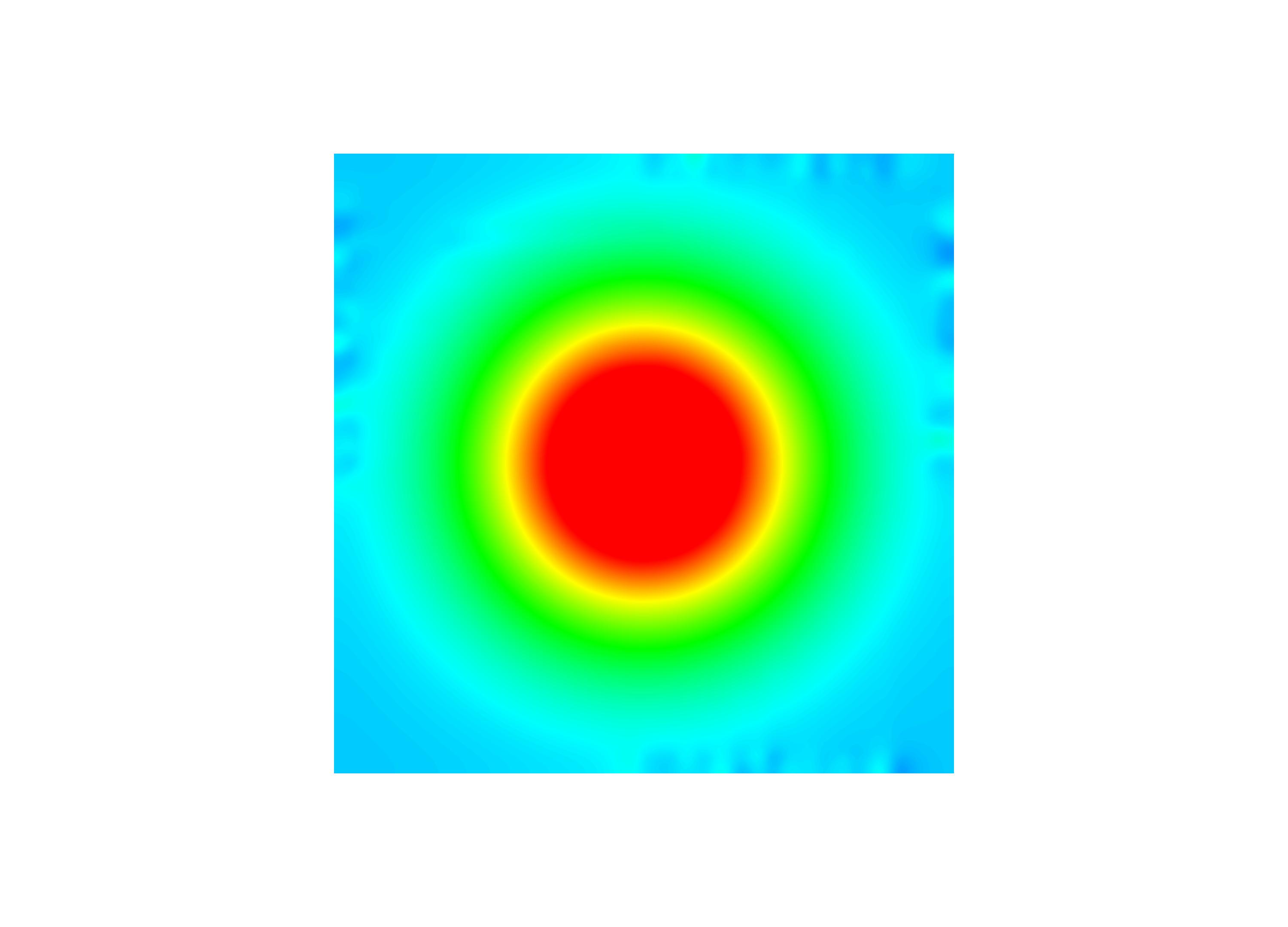}\\
        \includegraphics[width=1.3\linewidth,trim={13cm 5cm 13cm 13cm},clip]{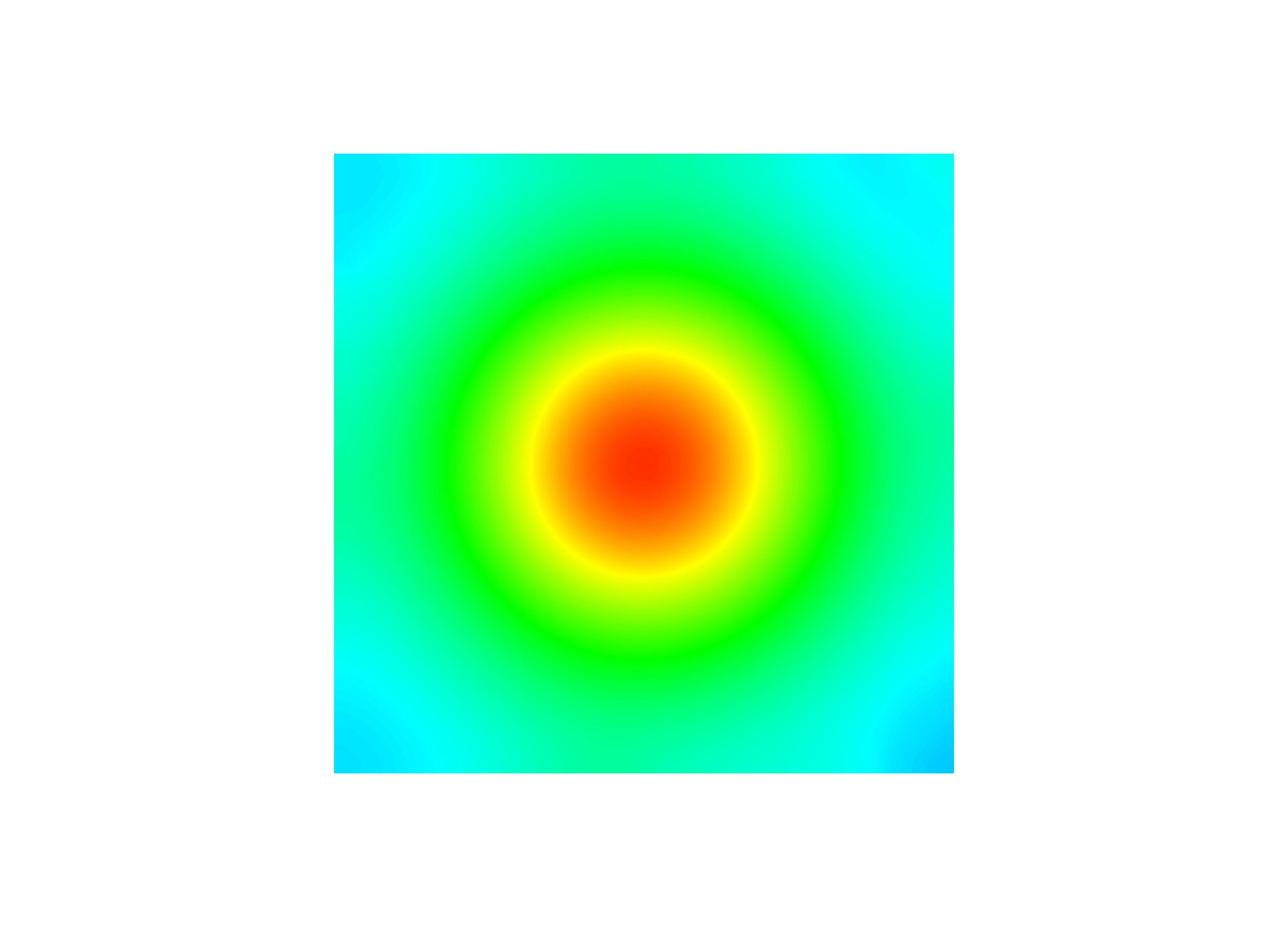}\\
        \includegraphics[width=1.3\linewidth,trim={13cm 5cm 13cm 13cm},clip]{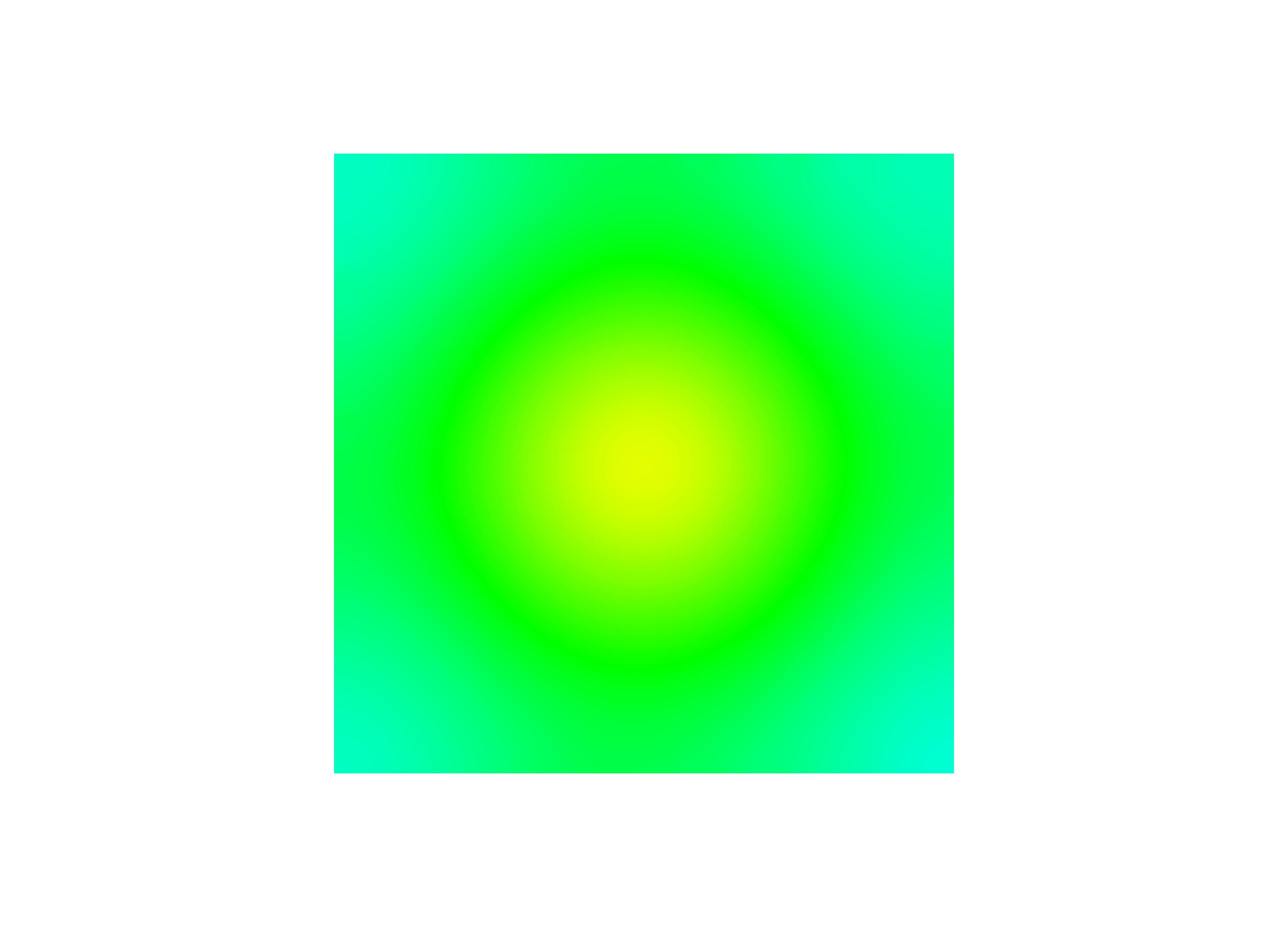}\\
        \includegraphics[width=1.3\linewidth,trim={13cm 5cm 13cm 13cm},clip]{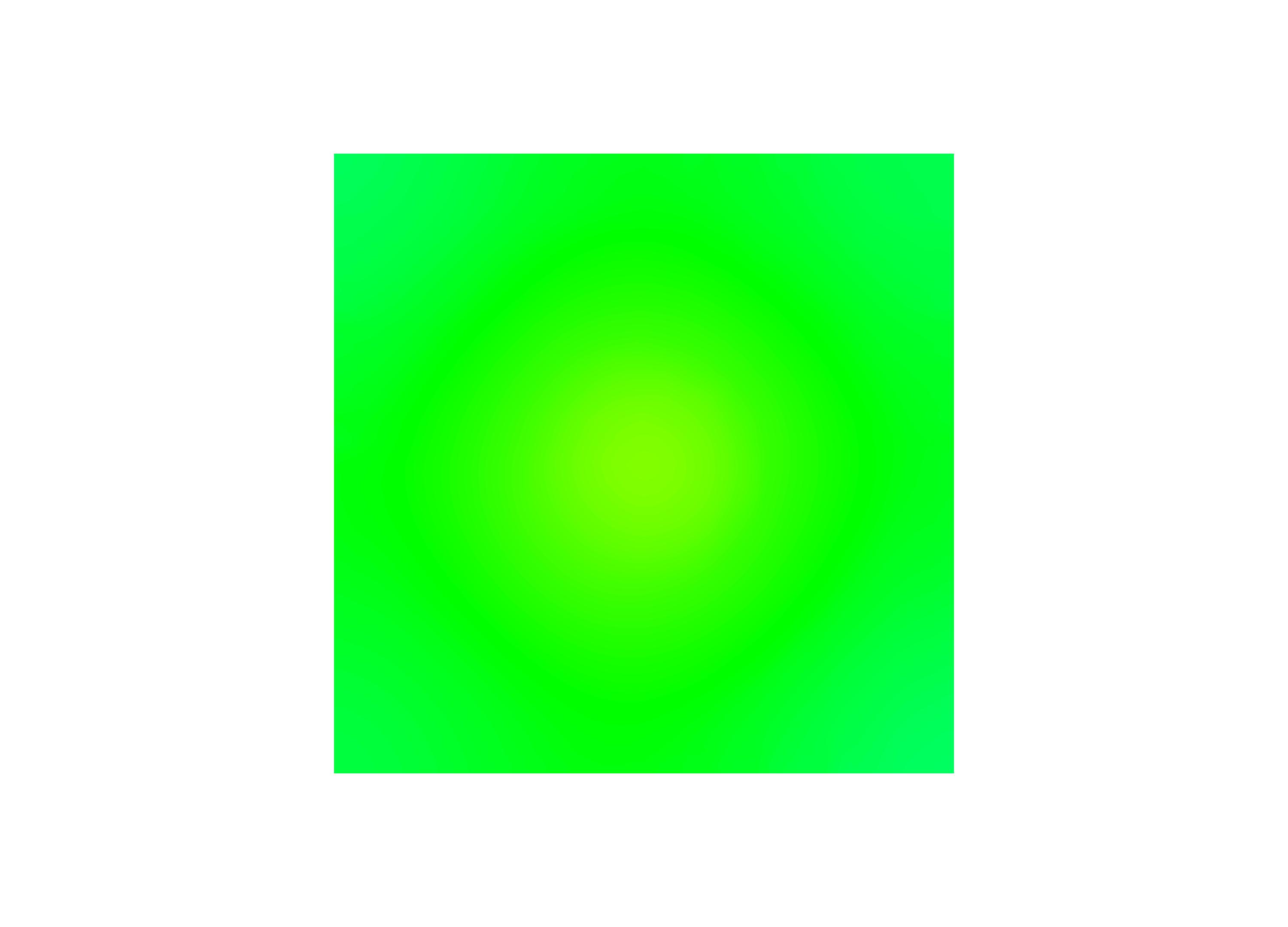}\\
        \includegraphics[width=1.3\linewidth,trim={13cm 5cm 13cm 13cm},clip]{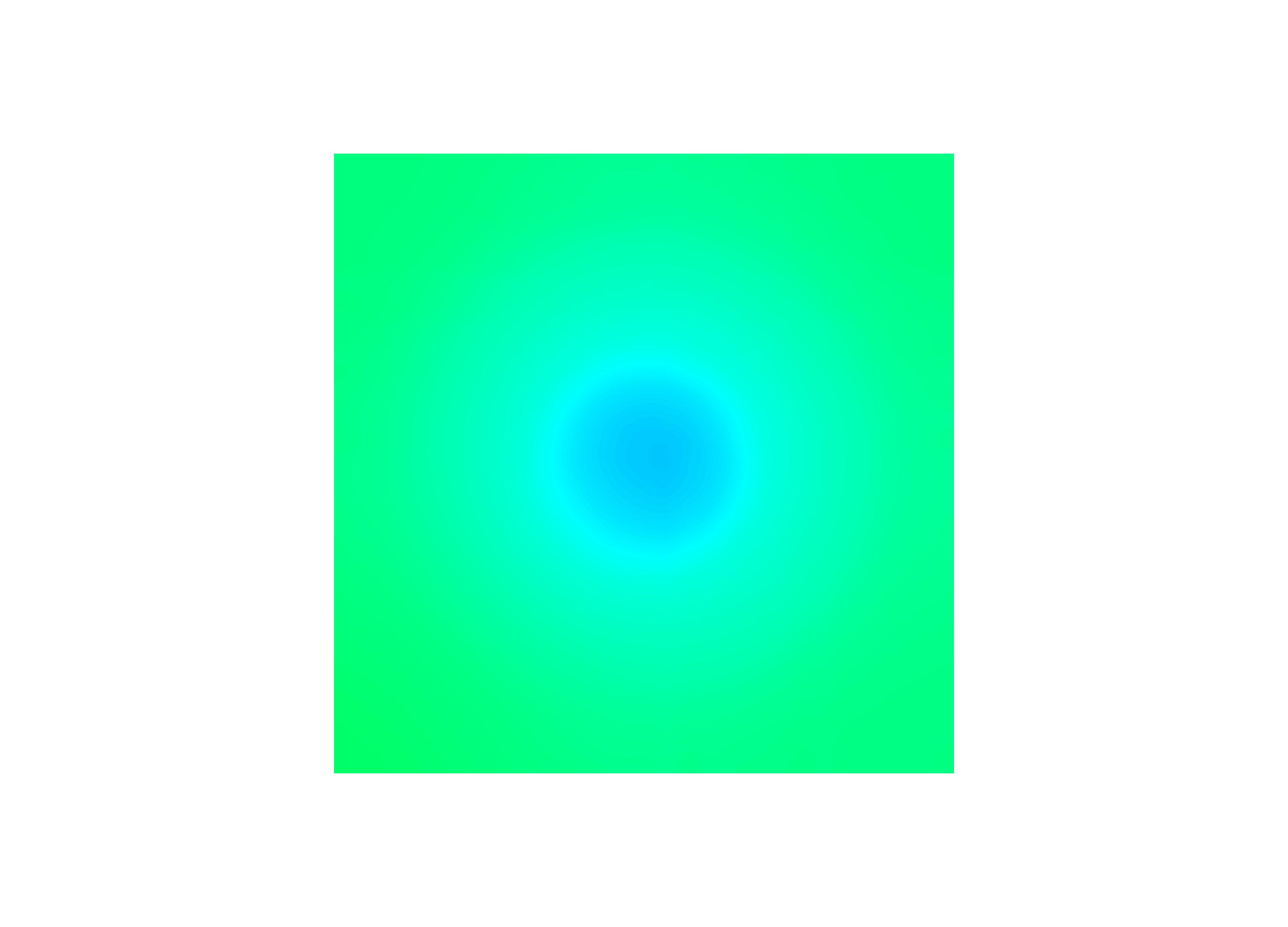}\\
        \includegraphics[width=1.3\linewidth,trim={13cm 5cm 13cm 13cm},clip]{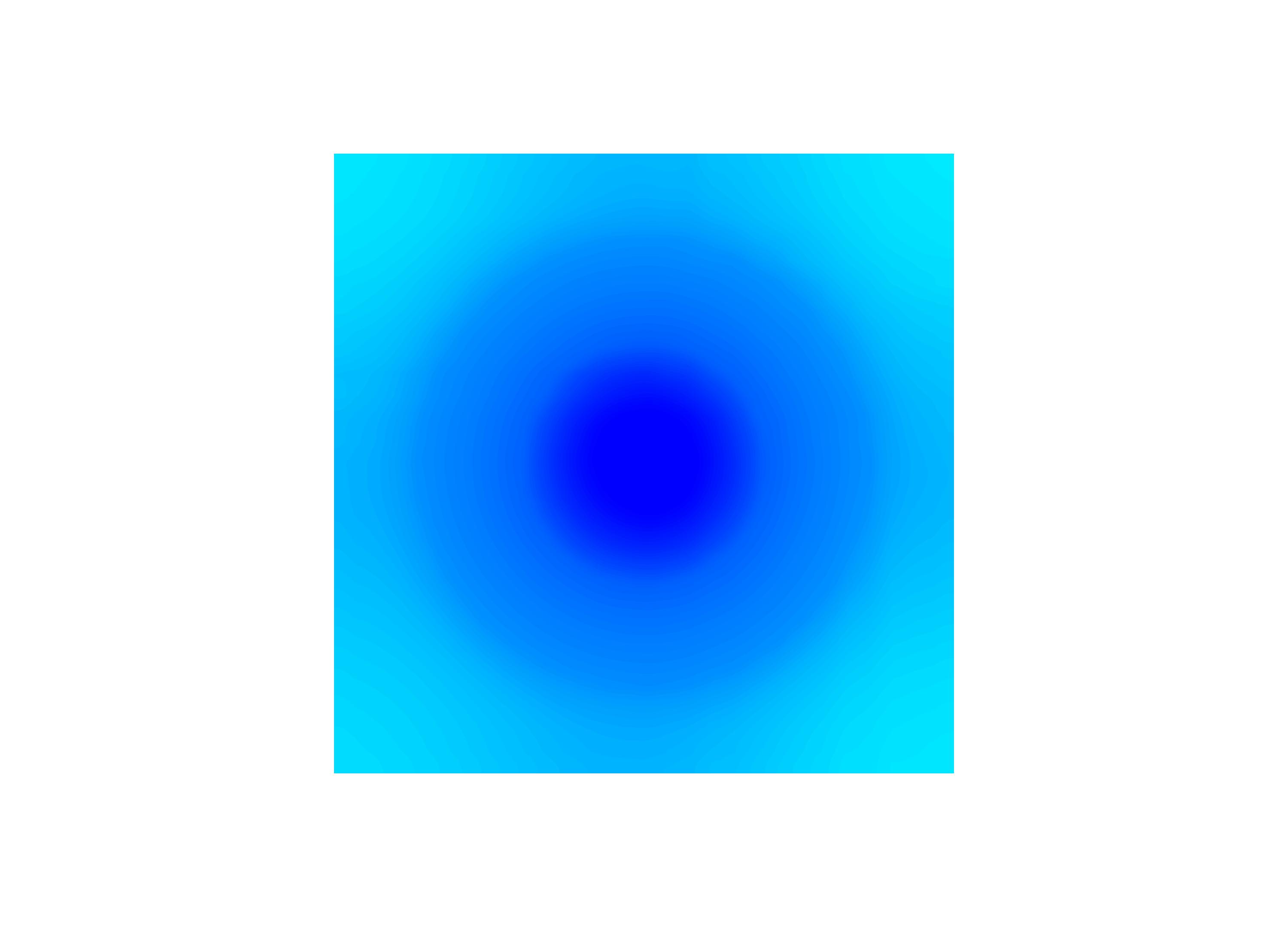}\\
        \includegraphics[width=1.3\linewidth,trim={13cm 5cm 13cm 13cm},clip]{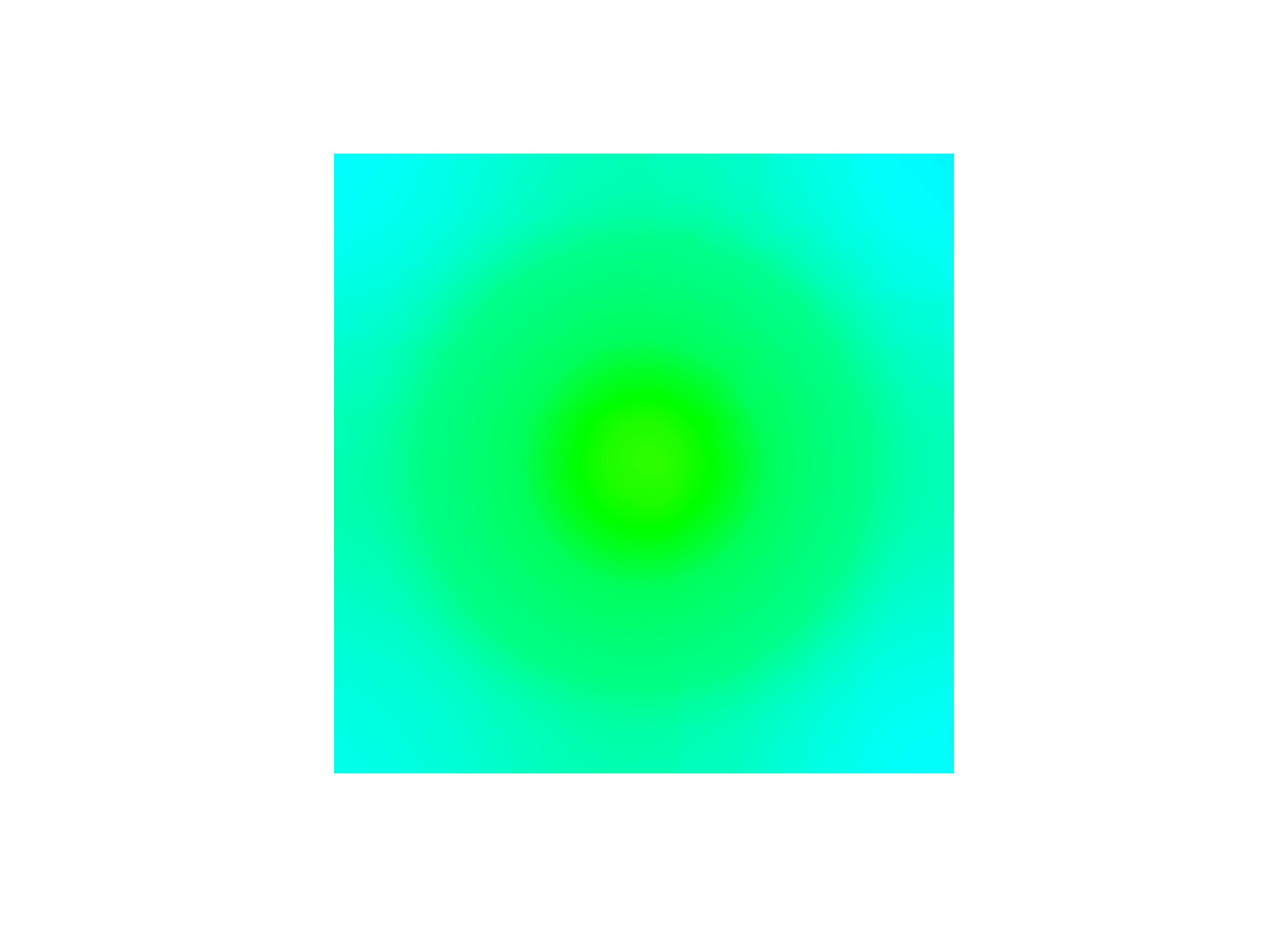}\\
    \end{minipage}
    \hspace{10pt}
    \begin{minipage}[t]{0.16\textwidth}
        \centering
        \hspace{7pt}Parameter\\
        \includegraphics[width=1.3\linewidth,trim={13cm 5cm 13cm 10cm},clip]{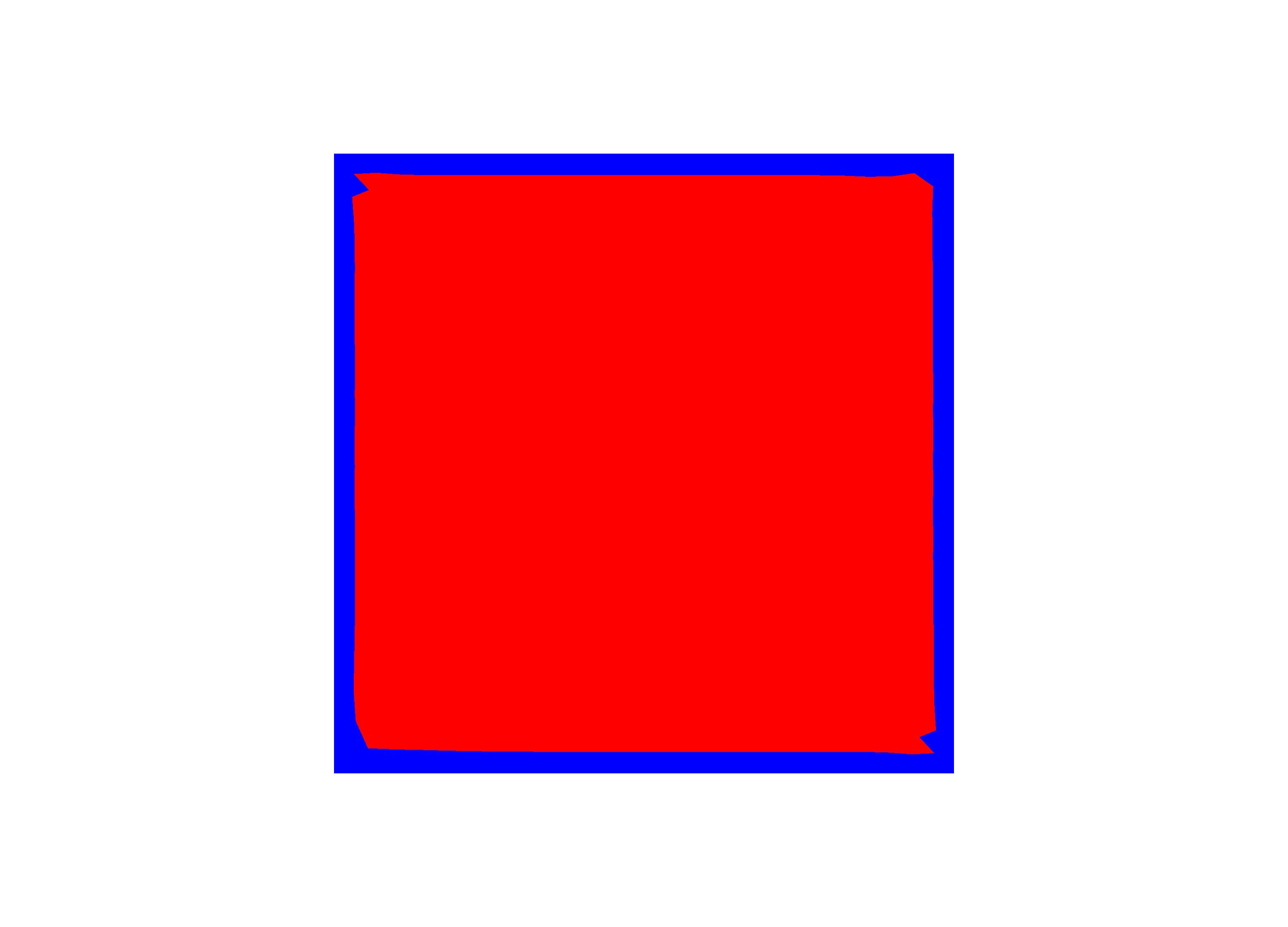}\\
        \includegraphics[width=1.3\linewidth,trim={13cm 5cm 13cm 13cm},clip]{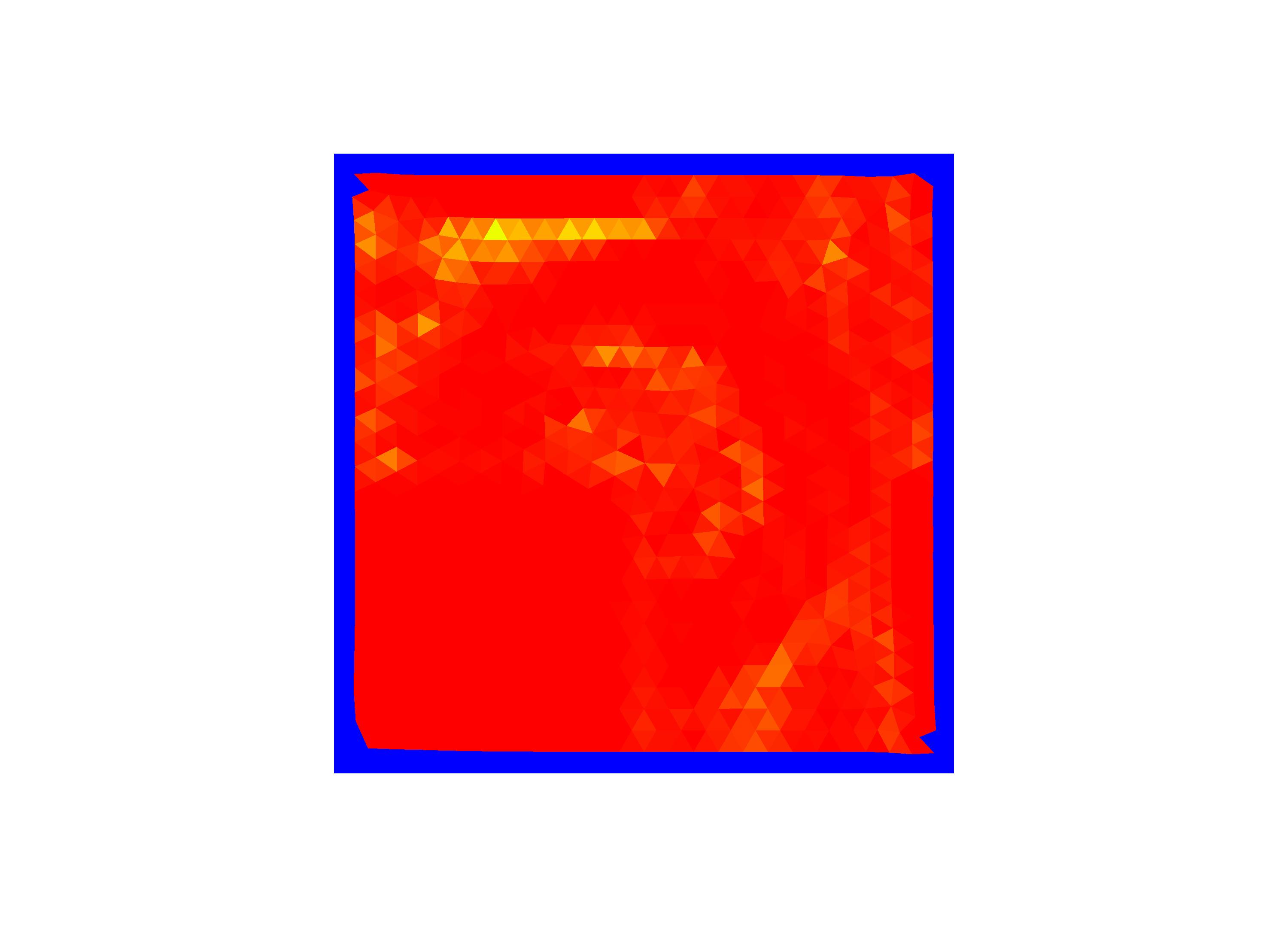}\\
        \includegraphics[width=1.3\linewidth,trim={13cm 5cm 13cm 13cm},clip]{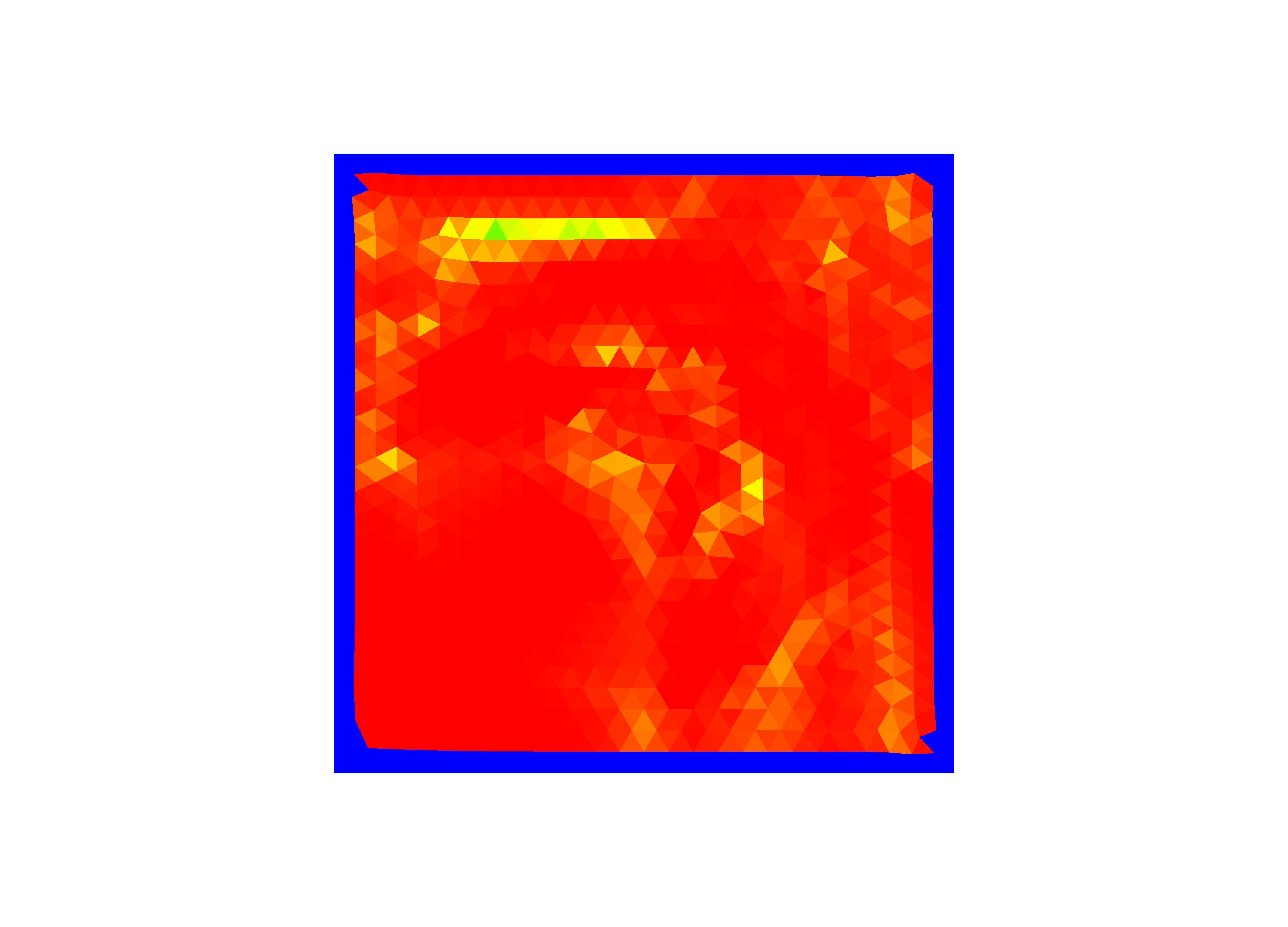}\\
        \includegraphics[width=1.3\linewidth,trim={13cm 5cm 13cm 13cm},clip]{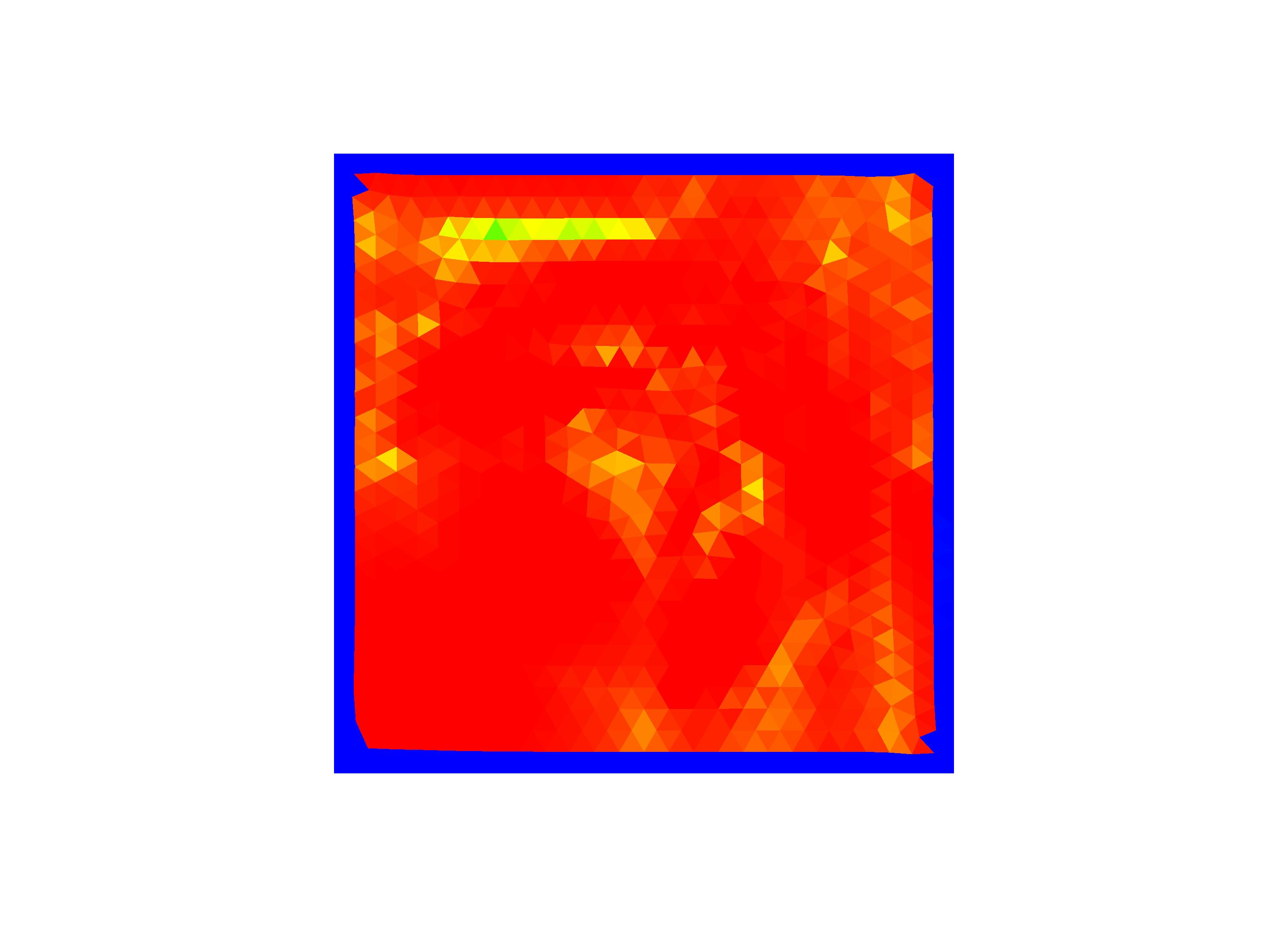}\\
        \includegraphics[width=1.3\linewidth,trim={13cm 5cm 13cm 13cm},clip]{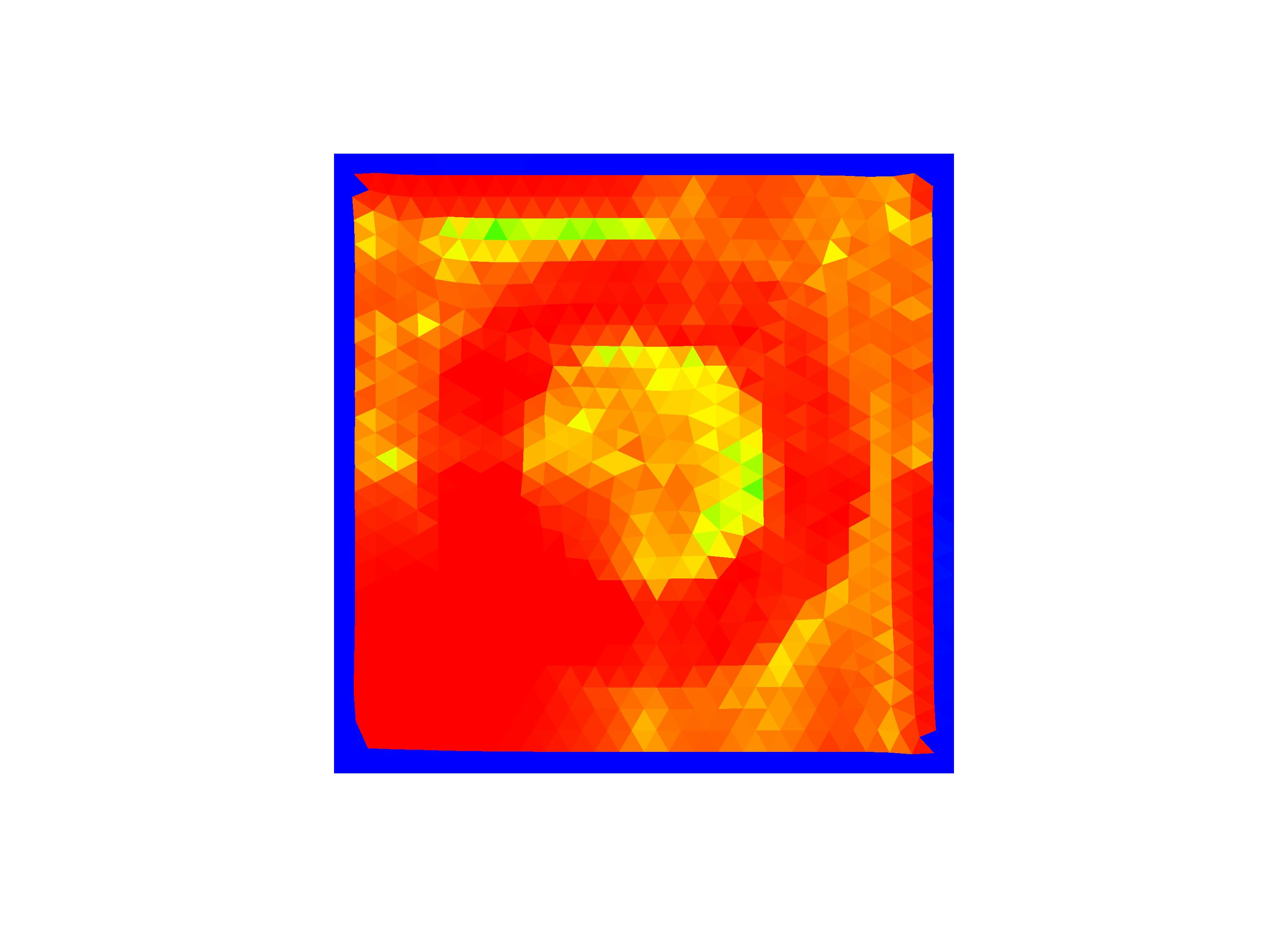}\\
        \includegraphics[width=1.3\linewidth,trim={13cm 5cm 13cm 13cm},clip]{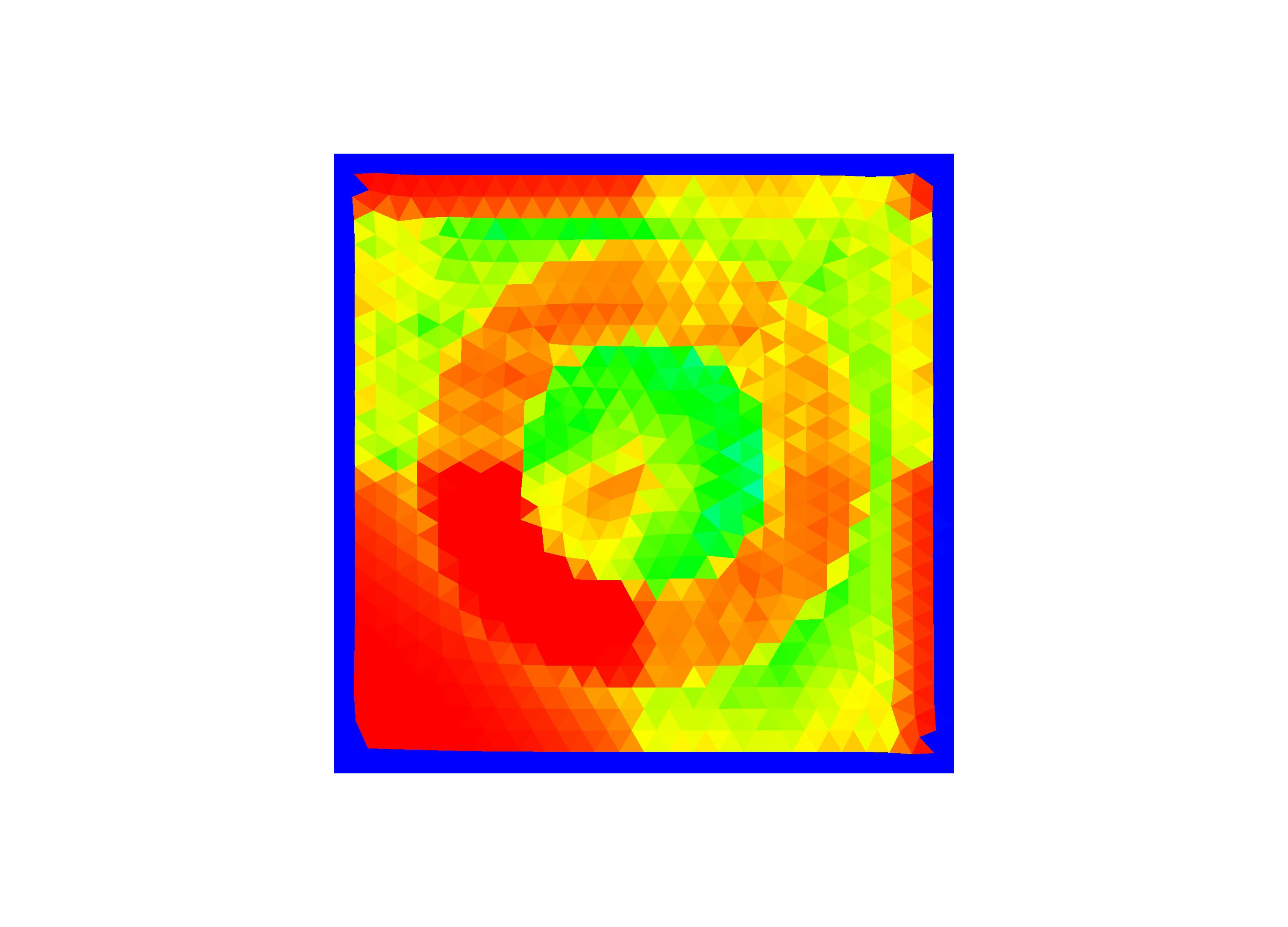}\\
        \includegraphics[width=1.3\linewidth,trim={13cm 5cm 13cm 13cm},clip]{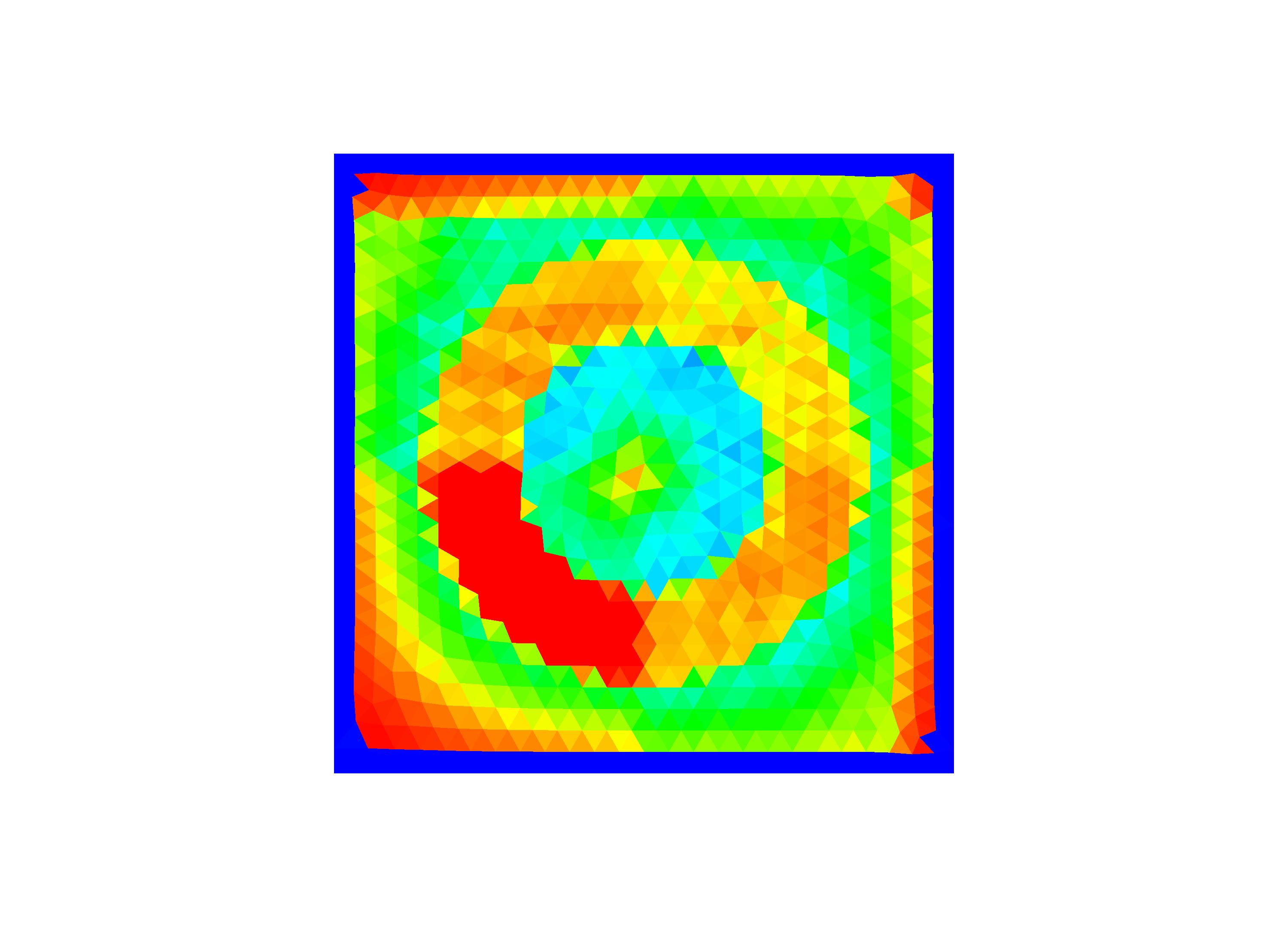}\\
        \includegraphics[width=1.3\linewidth,trim={13cm 5cm 13cm 13cm},clip]{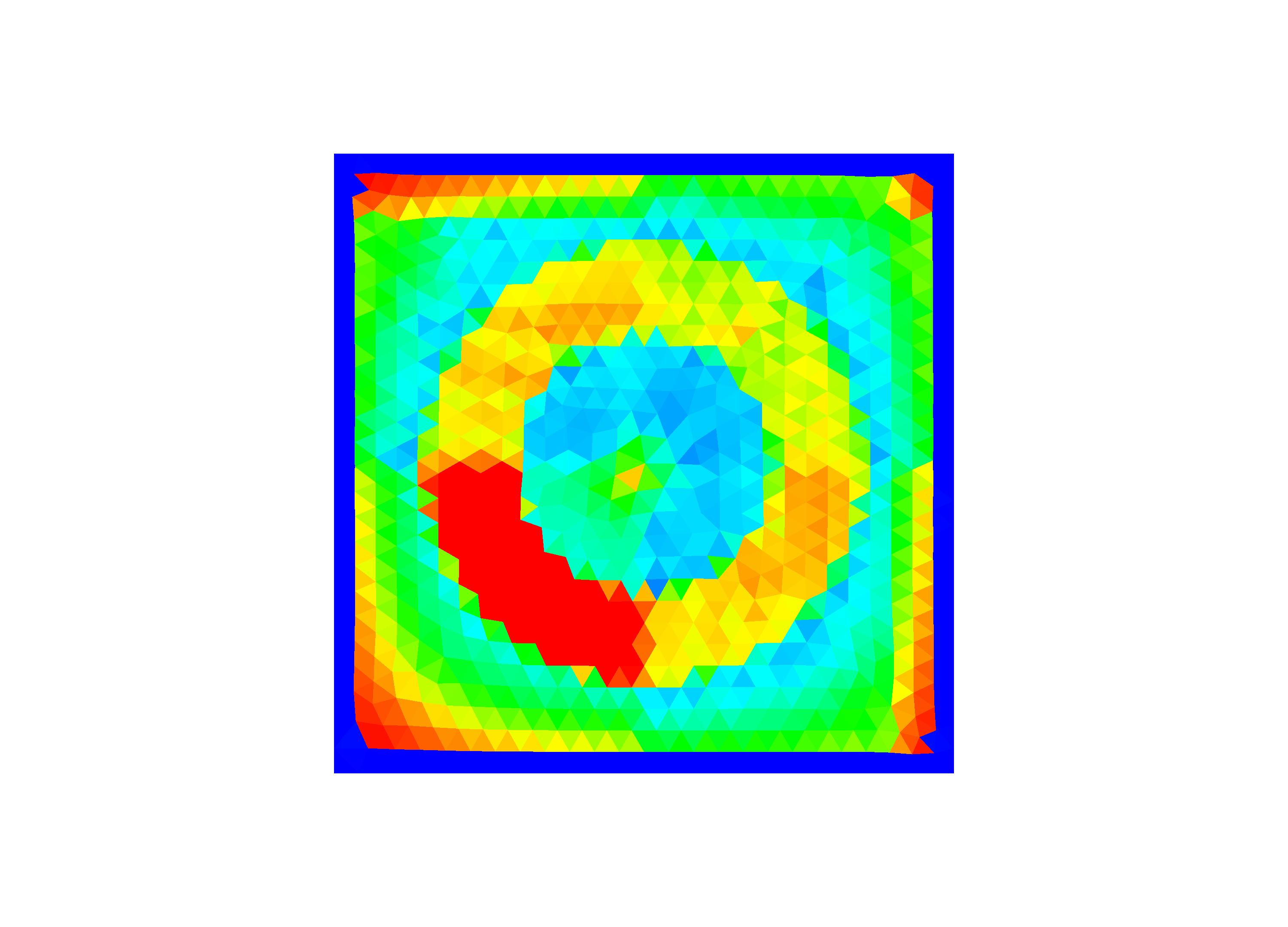}\\
    \end{minipage}
\textcolor{gray}{\hrule   }
\vspace{5pt}
\noindent

\hspace{29.6pt}
\begin{minipage}[c]{0.16\textwidth}
    \hspace{23.5pt}$u^\dagger(T)$\\[1ex]              
    \includegraphics[width=1.3\linewidth,trim={13cm 0cm 13cm 13cm},clip]{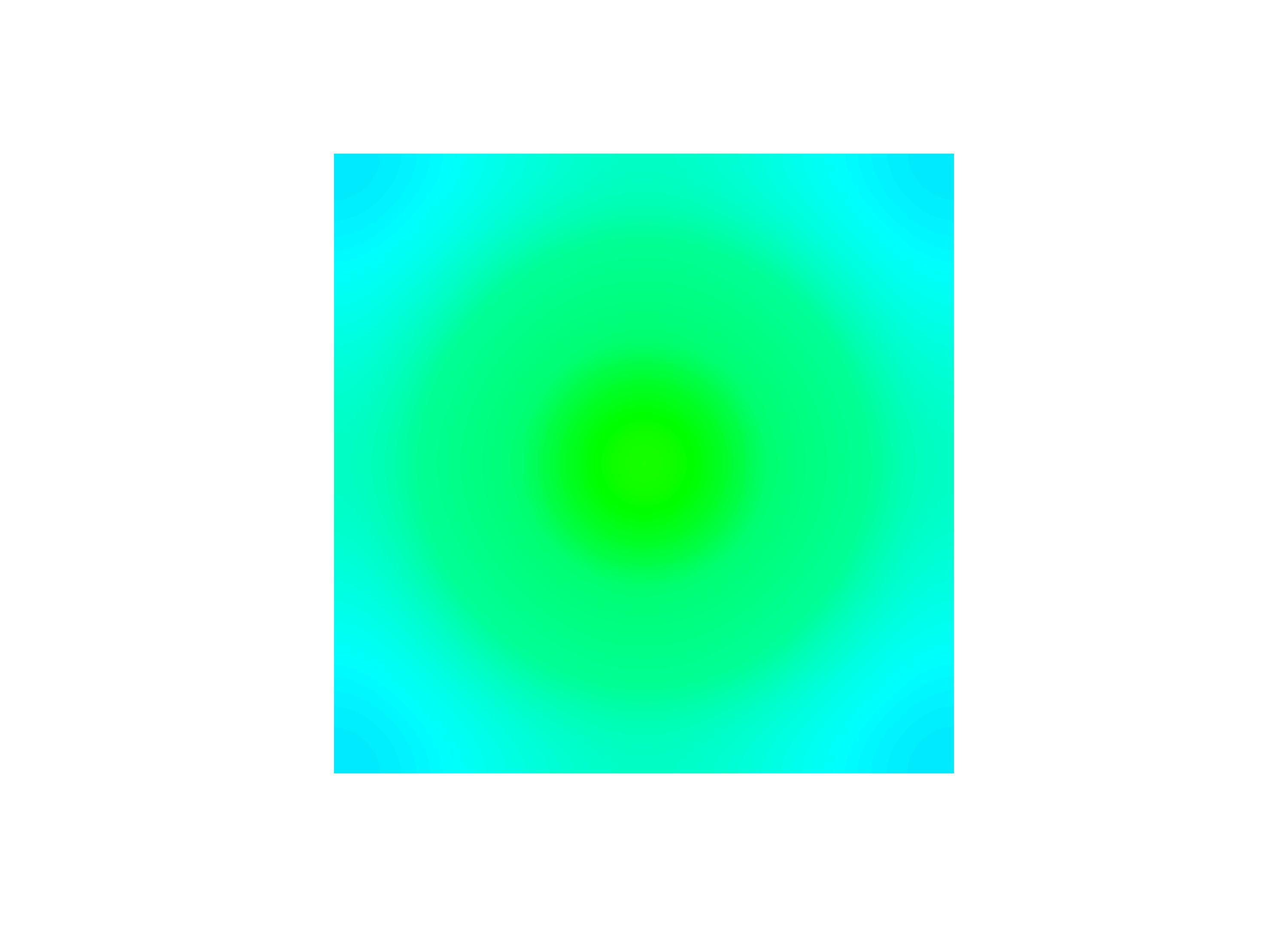}
    \hfill\pgfplotscolorbardrawstandalone[
        colormap/jet,
        colorbar horizontal,
        point meta min=0,
        point meta max=2,
        colorbar style={
            height = 0.3cm,
            width=1.8cm,
            /pgf/number format/fixed,
            /pgf/number format/precision=1,
            tick style={font=\tiny},
            xtick={0, 1, 2},
            xticklabels={0, 1, 2},
        }]
\end{minipage}
\hspace{10pt}
\begin{minipage}[c]{0.16\textwidth}
    \centering\hspace{14pt}$a^\dagger$\\[1ex]
    \includegraphics[width=1.3\linewidth,trim={13cm 0cm 13cm 13cm},clip]{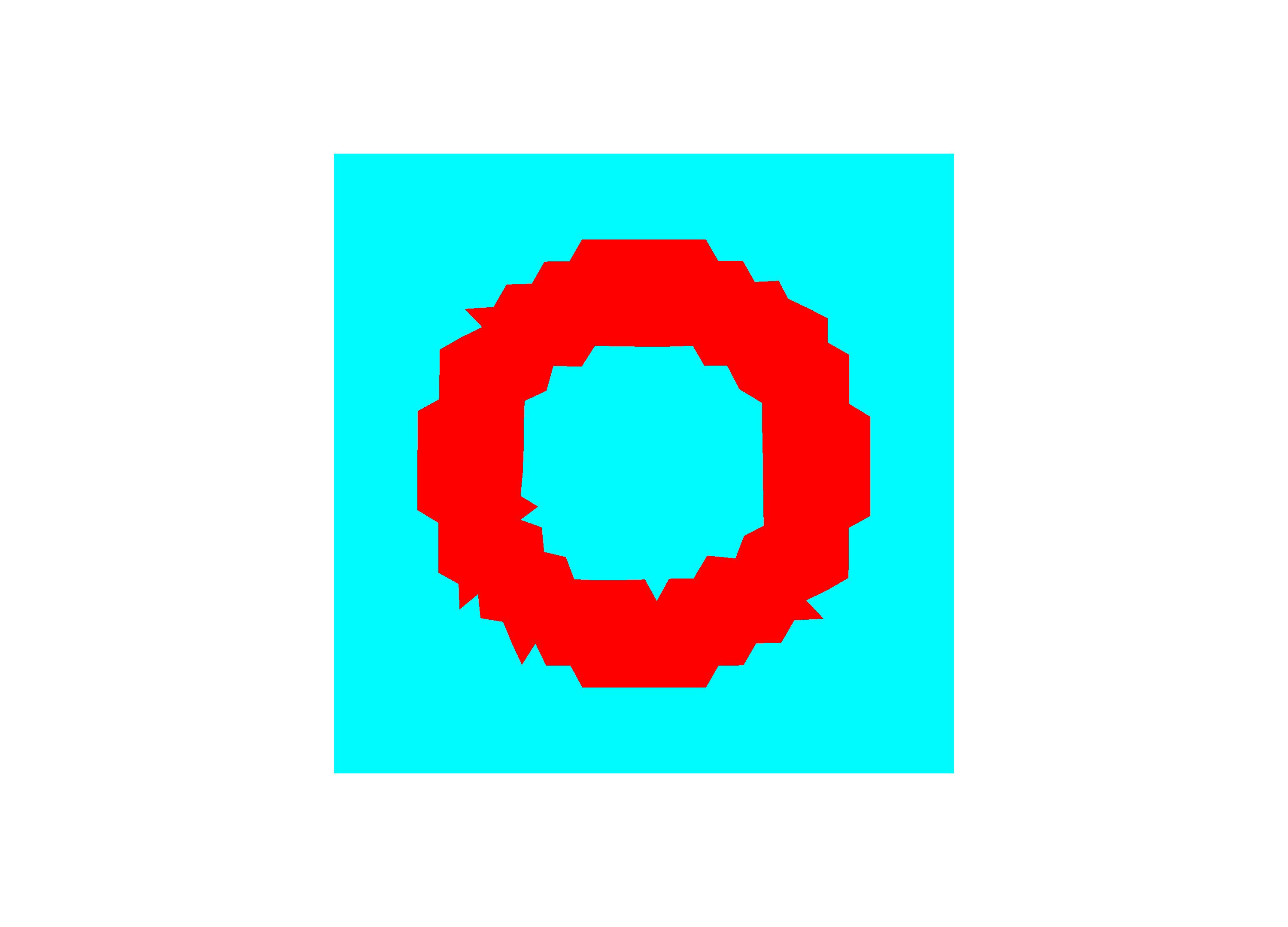}
        \pgfplotscolorbardrawstandalone[
        colormap/jet,
        colorbar horizontal,
        point meta min=0,
        point meta max=1,
        colorbar style={
            height = 0.3cm,
            width=1.8cm,
            /pgf/number format/fixed,
            /pgf/number format/precision=1,
            tick style={font=\tiny},
            xtick={0, 0.5, 1},
            xticklabels={0, 0.5, 1},                    
        }]\\

\end{minipage}%
\hspace{38pt}
\begin{minipage}[c]{0.16\textwidth}
    \centering \hspace{17pt}$u^\dagger(T)$\\[1ex]
    \includegraphics[width=1.3\linewidth,trim={13cm 0cm 13cm 13cm},clip]{fisher_u_exact.jpg}
    \pgfplotscolorbardrawstandalone[
        colormap/jet,
        colorbar horizontal,
        point meta min=0,
        point meta max=2,
        colorbar style={
            height = 0.3cm,
            width=1.8cm,
            /pgf/number format/fixed,
            /pgf/number format/precision=1,
            tick style={font=\tiny},
            xtick={0, 1, 2},
            xticklabels={0, 1, 2},
        }]\\
\end{minipage}%
\hspace{13pt}
\begin{minipage}[c]{0.16\textwidth}
    \centering \hspace{14pt}$a^\dagger$\\[1ex]
    \includegraphics[width=1.3\linewidth,trim={13cm 0cm 13cm 13cm},clip]{fisher_q_exact.jpg}
    \pgfplotscolorbardrawstandalone[
        colormap/jet,
        colorbar horizontal,
        point meta min=0,
        point meta max=1,
        colorbar style={
            height = 0.3cm,
            width=1.8cm,
            /pgf/number format/fixed,
            /pgf/number format/precision=1,
            tick style={font=\tiny},
            xtick={0, 0.5, 1},
            xticklabels={0, 0.5, 1},                    
        }]

\end{minipage}%

\caption{Fisher-KPP. Visualization of evolution of the state $u$ and diffusion parameter $a$ ran with clean data (left) and 3\% noisy data (right).}
\label{fig::fisher_evolution_big}
\end{figure}


    \begin{figure}
    \centering
    \begin{tabular}{cc}
        \begin{subfigure}[b]{0.5\textwidth}
            \captionsetup{labelformat=empty}
            \begin{center}
                \pgfplotscolorbardrawstandalone[
                    colormap/jet,    
                    colorbar horizontal,
                    point meta min= -0.5,
                    point meta max= 0.5,
                    colorbar style={
                        width=0.5\textwidth,
                        /pgf/number format/fixed,
                        /pgf/number format/precision=2,
                        xticklabel style={anchor=north},
                        every axis label/.append style={/pgf/number format/precision=3},
                        yticklabel style={/pgf/number format/none},
                        xtick distance=1/2, 
                        tick style={draw=none}}]
                \includegraphics[width=\linewidth]{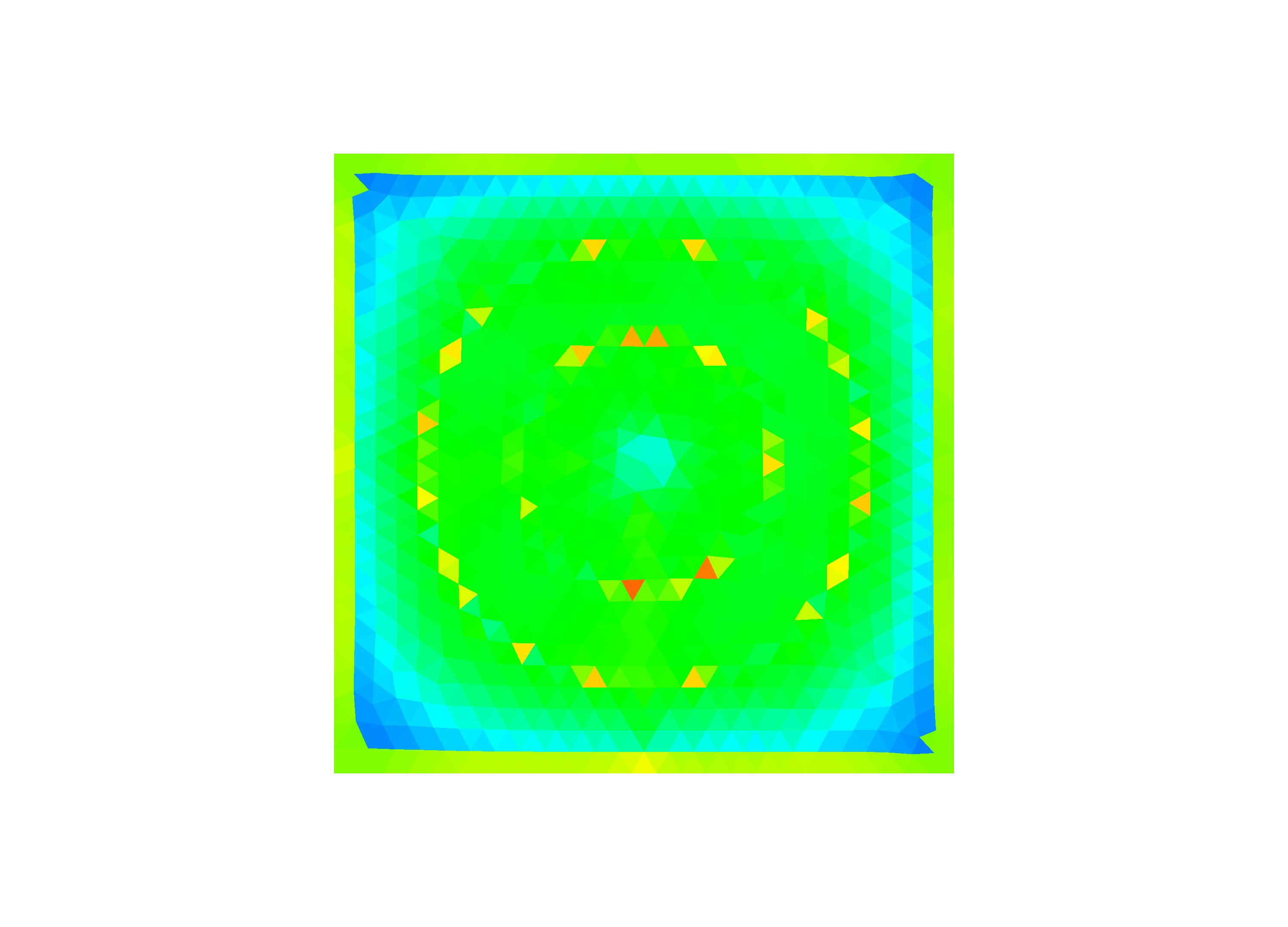} 
            \end{center}
        \end{subfigure}
        &
        \begin{subfigure}[b]{0.5\textwidth}
            \captionsetup{labelformat=empty}
            \begin{center}
                \pgfplotscolorbardrawstandalone[
                    colormap/jet,    
                    colorbar horizontal,
                    point meta min= -0.5,
                    point meta max= 0.5,
                    colorbar style={
                        width=0.5\textwidth,
                        /pgf/number format/fixed,
                        /pgf/number format/precision=2,
                        xticklabel style={anchor=north},
                        every axis label/.append style={/pgf/number format/precision=3},
                        yticklabel style={/pgf/number format/none},
                        xtick distance=1/2, 
                        tick style={draw=none}}]
                \includegraphics[width=\linewidth]{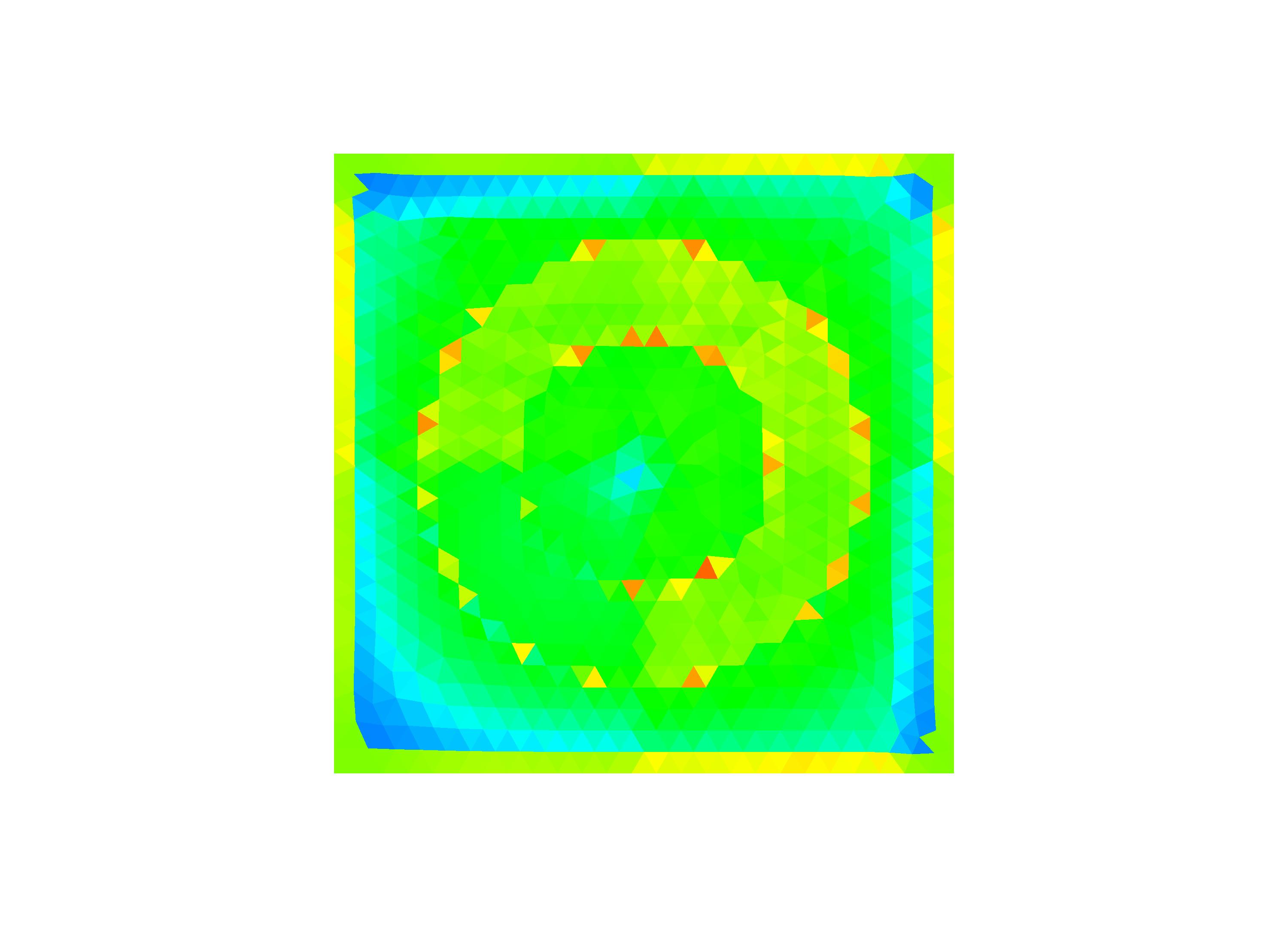} 
            \end{center}
        \end{subfigure}
    \end{tabular}
    \caption{Fisher-KPP. Error field $a(T)-a^\dagger$ for MRAS output $a$ given clean data (left) and 3\% noisy data (right).}
    \label{fig::fisher_diff}
\end{figure}

\section{Nonlinear potential:  parameter nonlinearity}\label{sec::nonlinear_cprob}

Having addressed the case of reaction-diffusion equations, we now turn our attention to PDEs that are nonlinear with respect to the unknown parameter. We distinguish between equations that have an inherently non-linear dependence on the parameter, and those that possess only a reducible non-linearity. One example of the latter is the classical Helmholtz equation with the model $f(c,u):=\Delta u+ c^2 u$ that is nonlinear in the wave number $c$. In practice, one often sets $q:=c^2$, reducing the equation to be linear in $q$, to simplify the reconstruction process. As such, this case could be addressed in a similar manner to that of the two previous examples.

Accordingly, we aim at showcasing the MRAS for an equation with truly irreducibly nonlinear dependence on the parameter. This equation will be the modified Allen-Cahn equation \eqref{eq:allen-cahn}, appearing in Section \ref{sec::allen-cahn}, which is irreducibly non-linear both in state and parameter.

To avoid compounding complexity, we will first take an intermediate step and consider the nonlinear potential problem with the unknown potential $c$. Although \emph{reducible} (by gathering and relabeling $c+c|c|^{2/3})$, the nonlinear dependence is analogous to that in the modified Allen-Cahn equation \eqref{eq:allen-cahn}, enabling us to streamline the analysis of the latter considerably.

As such, on a unit disk $B_\pi(0)$ with radius $\pi$, consider the parabolic PDE
\begin{equation}\label{eq:cproblem}
\begin{aligned}
D_t u -\Delta u +cu+c|c|^{\frac{2}{3}}u &= g \qquad \text{in }I\times\dom:=\I\times B_\pi(0)\\
u|_\bou&=h \qquad\text{in }I\\
u(t=0)&=u_0 \hspace{17pt} \text{in }\dom
\end{aligned}
\end{equation}
with positive boundary condition $h(t,x)\geq\underline{h}>0$ a.e.~in $I\times\bou$. We assume the exact state to be positive and uniformly bounded
\begin{align}\label{expanded_c-positivivty}
0<\underline{z}\leq\utrue(t,x)\leq \overline{z}\quad\text{a.e.~in } I\times\Omega
\end{align}
with some constants $\underline{z}, \overline{z}\in \R^+$.
We refer to \cite[Section 3.1]{Tram_Kaltenbacher_2021} for a detailed proof on unique existence of the true state $u^\dagger\in L^\infty ([0, T); H^2 (\Omega)) \embed L^\infty ([0, T) \times\Omega)$ and on positivity by means of a maximum principle \cite{Pao}.

\subsection{MRAS analysis}

Before diving into the analysis, we recall that $c^\dagger\in Q$ (equivalently, $c^\dagger\in\mathcal{Q}$ as a time-constant) refers to the true, spatially dependent parameter, while $\tilde{c}\in\mathcal{Q}$ is some arbitrary linearization point at which the model nonlinearity $f$ is Gâteaux differentiable. In addition, we employ the trace operator $\tr:H^1(\dom)\to L^2(\partial\Omega)$, and $C_{X\to Y}$, the norm of the continuous Sobolev embedding $X\embed Y$.

\begin{proposition}\label{prop:nonlinear-potential}
For the nonlinear potential problem \eqref{eq:cproblem} with unknown potential $c$ and data $\uobs$, the MRAS \eqref{mras} takes the form
\begin{equation}\label{mras-cproblem}
\begin{aligned}
D_t c +\sigma\left(D_t\uobs-\Delta \uobs+c\uobs+c|c|^\frac{2}{3}\uobs-g\right)
& = \uobs\left(1+\frac{5}{3}|\tilde{c}|^\frac{2}{3}\right)(u-\uobs), \\
D_tu -\Delta \uobs +c\uobs+c|c|^\frac{2}{3}\uobs + \Cc(\|c\|_H) (u-\uobs) & = g, \\
(c,u)(0) & = (c_0,u_0),
\end{aligned}
\end{equation}
where $\sigma=1$, with the state space $U:=H^1(\dom)$ and parameter space $H:=L^2(\dom)$. Above, the linear bounded operator $\Cc(\|c\|_H)$ is such that
\begin{align}\label{c-C}
\Cc(\|c\|_H)(u-z)
:= -\left(\frac{(L^{\tilde{c},\uobs}(\|c\|_H))^2}{2\underline{z}}+M\right)\Delta(u-z)
\end{align}
with the Liptschitz constant $L^{\tilde{c},z}(\|c\|):=\frac{5}{3}\,\overline{z}\, C_{H^1\to L^6}\left(\|c\|^{2/3}_{L^2}+\|c^\dagger\|^{2/3}_{L^2}+\|\tilde{c}\|^{2/3}_{L^2}\right)$, and
$\overline{z}, \underline{z}$ as in \eqref{expanded_c-positivivty} and with any $M>0$. Assumptions \ref{A-lip}, \ref{A-coe}, \ref{A-funcC} hold.
\end{proposition}

\begin{proof}
First of all,  since the PDE \eqref{eq:cproblem} is nonlinear in  $c$, the MRAS \eqref{mras} has $\sigma=1$. This yields $\sigma(D_t\uobs+f(c,z)-g)$ on the left hand side of the first equation in \eqref{mras-cproblem}; we keep $\sigma$ in the formula to better distinguish different terms.
The right hand side of this equation consists of
the adjoint derivative $f'_a(\tilde{c},\uobs)^*(u-\uobs)$,  and is computed as
\begin{align*}
f(c,u) & =-\Delta u +cu+ c|c|^\frac{2}{3}u,\\
f'_c(\tilde{c},\uobs)h & = h\uobs+\frac{2}{3}\text{sign}(\tilde{c})|\tilde{c}|^{-\frac{1}{3}}\tilde{c}h\uobs+|\tilde{c}|^{\frac{2}{3}}h\uobs=h\uobs+\frac{5}{3}|\tilde{c}|^\frac{2}{3}h\uobs,\\
\langle f'_c(\tilde{c},\uobs)h, u-\uobs\rangle & = \int_0^\infty\int_\dom h\uobs\left(1+\frac{5}{3}|\tilde{c}|^\frac{2}{3}\right)(u-\uobs)\d x\d t\\
& = \left\langle h, \uobs\left(1+\frac{5}{3}|\tilde{c}|^\frac{2}{3}\right)(u-\uobs)\right\rangle  =: \langle h, f'_a(\tilde{c},\uobs)^*(u-\uobs)\rangle \quad 
\end{align*}
for all $h\in \Qc$, all $u,\uobs\in\Uc$.
For the second equation of \eqref{mras-cproblem}, clearly $f(c,z)=-\Delta \uobs +c\uobs+c|c|^\frac{2}{3}\uobs$. However, derivation of the operator $\Cc(\|c\|_H)\in\Lc(U,U^*)$ fulfilling Assumption \ref{A-coe} requires careful attention. Firstly, coercivity \ref{A-coe} holds with
\begin{align}\label{c-coercive}
&\lan f(c,z)-f(c^\dagger,z),c-c^\dagger\ran_{Q^*,Q} \nonumber\\
&\quad= \lan cz-c^\dagger z,c-c^\dagger\ran+ \lan c|c|^{\frac{2}{3}}z-c^\dagger|c^\dagger|^{\frac{2}{3}}z,c-c^\dagger\ran\geq \underline{z}\|c-c^\dagger\|_{L^2}^2 \nonumber\\&\quad=:C_\mathrm{coe}\|c-c^\dagger\|_{H}^2 \quad \forall c\in Q
\end{align}
by invoking positivity  \eqref{expanded_c-positivivty} of the true state $\utrue$, the fact that $z=u^\dagger$ as well as monotonicity of the function $c\mapsto c|c|^{\frac{2}{3}}$.

Secondly, for the local Lipschitz property \ref{A-lip} we shall respectively employ the mean value theorem, H\"older's inequality $\int abc\d x \leq \|a\|_{L^2}\|b\|_{L^3}\|c\|_{L^6}$, inequality $(a+b)^{p}\leq |a|^p+|b|^p$ for $p\in(0,1)$, continuous embedding $H^1(\Omega)\embed L^6(\Omega)$ for $\dim(\Omega)\leq3$ and uniformity of the upper bound $\overline{z}$ in \eqref{expanded_c-positivivty}. More precisely, one estimates
\begin{align*}
\|&f(c,z)-f(c^\dagger,z)- f'_c(\tilde{c},z)(c-c^\dagger)\|_{U^*}\\
&=\sup_{\|v\|_U\leq1}\lan f(c,z)-f(c^\dagger,z)- f'_c(\tilde{c},z)(c-c^\dagger),v\ran_{U^*,U}\\
&=\sup_{\|v\|_U\leq1}\left\langle c|c|^{\frac{2}{3}}-c^\dagger|c^\dagger|^{\frac{2}{3}}- \frac{5}{3}|\tilde{c}|^\frac{2}{3}(c-c^\dagger),zv\right\rangle\\
&=\sup_{\|v\|_U\leq1}\left\langle \frac{5}{3} \int_0^1 \left(|c^\dagger+\lambda(c-c^\dagger)|^\frac{2}{3} - |\tilde{c}|^\frac{2}{3}\right)d\lambda (c-c^\dagger) ,zv \right\rangle\\
&\leq\sup_{\|v\|_U\leq1}\sup_{\lambda\in[0,1]} \frac{5}{3}\|z\|_{L^\infty}\|v\|_{L^6}\left\||c^\dagger+\lambda(c-c^\dagger)|^\frac{2}{3} - |\ctil|^\frac{2}{3}\right\|_{L^3}\|c-c^\dagger\|_{L^2}\\
&\leq\sup_{\|v\|_U\leq1} \frac{5}{3}\,\overline{z}\, C_{H^1\to L^6}\|v\|_{H^1}\left(\|c\|^{2/3}_{L^2}+\|c^\dagger\|^{2/3}_{L^2}+\|\tilde{c}\|^{2/3}_{L^2}\right)\|c-c^\dagger\|_{L^2}\\
&=L^{\tilde{c},z}(\|c\|_H)\|c-c^\dagger\|_H,
\end{align*}
verifying the local Lipschitz condition \ref{A-lip} with $U=H^1(\dom)$, $H=L^2(\dom)$.

This enables defining the linear bounded, coercive operator $\cC(\|c\|_H)$ as
\begin{align*}
   &\cC(\|c\|_H)v= \left( \frac{L^{c^\dagger,z}(\|c\|_H)^2}{2C_\mathrm{coe}}+ M\right)(-\Delta_{\tr}v), \quad C_\mathrm{coe}=\underline{z} \text{ as }\eqref{c-coercive}, M>0\\
   &\text{with } \Delta_{\tr} \text{ s.t. } \lan-\Delta_{\tr}u, v\ran = \int_\dom \nabla u\cdot\nabla v\d x + \int_\bou \tr (u)\tr(v)\d S \\
   &\text{then: } \lan \Cc(\|c\|_H) v,v\ran\geq \Mtil\|v\|^2_U, \quad \lan \Cc(\|c\|_H) v,w\ran\leq \Ntil\|v\|_U\|w\|_U 
\end{align*}
fulfilling Assumption \ref{A-C} with  $\Ntil=\Mtil:=\left( L^{c^\dagger,z}(\|c\|_H)^2/(2\underline{z})+ M\right)$ for coercivity and  boundedness. Here, we use $\|u\|_{H^1}:=\sqrt{\|\nabla u\|^2_{L^2}+\|\tr (u)\|^2_{L^2}}$, an equivalent norm to the standard one on $U=H^1(\dom)$. Since $u=z=h$ on the boundary $\bou$, one has $\Delta_{\tr}(u-z)=\Delta(u-z)$, yielding claimed \eqref{c-C} and completing the proof.
\end{proof}

Before proceeding with discretization, we make a useful observation.
\begin{remark}\label{rem:C-scale}
    The choice for the linear bounded, coercive operator $\cC(\|c\|_H)$ is not unique. Indeed, one can define 
\[
    \cC(\|c\|_H):= C\left( \frac{L^{c^\dagger,z}(\|c\|_H)^2}{2C_\mathrm{coe}}+M\right)(-\Delta_{\tr})
\]
for any constant $C\geq 1$, then scale the bounds $\Mtil, \Ntil$ accordingly.

In implementation, we specifically chose $C:=4\left(\frac{3}{5}\right)^2>1$ and $M:=1/C$ for convenience. Hence \begin{align*}
        &\cC(\|c\|_H)(u-v)\\
        &\quad=\left(\frac{2}{\underline{z}}\left[C_{H^1\to L^6}\overline{z}(\|c_{n}\|_{L^2}^{2/3}+\|c^\dagger\|_{L^2}^{2/3}+\|\tilde{c}\|_{L^2}^{2/3})\right]^2+1\right)(-\Delta)(u-v).
    \end{align*}
\end{remark}

The following discrete form outlines our implementation strategy. In contrast to the previous examples, we formulate the discretized MRAS in an \emph{incremental form} in order to conveniently treat the nonhomogeneous boundary. Here and in what follows, $\nvec$ denotes the outward normal vector on $\bou$.

\begin{corollary}\label{corr:mras-cproblem-weak}

The weak form of the MRAS \eqref{mras-cproblem} under a semi-implicit Euler scheme, with  $\cC(\|c\|_H)$ as in Remark \ref{rem:C-scale} and with $\sigma=1$, reads as
\begin{alignat}{2}
\label{pottential-dis-c}
    \int_\Omega( c_{n+1}&-c_n)\,s\d x  + \sigma\Delta t \int_\Omega\left((c_{n+1}-c_n)\uobs_{n+1}+|c_{n}|^\frac{2}{3}(c_{n+1}-c_n)\uobs_{n}\right)s\d x \notag\\ 
    = & +\Delta t\int_\Omega  \uobs_{n}\left(1+\tfrac{5}{3}|{\tilde{c}}|^\frac{2}{3}\right)(u_{n}-\uobs_{n})s \d x \\
    &-\sigma\Delta t \int_\dom \nabla \uobs_{n+1}\cdot\nabla s\d x+\sigma\Delta t\int_{\partial \Omega}\nabla \uobs_{n+1}\,s \cdot \nvec\d S\notag\\&-\sigma\Delta t \int_\Omega\left(D_t \uobs_{n+1}+{c_{n}\uobs_{n+1}+|c_{n}|^\frac{2}{3}c_{n}\uobs_{n}}-g_{n+1}\right)\,s \d x, \\
    \label{pottential-dis-u}
    \int_\Omega (u_{n+1}&-u_n)v \d x +\Delta t\int_\Omega \left((c_{n+1}-c_n)+|c_n|^\frac{2}{3}(c_{n+1}-c_n)\right)\uobs_{n+1}v\d x\notag\\
    & +\Delta t\int_\Omega \,C_{c_n}\nabla (u_{n+1}-{u_n})\cdot  \nabla v\d x\notag\\
    = & -\Delta t \int_\Omega \left({c_n\,\uobs_{n+1}+|c_n|^\frac{2}{3}c_{n}\uobs_{n+1}}-g_{n+1}\right)\,v\d x\\
    &-\Delta t\int_\Omega\nabla \uobs_{n+1}\cdot\nabla v\d x  +\Delta t\int_{\partial \Omega}\nabla \uobs_{n+1}\,v\,\cdot \nvec\d S \nonumber\\
    &+\Delta t\int_\Omega \,C_{c_n}\nabla \left(\uobs_{n} - {u_n}\right)\cdot\nabla v\d x, \nonumber\\
    (c,u)(0) & = (c_0,u_0)
    \end{alignat} for any $v\in U=H^1(\dom)$, $s\in H=L^2(\dom)$ and with the constant  $C_{c_n}:= C^2_{H^1\to L^6}\frac{2\overline{z}^2}{\underline{z}}\left(\|c_{n}\|_{L^2}^{2/3}+\|c^\dagger\|_{L^2}^{2/3}+\|\ctil\|_{L^2}^{2/3}\right)^2+1$. 
\end{corollary}

\begin{proof}
First of all, the constant $C_{c_n}$ results from the weak form of $\mathcal{C}(\|c_n\|_H)$, utilizing integration by parts, Proposition \ref{prop:nonlinear-potential} and Remark \ref{rem:C-scale}. We now write the MRAS in the semi-implicit form as
\begin{align*}
\frac{c_{n+1}-c_n}{\Delta t} &+ \sigma\Big(D_t \uobs_{n+1}+f(c_{n+1};c_{n},\uobs_{n+1})-g_{n+1}\Big)=f'_c({\tilde{c}},\uobs_n)^*(u_n-\uobs_n), \\
\frac{u_{n+1}-u_n}{\Delta t} &+ f(c_n;c_{n+1},\uobs_{n+1}) + \cC(\|c_n\|)(u_{n+1}-\uobs_{n+1})  = g_{n+1},
\end{align*}
recalling that $f(c_{n+1};c_{n},\uobs_{n+1})$ should be chosen as a reformulation of $f(c,z)$ that it is linear $c_{n+1}$ and nonlinear in $c_n$. Explicitly, the choice
\begin{align}\label{potential-f-discrete}
    f(c_{n+1};c_{n},\uobs_{n+1}) :=
    -\Delta\uobs_{n+1} + c_{n+1}\uobs_{n+1} + c_{n+1}|c_n|^{2/3}\uobs_{n+1}
\end{align}
leads to 
\begin{align}
\label{pottential-dis1-c}
    (c_{n+1}&-c_n) +\sigma \,\Delta t\,\left(c_{n+1}+|c_n|^\frac{2}{3}c_{n+1}\right)\uobs_{n+1} \\
&=\Delta t\,\uobs_{n}\left(1+\tfrac{5}{3}|{\tilde{c}}|^\frac{2}{3}\right)(u_n-\uobs_{n}) \notag
+\sigma\Delta t\Delta z_{n+1} -\sigma\,\Delta t(D_t\uobs_{n+1}-g_{n+1}) \nonumber\\
\label{pottential-dis1-u}
(u_{n+1}&-u_n)  +\Delta t ( c_{n+1} + |c_n|^\frac{2}{3}c_{n+1})\uobs_{n+1}
- \Delta t\,C_{c_n}\Delta u_{n+1} \\
&=\Delta t\,g_{n+1}+ \Delta t\,\Delta \uobs_{n+1} - \Delta t\,C_{c_n}\Delta \uobs_{n+1} \notag
\end{align}
Writing these expressions into an incremental form for $c_{n+1}-c_n$ and $u_{n+1}-u_n$ such that the nonhomogeneous boundary terms of $u$ and $z$ cancel, we add $\sigma\Delta t(-c_{n}\uobs_{n+1}-|c_{n}|^\frac{2}{3}c_{n}\uobs_{n})$ to both sides of \eqref{pottential-dis1-c}, while for \eqref{pottential-dis1-u}, we add $\Delta t(-c_n\,\uobs_{n+1}-|c_n|^\frac{2}{3}c_{n}\uobs_{n+1})$  and $\Delta t C_{c_n}\Delta u_n$ to both sides.

In the last step, testing the resulted equations with $(v,s)\in U\times H$ and perform partial integrations for the Laplacians, we arrive at \eqref{pottential-dis-c}-\eqref{pottential-dis-u}. Note that, except for the data term $\Delta z_{n+1}$, all other terms have zero boundary thanks to the incremental form. 

\end{proof}

\subsection{Numerical results}

\paragraph{Data  and discretization}

In this example, we predetermine a positive state ${u}_0$ that admits no simple closed form representation, but can be seen on the top left of Figure \ref{fig::nc_evolution_big}, and define $u^\dagger(\cdot,t):=u_0 \tfrac{6 - t}{6}+\tfrac{t}{6}$ for $t\in[0,T]$.

For the parameter, we consider the cone-shaped function $c^\dagger(x) = \pi-\sqrt{x_1^2+x_2^2}$ as the exact potential. 
From $u^\dagger$ and $c^\dagger$, we compute the corresponding source $g=D_t \utrue+f(\qtrue,\utrue)=D_t \utrue-\Delta \utrue +c^\dagger\utrue+\utrue\,|c^\dagger|^{\frac{2}{3}}c^\dagger$.
We take this opportunity to remark that for the case of noise-free data, i.e.~$\uobs=\utrue$, this leads to the term $\Delta \utrue$ in $g$ canceling out the $\Delta z$ term in \eqref{pottential-dis-u}. Moreover, the linearization point and initial guess for the parameter are chosen as $\ctil:=c_0:=0$. The MRAS reconstruction is now performed both noise-free, and with $5\%$ relative noise in the measured data.

Regarding discretization, a FEM mesh with size $h_\mathrm{max} = 0.1$ and temporal resolution $\Delta t= 0.001$ were implemented. For the state space, one again has $U_h\subset H^1_0(\dom)$ of polynomial order $k=3$, and $H_h\subset L^2(\dom)$ of order $k=0$ for the parameter space. Table \ref{tab:nonlinear-c} details all hyperparameters. 

\begin{table}
    \centering
    \begin{tabular}{c|c}
         Spatial domain, mesh size & $B_0(\pi)$, $h_\mathrm{max} = 0.1$\\
         $\#$dofs for $U_h$, $\#$dofs for $Q_h$  &34279, 7552\\
         Max time, time step, \#steps & $T=5$, $\Delta t=0.001$, 5000 steps\\
         Source term & $D_t \utrue-\Delta \utrue +c^\dagger\utrue+\utrue\,|c^\dagger|^{\frac{2}{3}}c^\dagger$ \\
         Relative noise levels & 0\%, 5\%\\
    \end{tabular}
    \caption{Setup for the nonlinear potential problem.}
    \label{tab:nonlinear-c}
\end{table}

\paragraph{Numerical results}
Figure \ref{fig::nc_evolution_big} displays the reconstruction result with noise-free data (left columns) and data with 5\% relative noise (right columns). In both cases, the MRAS \eqref{mras} begins from a zero initial guess of the parameter, meaning without prior information about the ground truth. Some characteristics of the state $u$ possibly affect the evolution of the parameter estimate $c$, e.g.~the gap-like structure in $c$ at early times. Once overcome, convergence towards the true state $c^\dagger$ is rather rapid. Meanwhile, the state reconstruction converges rather well.

When 5\% random noise is introduced to the observed data $\uobs$, the reconstruction somewhat deteriorates, but well captures qualitative aspects of the true solution, both for the parameter $c^\dagger$ and the true state $\utrue$. The error fields in Figure \ref{fig::nc_diff} again confirm acceptably good quality; in particular, no qualitative relationship can be seen between these error fields and the features of the ground truths.

\begin{figure}
    \centering
    \begin{minipage}[t]{0.095\textwidth}
        \vspace{7ex}
        \raggedright
        t=0.005\\[9.6ex]
        t=0.025\\[9.6ex]
        t=0.05\\[9.6ex]
        t=0.1\\[9.6ex]
        t=0.2\\[9.6ex]
        t=0.5\\[9.6ex]
        t=1.5\\[9.6ex]
        t=5
    \end{minipage}
    \begin{minipage}[t]{0.16\textwidth}
        \centering
        \hspace{15pt}{State} \\
        \includegraphics[width=1.3\linewidth,trim={10.1cm 5cm 10.1cm 5cm},clip]{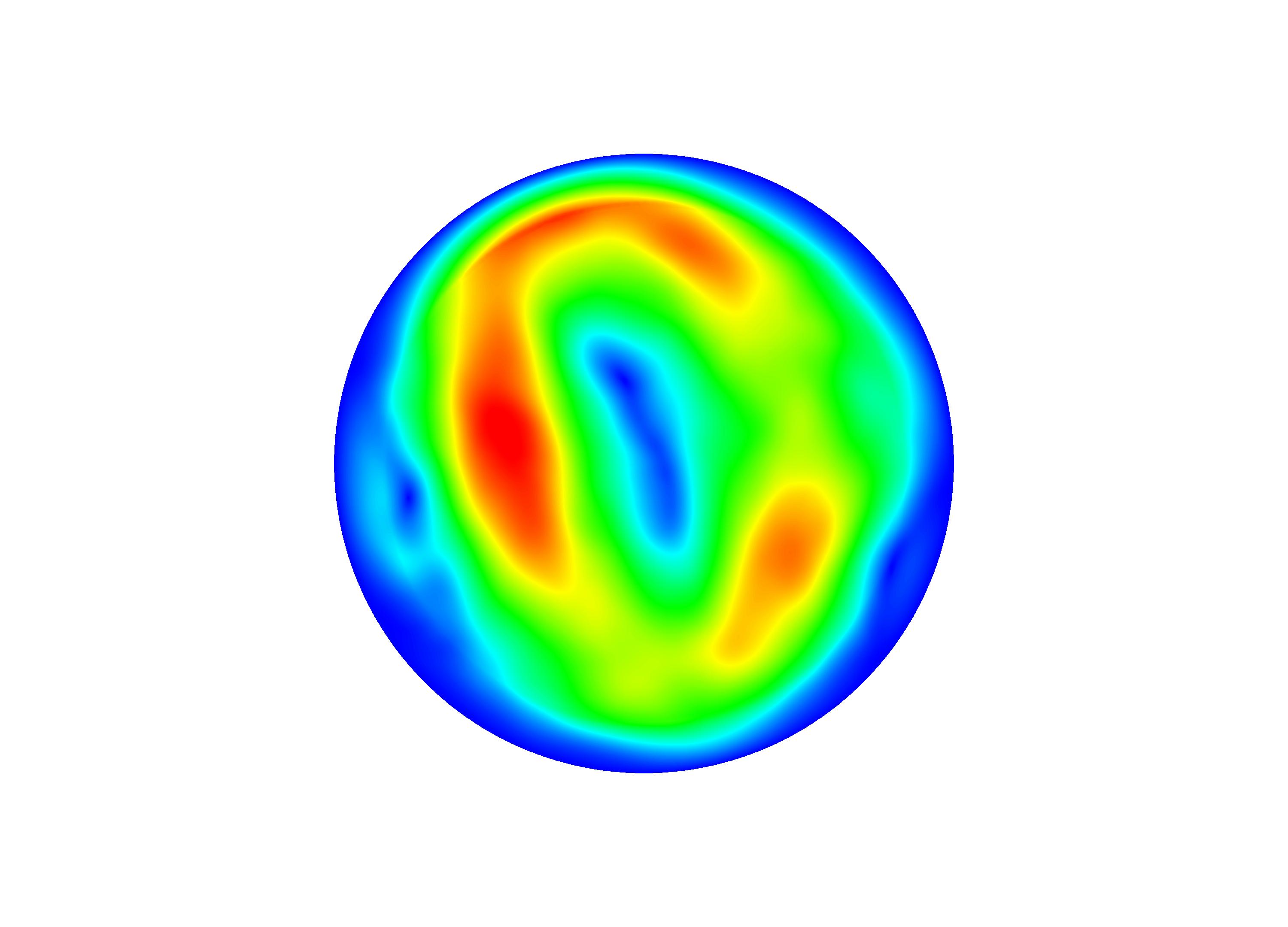}\\[0cm]
        \includegraphics[width=1.3\linewidth,trim={10.1cm 5cm 10.1cm 5cm},clip]{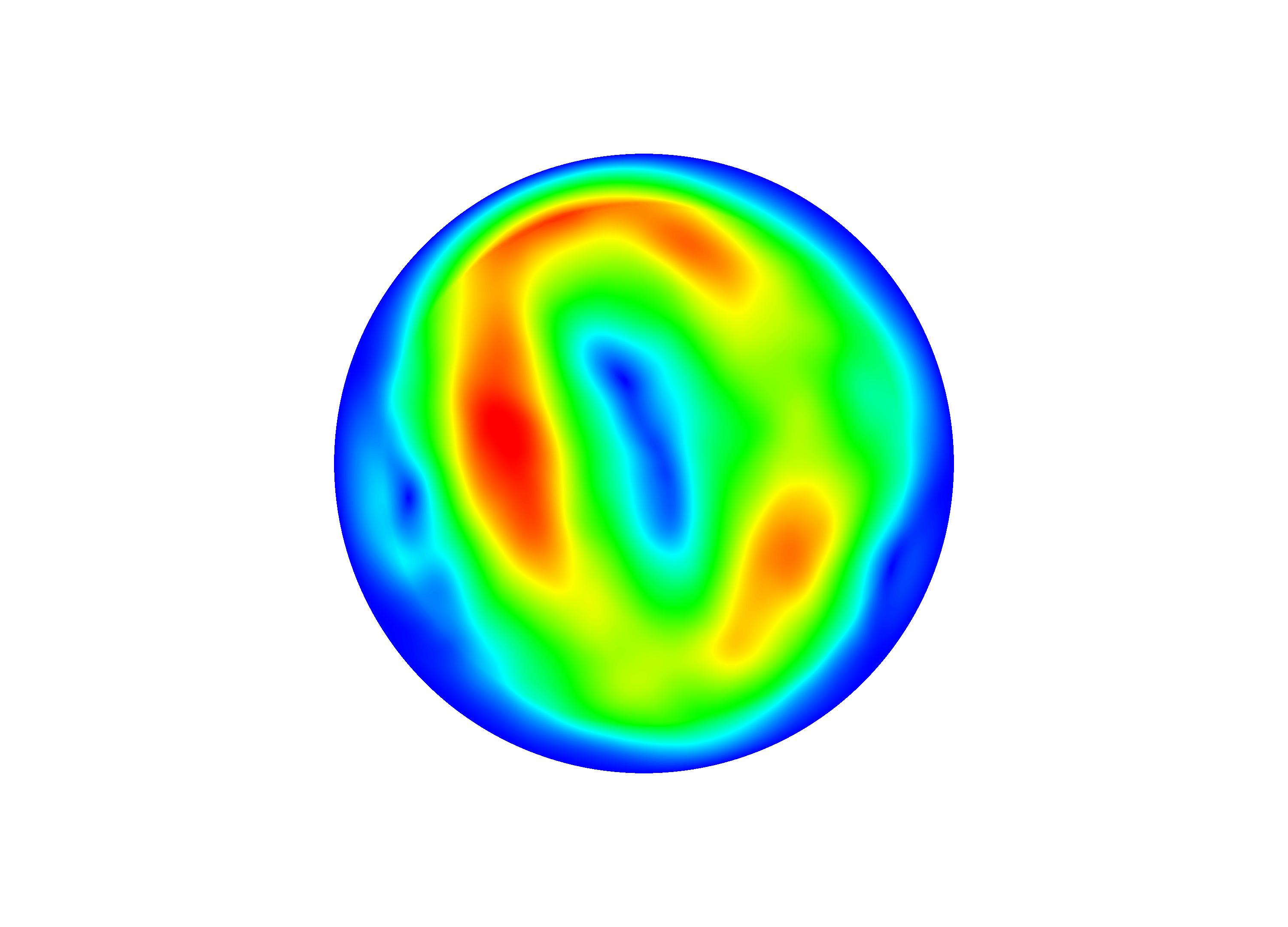}\\[0cm]
        \includegraphics[width=1.3\linewidth,trim={10.1cm 5cm 10.1cm 5cm},clip]{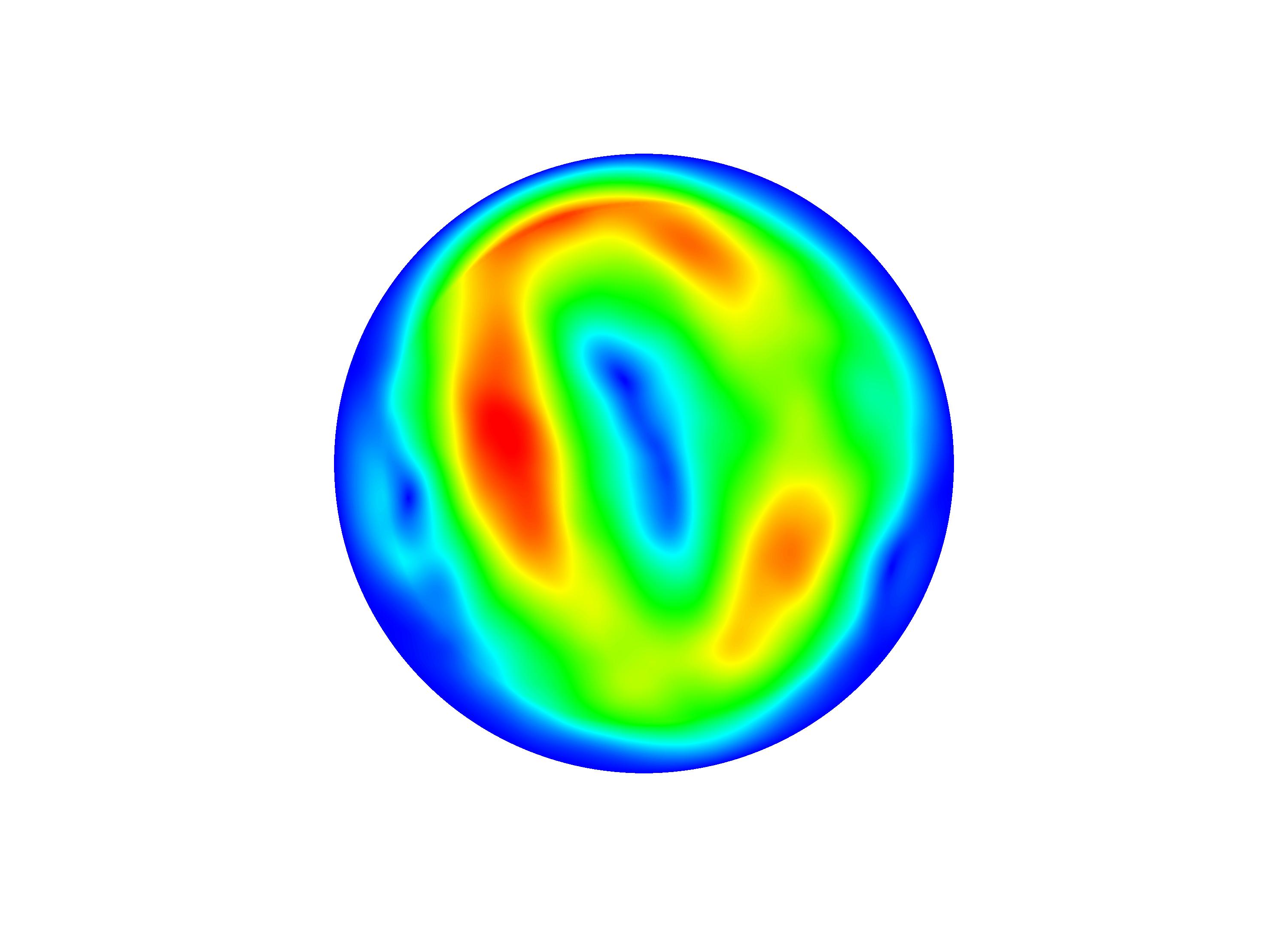}\\[0cm]
        \includegraphics[width=1.3\linewidth,trim={10.1cm 5cm 10.1cm 5cm},clip]{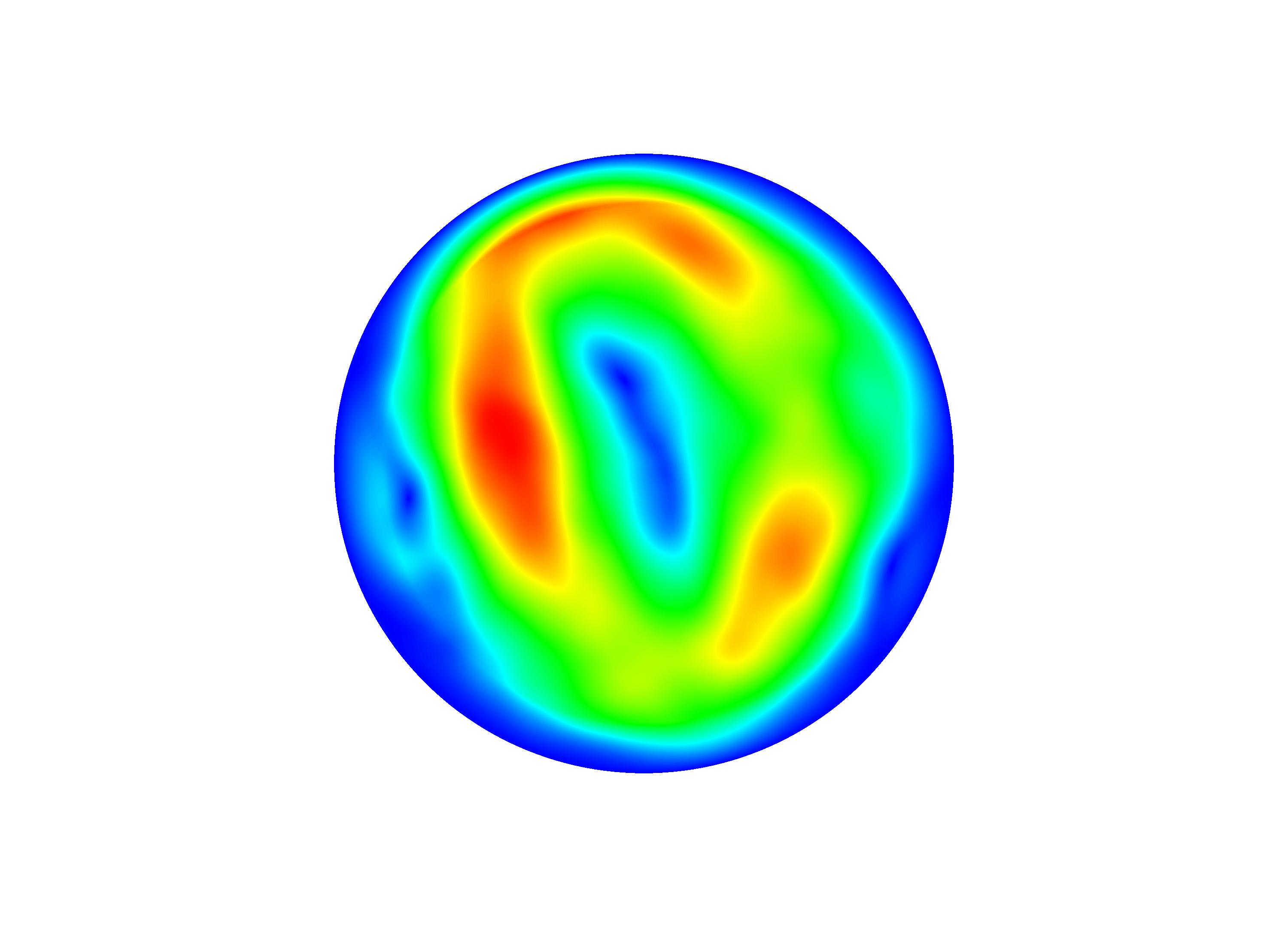}\\[0cm]
        \includegraphics[width=1.3\linewidth,trim={10.1cm 5cm 10.1cm 5cm},clip]{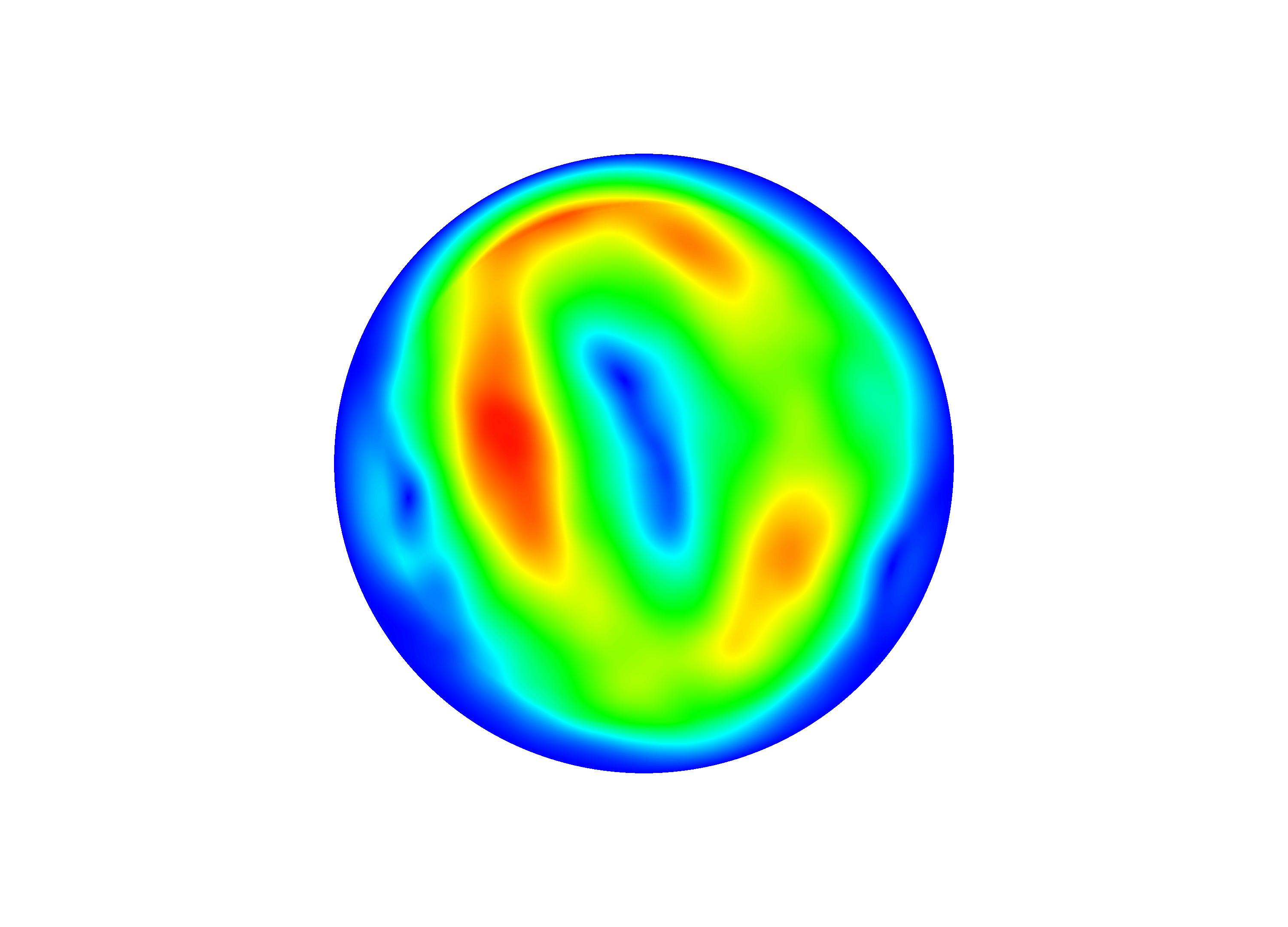}\\[0cm]
        \includegraphics[width=1.3\linewidth,trim={10.1cm 5cm 10.1cm 5cm},clip]{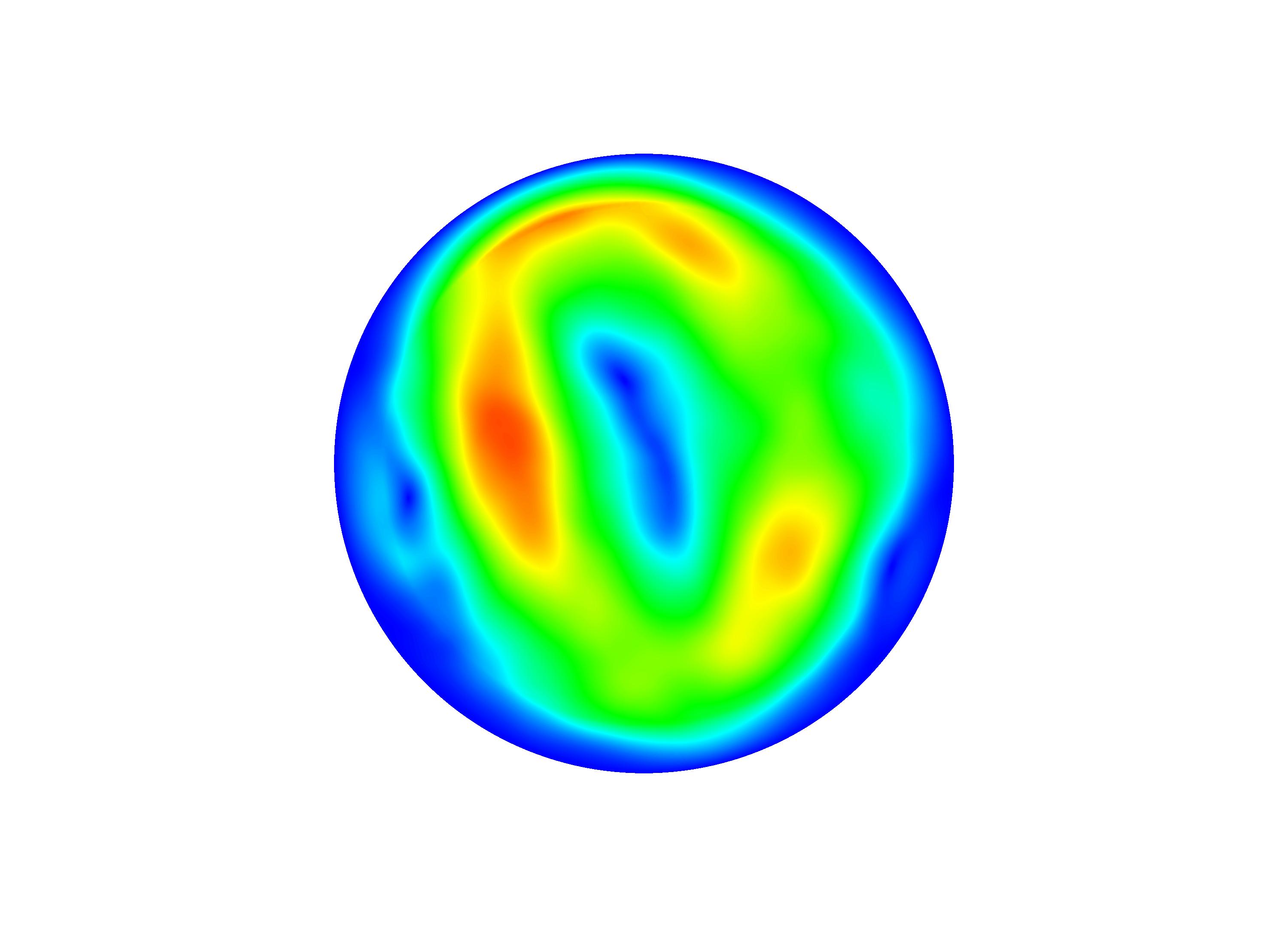}\\[0cm]
        \includegraphics[width=1.3\linewidth,trim={10.1cm 5cm 10.1cm 5cm},clip]{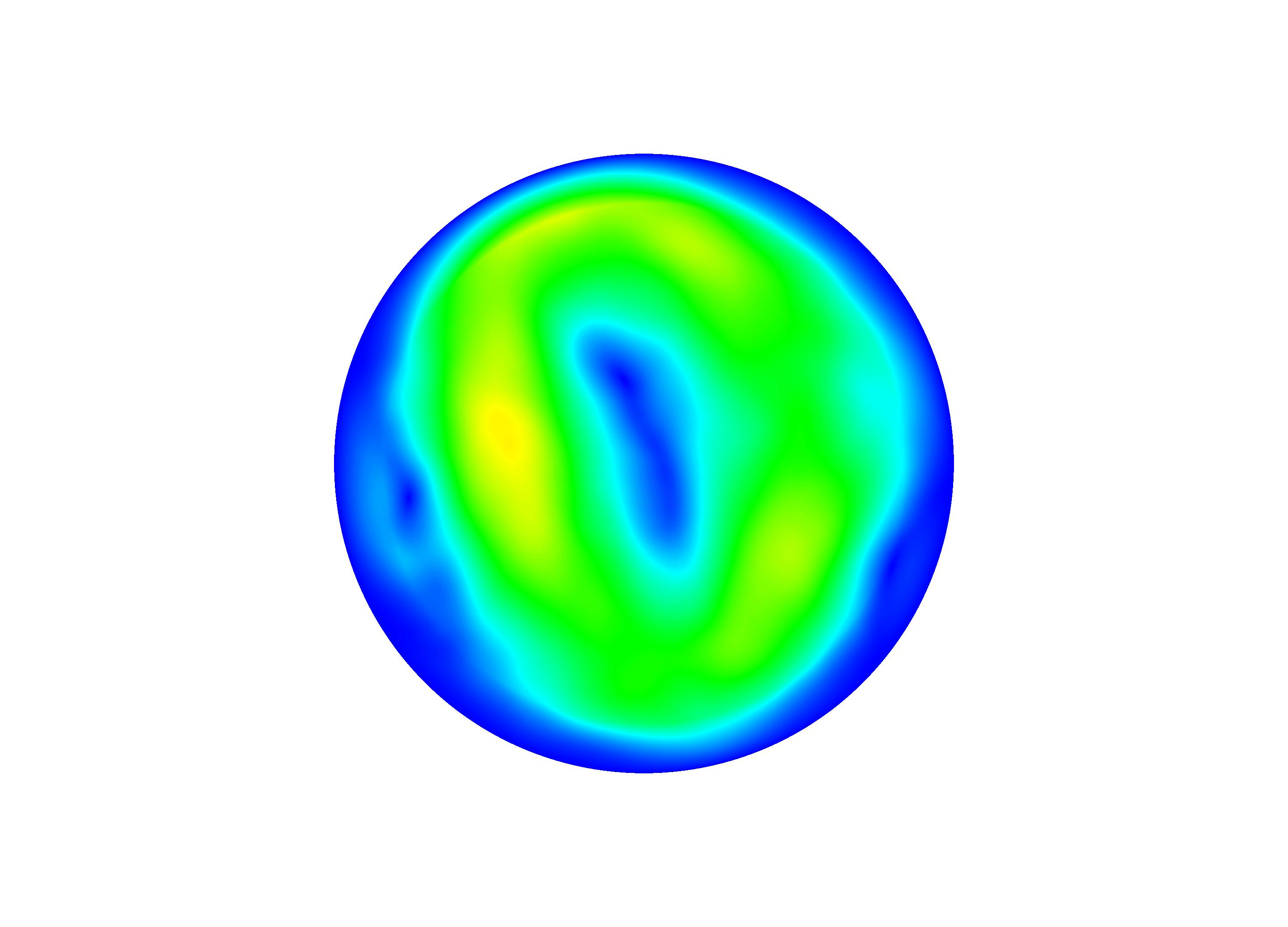}\\[0cm]
        \includegraphics[width=1.3\linewidth,trim={10.1cm 5cm 10.1cm 5cm},clip]{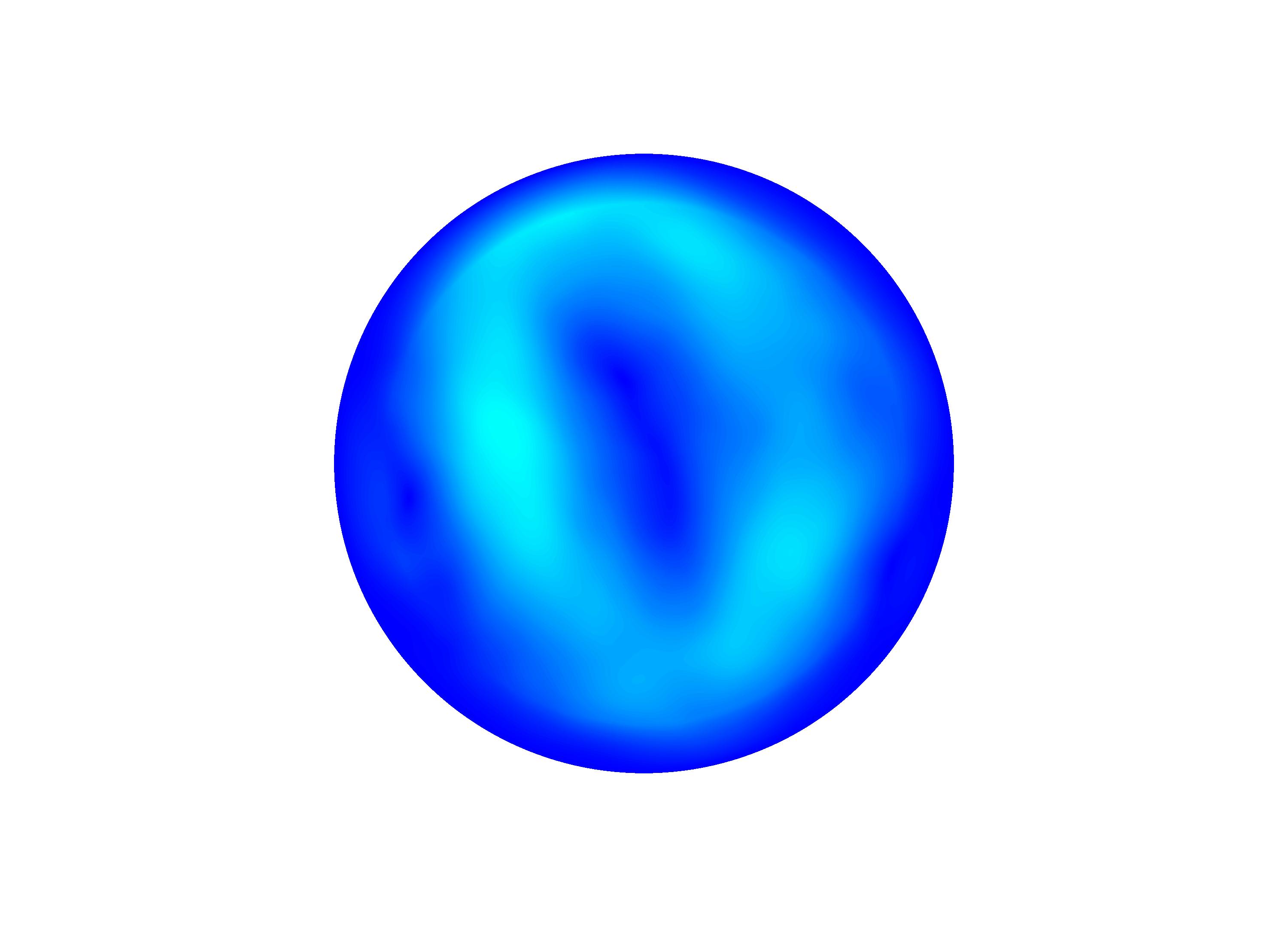}\\[0cm]
    \end{minipage}
    \hfill
    \begin{minipage}[t]{0.16\textwidth}
        \centering
        \hspace{7.5pt}Parameter\\

        \includegraphics[width=1.3\linewidth,trim={10.1cm 5cm 10.1cm 5cm},clip]{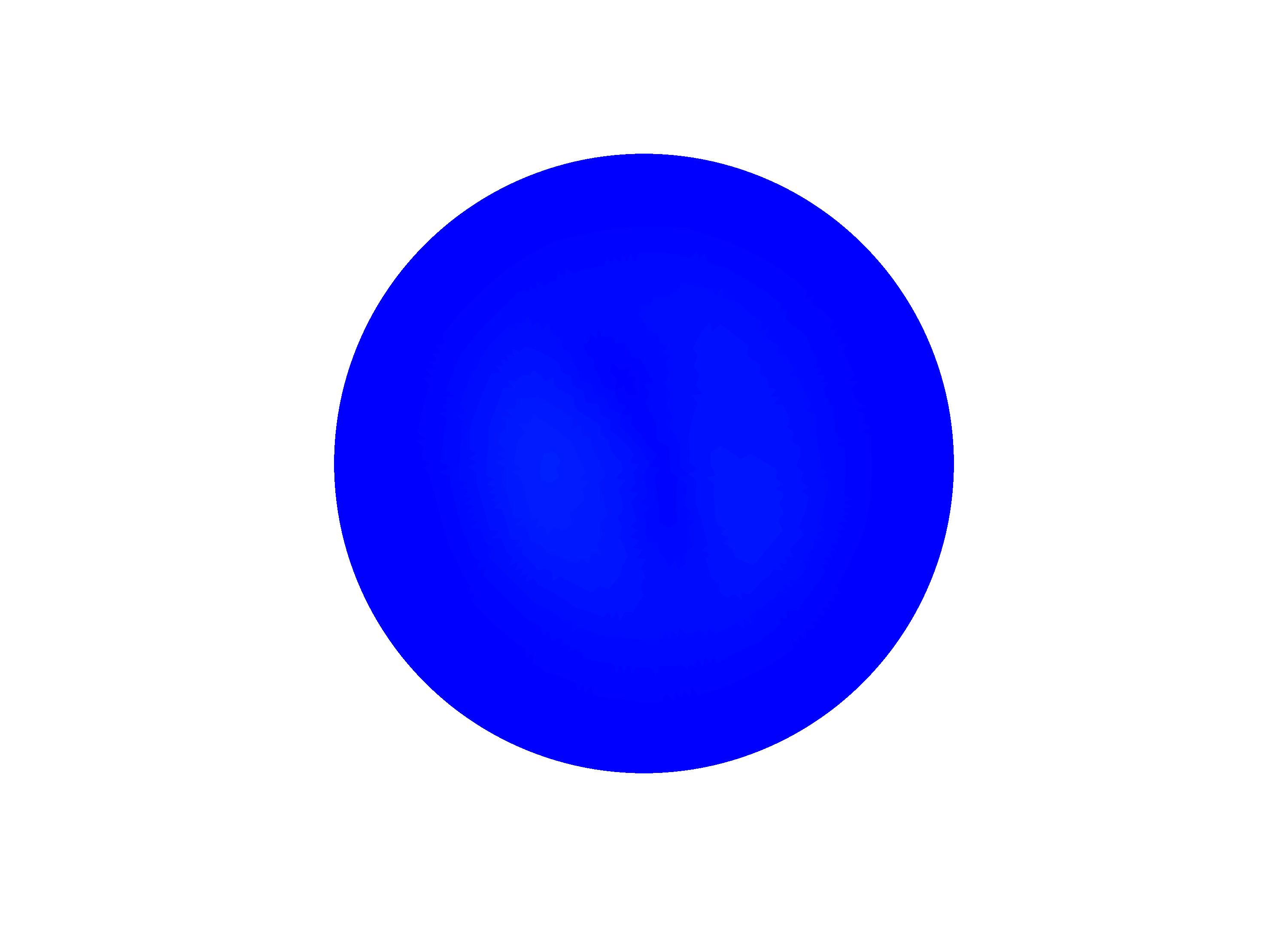}\\[0cm]
        \includegraphics[width=1.3\linewidth,trim={10.1cm 5cm 10.1cm 5cm},clip]{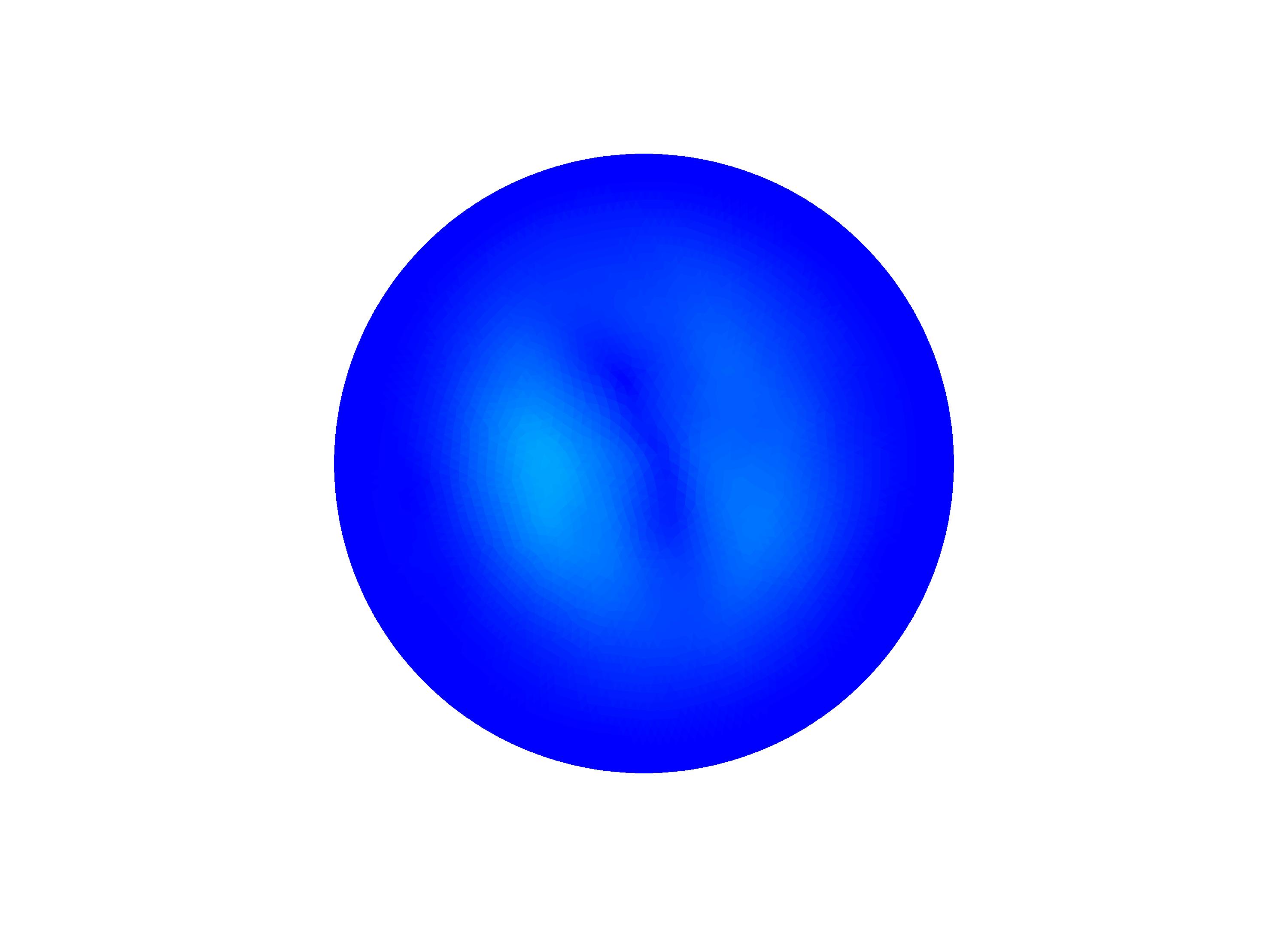}\\[0cm]
        \includegraphics[width=1.3\linewidth,trim={10.1cm 5cm 10.1cm 5cm},clip]{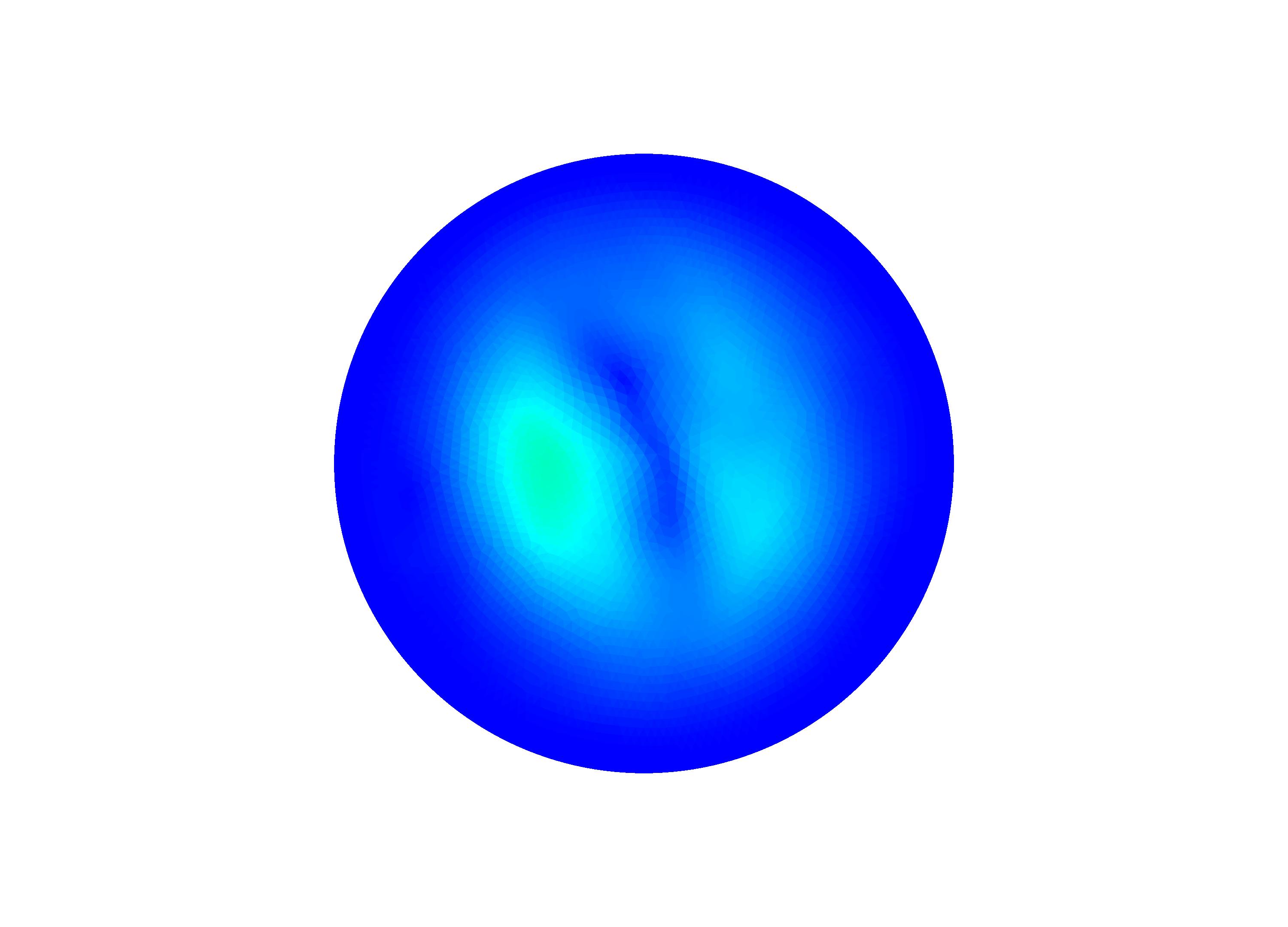}\\[0cm]
        \includegraphics[width=1.3\linewidth,trim={10.1cm 5cm 10.1cm 5cm},clip]{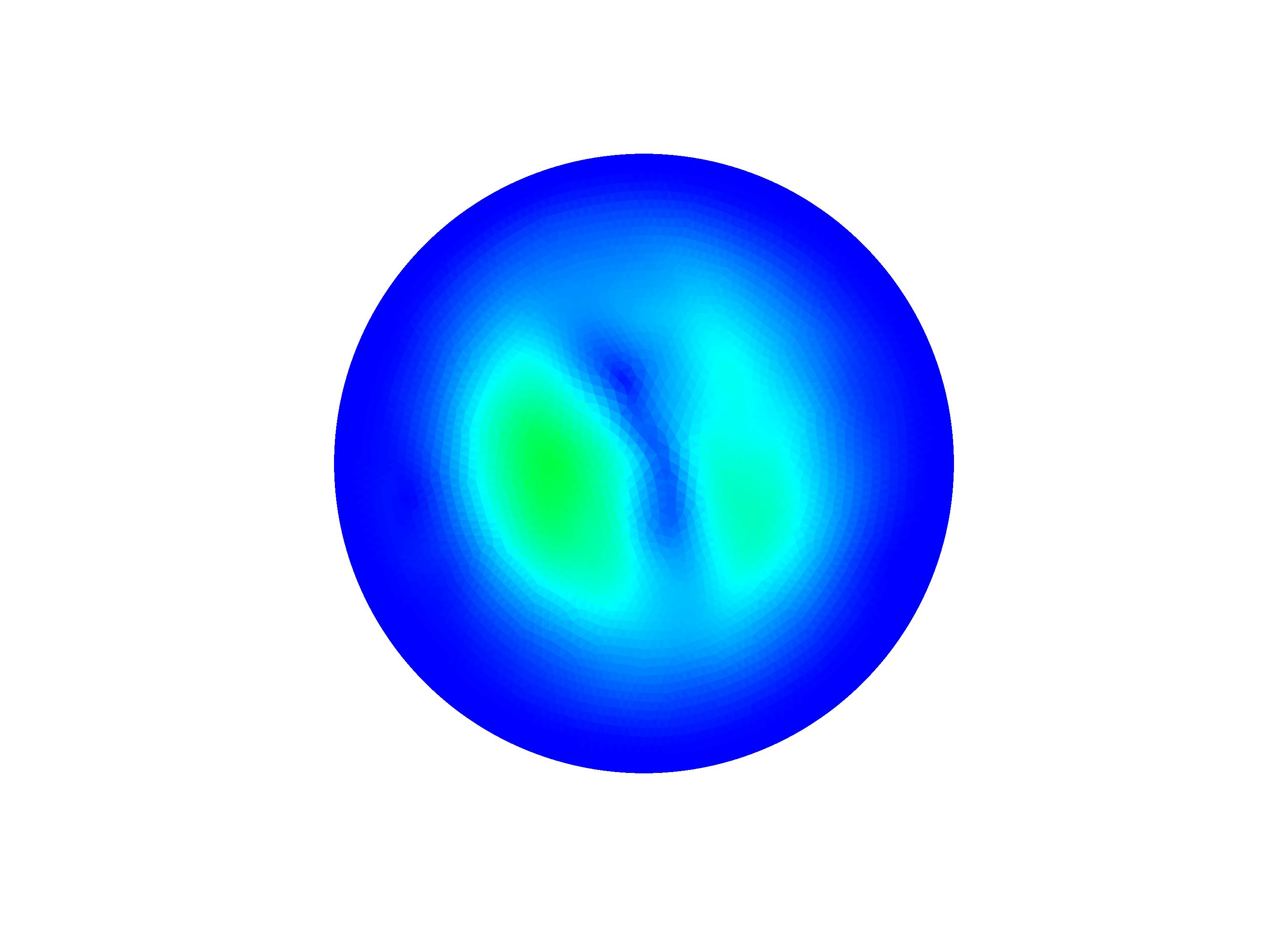}\\[0cm]
        \includegraphics[width=1.3\linewidth,trim={10.1cm 5cm 10.1cm 5cm},clip]{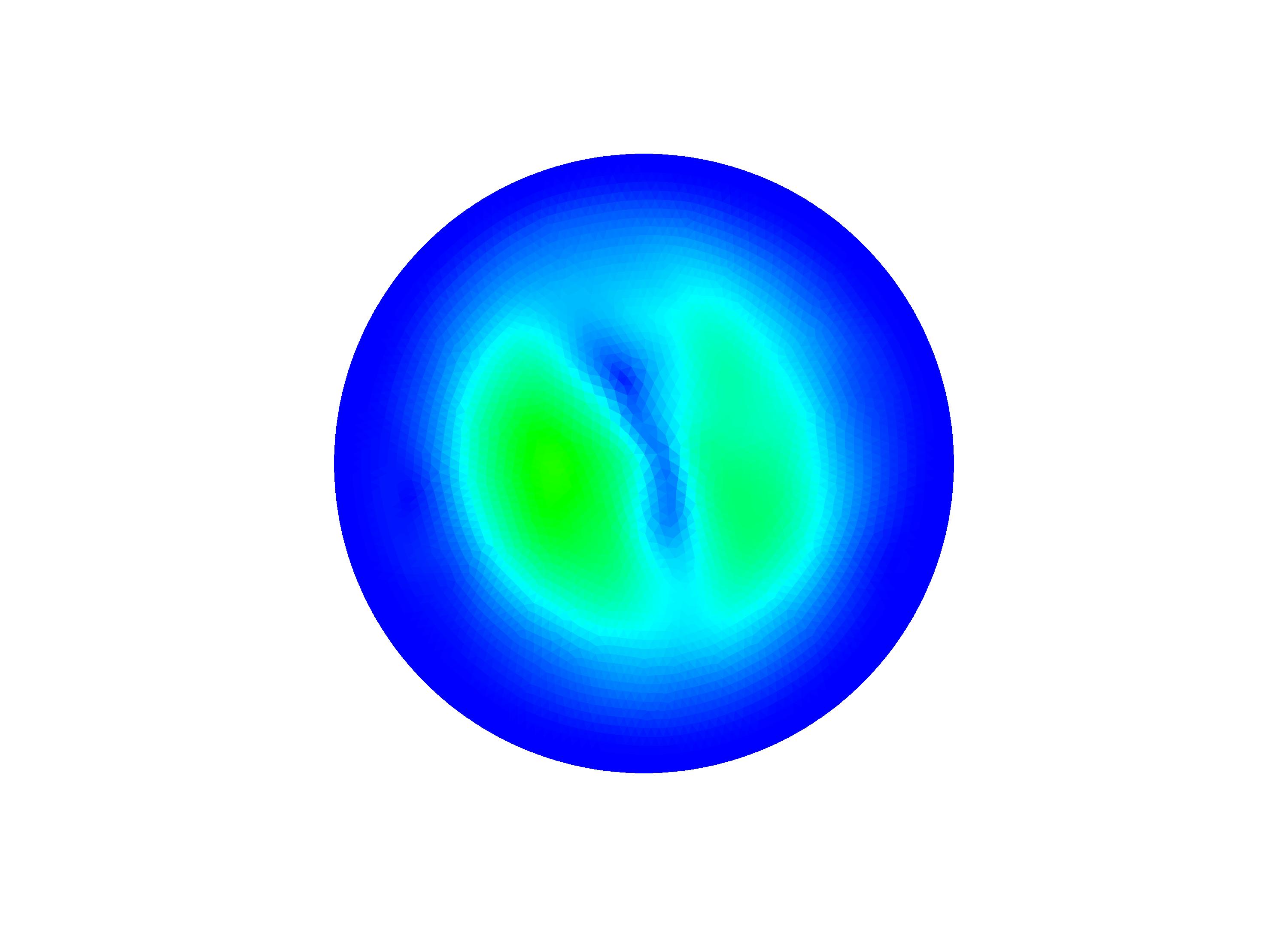}\\[0cm]
        \includegraphics[width=1.3\linewidth,trim={10.1cm 5cm 10.1cm 5cm},clip]{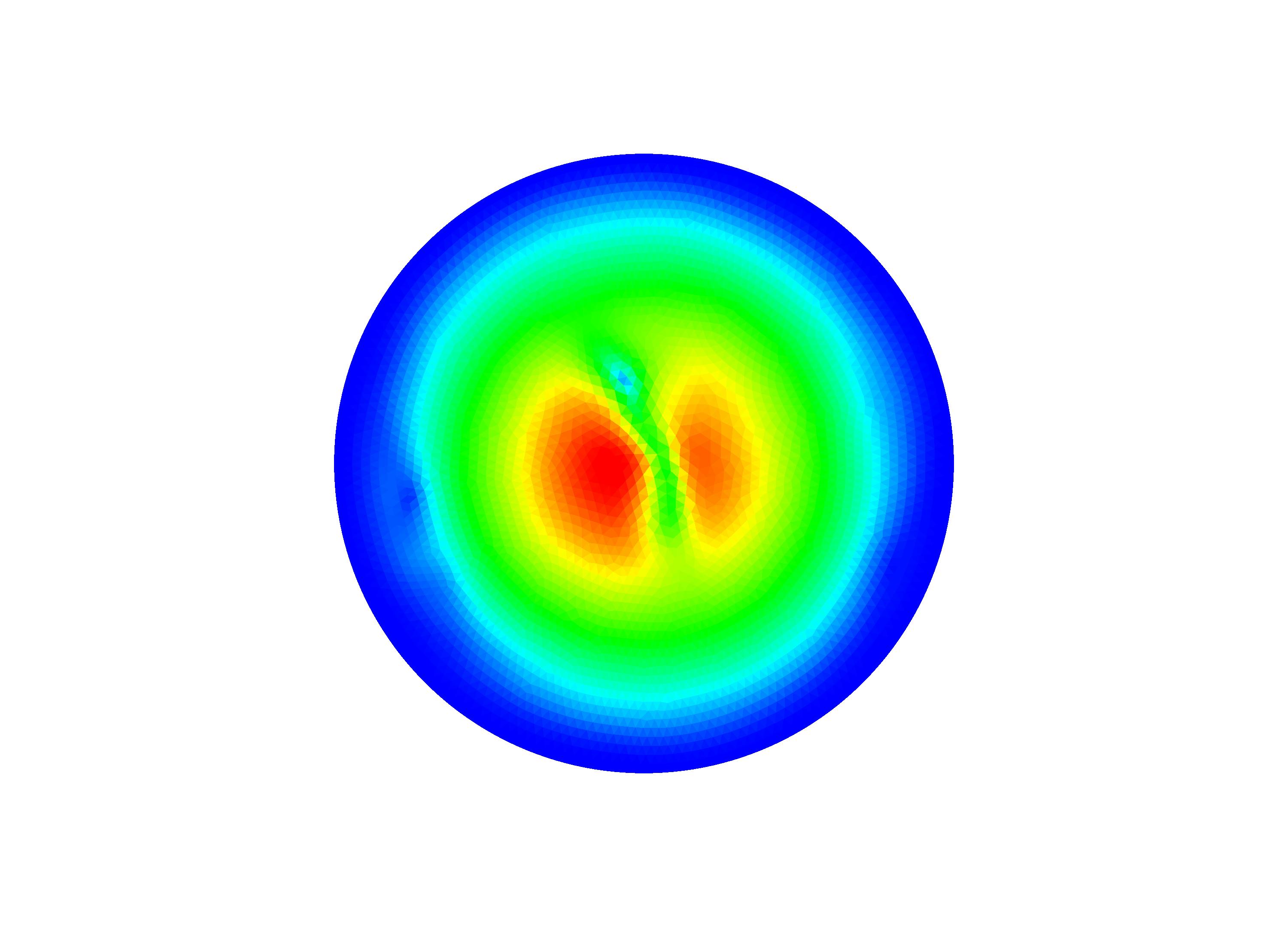}\\[0cm]
        \includegraphics[width=1.3\linewidth,trim={10.1cm 5cm 10.1cm 5cm},clip]{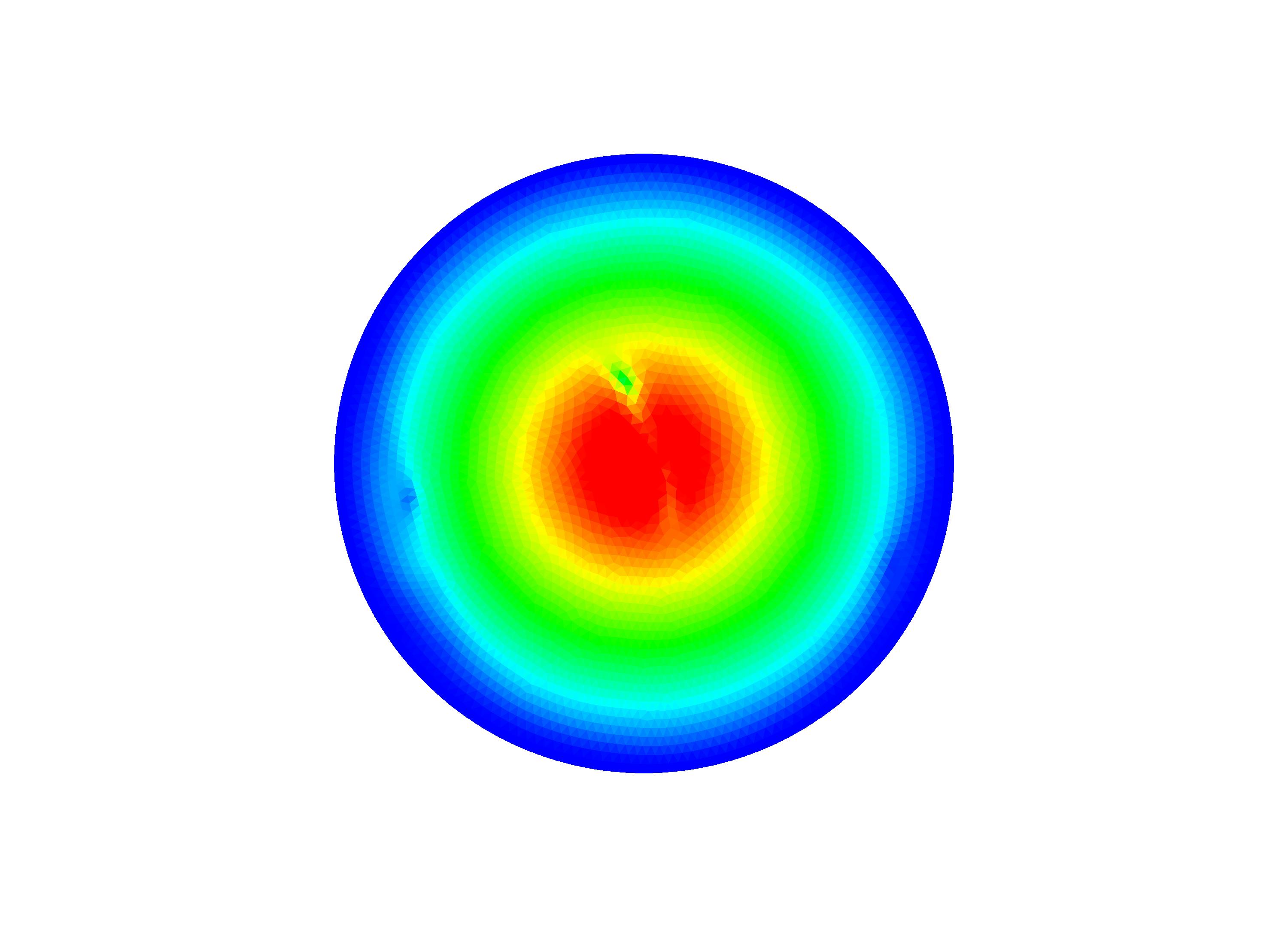}\\[0cm]
        \includegraphics[width=1.3\linewidth,trim={10.1cm 5cm 10.1cm 5cm},clip]{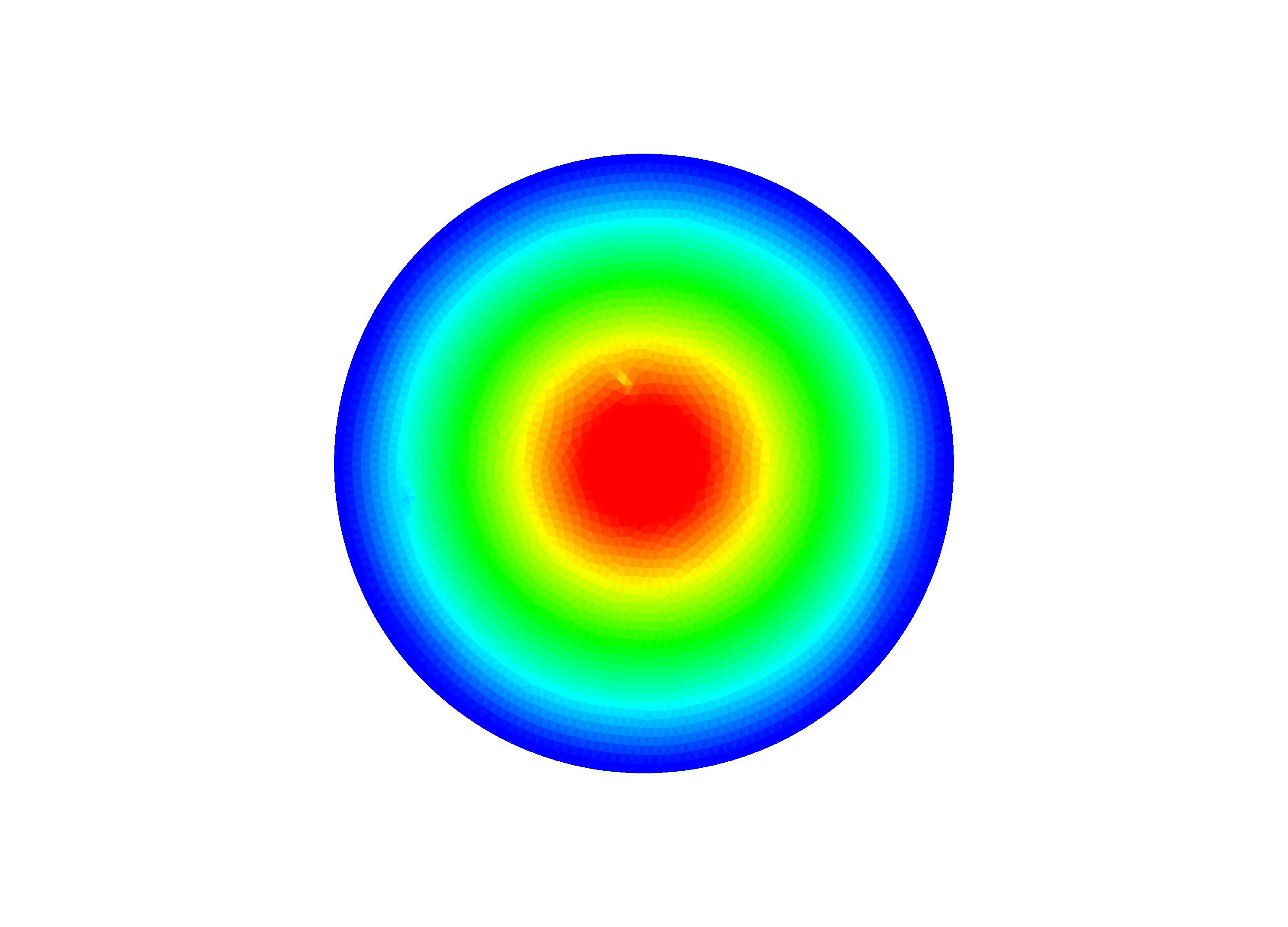}\\[0.2cm]
    \end{minipage}
    \hspace{32pt}
    \textcolor{gray}{\vline}
    \hspace{28pt}
    \begin{minipage}[t]{0.16\textwidth}
        \centering
        \hspace{15pt}{State} \\
        \includegraphics[width=1.3\linewidth,trim={10.1cm 5cm 10.1cm 5cm},clip]{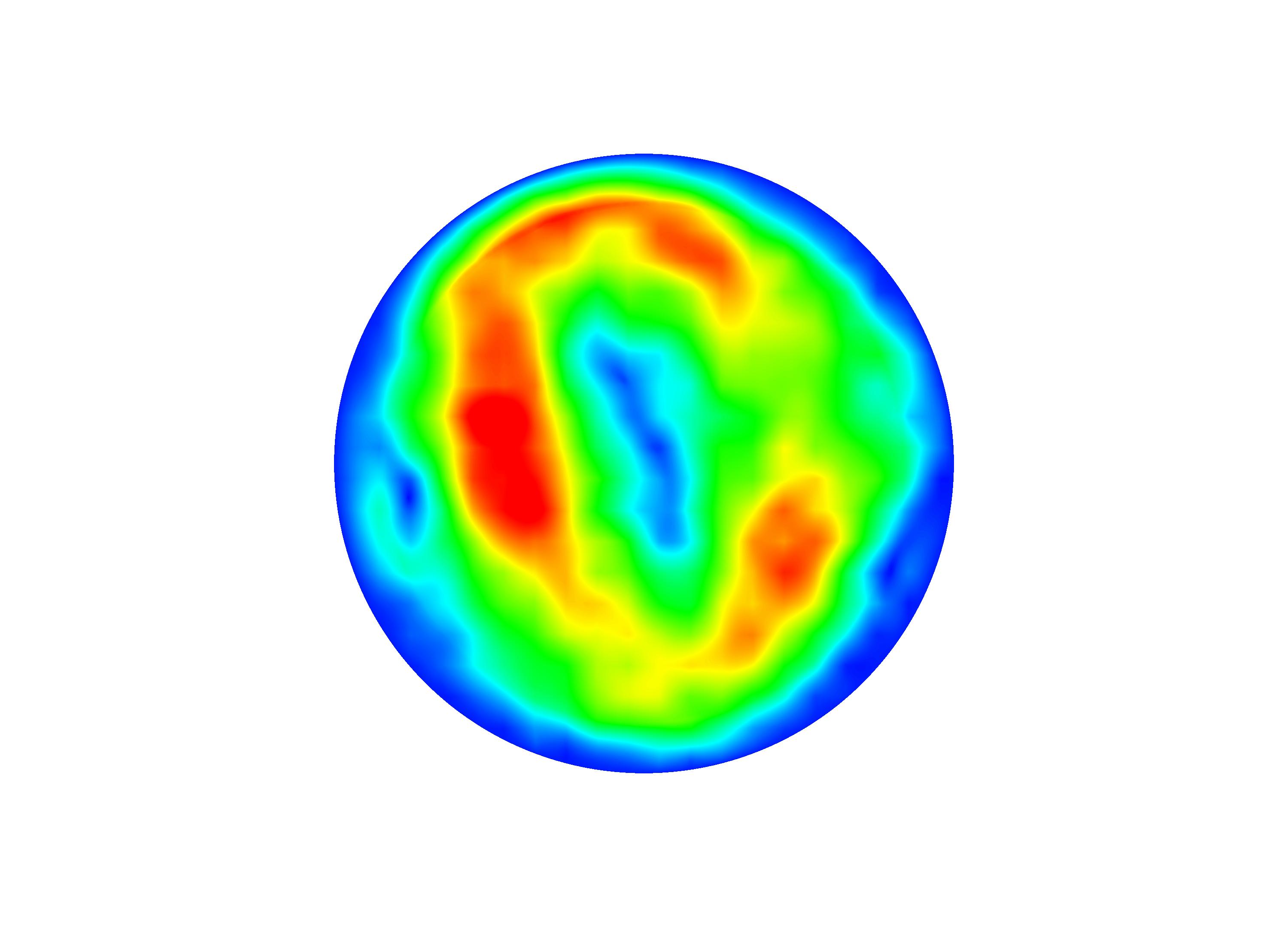}\\[0cm]
        \includegraphics[width=1.3\linewidth,trim={10.1cm 5cm 10.1cm 5cm},clip]{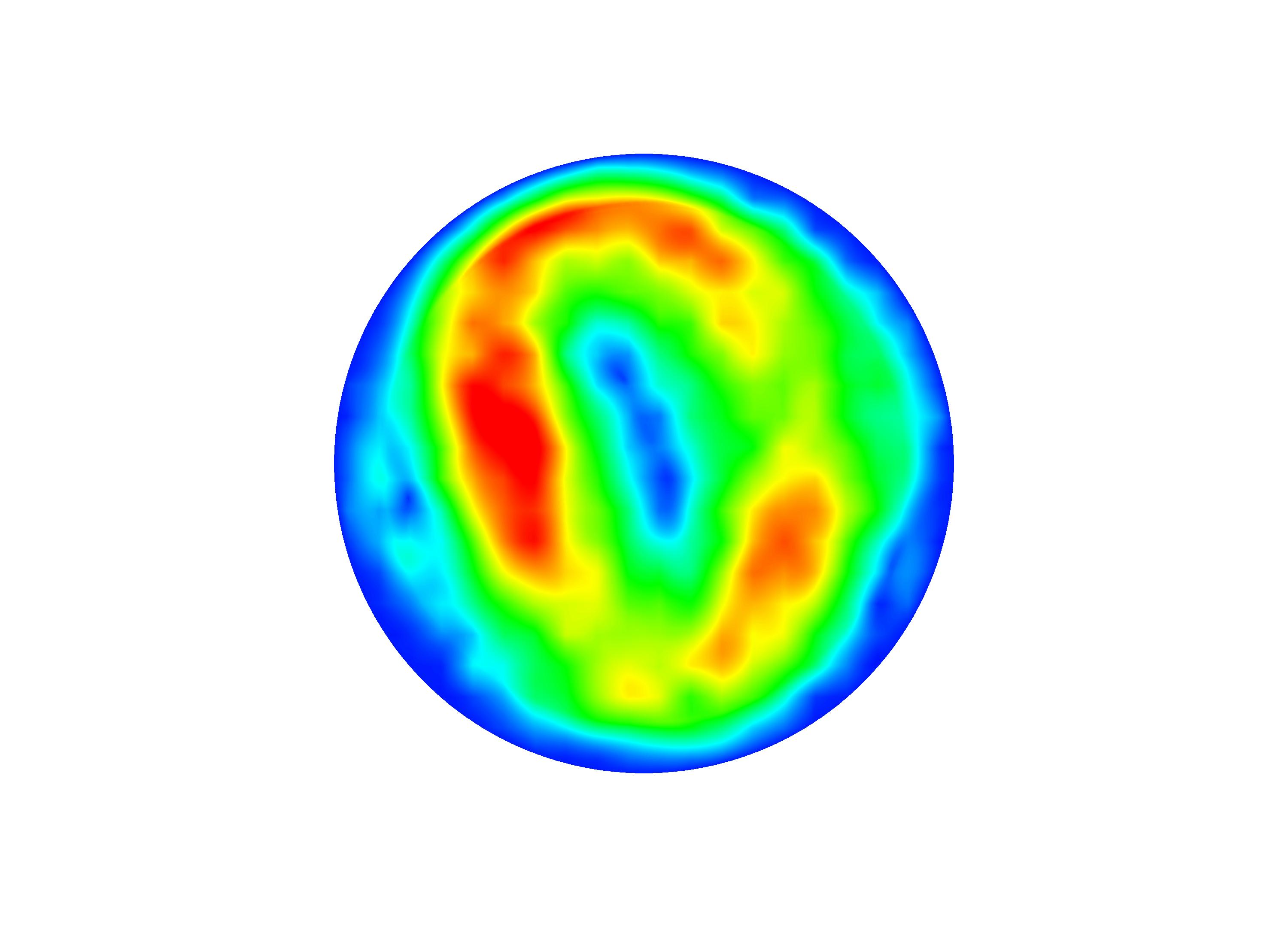}\\[0cm]
        \includegraphics[width=1.3\linewidth,trim={10.1cm 5cm 10.1cm 5cm},clip]{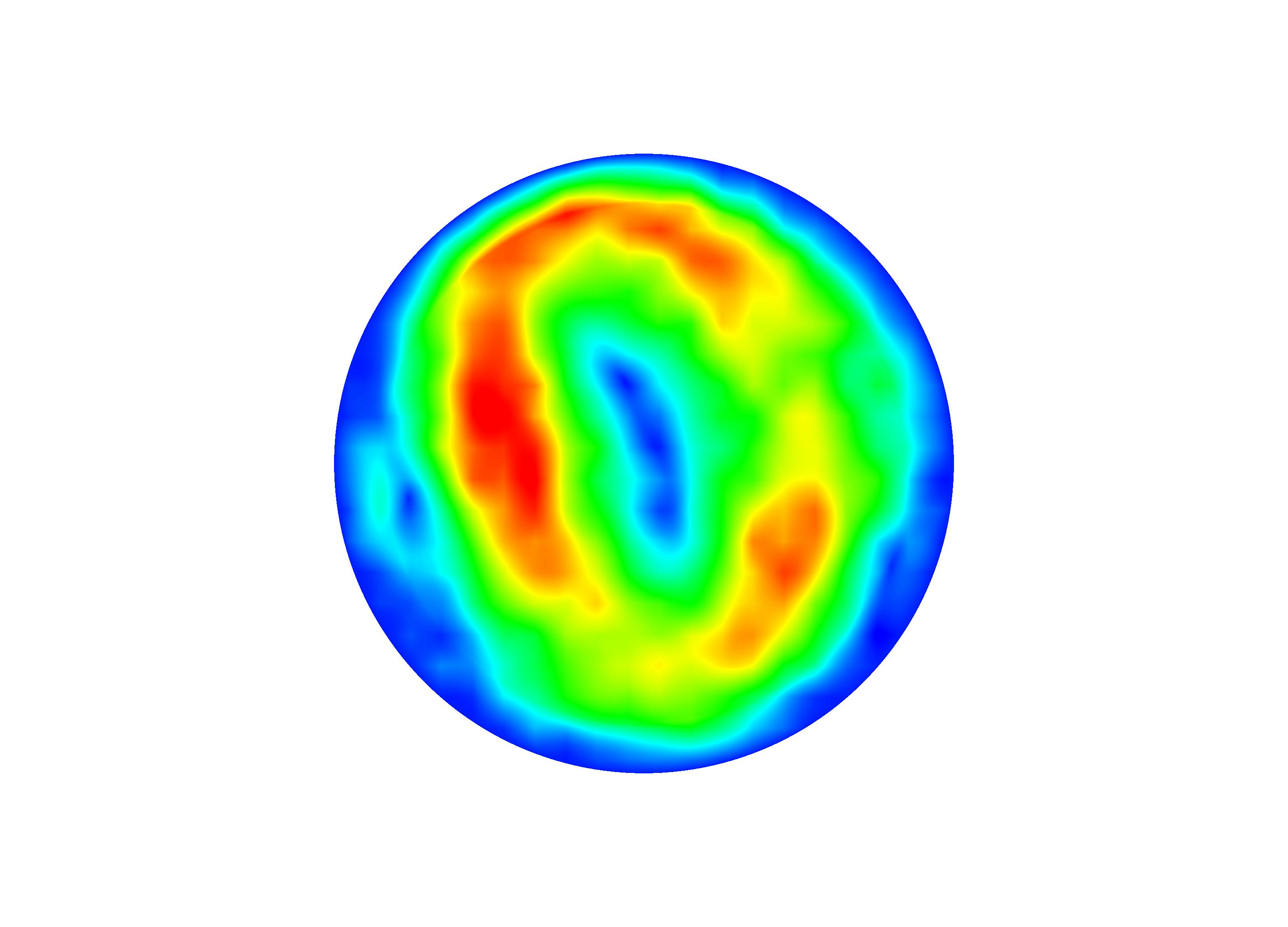}\\[0cm]
        \includegraphics[width=1.3\linewidth,trim={10.1cm 5cm 10.1cm 5cm},clip]{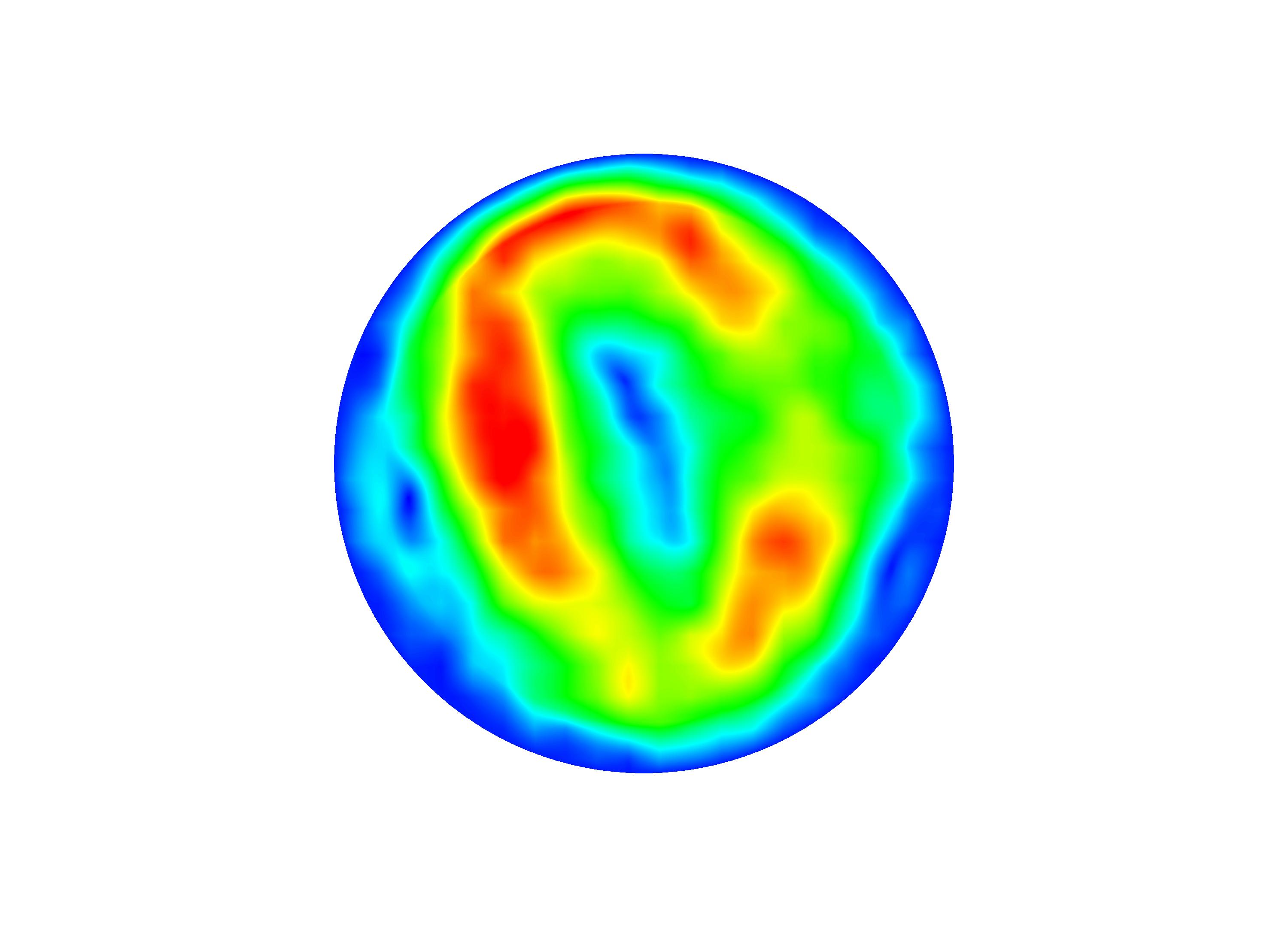}\\[0cm]
        \includegraphics[width=1.3\linewidth,trim={10.1cm 5cm 10.1cm 5cm},clip]{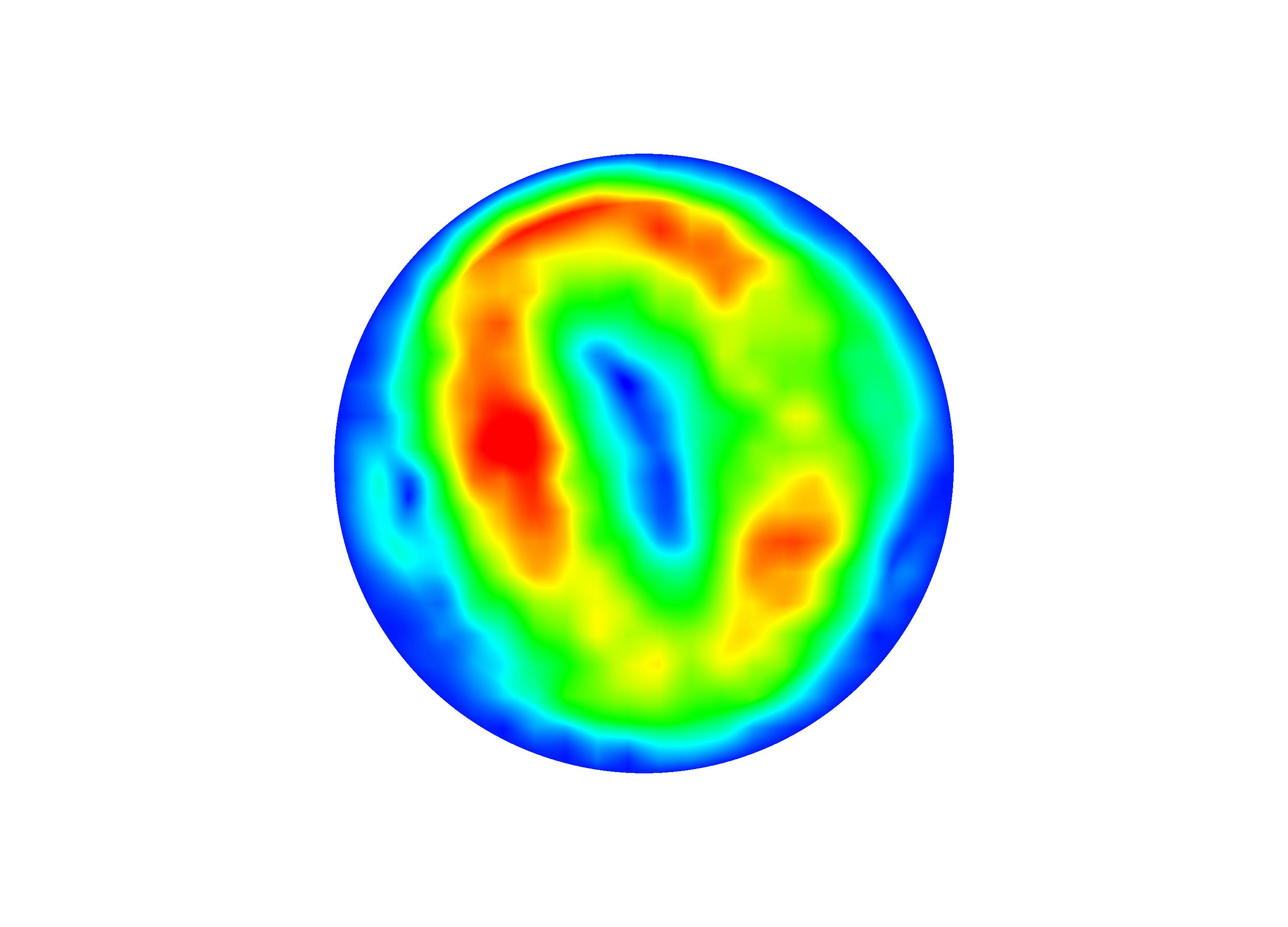}\\[0cm]
        \includegraphics[width=1.3\linewidth,trim={10.1cm 5cm 10.1cm 5cm},clip]{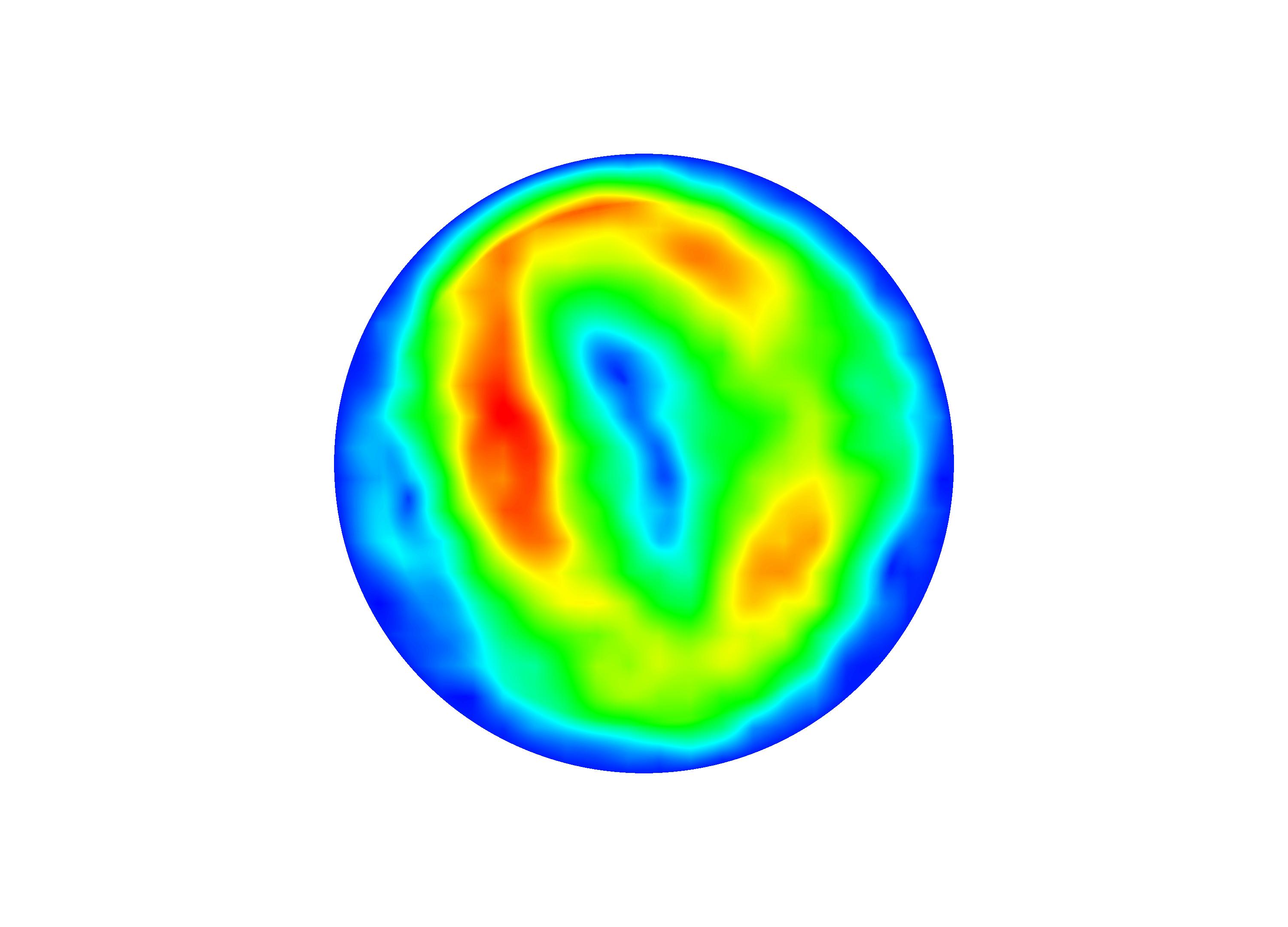}\\[0cm]
        \includegraphics[width=1.3\linewidth,trim={10.1cm 5cm 10.1cm 5cm},clip]{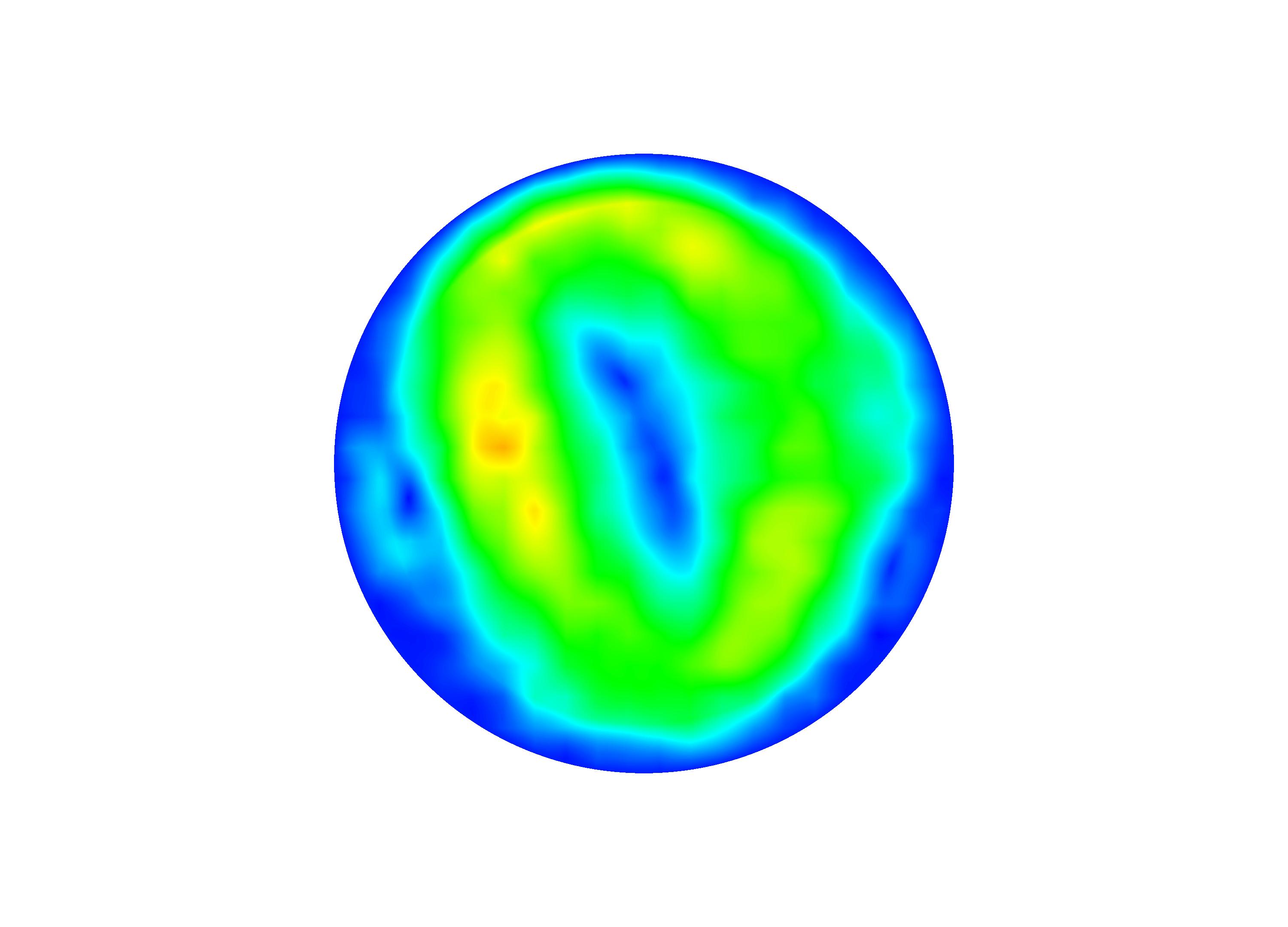}\\[0cm]
        \includegraphics[width=1.3\linewidth,trim={10.1cm 5cm 10.1cm 5cm},clip]{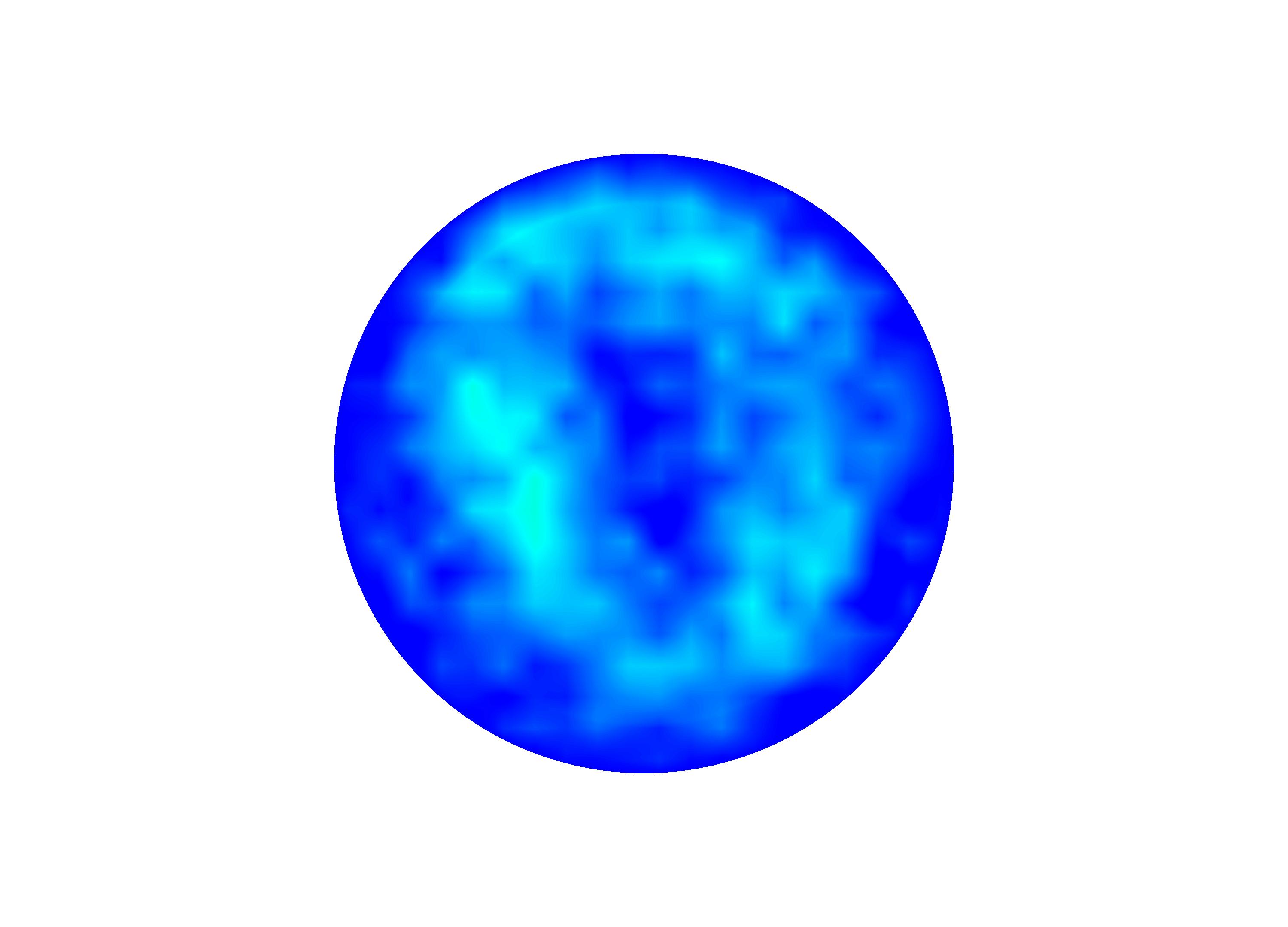}\\[0.2cm]
    \end{minipage}
    \hfill
    \begin{minipage}[t]{0.16\textwidth}
        \centering
        \hspace{7.5pt}Parameter\\
        \includegraphics[width=1.3\linewidth,trim={10.1cm 5cm 10.1cm 5cm},clip]{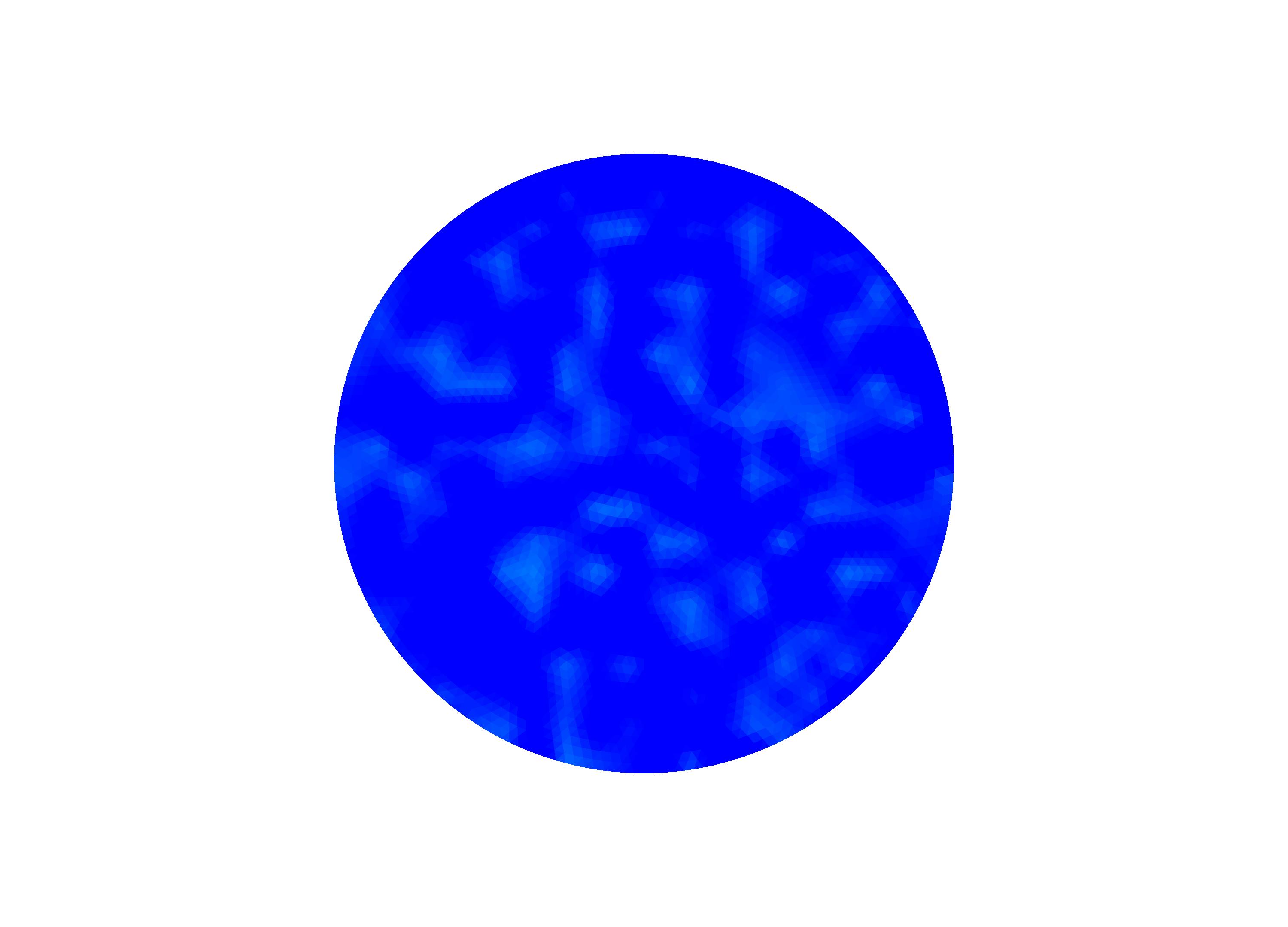}\\[0cm]
        \includegraphics[width=1.3\linewidth,trim={10.1cm 5cm 10.1cm 5cm},clip]{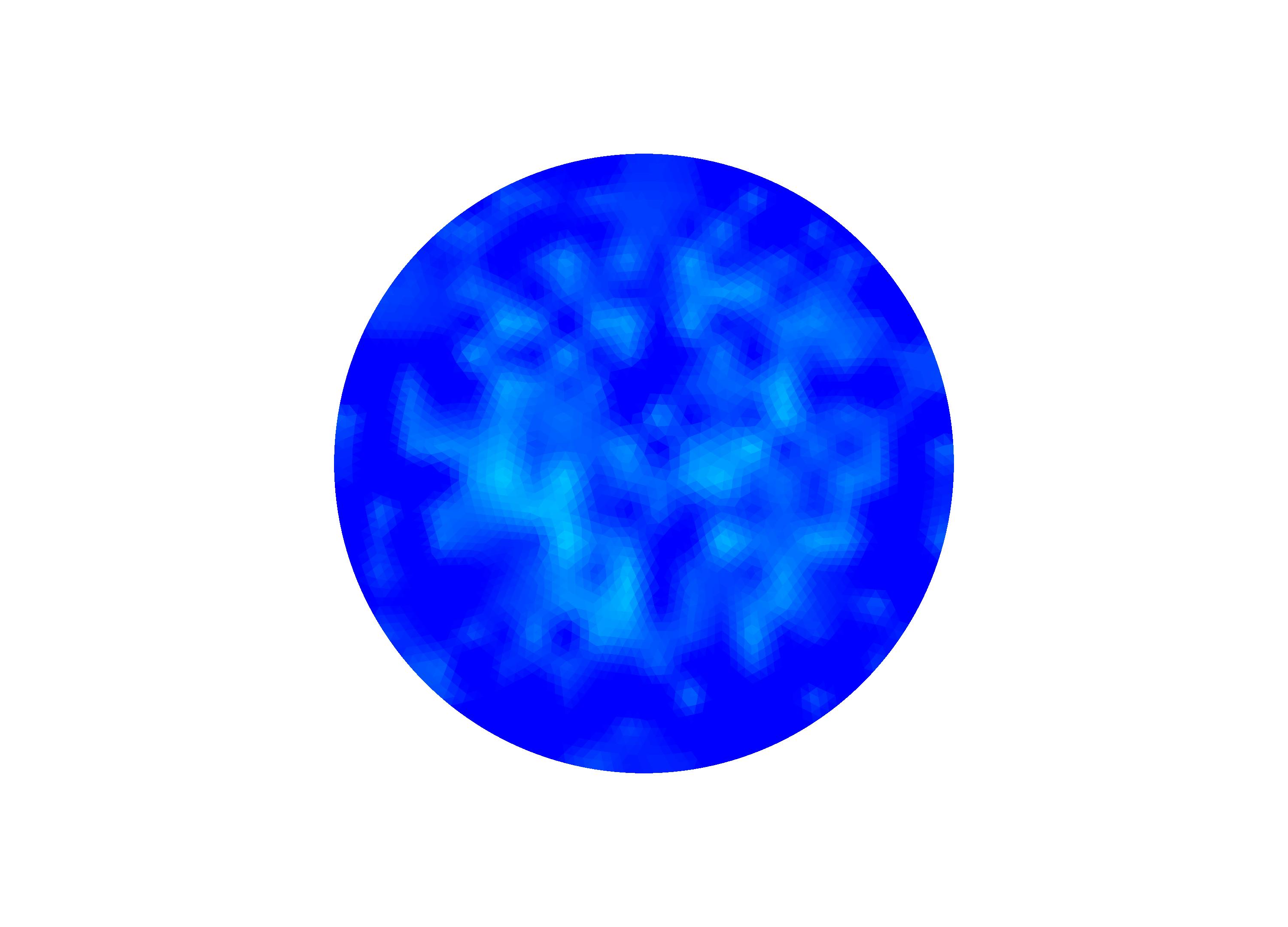}\\[0cm]
        \includegraphics[width=1.3\linewidth,trim={10.1cm 5cm 10.1cm 5cm},clip]{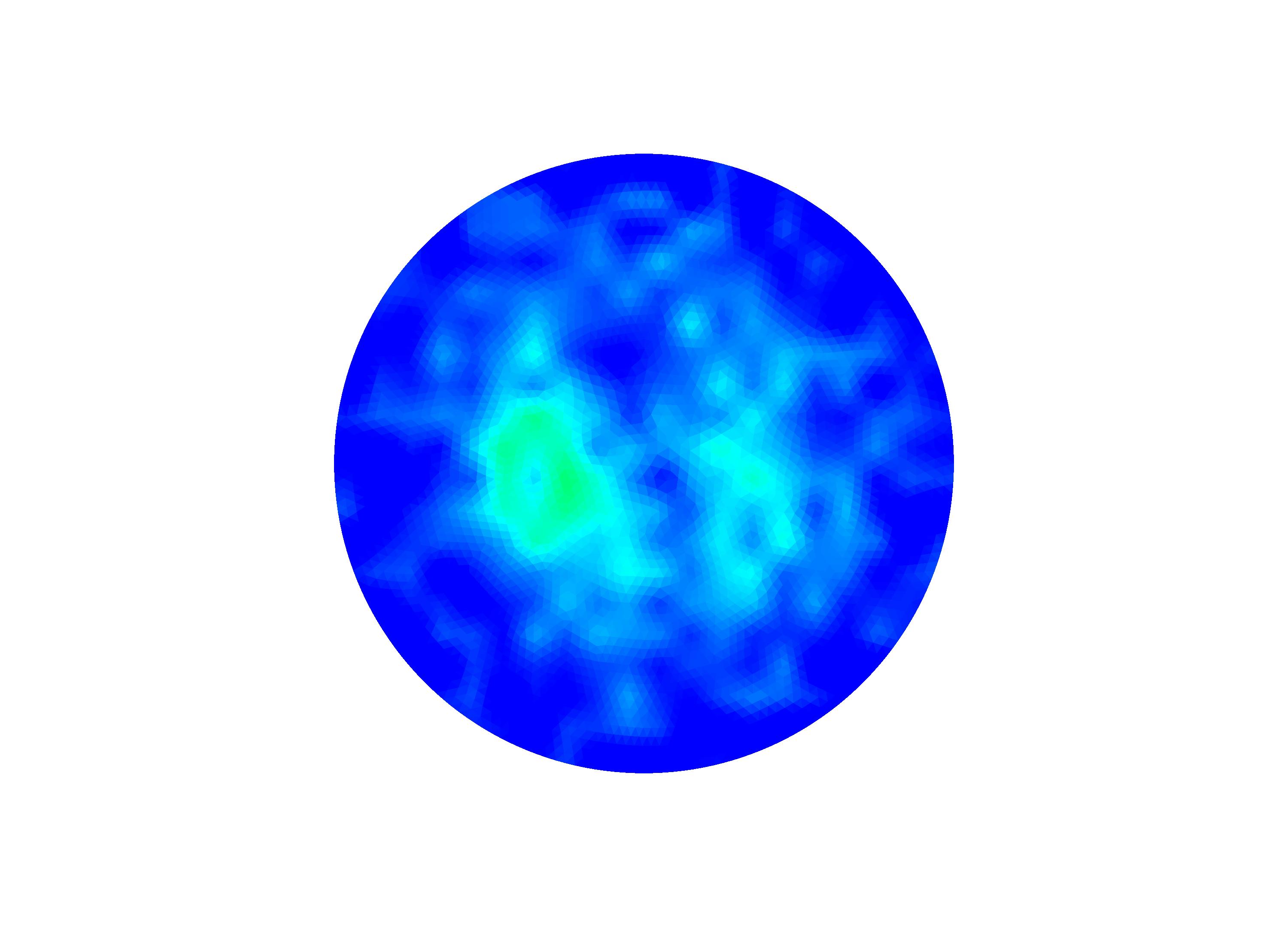}\\[0cm]
        \includegraphics[width=1.3\linewidth,trim={10.1cm 5cm 10.1cm 5cm},clip]{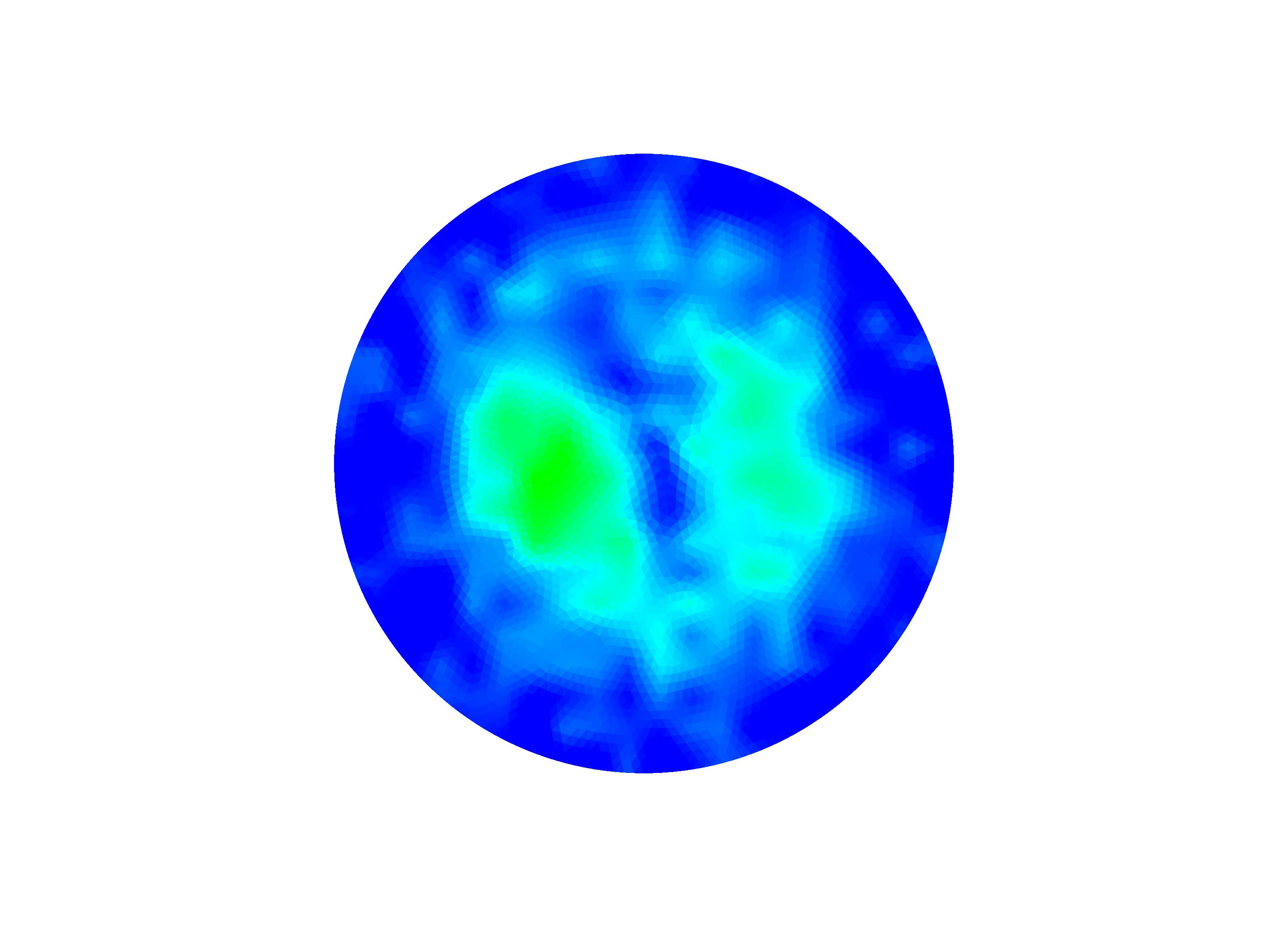}\\[0cm]
        \includegraphics[width=1.3\linewidth,trim={10.1cm 5cm 10.1cm 5cm},clip]{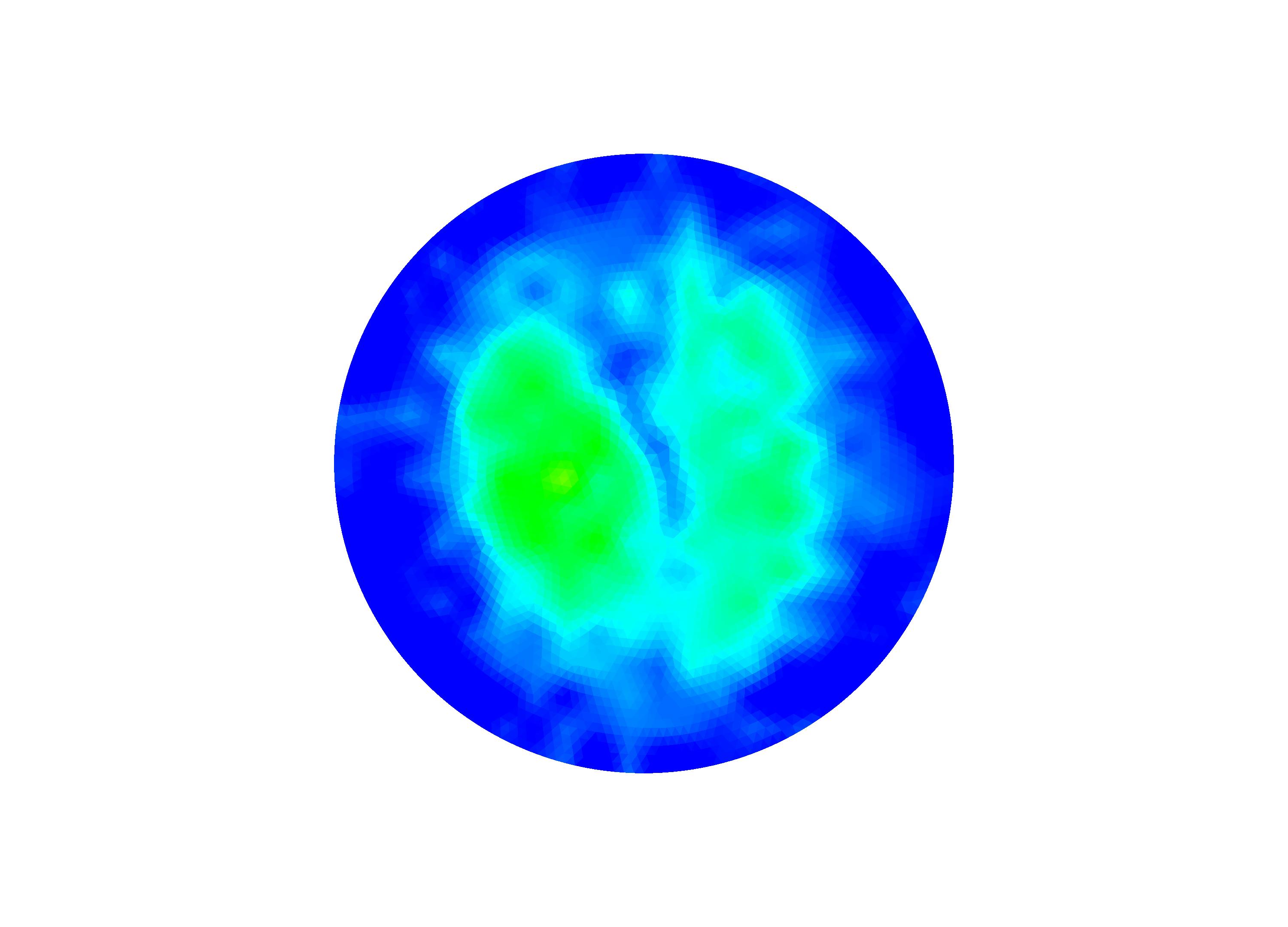}\\[0cm]
        \includegraphics[width=1.3\linewidth,trim={10.1cm 5cm 10.1cm 5cm},clip]{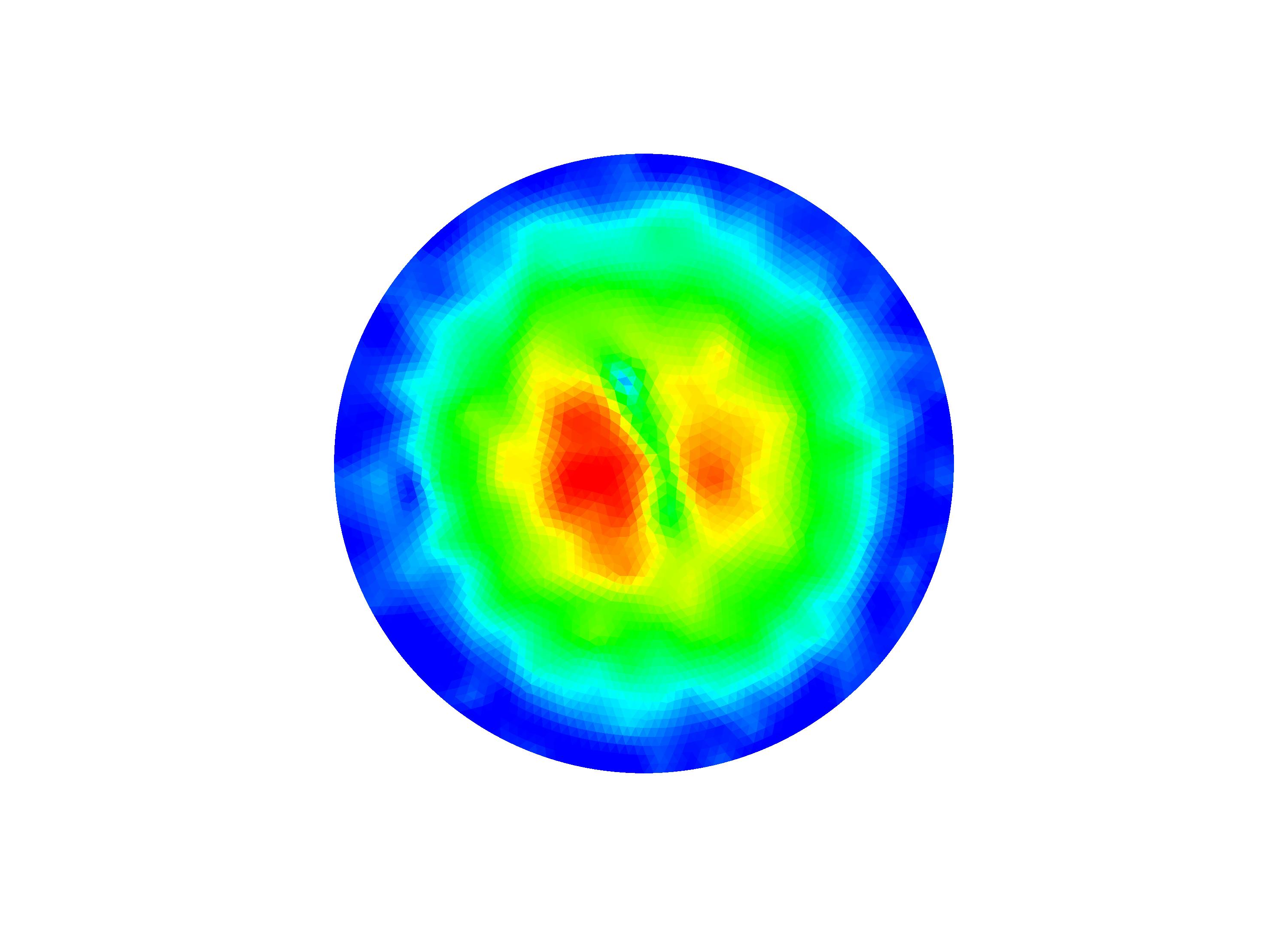}\\[0cm]
        \includegraphics[width=1.3\linewidth,trim={10.1cm 5cm 10.1cm 5cm},clip]{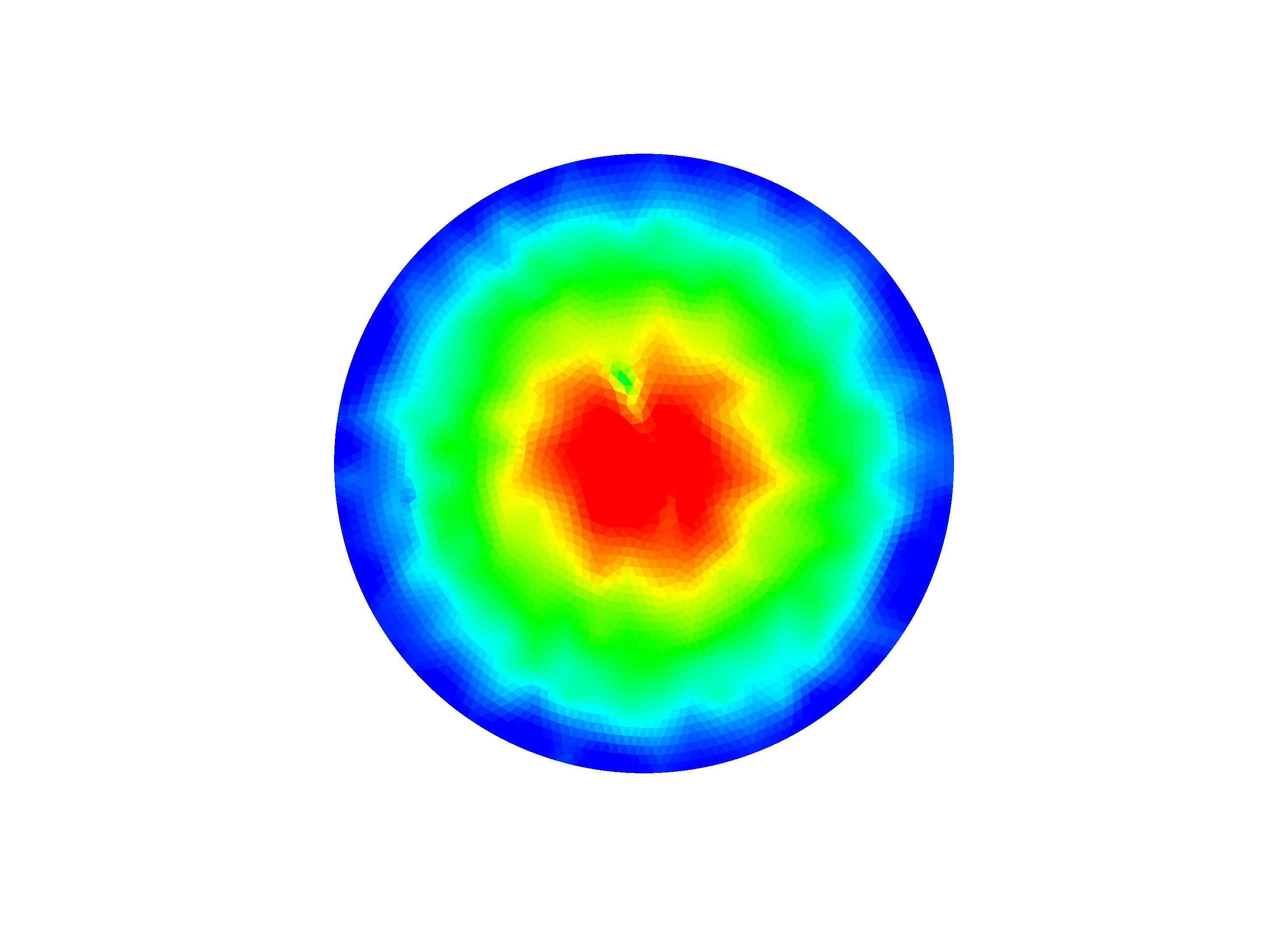}\\[0cm]
        \includegraphics[width=1.3\linewidth,trim={10.1cm 5cm 10.1cm 5cm},clip]{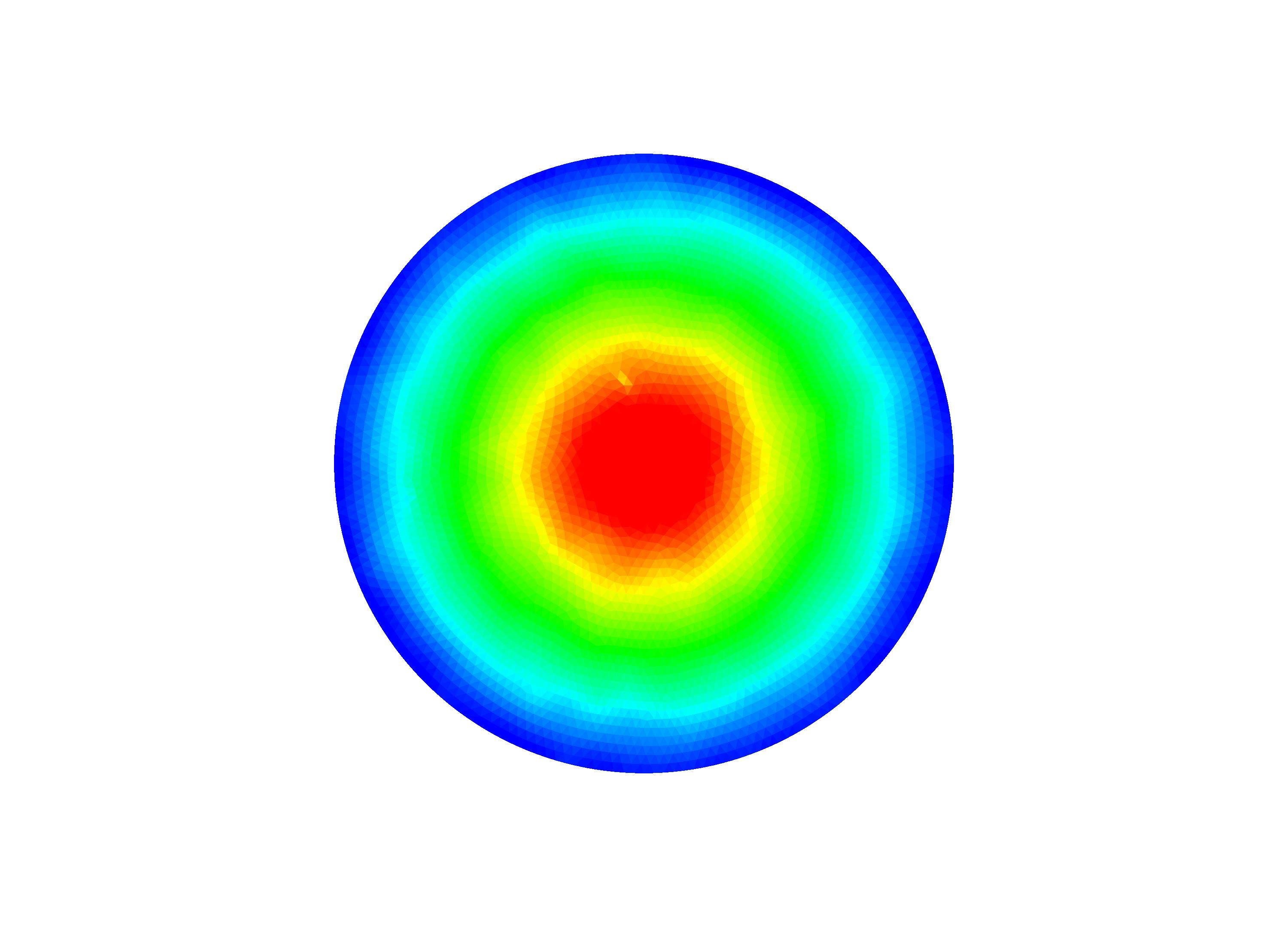}\\
    \end{minipage}  
    \textcolor{gray}{\hrule}
    \vspace{5pt}
    \hspace{34pt}
    \begin{minipage}[t]{0.16\textwidth}
    \hspace{15pt}\centering$u^\dagger(T)$
    \includegraphics[width=1.3\linewidth,trim={10.1cm 5cm 10.1cm 5cm},clip]{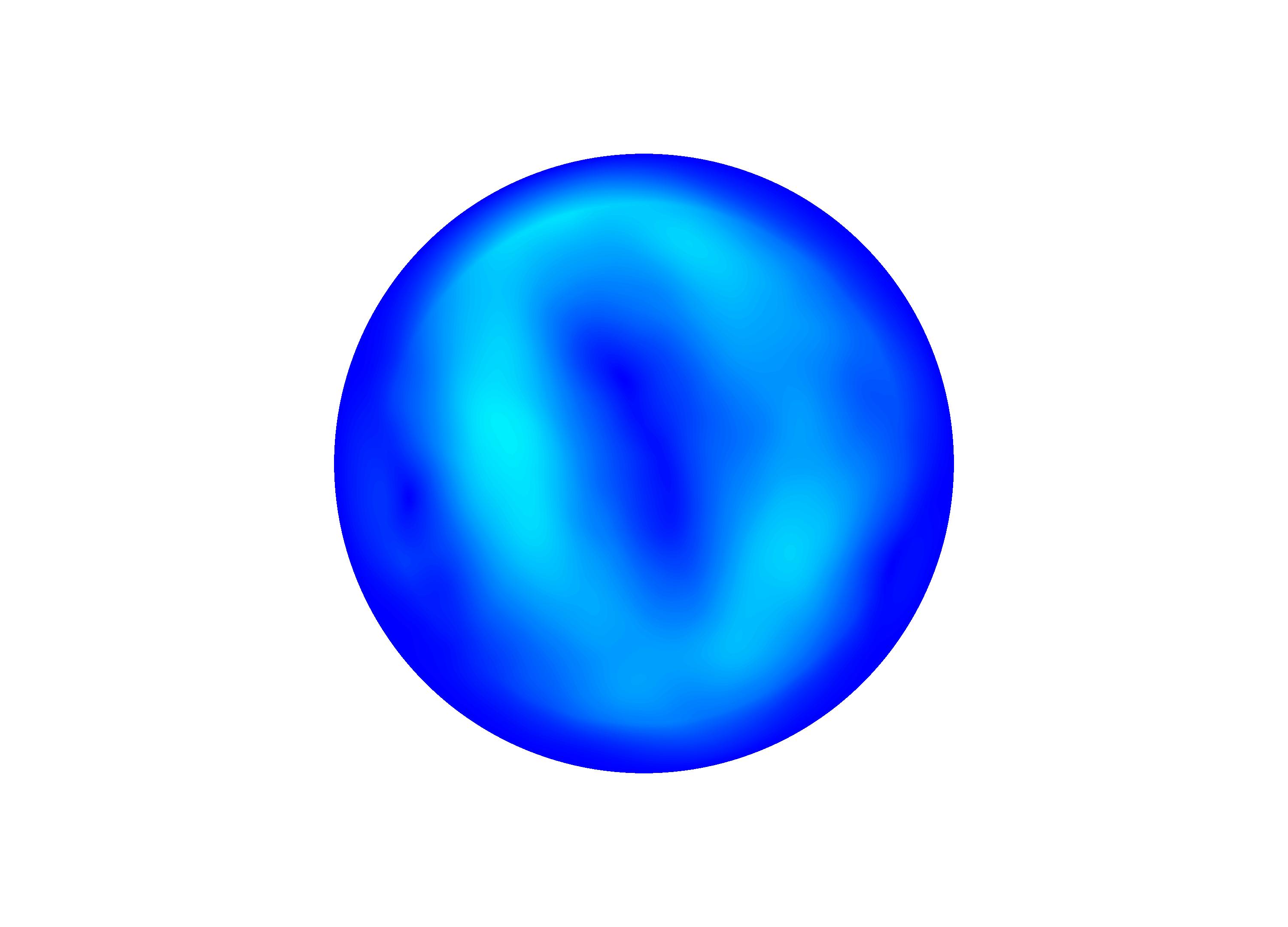}\\[0ex]
        \pgfplotscolorbardrawstandalone[
            colormap/jet,
            colorbar horizontal,
            point meta min=0,
            point meta max=3.5,
            colorbar style={
                height = 0.2cm,
                width=0.85\textwidth,
                /pgf/number format/fixed,
                /pgf/number format/precision=1.75,
                tick style={font=\tiny},
                xtick={0, 1.75,3.5},
                xticklabels={0, 1.75,3.5},
            }
        ]\\[0.2cm]
    \end{minipage}
    \hfill
    \begin{minipage}[t]{0.16\textwidth}
    \hspace{15pt}\centering$c^\dagger$
    \includegraphics[width=1.3\linewidth,trim={10.1cm 5cm 10.1cm 5cm},clip]{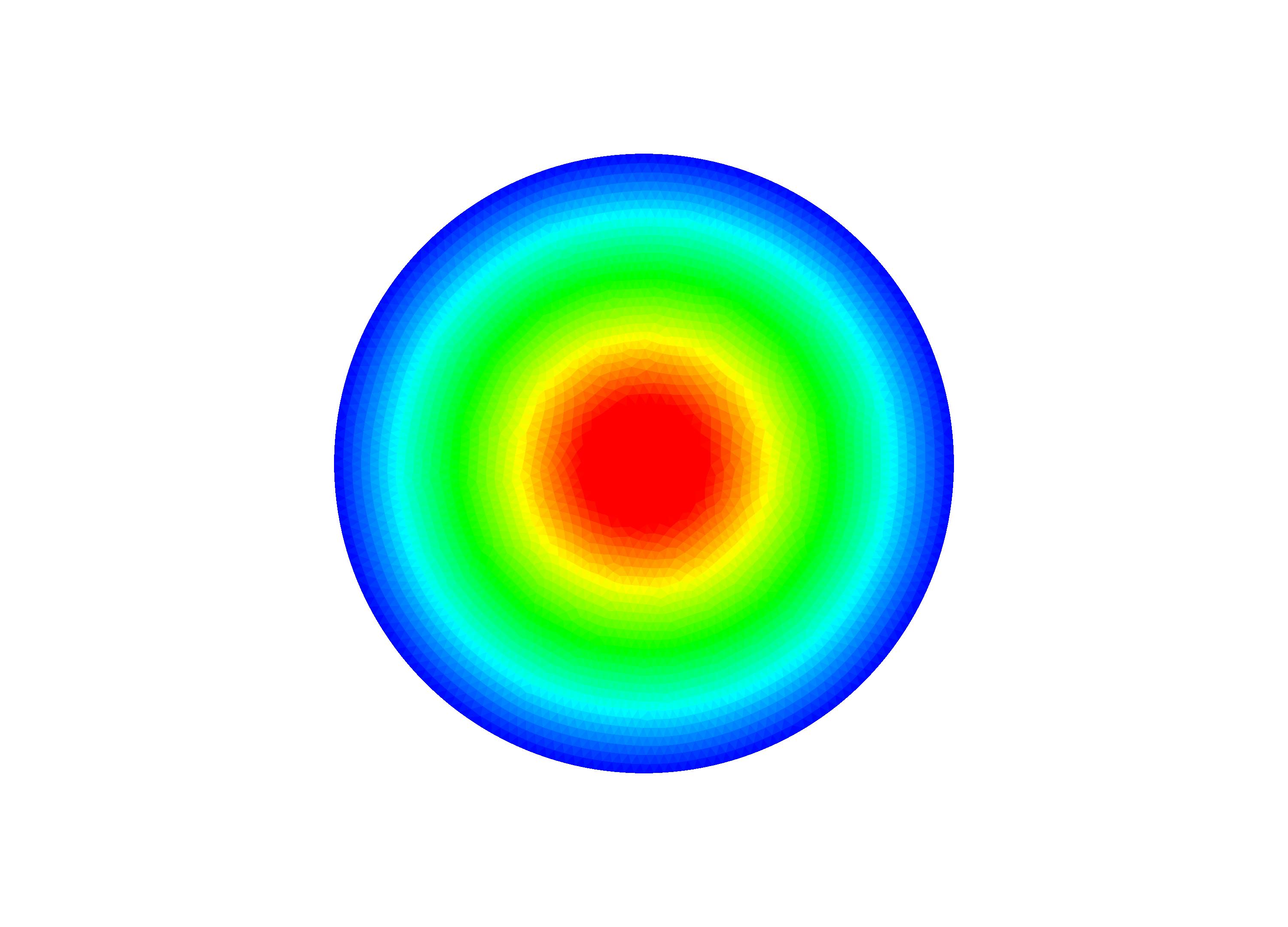}\\[0.5ex]
    \pgfplotscolorbardrawstandalone[
            colormap/jet,
            colorbar horizontal,
            point meta min=0,
            point meta max=2.5,
            colorbar style={
                height = 0.2cm,
                width=0.85\textwidth,
                /pgf/number format/fixed,
                /pgf/number format/precision=1.25,
                tick style={font=\tiny},
                xtick={0.0, 1.25,2.5},
                xticklabels={0.0, 1.25,2.5},
            }
        ]\\[0.2cm]
    \end{minipage}  
    \hspace{30pt}
    \hspace{30pt}
    \begin{minipage}[t]{0.16\textwidth}
    \hspace{15pt}\centering$u^\dagger(T)$
    \includegraphics[width=1.3\linewidth,trim={10.1cm 5cm 10.1cm 5cm},clip]{nonlinear_cprob_u_exact.jpg}\\[0ex]
            \pgfplotscolorbardrawstandalone[
            colormap/jet,
            colorbar horizontal,
            point meta min=0,
            point meta max=3.5,
            colorbar style={
                height = 0.2cm,
                width=0.85\textwidth,
                /pgf/number format/fixed,
                /pgf/number format/precision=1.25,
                tick style={font=\tiny},
                xtick={0, 1.75,3.5},
                xticklabels={0, 1.75,3.5},
            }
        ]\\[0.2cm]
    \end{minipage} 
    \hfill
    \begin{minipage}[t]{0.16\textwidth}
    \hspace{15pt}\centering $c^\dagger$
    \includegraphics[width=1.3\linewidth,trim={10.1cm 5cm 10.1cm 5cm},clip]{nonlinear_cprob_q_exact.jpg}\\[0.5ex]
            \pgfplotscolorbardrawstandalone[
            colormap/jet,
            colorbar horizontal,
            point meta min=0,
            point meta max=2.5,
            colorbar style={
                height = 0.2cm,
                width=0.85\textwidth,
                /pgf/number format/fixed,
                /pgf/number format/precision=1.25,
                tick style={font=\tiny},
                xtick={0, 1.25,2.5},
                xticklabels={0, 1.25,2.5},
            }
        ]\\[0.2cm]
    \end{minipage}
    \caption{Nonlinear potential. Evolution of the state $u$ and potential parameter $c$ ran with clean data (left) and data with 5\% noise (right).}
    \label{fig::nc_evolution_big}
\end{figure}

    \begin{figure}
    \centering
    \begin{tabular}{cc}
        \begin{subfigure}[b]{0.5\textwidth}
            \captionsetup{labelformat=empty}
            \begin{center}
                \pgfplotscolorbardrawstandalone[
                    colormap/jet,    
                    colorbar horizontal,
                    point meta min= -0.5,
                    point meta max= 0.5,
                    colorbar style={
                        width=0.5\textwidth,
                        /pgf/number format/fixed,
                        /pgf/number format/precision=2,
                        xticklabel style={anchor=north},
                        every axis label/.append style={/pgf/number format/precision=3},
                        yticklabel style={/pgf/number format/none},
                        xtick distance=1/2, 
                        tick style={draw=none}}]
                \includegraphics[width=\linewidth]{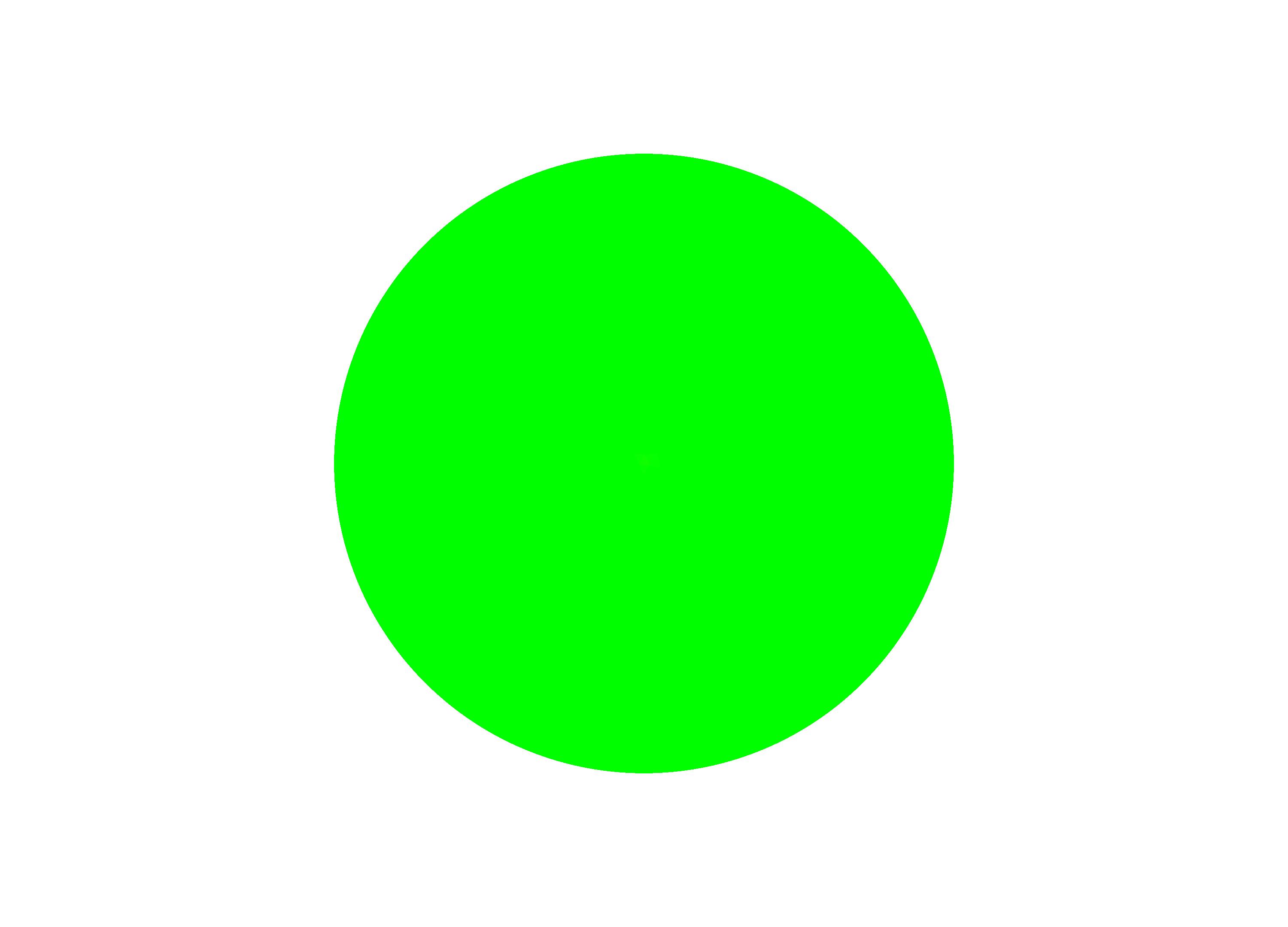} 
            \end{center}
        \end{subfigure}
        &
        \begin{subfigure}[b]{0.5\textwidth}
            \captionsetup{labelformat=empty}
            \begin{center}
                \pgfplotscolorbardrawstandalone[
                    colormap/jet,    
                    colorbar horizontal,
                    point meta min= -0.5,
                    point meta max= 0.5,
                    colorbar style={
                        width=0.5\textwidth,
                        /pgf/number format/fixed,
                        /pgf/number format/precision=2,
                        xticklabel style={anchor=north},
                        every axis label/.append style={/pgf/number format/precision=3},
                        yticklabel style={/pgf/number format/none},
                        xtick distance=1/2, 
                        tick style={draw=none}}]
                \includegraphics[width=\linewidth]{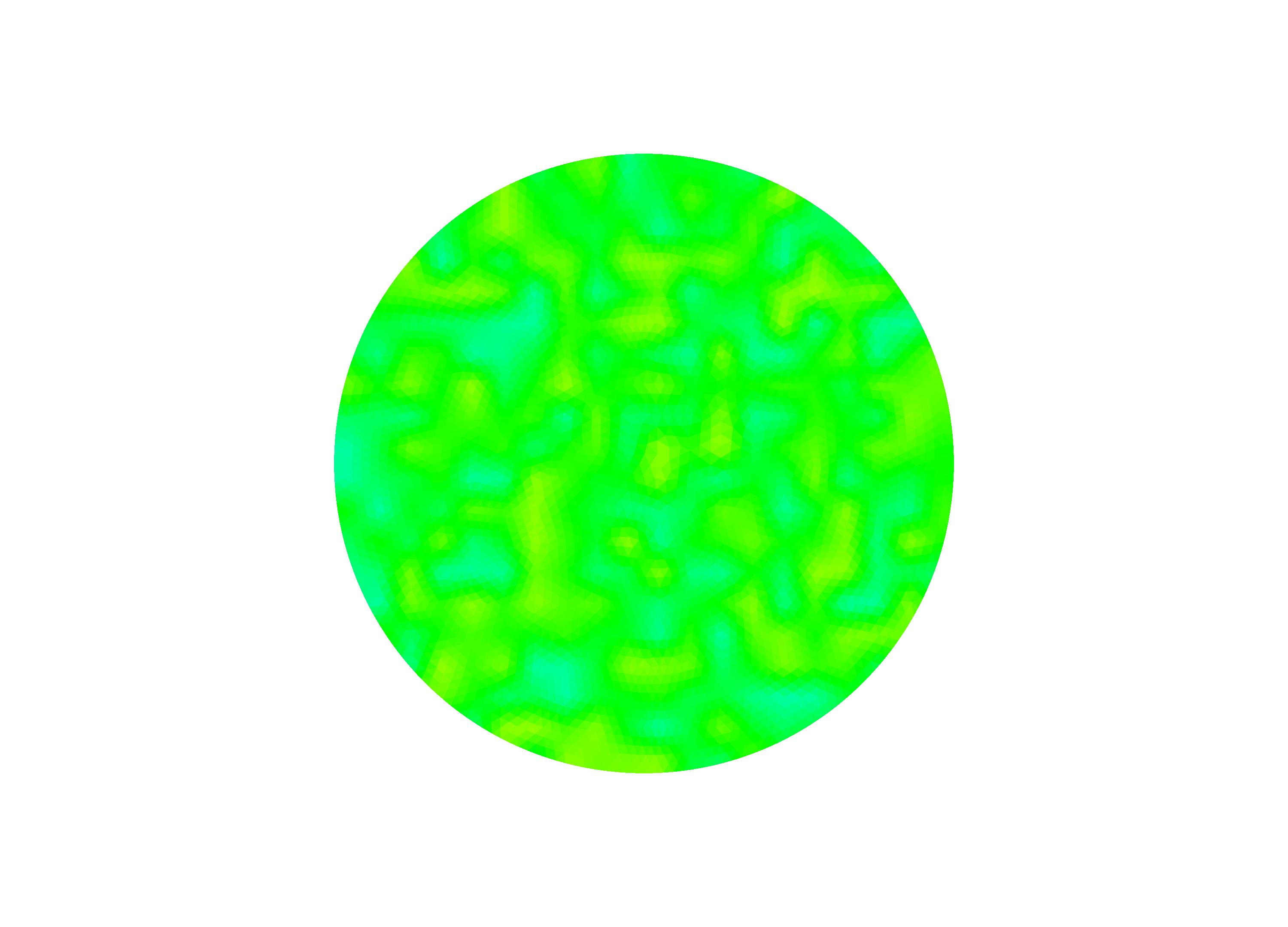} 
            \end{center}
        \end{subfigure}
    \end{tabular}
    \caption{Nonlinear potential. Error field $c(T)-c^\dagger$ resulted from MRAS ran with clean data (left) and data with 5\% noise (right).}    \label{fig::nc_diff}
\end{figure}

\section{Modified Allen-Cahn equation: state and parameter nonlinearity}\label{sec::allen-cahn}

In our last example, we consider a variation of the potential problem that is nonlinear in both state and parameter, and whose nonlinearity is truly irreducible. The PDE model consists of the Allen-Cahn-type cubic nonlinear reaction. This is paired with the same kind of nonlinear parameter dependence as in \eqref{eq:cproblem}; however, the nonlinear term $u^3$ here prevents reduction to a linear parameter dependence. This example represents the highly complex kind of nonlinearities that our proposed MRAS can handle, as was claimed in \cite{Tram_Kaltenbacher_2021}. We refer to this parabolic equation as the modified Allen-Cahn equation with unknown potential $c$; it is given by
\begin{equation}\label{eq:allen-cahn}
\begin{aligned}
D_t u -\Delta u +cu^3+u\,|c|^{\frac{2}{3}}c &= g \qquad \text{in }I\times\dom := \I\times [-2,2]\times[-1,1]\\
u|_\bou&=h \qquad\text{in }I\\
u(t=0)&=u_0 \hspace{17pt} \text{in }\dom.
\end{aligned}
\end{equation}

We moreover adopt the same positivity assumptions on $h$, $u^\dagger$ as in Section \ref{sec::nonlinear_cprob}. The Allen-Cahn equation is present in natural convection, phase separation \cite{AllenCahn} and traveling wave dynamics \cite{GildingKersner}, to name a few applications.

\subsection{MRAS analysis}

Due to their structural similarities, the analysis for the MRAS of the modified Allen-Cahn equation largely agrees with that presented in Proposition \ref{prop:nonlinear-potential} for the nonlinear potential problem, with only minor changes.

\begin{proposition}

For the modified Allen-Cahn equation \eqref{eq:allen-cahn} with unknown potential $c$ and data $\uobs$, the MRAS \eqref{mras} takes the form 
\begin{equation}\label{mras-allen-cahn}
\begin{aligned}
D_t c +\sigma\left(D_t\uobs-\Delta \uobs+c\uobs^3+c|c|^\frac{2}{3}\uobs-g\right) & = \uobs\left(1+\frac{5}{3}|\tilde{c}|^\frac{2}{3}\right)(u-\uobs), \\
D_tu -\Delta \uobs +c\uobs^3+c|c|^\frac{2}{3}\uobs + \Cc(\|c\|_H) (u-\uobs) & = g, \\
(c,u)(0) & = (c_0,u_0),
\end{aligned}
\end{equation}
with $\sigma=1$, state space $U:=H^1(\dom)$ and parameter space $H:=L^2(\dom)$. The linear bounded operator $\Cc(\|c\|_H)$ is given as
\begin{align}\label{c-C-ac}
\Cc(\|c\|_H)(u-z)
:= -\left(\frac{(L^{\tilde{c},\uobs}(\|c\|_H))^2}{2\underline{z}}+M\right)\Delta(u-z)
\end{align}
with $L^{\tilde{c},z}(\|c\|)=\frac{5}{3}\,\overline{z}\, C_{H^1\to L^6}\left(\|c\|^{2/3}_{L^2}+\|c^\dagger\|^{2/3}_{L^2}+\|\tilde{c}\|^{2/3}_{L^2}\right)$, and
$\overline{z}, \underline{z}$ as in \eqref{expanded_c-positivivty} and with any $M>0$. Assumptions \ref{A-lip}, \ref{A-coe}, \ref{A-funcC} hold.
\end{proposition}

\begin{proof}  Since \eqref{eq:allen-cahn} is a modification of \eqref{eq:cproblem} with the reaction law $cu^3$ instead of $cu$. As this is nonlinearity in the state, one need only change $cz$ to $cz^3$ in $f(c,z)$ whenever it appears.
Also by this reason, Lipschitz continuity w.~r.~t parameter remains unchanged, resulting in the same linear bounded, coercive operator $\Cc(\|c\|_H)$ as in Proposition \ref{prop:nonlinear-potential}.
\end{proof}

\begin{corollary}

The weak form of the MRAS \eqref{mras-allen-cahn} with  $\cC(\|c\|_H)$ as in Remark \ref{rem:C-scale} in a semi-implicit Euler scheme is
\begin{alignat}{2}
\label{ac-c}
    \int_\Omega( c_{n+1}&-c_n)\,s\d x  + \sigma\Delta t \int_\Omega\left((c_{n+1}-c_n)\uobs_{n+1}^3+|c_{n}|^\frac{2}{3}(c_{n+1}-c_n)\uobs_{n}\right)s\d x \notag\\ 
    = & +\Delta t\int_\Omega  \uobs_{n}\left(1+\tfrac{5}{3}|{\tilde{c}}|^\frac{2}{3}\right)(u_{n}-\uobs_{n})s \d x \\
    &-\sigma\Delta t \int_\dom \nabla \uobs_{n+1}\cdot\nabla s\d x+\sigma\Delta t\int_{\partial \Omega}\nabla \uobs_{n+1}\,s \cdot \nvec\d S\notag\\&-\sigma\Delta t \int_\Omega\left(D_t \uobs_{n+1}+{c_{n}\uobs_{n+1}^3+|c_{n}|^\frac{2}{3}c_{n}\uobs_{n}}-g_{n+1}\right)\,s \d x  \nonumber\\
    \label{ac-u}
    \int_\Omega (u_{n+1}&-u_n)v \d x +\Delta t\int_\Omega \left((c_{n+1}-c_n)\uobs_{n+1}^3+|c_n|^\frac{2}{3}(c_{n+1}-c_n)\uobs_{n+1}\right)v\d x\notag\\
    & +\Delta t\int_\Omega \,C_{c_n}\nabla (u_{n+1}-{u_n})\cdot  \nabla v\d x\notag\\
    = & -\Delta t \int_\Omega \left({c_n\,\uobs_{n+1}^3+|c_n|^\frac{2}{3}c_{n}\uobs_{n+1}}-g_{n+1}\right)\,v\d x\\
    &-\Delta t\int_\Omega\nabla \uobs_{n+1}\cdot\nabla v\d x  +\Delta t\int_{\partial \Omega}\nabla \uobs_{n+1}\,v\,\cdot \nvec\d S \nonumber\\
    &+\Delta t\int_\Omega \,C_{c_n}\nabla \left(\uobs_{n} - {u_n}\right)\cdot\nabla v\d x\nonumber\\
    (c,u)(0) & = (c_0,u_0)
\end{alignat} 
for any $v\in U=H^1(\dom)$, $s\in H=L^2(\dom)$, with $\sigma=1$ and with the constant $C_{c_n}:= C^2_{H^1\to L^6}\frac{2\overline{z}^2}{\underline{z}}\left(\|c_{n}\|_{L^2}^{2/3}+\|c^\dagger\|_{L^2}^{2/3}+\|\ctil\|_{L^2}^{2/3}\right)^2+1$.
\end{corollary}

\begin{proof}

This proof is exactly analogous to that of Corollary \ref{corr:mras-cproblem-weak}; the only modification is that  $c_n\uobs_{n+1}$ is now replaced by $c_n\uobs_{n+1}^3$.

\end{proof}
\subsection{Numerical results}
For the ground truth parameter, we consider the three-part material described by the equation
\[
c^\dagger(x)
:=
\begin{cases}
1, & \text{if } x_1 + x_2 < 1 \text{ and } x_1 - 2x_2 < - 0.4,\\
4, & \text{if } x_1 + x_2 \ge 1,\\
2, & \text{if } x_1 + x_2 < 1 \text{ and } x_1 - 2x_2 \ge - 0.4.
\end{cases}
\]
We complement this choice with an simple sinusoidal true state
\[
\utrue(x,t) := \sin\left(\dfrac{\pi}{4}(x_1-2)\right)\sin\left(\dfrac{\pi}{2}+(x_2-1)\right)\dfrac{10-t}{10}+1,
\]
which decays linearly in time towards the constant offset $1$, ensuring the positivity required by the analysis. The source term is then computed as $g:=D_t u^\dagger -\Delta u^\dagger +c(u^\dagger)^3+u^\dagger\,|c|^{\frac{2}{3}}c$. We run the MRAS with a total of four different relative noise levels: $0\%$, $5\%$, $10\%$ and $20\%$

Regarding discretization in time and space, a coarse mesh with $h_\mathrm{max}=0.1$ and a constant time step $\Delta t=0.001$ are employed, and the evolution is again observed until the final time $T=5$ on the usual function space setup. The initial state $u_0$ is set to $u^\dagger(0)$, and the initial parameter is $c_0:=0$. These hyperparameters are summarised in Table \ref{table:allen-cahn}.
\begin{table}[!ht]
    \centering
    \begin{tabular}{c|c}
         Spatial domain, mesh size & $\Omega=[-2,2]\times[-1,1]$, $h_\mathrm{max} = 0.1$\\
         $\#$dofs for $U_h$, $\#$dofs for $Q_h$  &8542, 1858\\
         Max time, time step, \#steps & $T=5$, $\Delta t=0.001$, 5000 steps\\
         Source term & $D_t u^\dagger -\Delta u^\dagger +c(u^\dagger)^3+u^\dagger\,|c|^{\frac{2}{3}}c$ \\
         Noise level & $0\%$, $5\%$, $10\%$, $20\%$\\
    \end{tabular}
    \caption{Setup for the modified Allen-cahn equation.}
    \label{table:allen-cahn}
\end{table}
\paragraph{Numerical results}

\begin{figure}
    \centering
    \begin{minipage}[t]{0.98\textwidth}
        {\small
        \hspace{30pt}\centering
        t=0.001\hspace{6pt}
        t=0.05\hspace{6pt}
        t=0.075\hspace{6.5pt}
        t=0.2\hspace{10pt}
        t=0.5\hspace{15pt}
        t=1.5\hspace{16pt}
        t=5\hspace{20pt}
        \quad $c^\dagger$}
    \end{minipage}
    \begin{minipage}[t]{1\textwidth}
        \centering
        \begin{subfigure}[b]{0.08\textwidth}
            \pgfplotscolorbardrawstandalone[
                    colormap/jet,
                    point meta min=0,
                    point meta max=4.25,
                    colorbar style={
                        height = 0.8cm,
                        width=0.15\textwidth,
                        /pgf/number format/fixed,
                        /pgf/number format/precision=1,
                        tick style={font=\tiny},
                        ytick={0, 2,4.25},
                        yticklabels={0,2,4.25},                    
                    }
                ]
        \end{subfigure}\hfill
        \begin{subfigure}[b]{0.11\textwidth}
            \includegraphics[width=\linewidth,trim={15cm 0cm 15cm 15cm},clip]{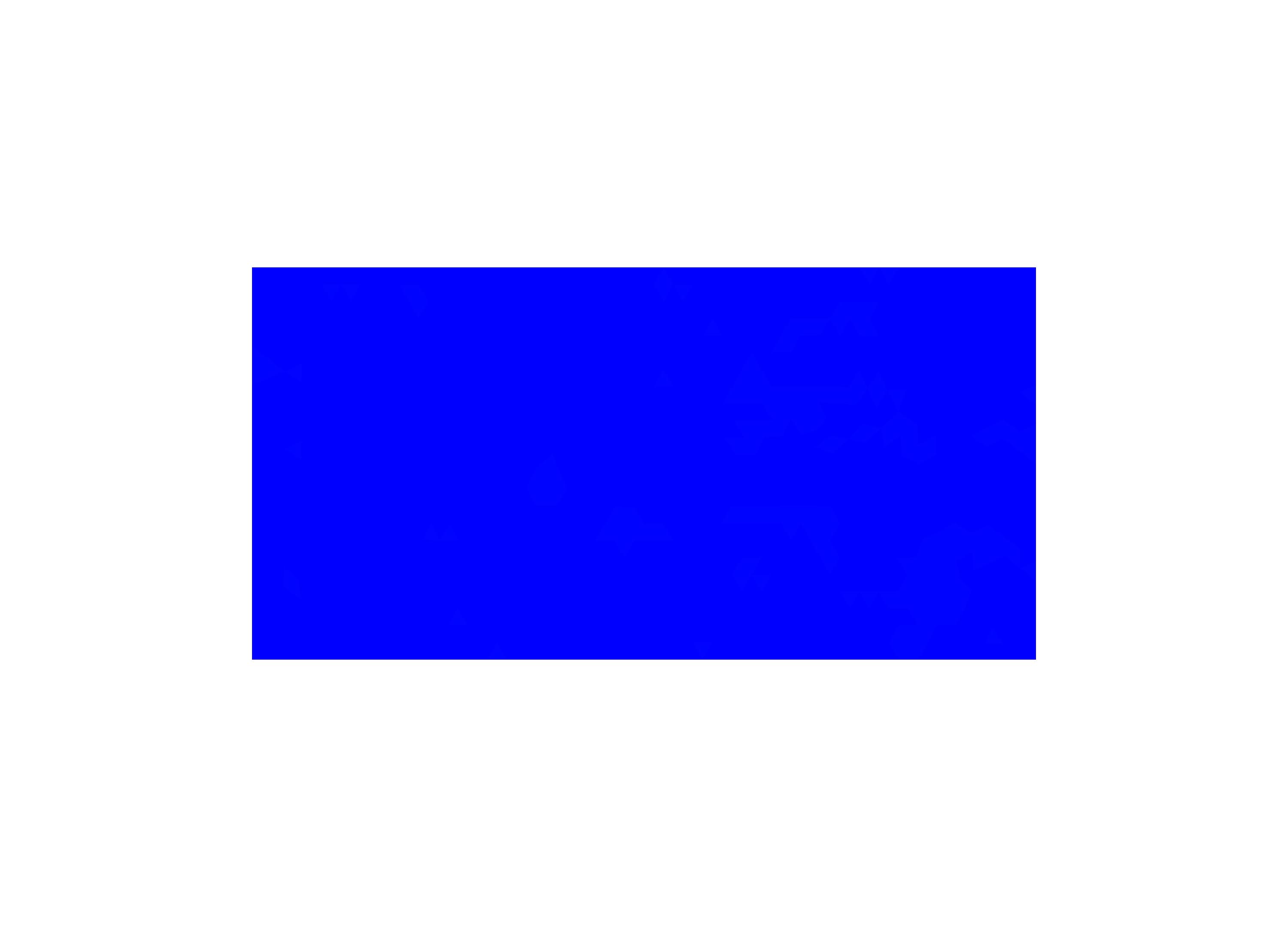}
        \end{subfigure}\hfill
        \begin{subfigure}[b]{0.11\textwidth}
            \includegraphics[width=\linewidth,trim={15cm 0cm 15cm 15cm},clip]{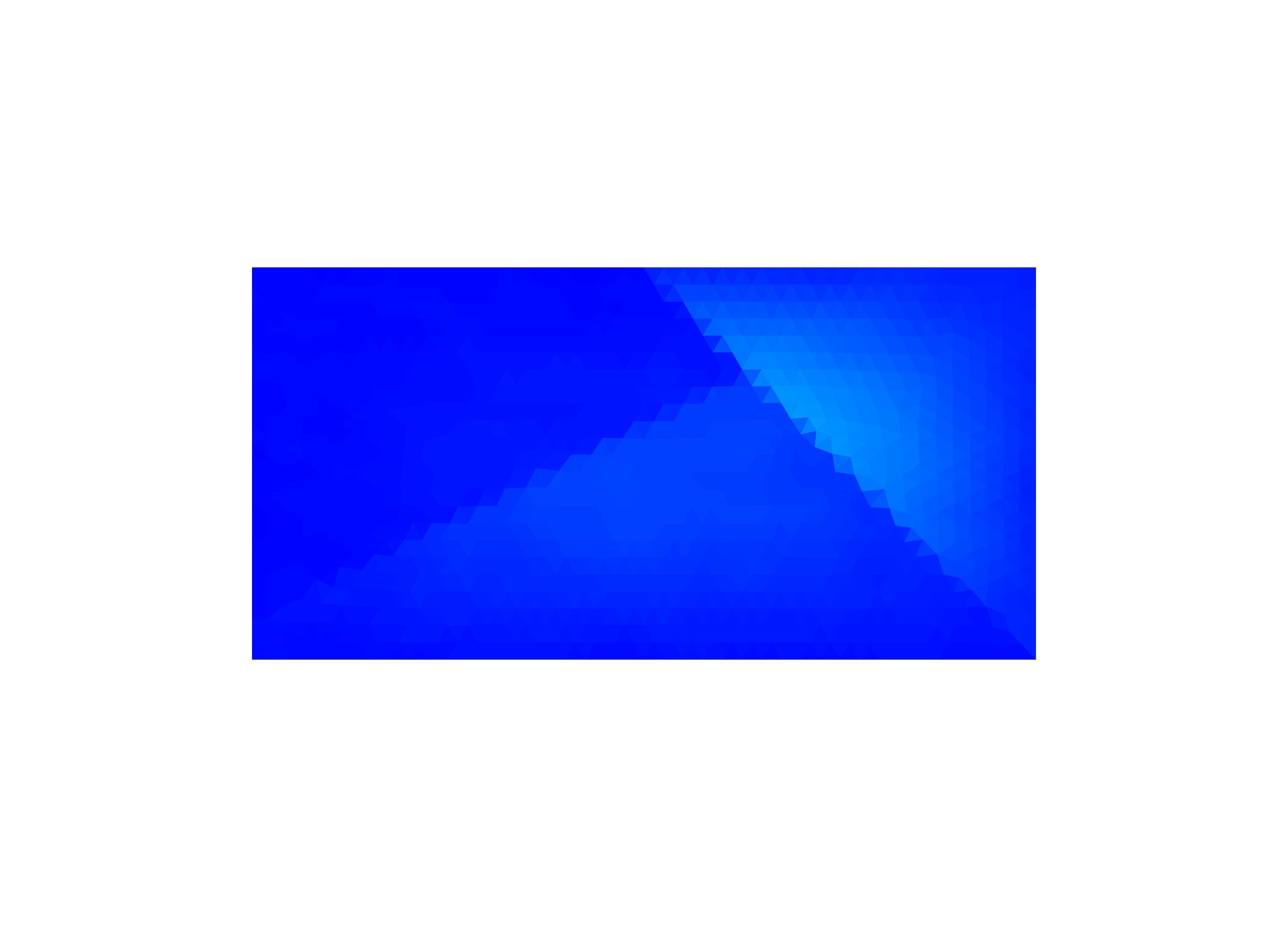}
        \end{subfigure}\hfill
        \begin{subfigure}[b]{0.11\textwidth}
            \includegraphics[width=\linewidth,trim={15cm 0cm 15cm 15cm},clip]{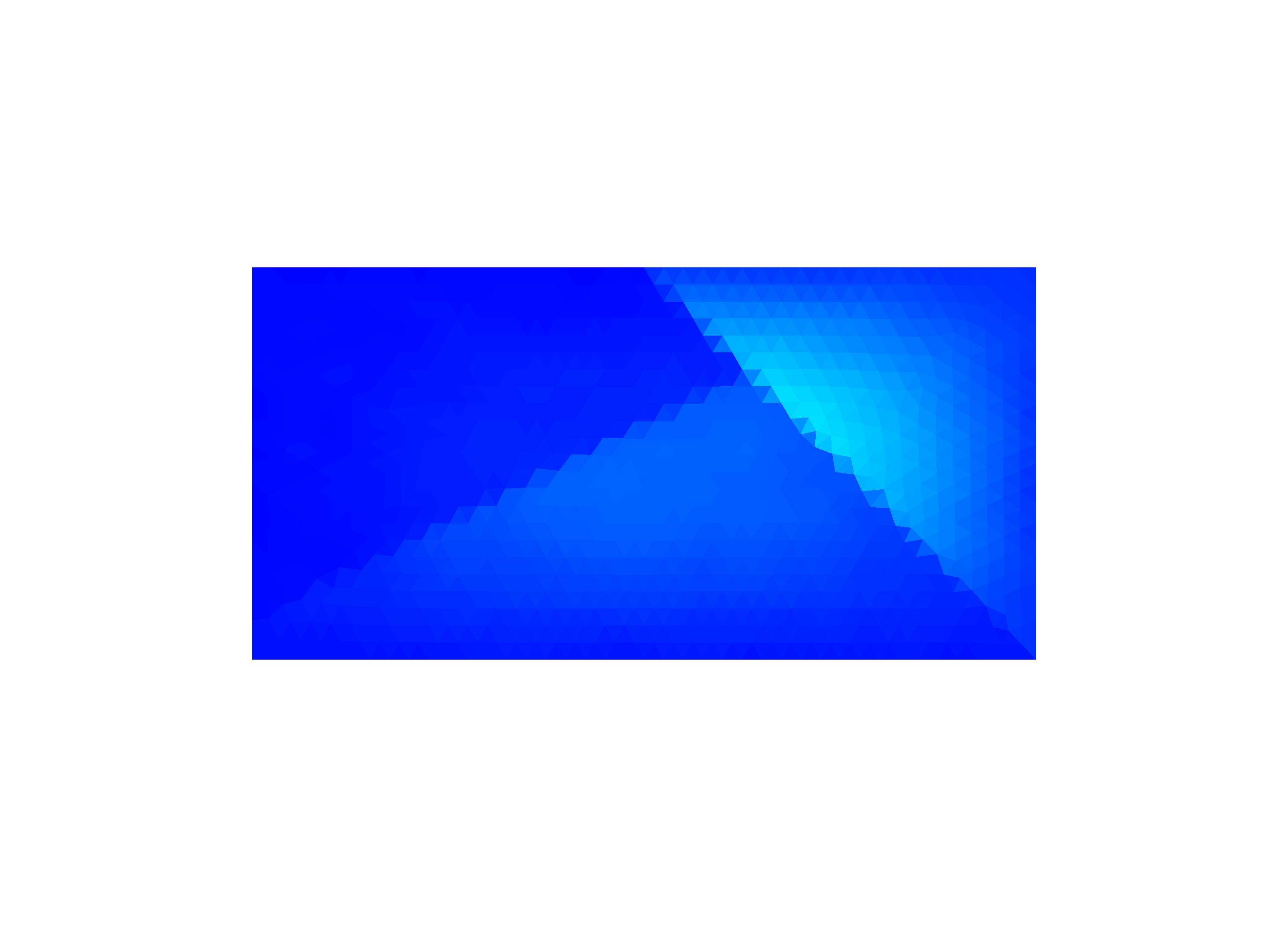}
        \end{subfigure}\hfill
        \begin{subfigure}[b]{0.11\textwidth}
            \includegraphics[width=\linewidth,trim={15cm 0cm 15cm 15cm},clip]{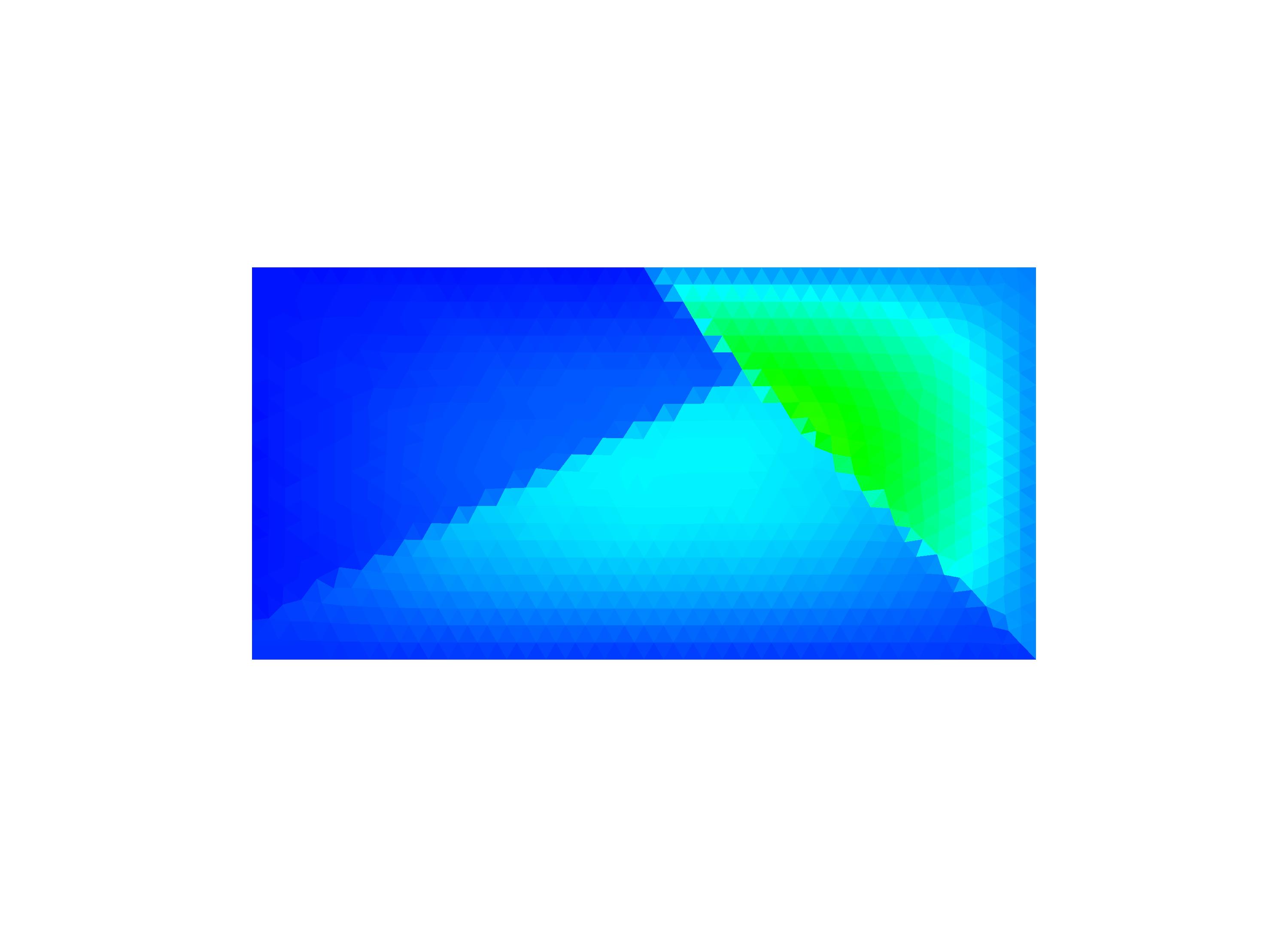}
        \end{subfigure}\hfill
        \begin{subfigure}[b]{0.11\textwidth}
            \includegraphics[width=\linewidth,trim={15cm 0cm 15cm 15cm},clip]{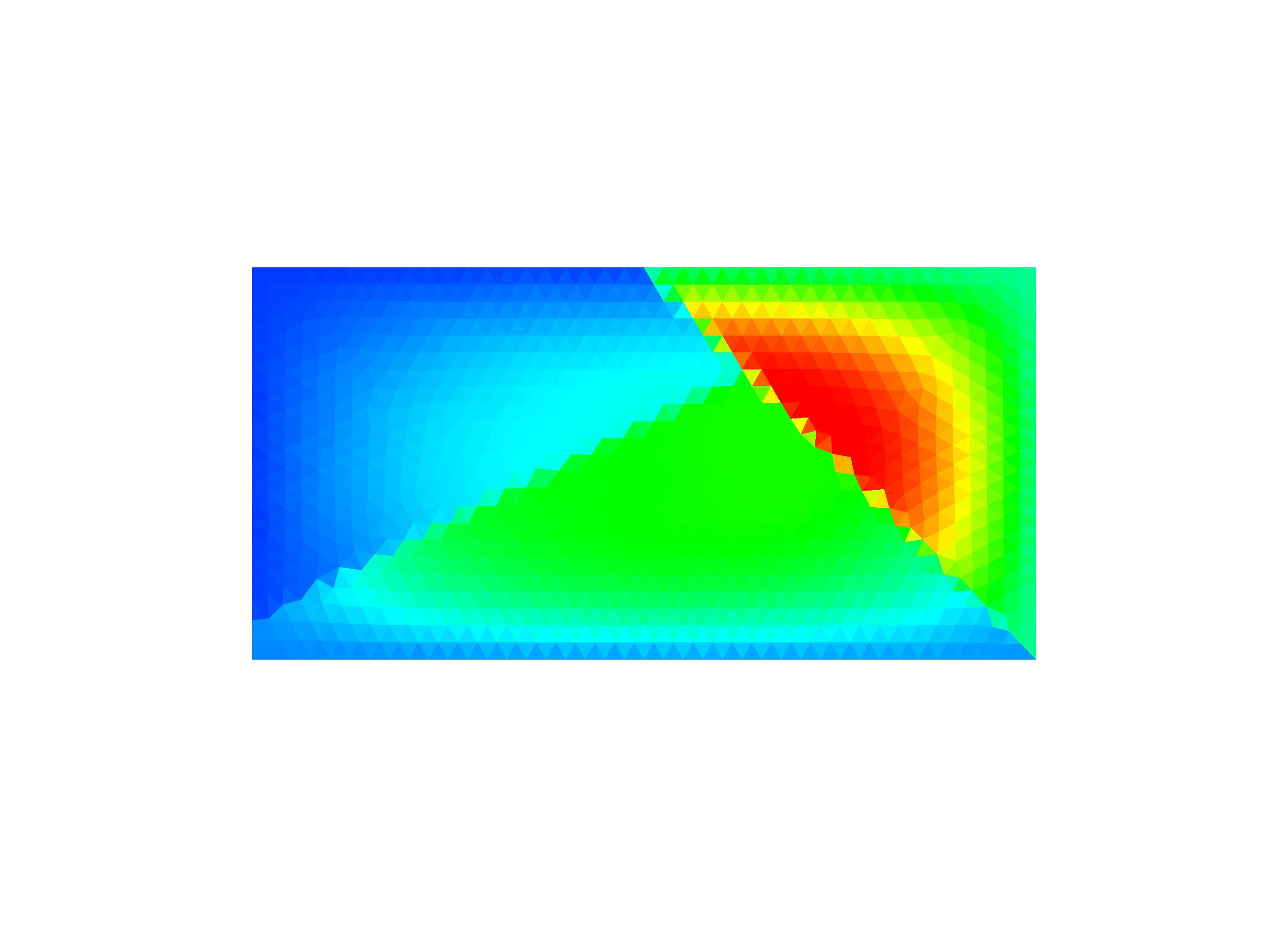}
        \end{subfigure}\hfill
        \begin{subfigure}[b]{0.11\textwidth}
            \includegraphics[width=\linewidth,trim={15cm 0cm 15cm 15cm},clip]{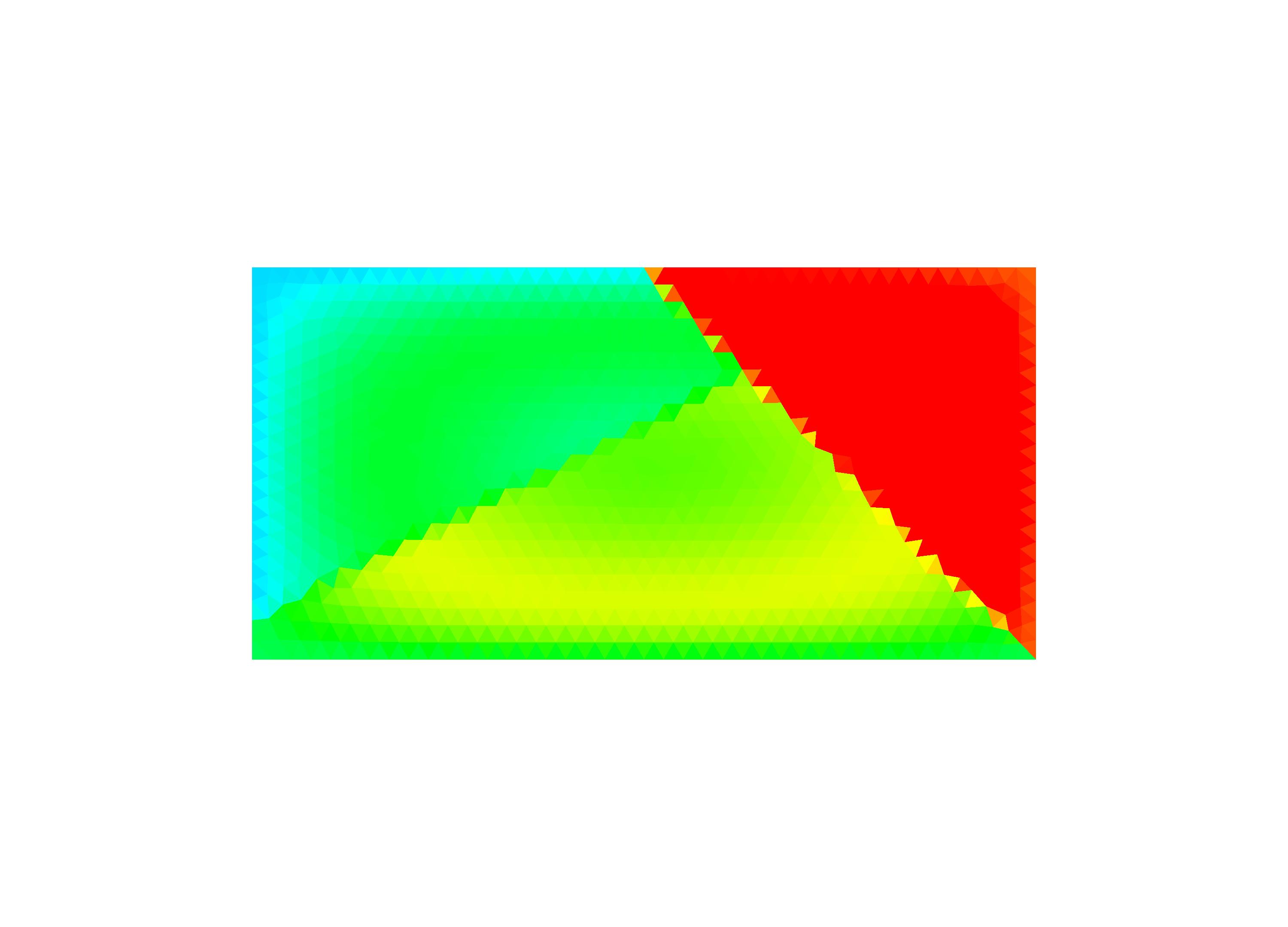}
        \end{subfigure}\hfill
        \begin{subfigure}[b]{0.11\textwidth}
            \includegraphics[width=\linewidth,trim={15cm 0cm 15cm 15cm},clip]{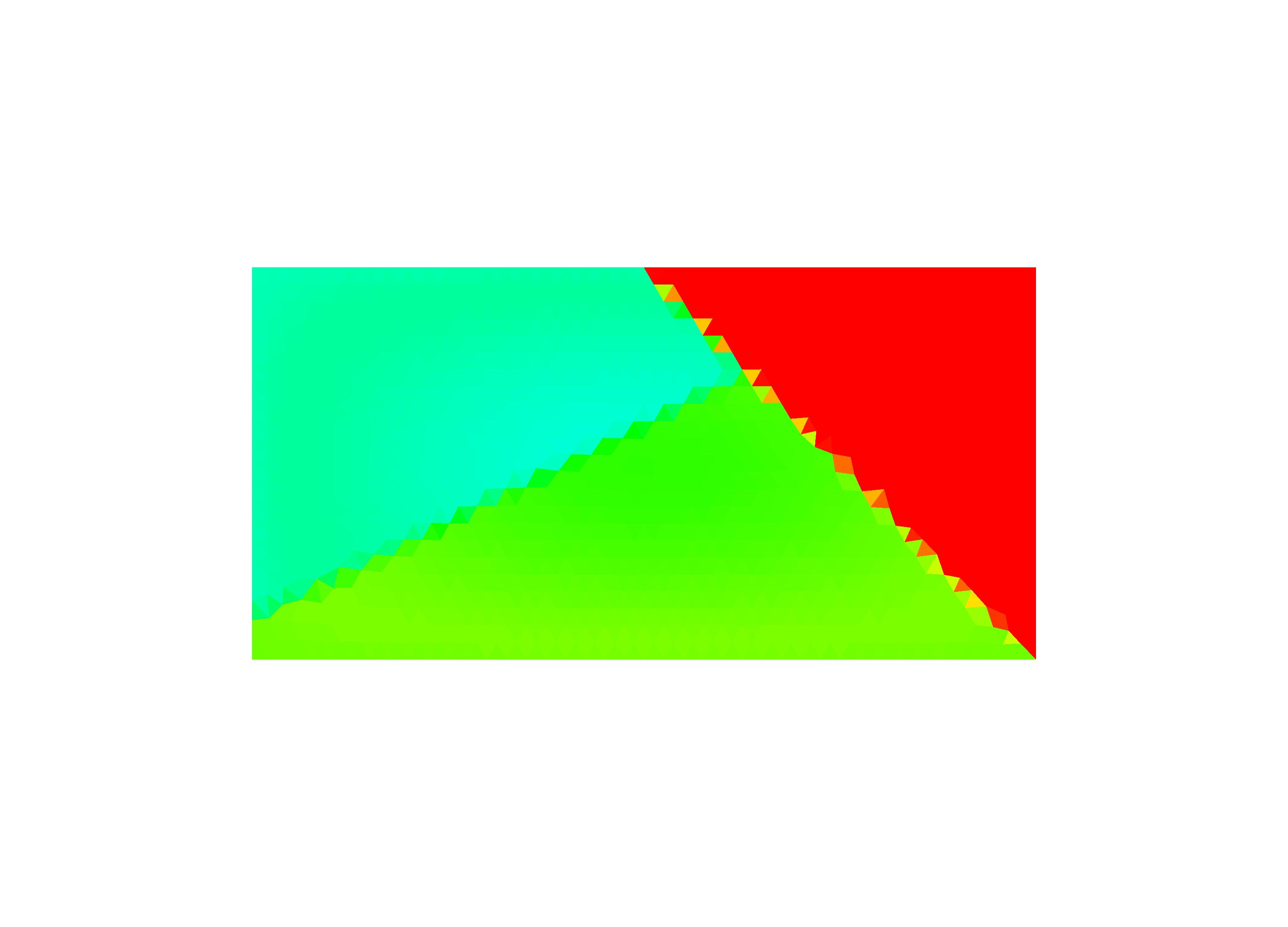}
        \end{subfigure}\hfill
        \hspace{1pt}\hfill
        \begin{subfigure}[b]{0.11\textwidth}
            \includegraphics[width=\linewidth,trim={15cm 0cm 15cm 15cm},clip]{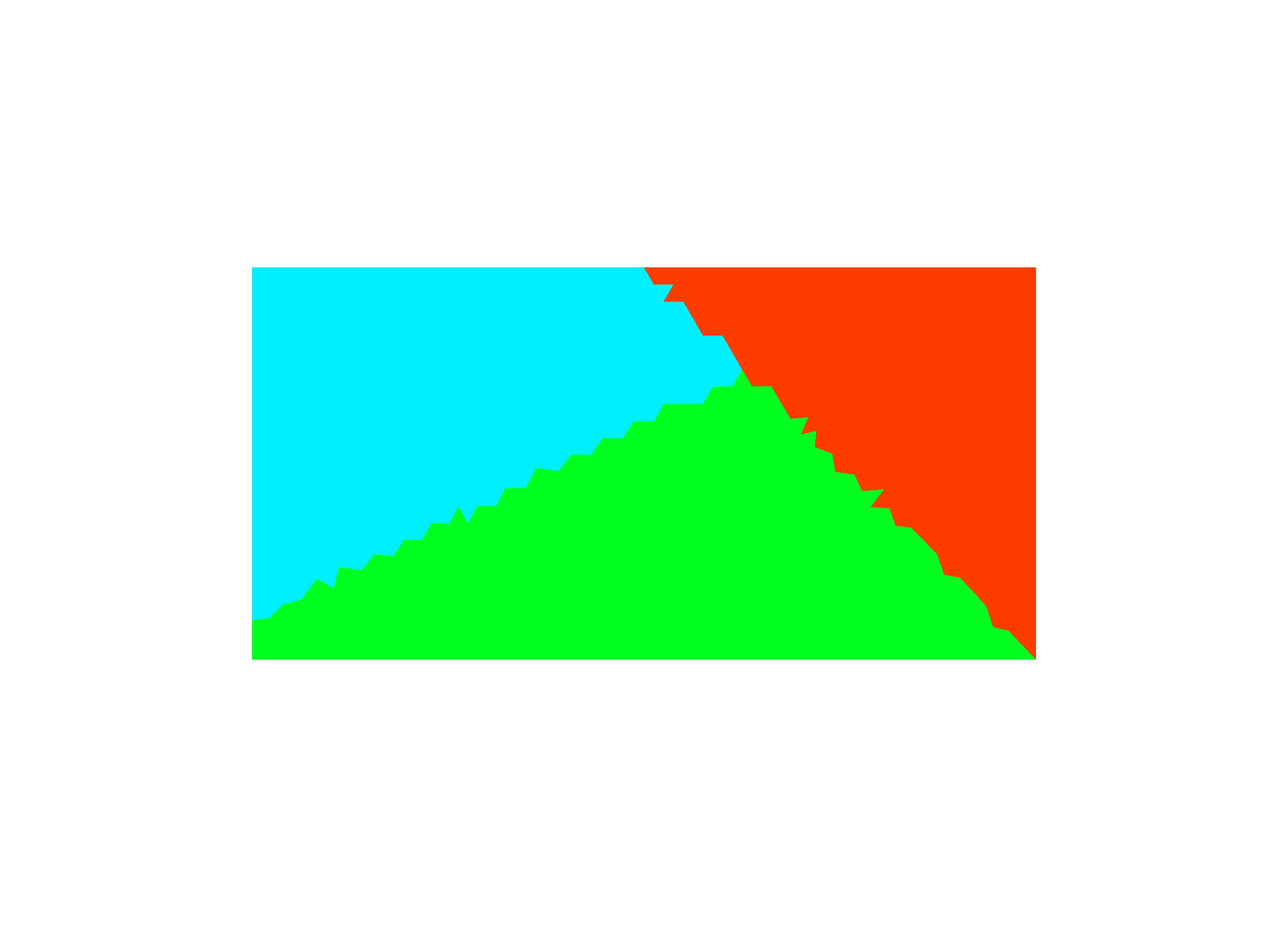}
        \end{subfigure}   
        \\
        \vspace{0.3cm}
        \begin{subfigure}[b]{0.08\textwidth}
            \pgfplotscolorbardrawstandalone[
                    colormap/jet,
                    point meta min=0,
                    point meta max=4.5,
                    colorbar style={
                        height = 0.8cm,
                        width=0.15\textwidth,
                        /pgf/number format/fixed,
                        /pgf/number format/precision=1,
                        tick style={font=\tiny},
                        ytick={0, 2,4.25},
                        yticklabels={0,2,4.25},                    
                    }
                ]
        \end{subfigure}\hfill
        \begin{subfigure}[b]{0.11\textwidth}
            \includegraphics[width=\linewidth,trim={15cm 0cm 15cm 15cm},clip]{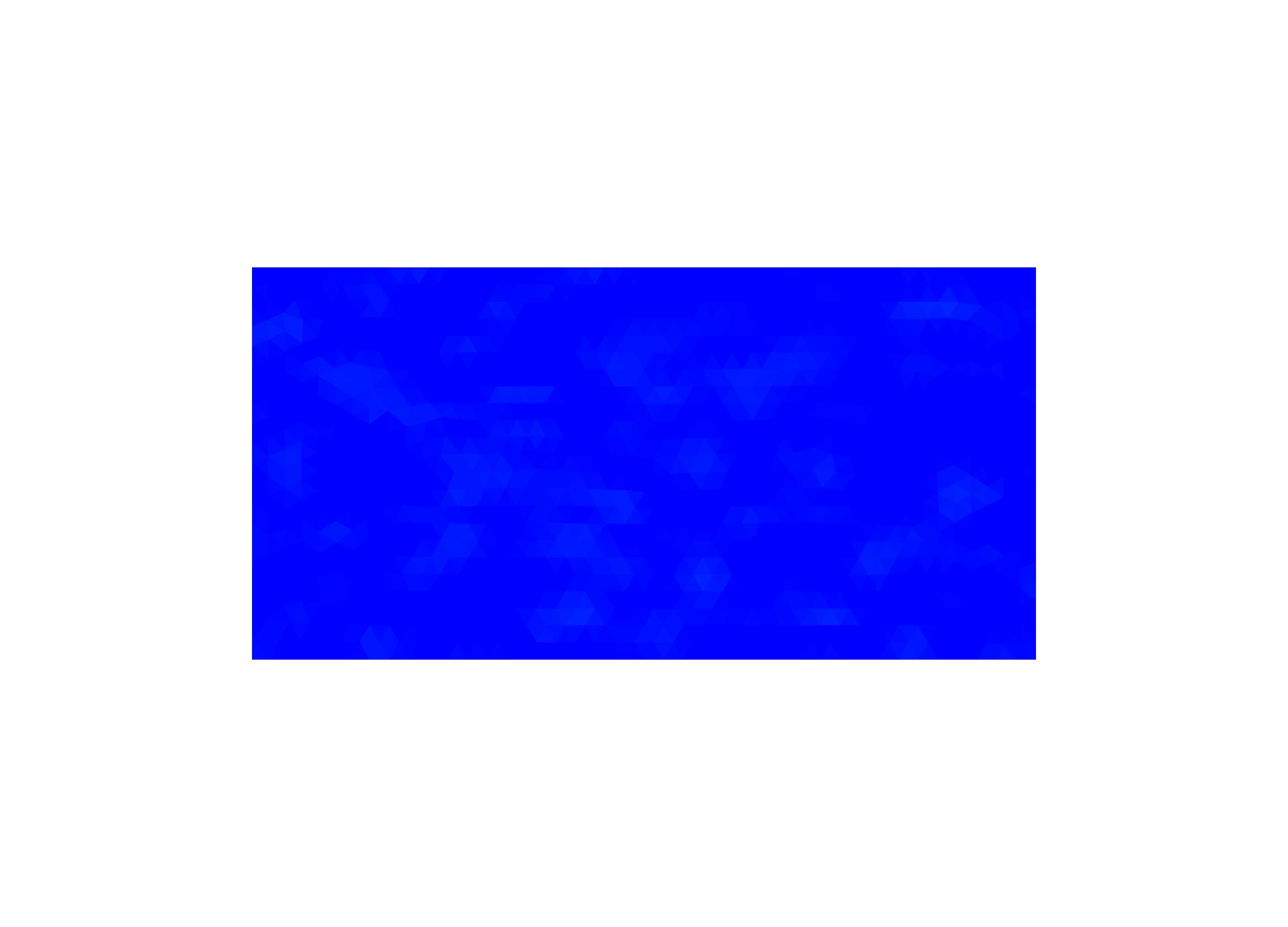}
        \end{subfigure}\hfill
        \begin{subfigure}[b]{0.11\textwidth}
            \includegraphics[width=\linewidth,trim={15cm 0cm 15cm 15cm},clip]{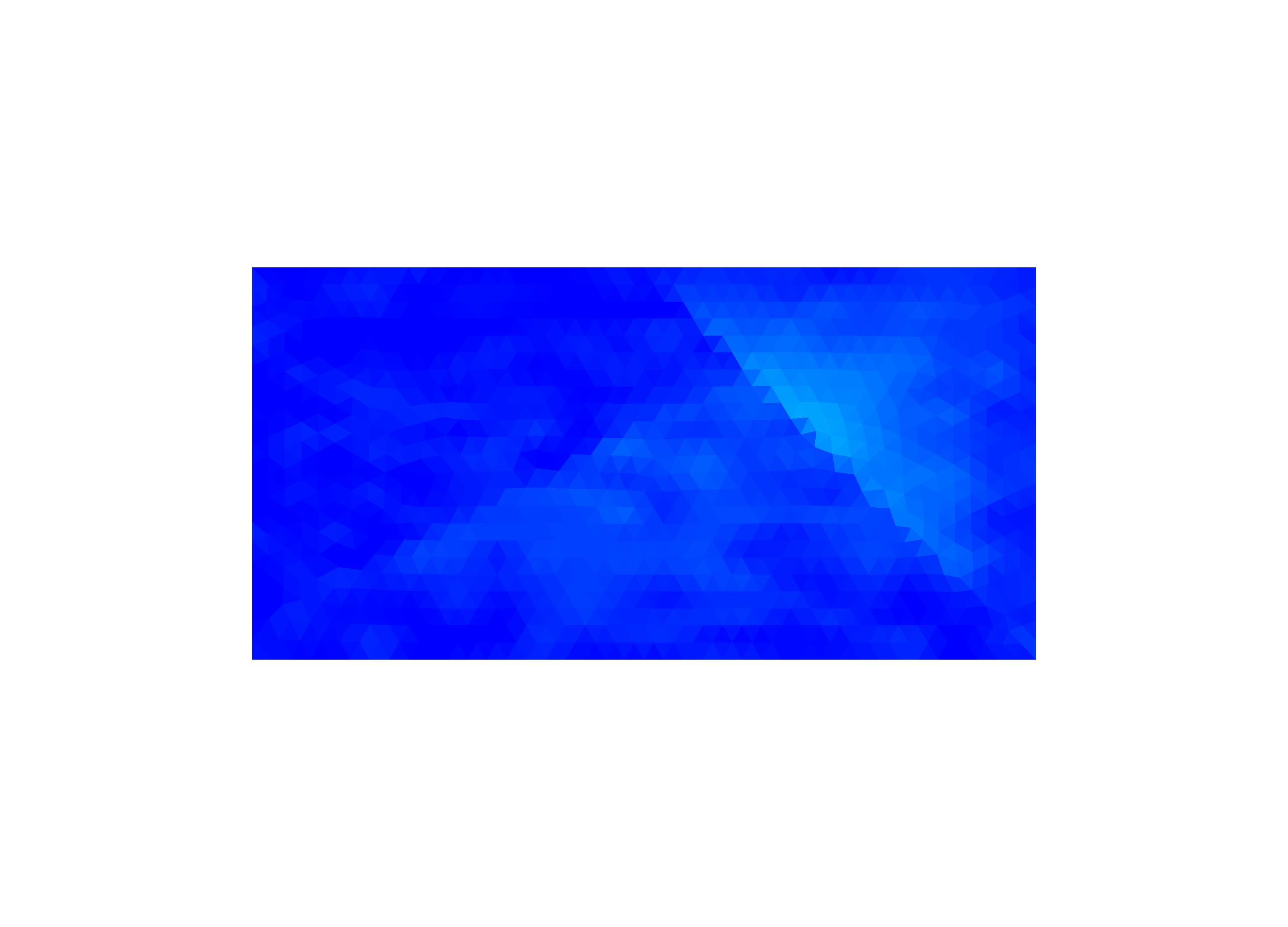}
        \end{subfigure}\hfill
        \begin{subfigure}[b]{0.11\textwidth}
            \includegraphics[width=\linewidth,trim={15cm 0cm 15cm 15cm},clip]{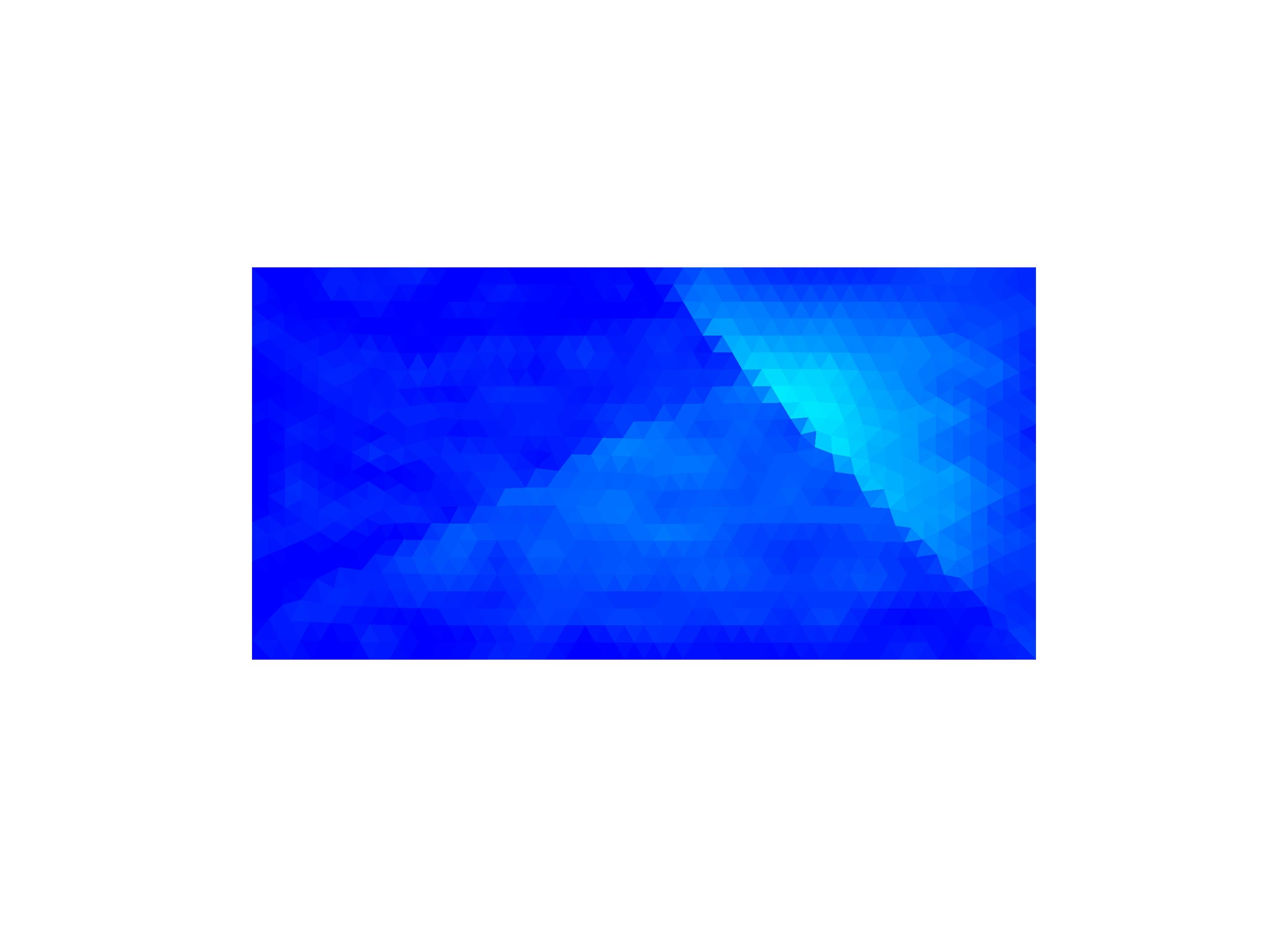}
        \end{subfigure}\hfill
        \begin{subfigure}[b]{0.11\textwidth}
            \includegraphics[width=\linewidth,trim={15cm 0cm 15cm 15cm},clip]{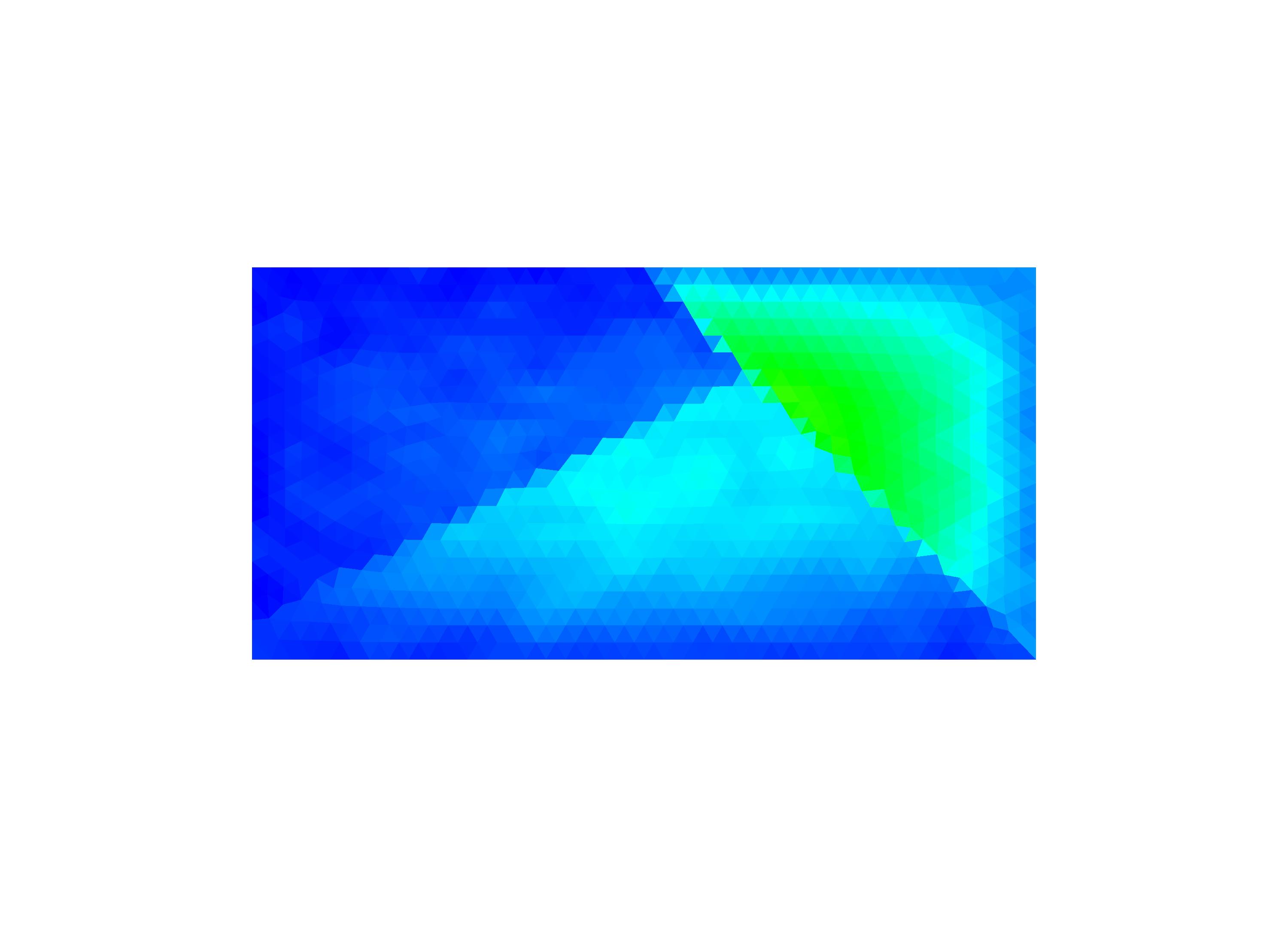}
        \end{subfigure}\hfill
        \begin{subfigure}[b]{0.11\textwidth}
            \includegraphics[width=\linewidth,trim={15cm 0cm 15cm 15cm},clip]{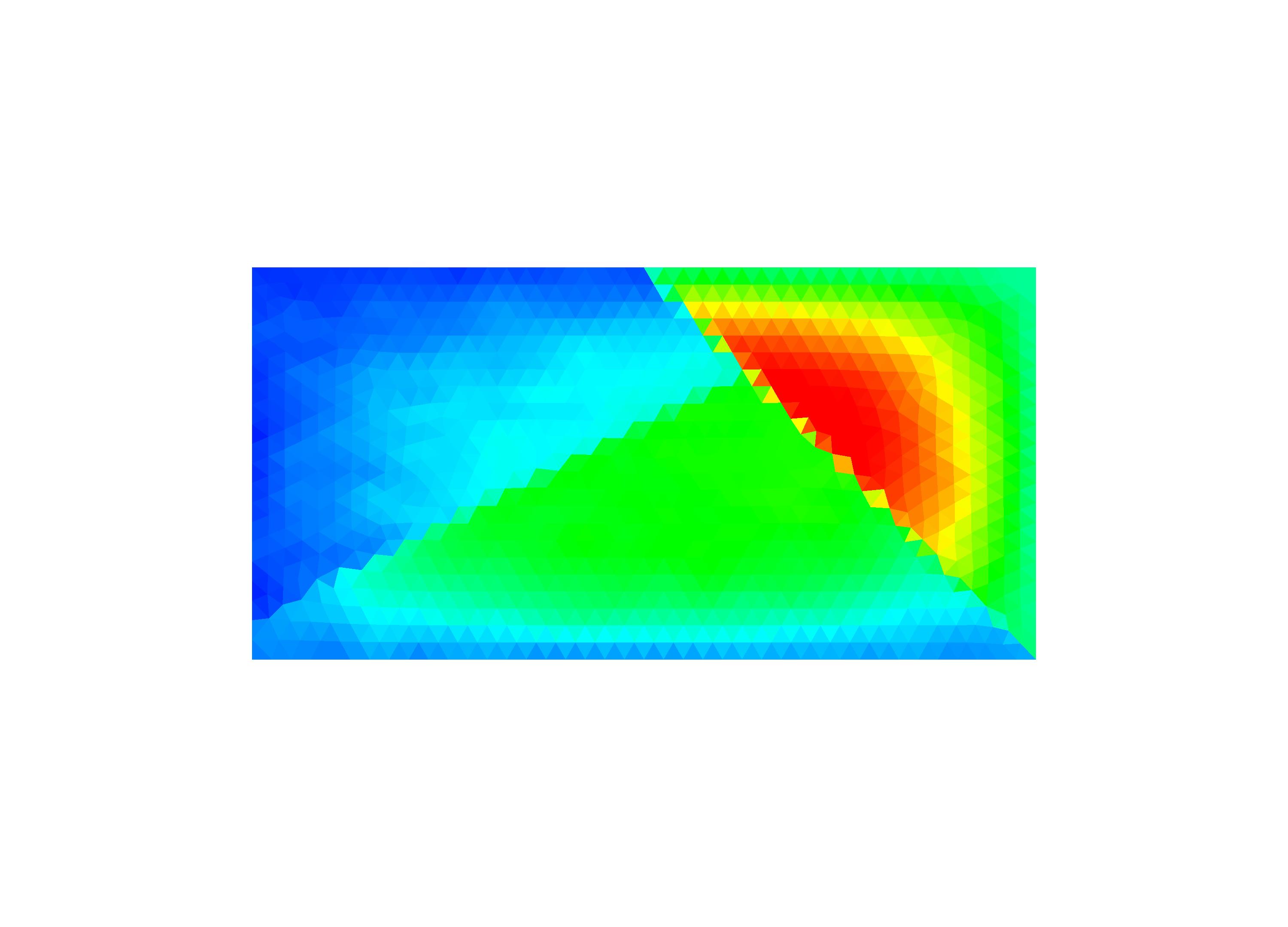}
        \end{subfigure}\hfill
        \begin{subfigure}[b]{0.11\textwidth}
            \includegraphics[width=\linewidth,trim={15cm 0cm 15cm 15cm},clip]{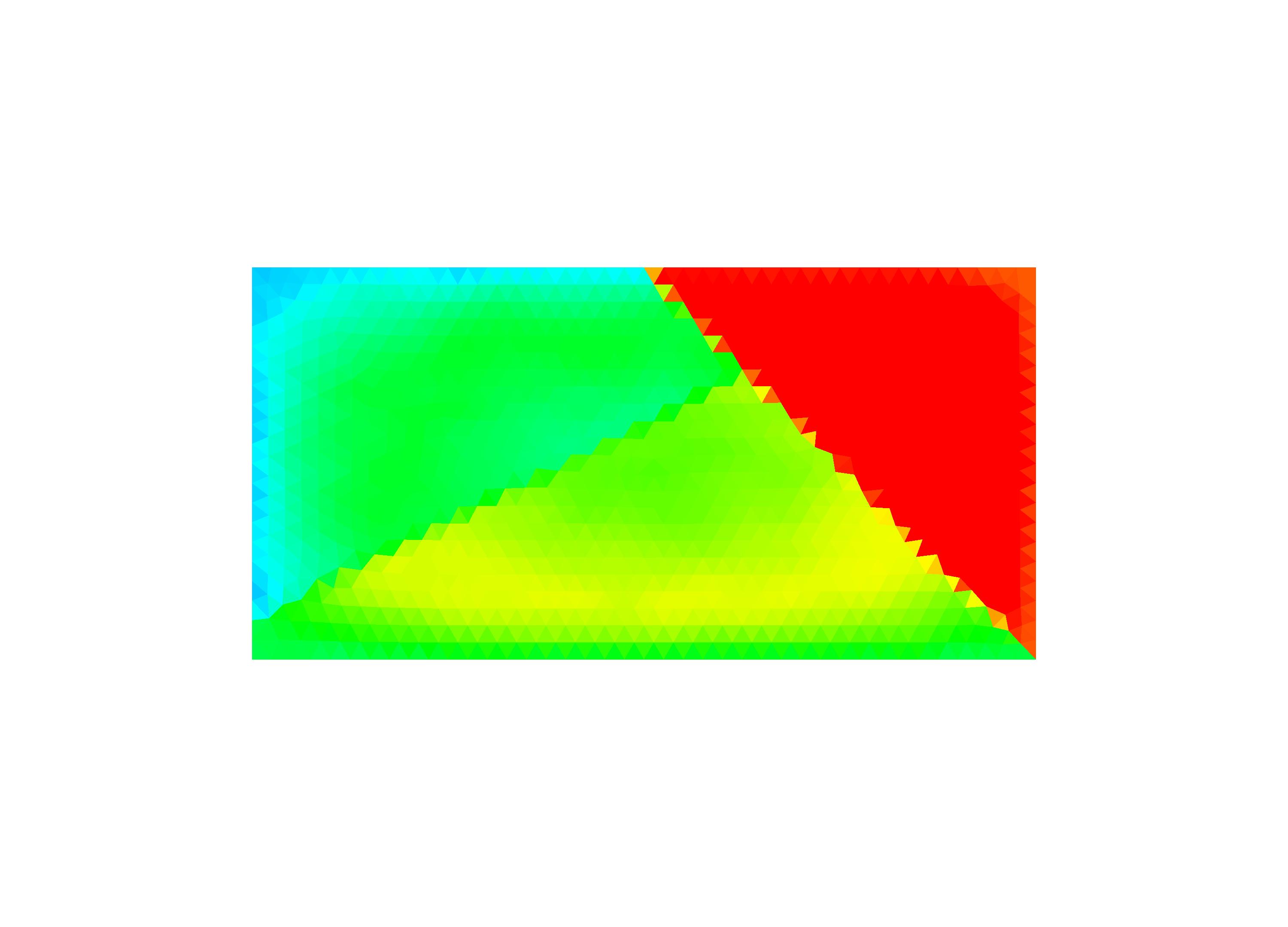}
        \end{subfigure}\hfill
        \begin{subfigure}[b]{0.11\textwidth}
            \includegraphics[width=\linewidth,trim={15cm 0cm 15cm 15cm},clip]{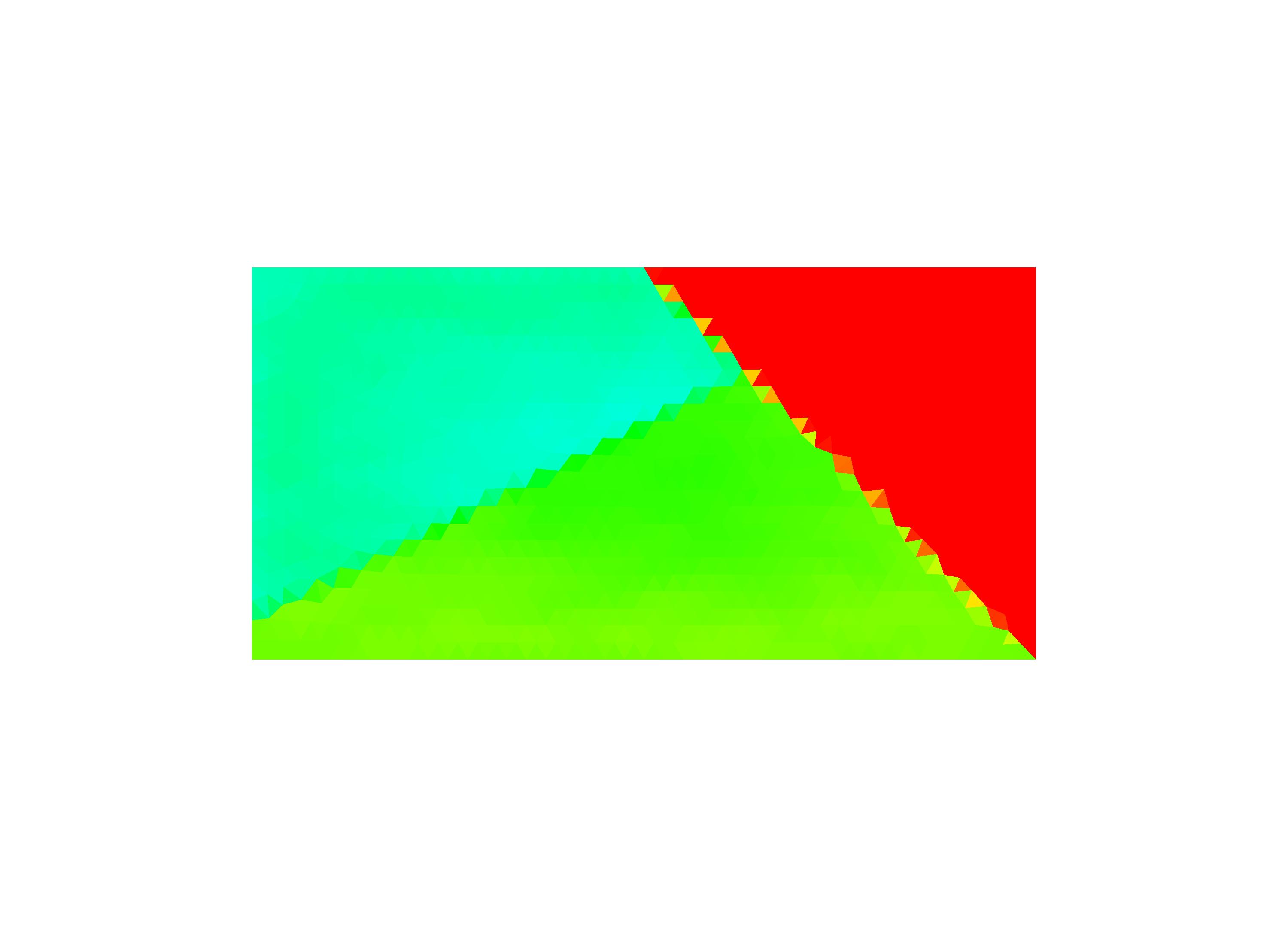}
        \end{subfigure}\hfill
        \hspace{1pt}\hfill
        \begin{subfigure}[b]{0.11\textwidth}
            \includegraphics[width=\linewidth,trim={15cm 0cm 15cm 15cm},clip]{allen_cahn_q_exact.jpg}
        \end{subfigure}
        \\
        \vspace{0.3cm}
        \begin{subfigure}[b]{0.08\textwidth}
            \pgfplotscolorbardrawstandalone[
                    colormap/jet,
                    point meta min=0,
                    point meta max=4.25,
                    colorbar style={
                        height = 0.8cm,
                        width=0.15\textwidth,
                        /pgf/number format/fixed,
                        /pgf/number format/precision=1,
                        tick style={font=\tiny},
                        ytick={0, 2,4.25},
                        yticklabels={0,2,4.25},                    
                    }
                ]
        \end{subfigure}\hfill
        \begin{subfigure}[b]{0.11\textwidth}
            \includegraphics[width=\linewidth,trim={15cm 0cm 15cm 15cm},clip]{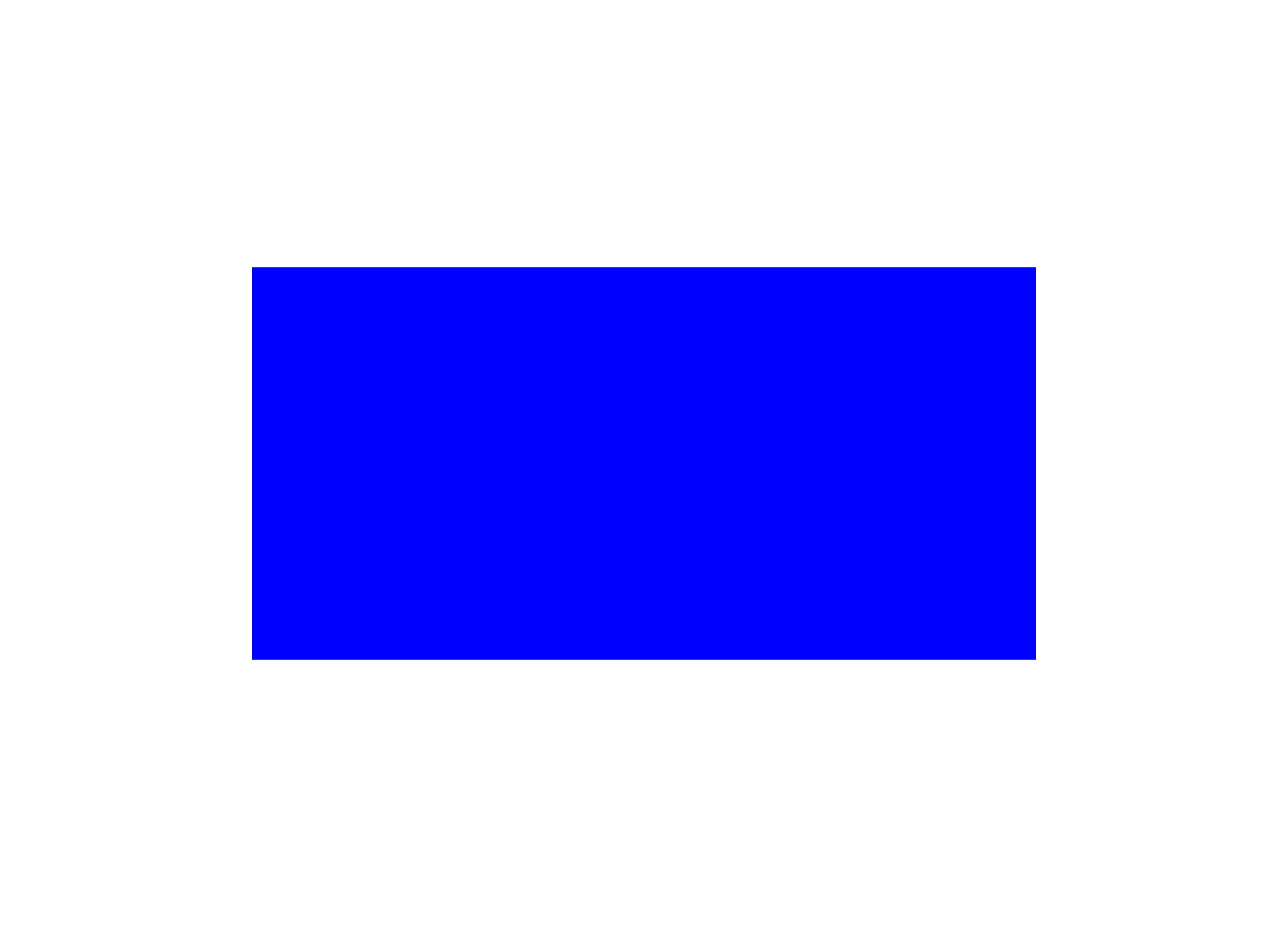}
        \end{subfigure}\hfill
        \begin{subfigure}[b]{0.11\textwidth}
            \includegraphics[width=\linewidth,trim={15cm 0cm 15cm 15cm},clip]{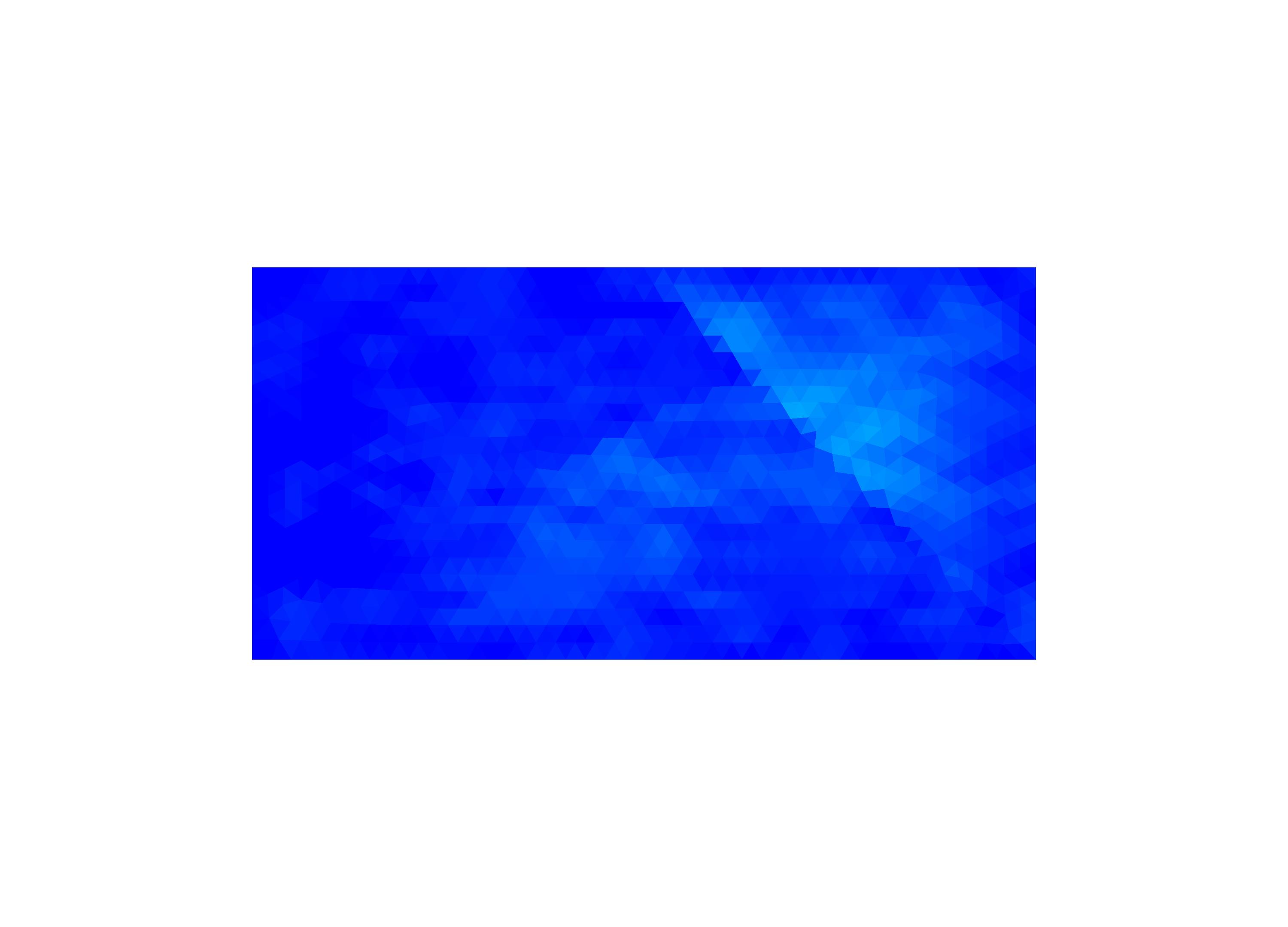}
        \end{subfigure}\hfill
        \begin{subfigure}[b]{0.11\textwidth}
            \includegraphics[width=\linewidth,trim={15cm 0cm 15cm 15cm},clip]{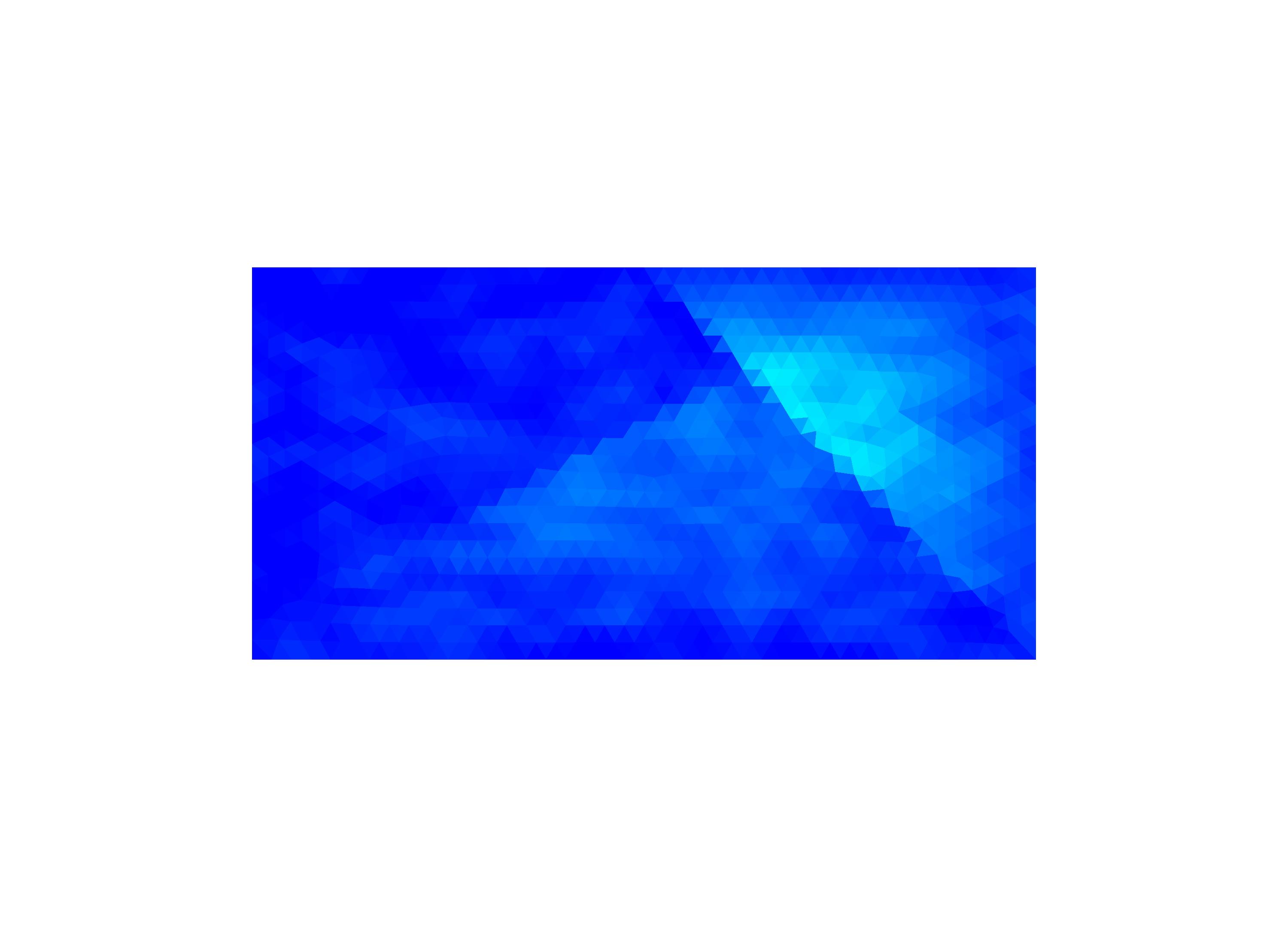}
        \end{subfigure}\hfill
        \begin{subfigure}[b]{0.11\textwidth}
            \includegraphics[width=\linewidth,trim={15cm 0cm 15cm 15cm},clip]{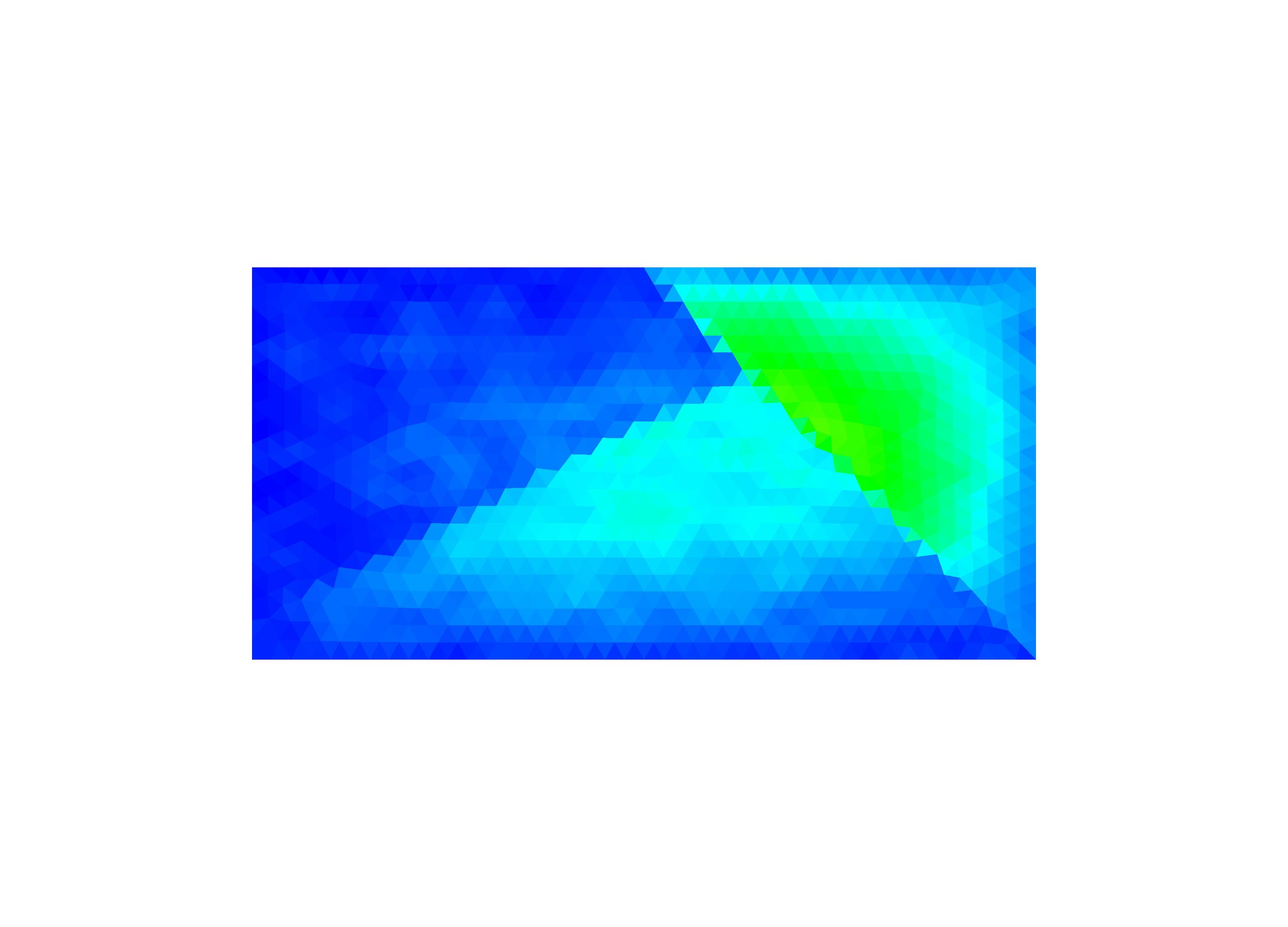}
        \end{subfigure}\hfill
        \begin{subfigure}[b]{0.11\textwidth}
            \includegraphics[width=\linewidth,trim={15cm 0cm 15cm 15cm},clip]{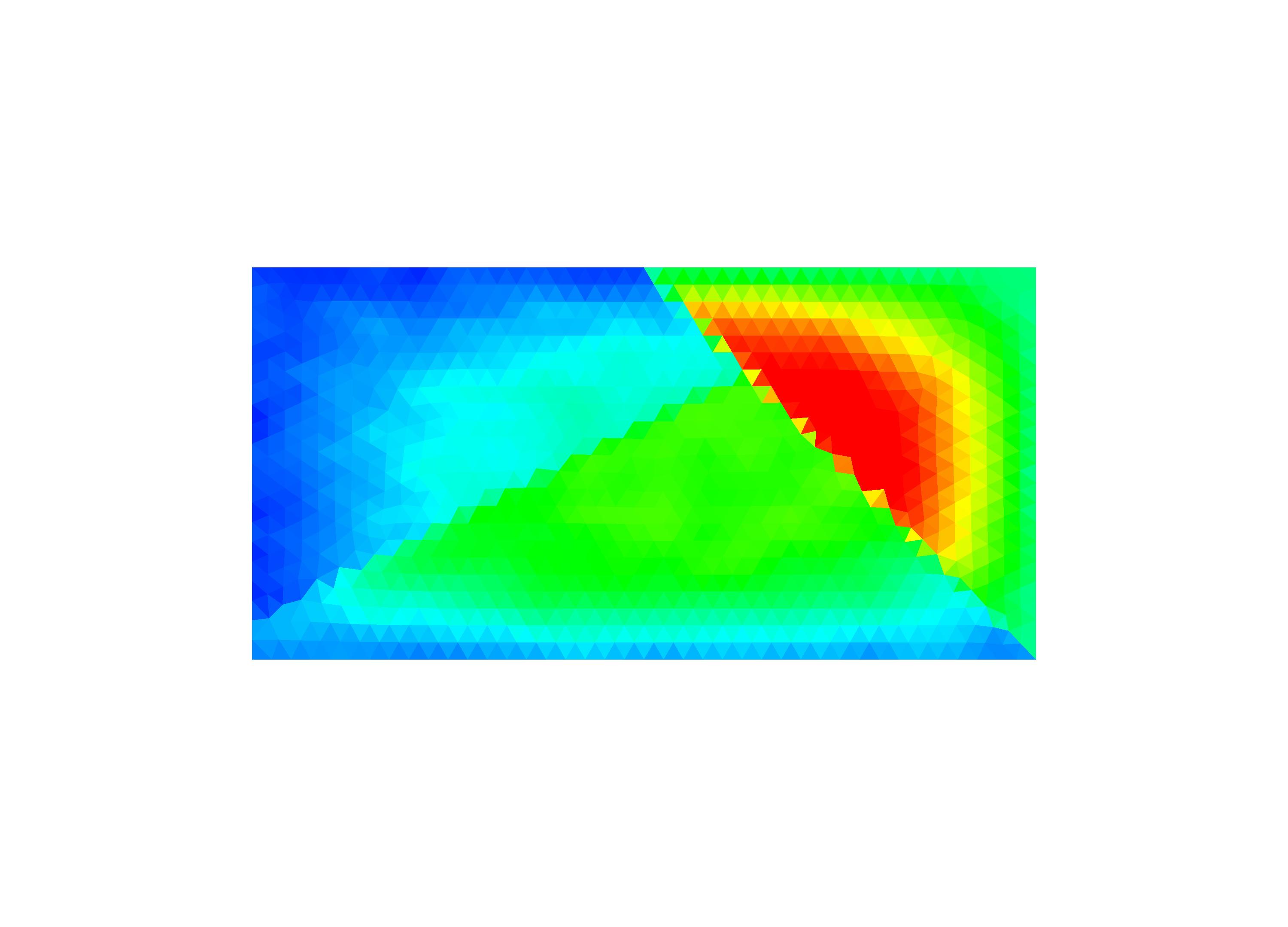}
        \end{subfigure}\hfill
        \begin{subfigure}[b]{0.11\textwidth}
            \includegraphics[width=\linewidth,trim={15cm 0cm 15cm 15cm},clip]{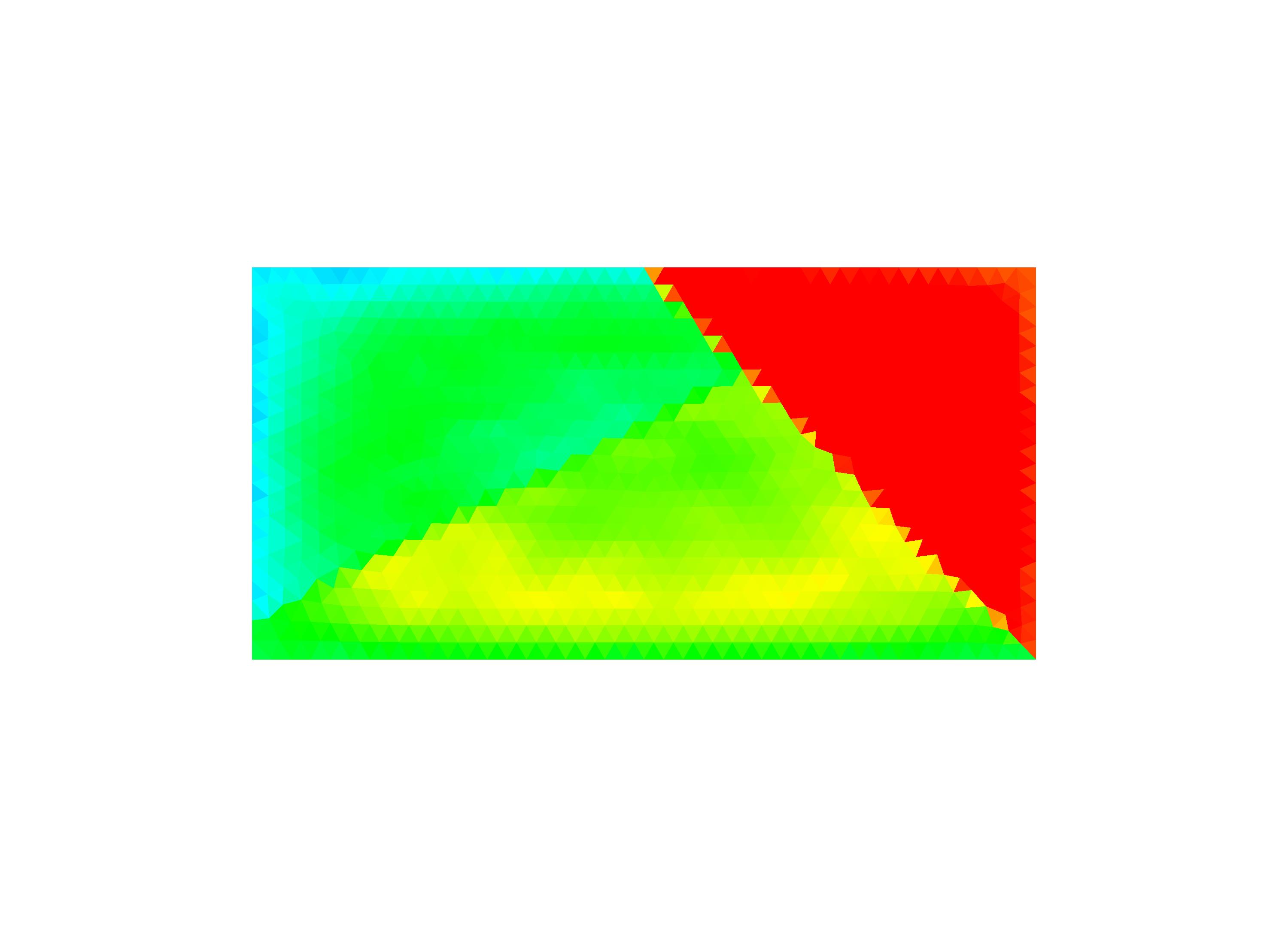}
        \end{subfigure}\hfill
        \begin{subfigure}[b]{0.11\textwidth}
            \includegraphics[width=\linewidth,trim={15cm 0cm 15cm 15cm},clip]{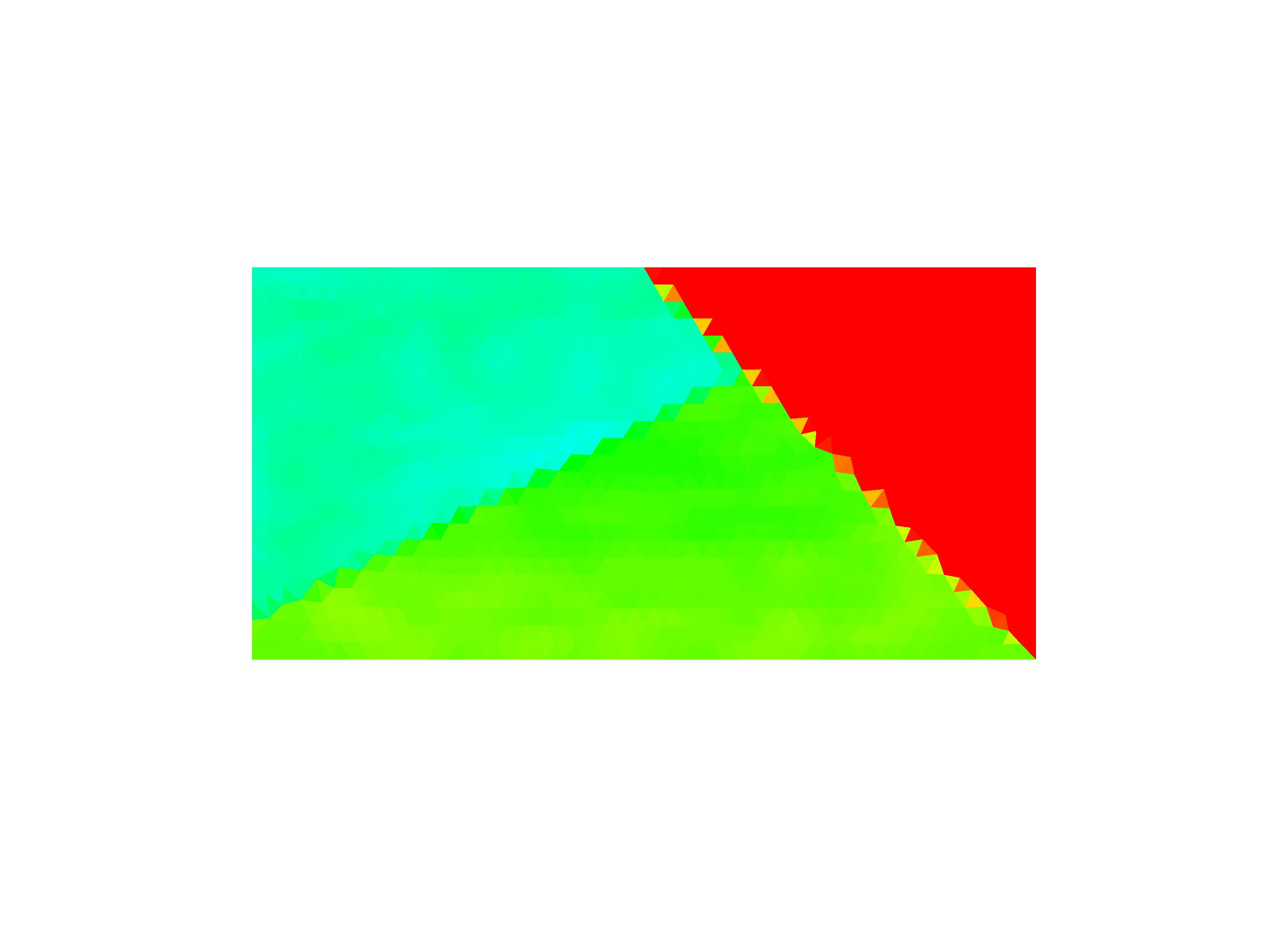}
        \end{subfigure}\hfill 
        \raisebox{-42pt}[0pt][0pt]{\rule{0.5pt}{0.28\textheight}}\hfill
        \begin{subfigure}[b]{0.11\textwidth}
            \includegraphics[width=\linewidth,trim={15cm 0cm 15cm 15cm},clip]{allen_cahn_q_exact.jpg}
        \end{subfigure}
    \\
        \vspace{0.3cm}
        \begin{subfigure}[b]{0.08\textwidth}
            \pgfplotscolorbardrawstandalone[
                    colormap/jet,
                    point meta min=0,
                    point meta max=4.25,
                    colorbar style={
                        height = 0.8cm,
                        width=0.15\textwidth,
                        /pgf/number format/fixed,
                        /pgf/number format/precision=1,
                        tick style={font=\tiny},
                        ytick={0, 2,4.25},
                        yticklabels={0,2,4.25},
                    }
                ]
        \end{subfigure}\hfill
        \begin{subfigure}[b]{0.11\textwidth}
            \includegraphics[width=\linewidth,trim={15cm 0cm 15cm 15cm},clip]{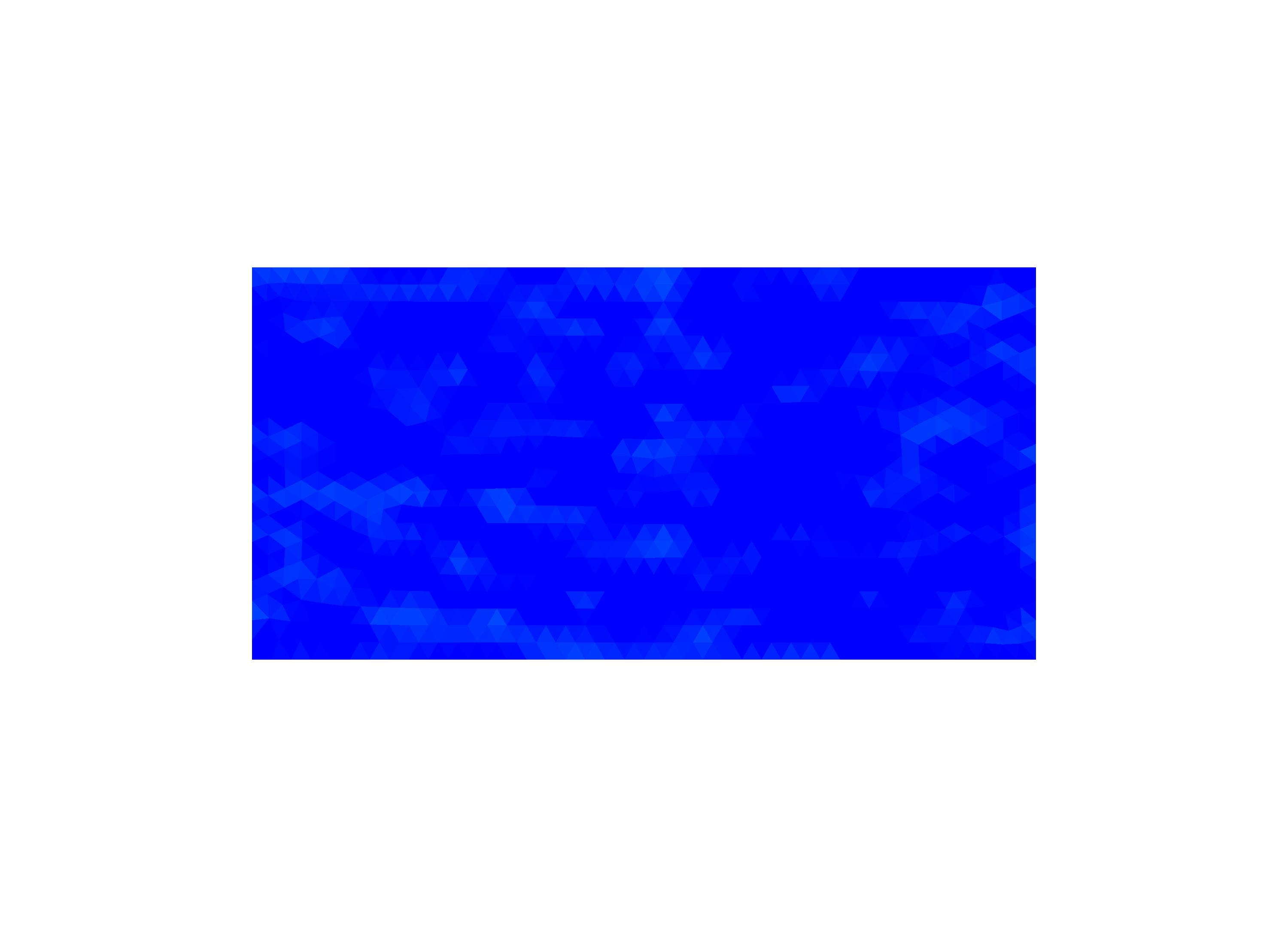}
        \end{subfigure}\hfill
        \begin{subfigure}[b]{0.11\textwidth}
            \includegraphics[width=\linewidth,trim={15cm 0cm 15cm 15cm},clip]{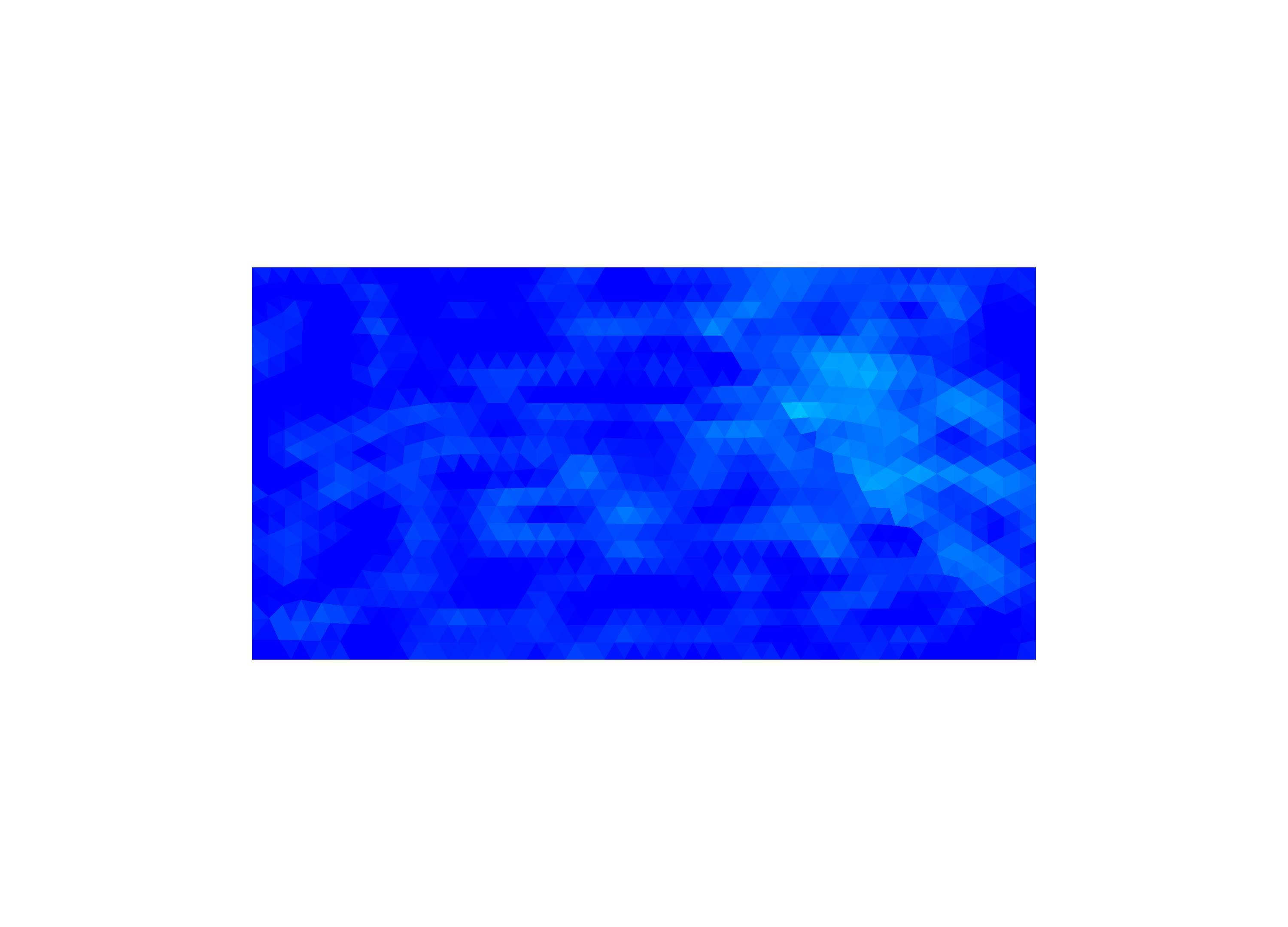}
        \end{subfigure}\hfill
        \begin{subfigure}[b]{0.11\textwidth}
            \includegraphics[width=\linewidth,trim={15cm 0cm 15cm 15cm},clip]{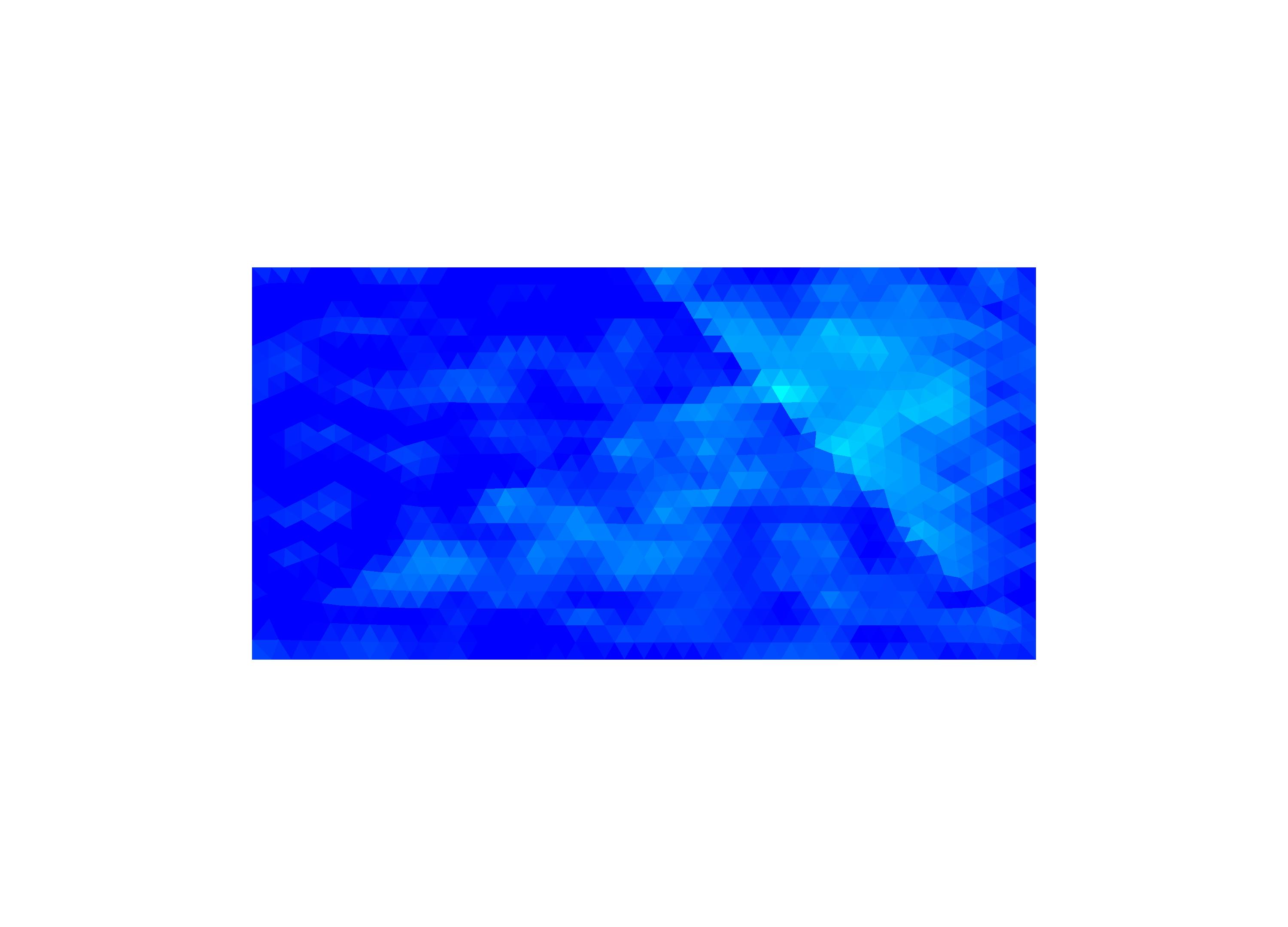}
        \end{subfigure}\hfill
        \begin{subfigure}[b]{0.11\textwidth}
            \includegraphics[width=\linewidth,trim={15cm 0cm 15cm 15cm},clip]{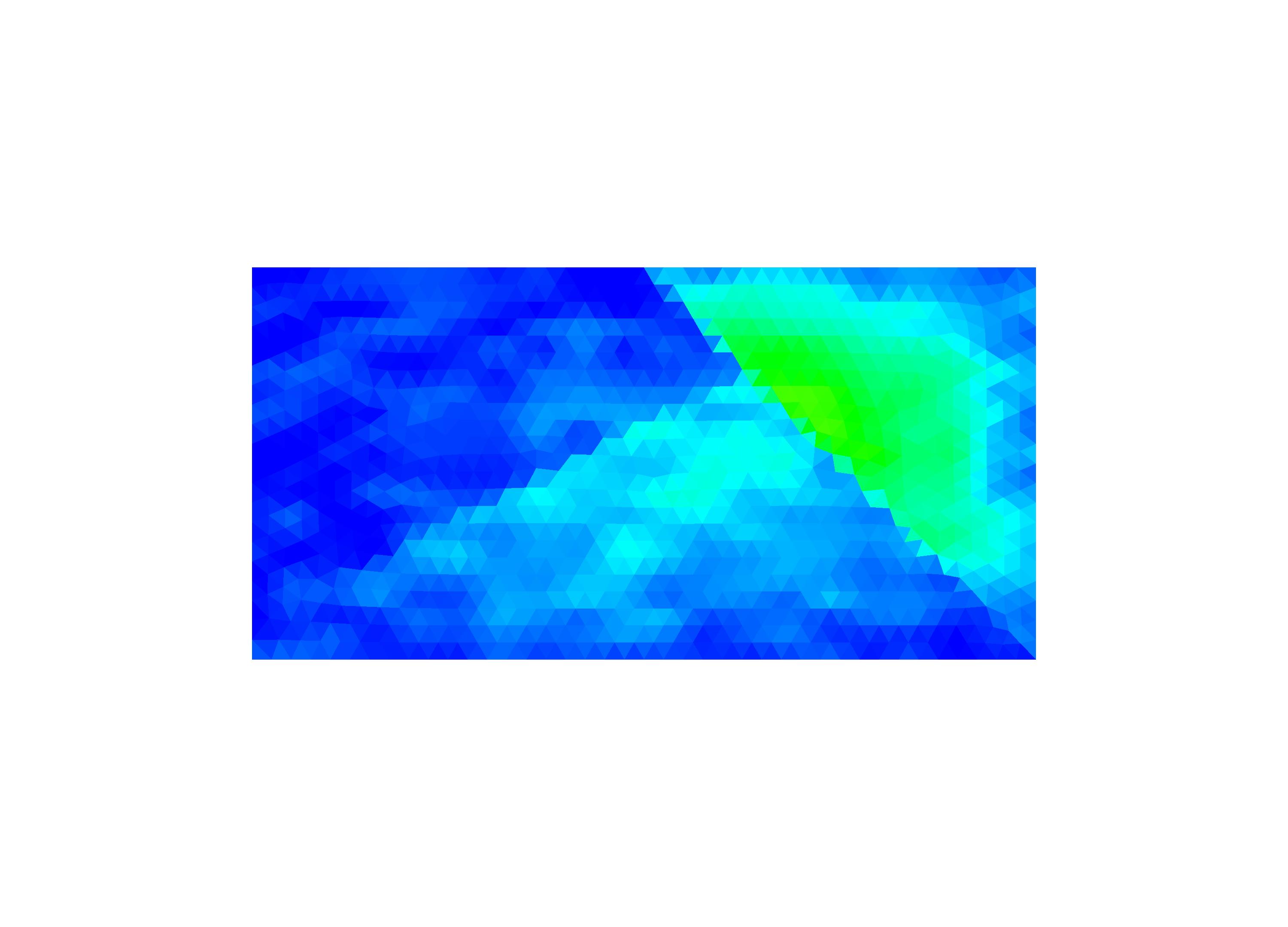}
        \end{subfigure}\hfill
        \begin{subfigure}[b]{0.11\textwidth}
            \includegraphics[width=\linewidth,trim={15cm 0cm 15cm 15cm},clip]{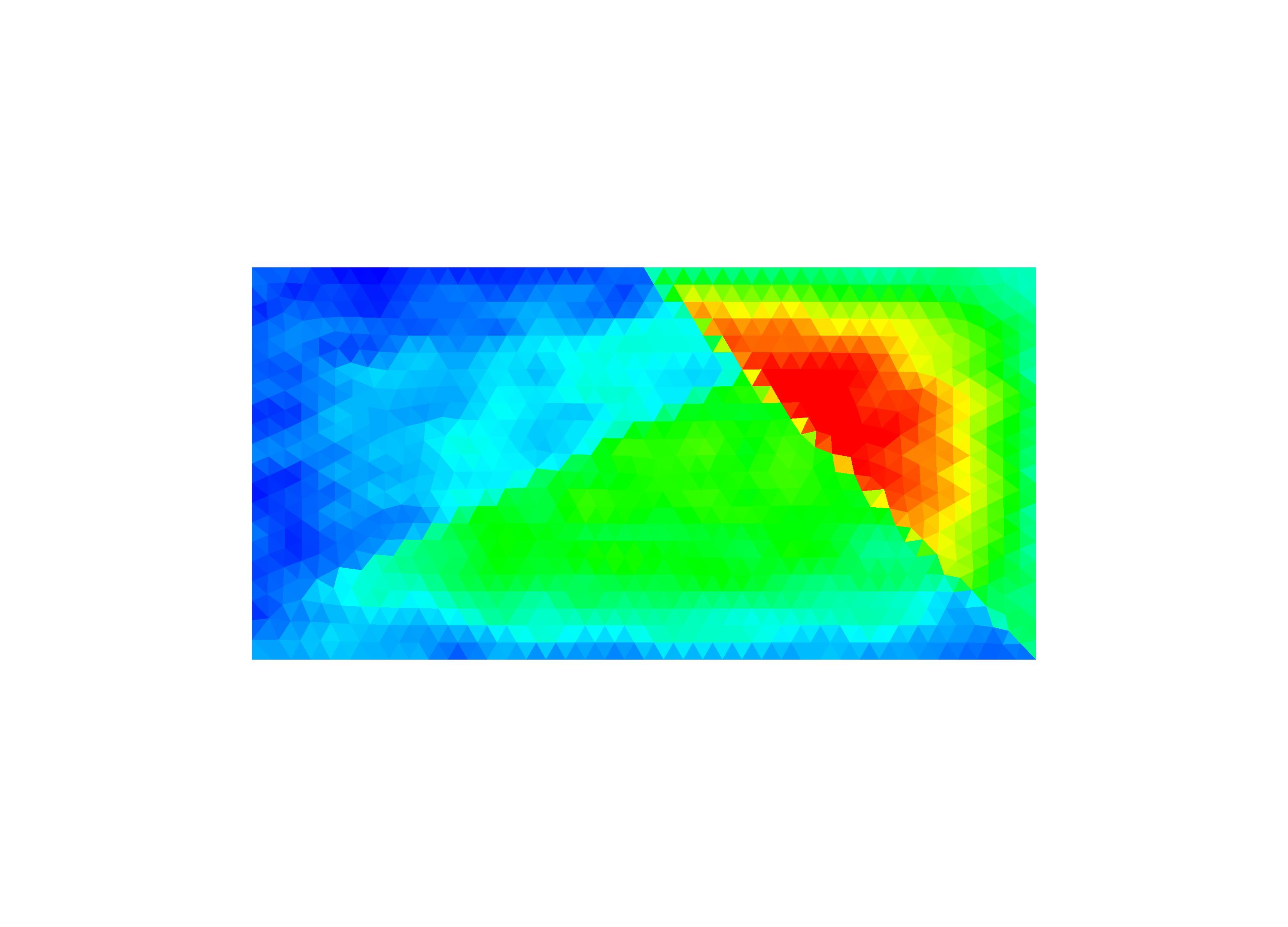}
        \end{subfigure}\hfill
        \begin{subfigure}[b]{0.11\textwidth}
            \includegraphics[width=\linewidth,trim={15cm 0cm 15cm 15cm},clip]{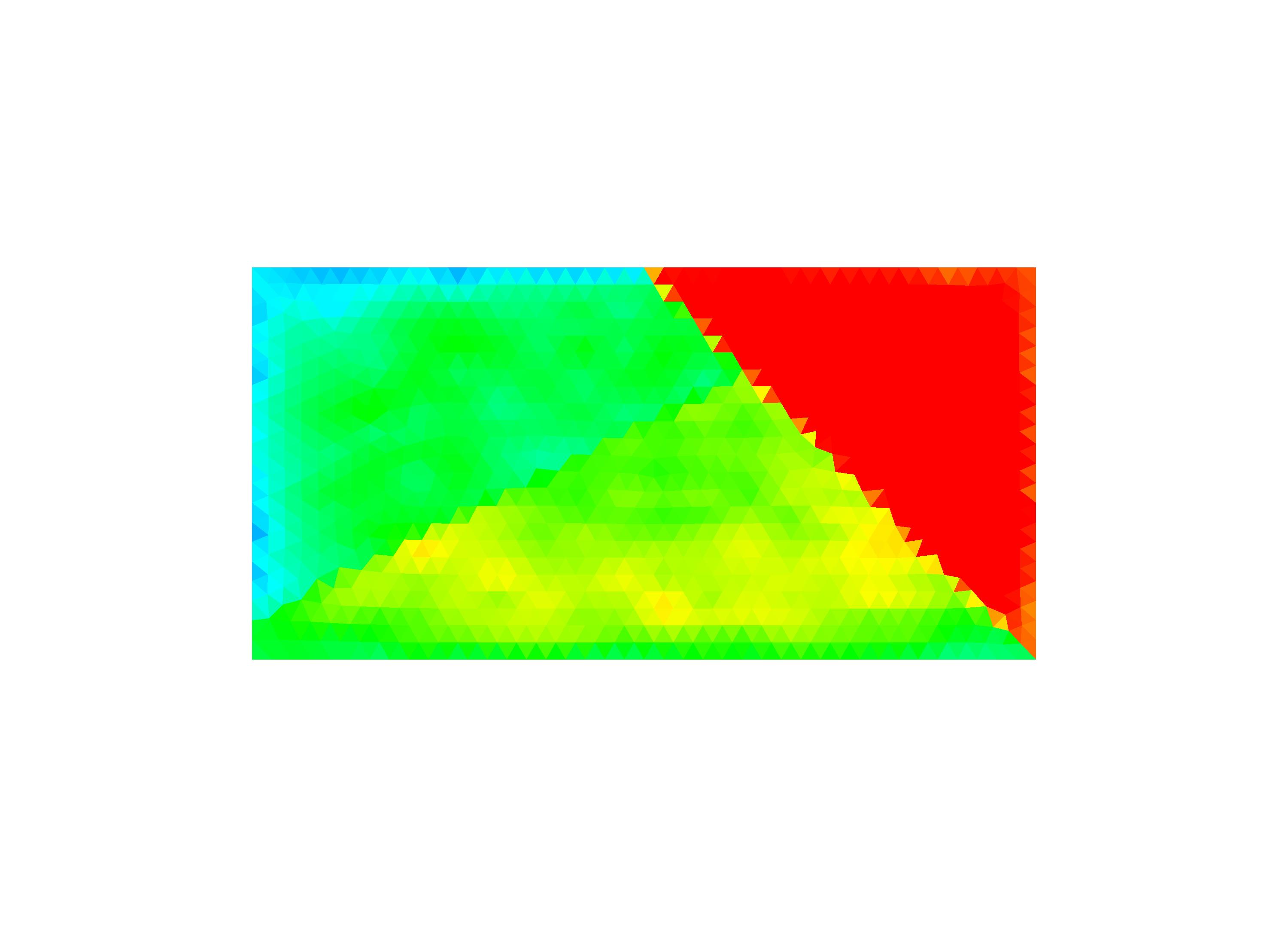}
        \end{subfigure}\hfill
        \begin{subfigure}[b]{0.11\textwidth}
            \includegraphics[width=\linewidth,trim={15cm 0cm 15cm 15cm},clip]{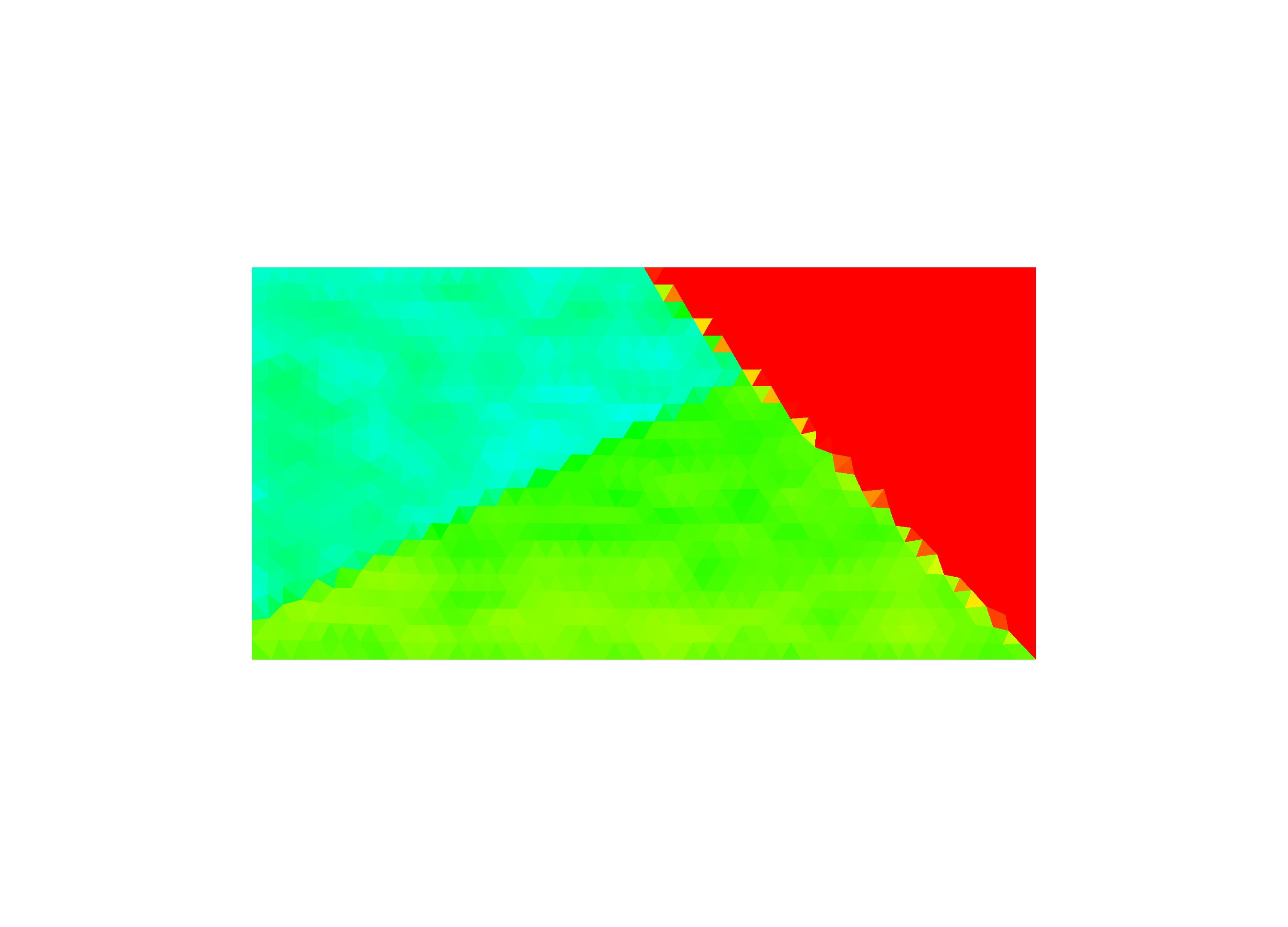}
        \end{subfigure}\hfill
        \hspace{1pt}\hfill
        \begin{subfigure}[b]{0.11\textwidth}
            \includegraphics[width=\linewidth,trim={15cm 0cm 15cm 15cm},clip]{allen_cahn_q_exact.jpg}
        \end{subfigure}\hfill
    \end{minipage}
    \caption{Modified Allen-Cahn equation. Visualization of evolution of the potential $c$ computed from MRAS ran without noise (first row), with 5\% noise (second row), 10\% noise (third row), 20\% noise (fourth row) and the true parameter $c^\dagger$.}
    \label{fig::ac_evolution}

\end{figure}

\begin{figure}
        \begin{minipage}[b]{0.48\textwidth}
            \captionsetup{labelformat=empty}
            \begin{center}
                0\% noise
                \includegraphics[width=0.8\linewidth,trim={15cm 10cm 15cm 15cm},clip]{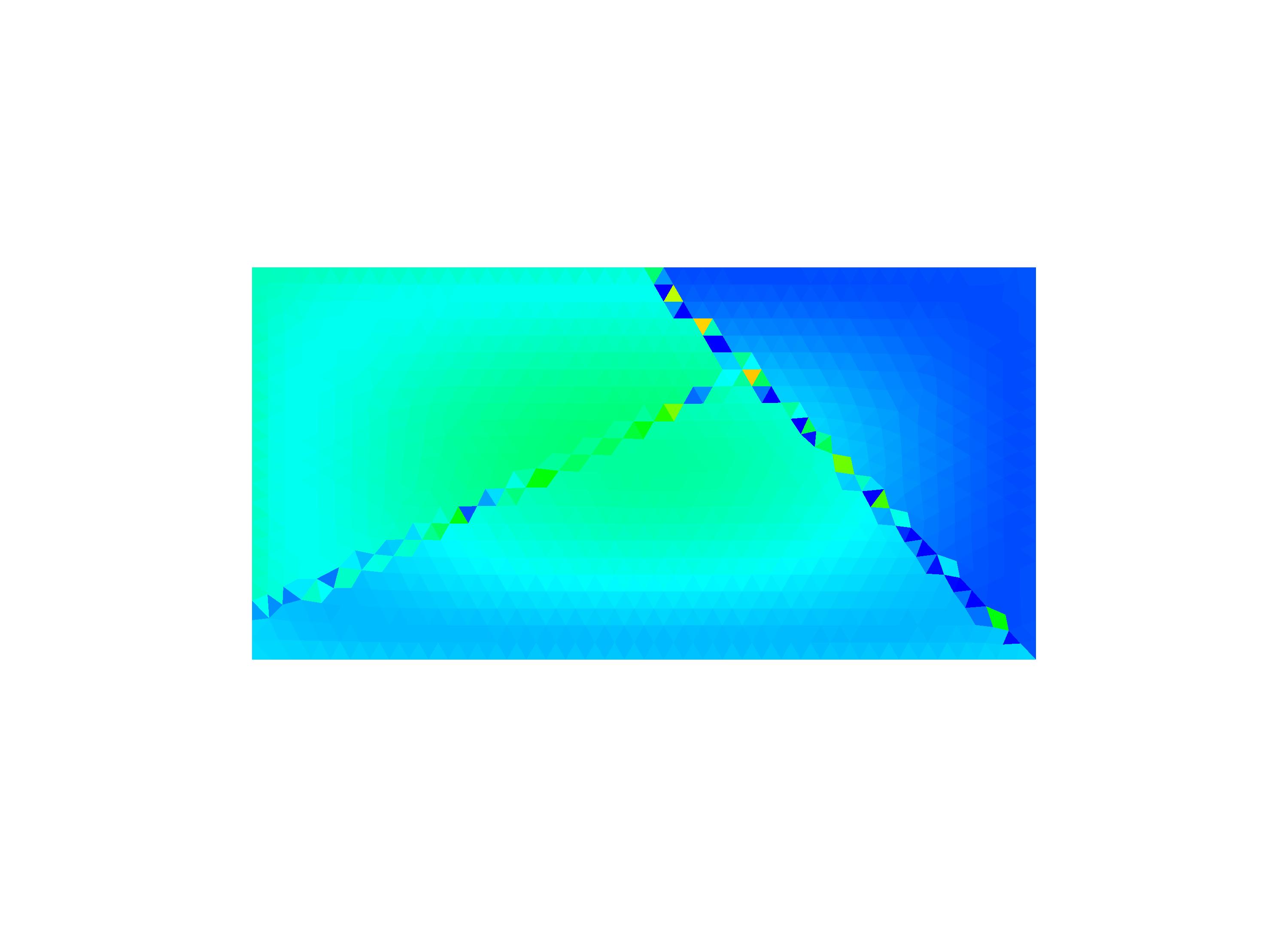} \\
                10\% noise
                \includegraphics[width=0.8\linewidth,trim={15cm 10cm 15cm 15cm},clip]{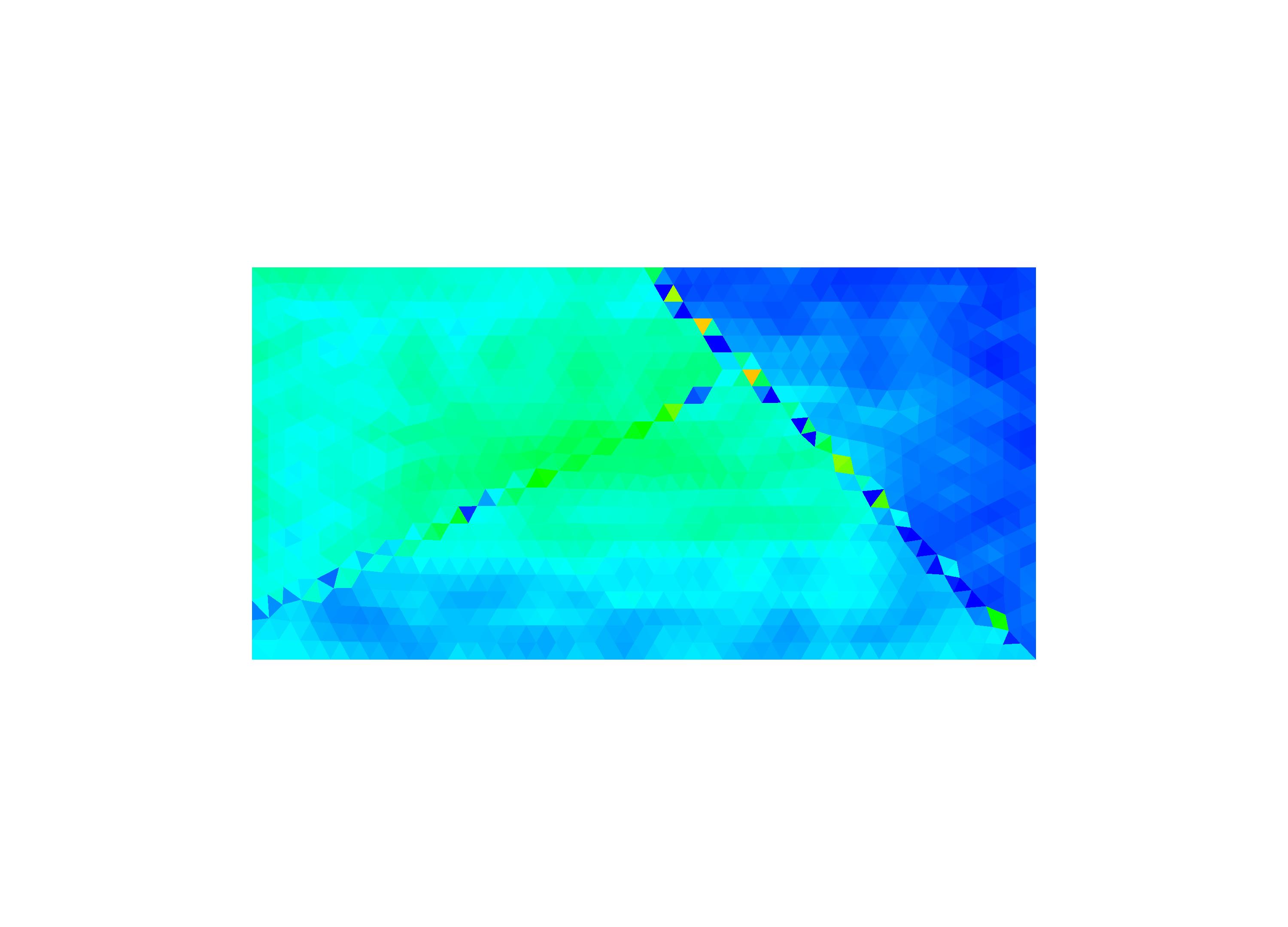}
                \pgfplotscolorbardrawstandalone[
                    colormap/jet,    
                    colorbar horizontal,
                    point meta min= -1,
                    point meta max= 1,
                    colorbar style={
                        width=0.5\textwidth,
                        /pgf/number format/fixed,
                        /pgf/number format/precision=2,
                        xticklabel style={anchor=north},
                        every axis label/.append style={/pgf/number format/precision=3},
                        yticklabel style={/pgf/number format/none},
                        tick style={draw=none}}]
            \end{center}
        \end{minipage}
        \begin{minipage}[b]{0.48\textwidth}
            \captionsetup{labelformat=empty}
            \begin{center}
                5\% noise
                \includegraphics[width=0.8\linewidth,trim={15cm 10cm 15cm 15cm},clip]{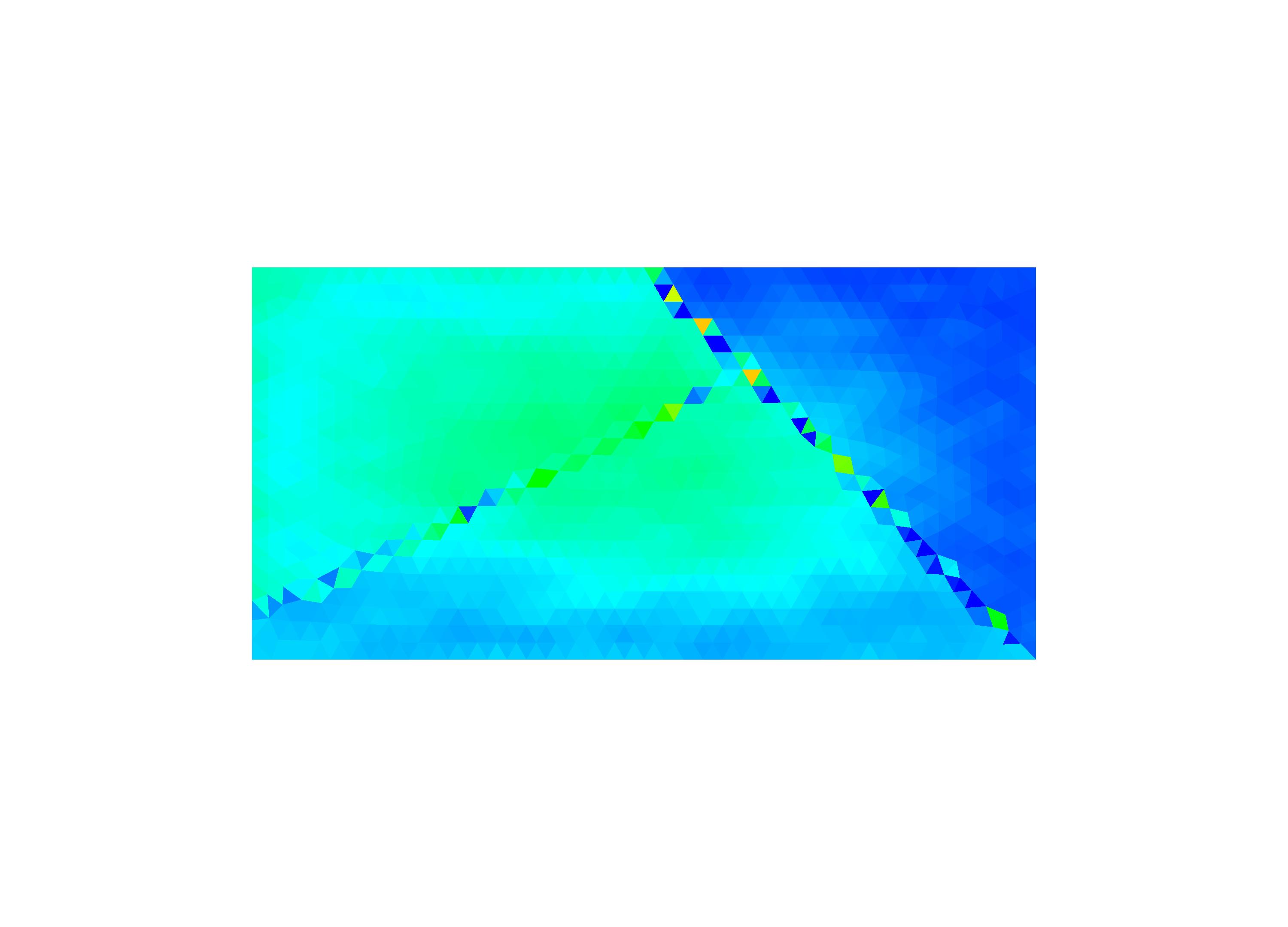} \\
                20\% noise
                \includegraphics[width=0.8\linewidth,trim={15cm 10cm 15cm 15cm},clip]{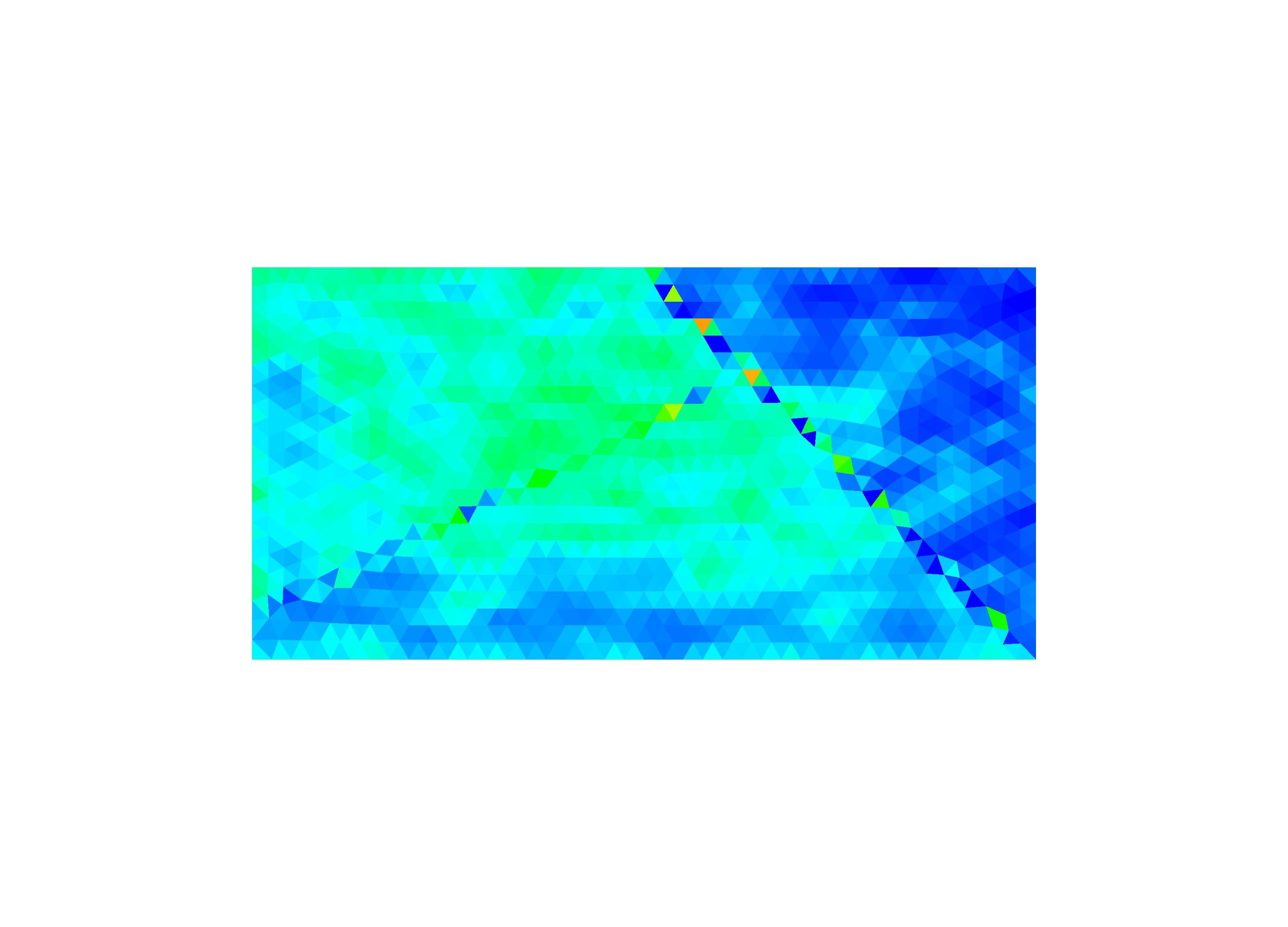}
                \pgfplotscolorbardrawstandalone[
                    colormap/jet,    
                    colorbar horizontal,
                    point meta min= -1,
                    point meta max= 1,
                    colorbar style={
                        width=0.5\textwidth,
                        /pgf/number format/fixed,
                        /pgf/number format/precision=2,
                        xticklabel style={anchor=north},
                        every axis label/.append style={/pgf/number format/precision=3},
                        yticklabel style={/pgf/number format/none},
                        tick style={draw=none}}]
            \end{center}
        \end{minipage}
    \caption{Modified Allen-Cahn equation.  Error field $c(T)-c^\dagger$ resulting from the MRAS with different relative noise levels.}
    \label{fig::ac_diff}
\end{figure}

\begin{figure}[!ht]
    \centering
    \begin{tabular}{cc}
        \begin{subfigure}[b]{0.45\textwidth}
            \captionsetup{labelformat=empty}
            \centering
            \includegraphics[trim={0.5cm 0 0.5cm 1cm},clip=true,width=\linewidth]{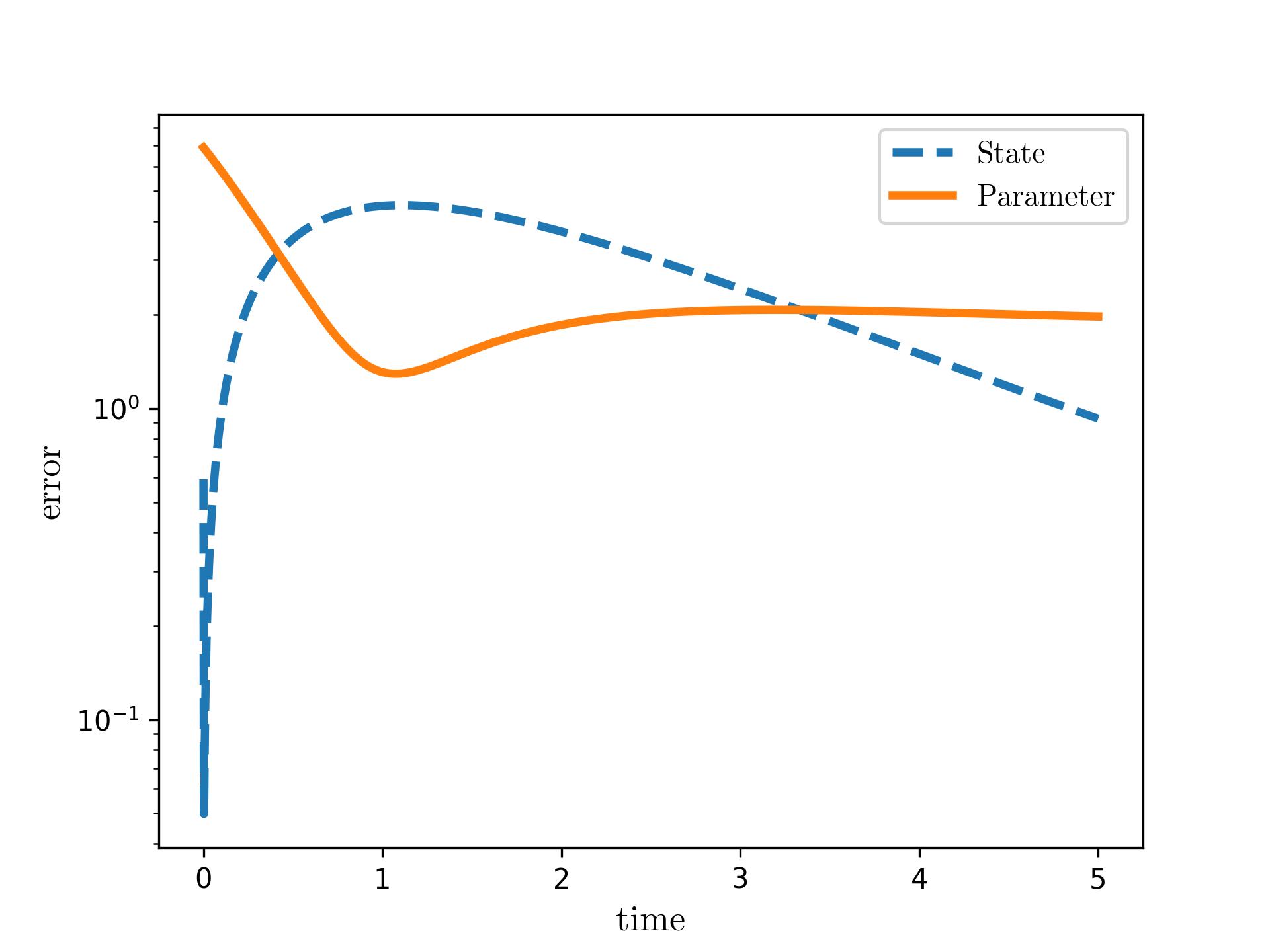}
            \caption{no noise}
        \end{subfigure}
        &
        \begin{subfigure}[b]{0.45\textwidth}
            \centering
            \captionsetup{labelformat=empty}
            \includegraphics[trim={1cm 0 0 1},clip=true,width=\linewidth]{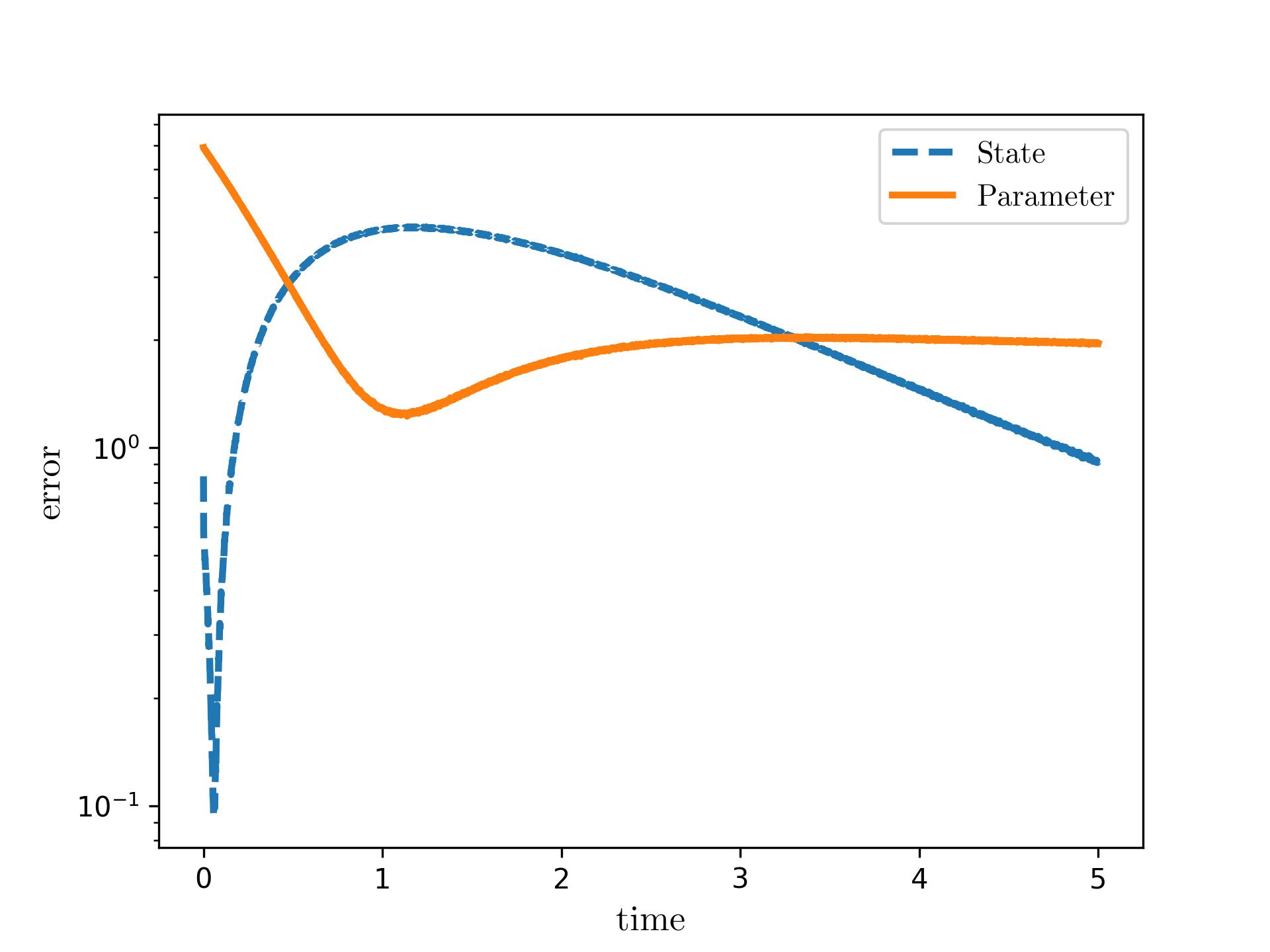} 
            \caption{$5\%$ noise}
        \end{subfigure}
        \\
        \begin{subfigure}[b]{0.45\textwidth}
            \captionsetup{labelformat=empty}
            \centering
            \includegraphics[trim={0.5cm 0 0.5cm 1cm},clip=true,width=\linewidth]{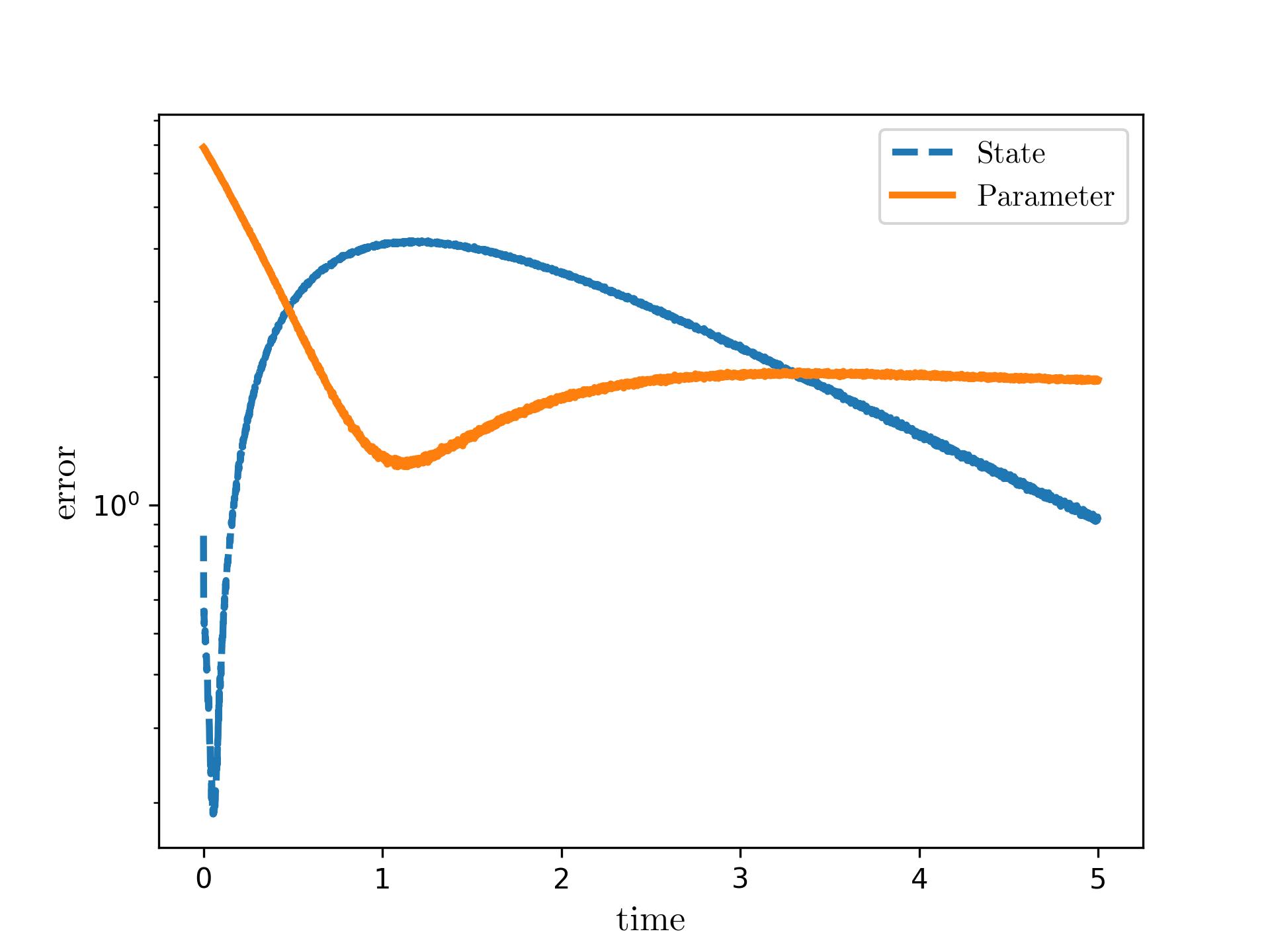}
            \caption{$10\%$ noise}
        \end{subfigure}
        &
        \begin{subfigure}[b]{0.45\textwidth}
            \centering
            \captionsetup{labelformat=empty}
            \includegraphics[trim={1cm 0 0 1},clip=true,width=\linewidth]{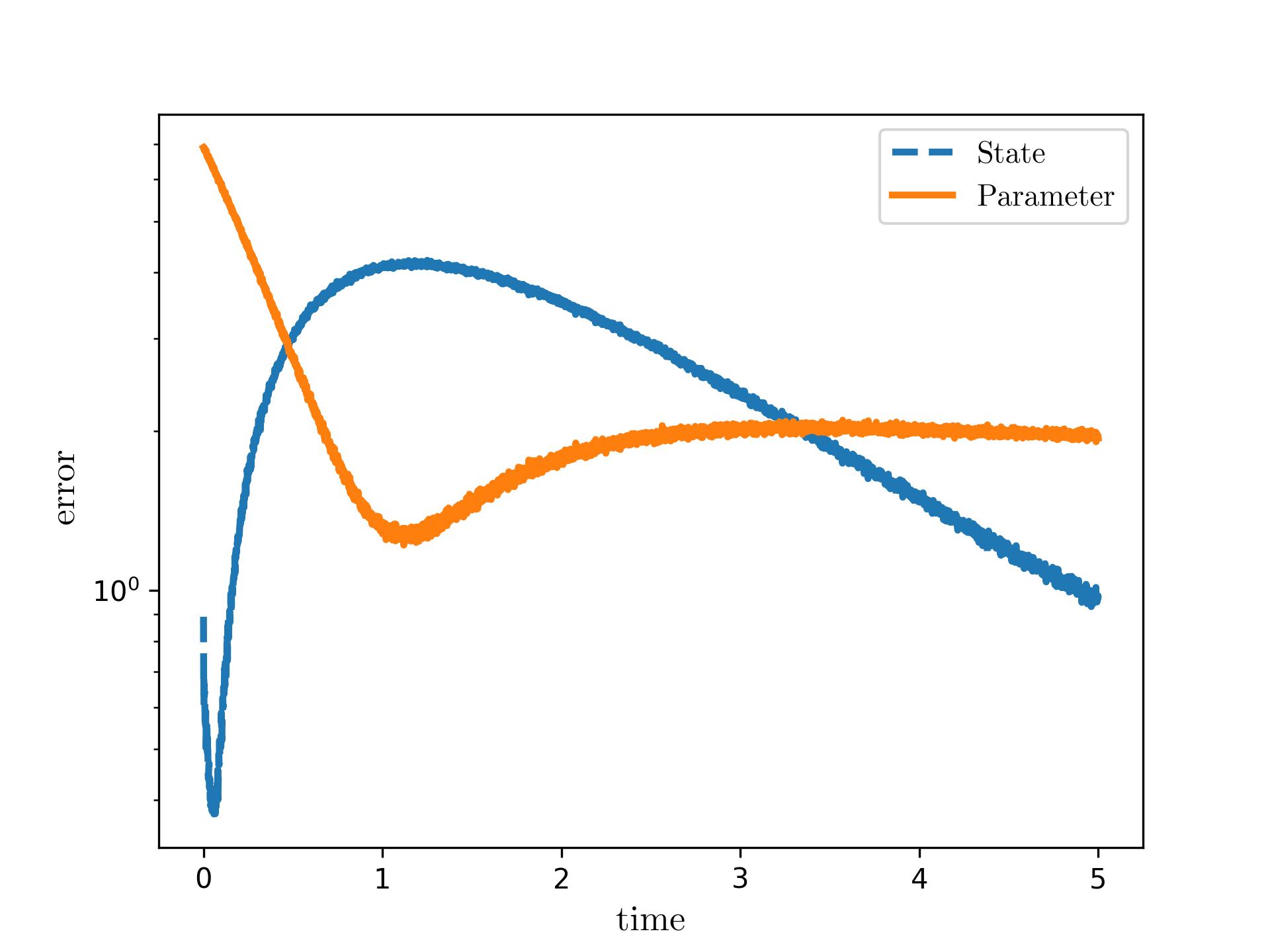} 
            \caption{$20 \%$ noise}
        \end{subfigure}
    \end{tabular}
    \caption{Modified Allen-Cahn-equation. $L^2$-Error plots for the potential $c$ and the state $u$. 
    }
    \label{fig:error_ac}
\end{figure}

Figure \ref{fig::ac_evolution} displays the reconstructed parameter in four different noise scenarios. It is noteworthy that even though the relative error of the reconstruction is not insignificant, a qualitatively good reconstruction is obtained even for the extreme case of $20\%$, relative noise in the observed state. Indeed, the three distinct sub-domains can clearly be distinguished.

It is moreover interesting to note that edges are well preserved, even for higher noise levels, compared to the effect observed for the noisy Fisher-KPP equation, as seen in Figure \ref{fig::fisher_evolution_big}. Indeed, the difference fields $c(T)-c^\dagger$ in Figure \ref{fig::ac_diff} expose that a majority of the error is spread over the spatial domain, as opposed to being extremely concentrated at the edges. Finally, Figure \ref{fig:error_ac} demonstrates stability of the MRAS over long runtimes, where the parameter estimation error remains bounded.

\color{black}

\section{Sensitivity to noise, discretization and initial guess}\label{sec::convergence_studies}

\begin{table}[!ht]
\begin{minipage}[b]{\linewidth}
\begin{center}
\scriptsize
\setlength{\tabcolsep}{4pt}
\renewcommand{\arraystretch}{1.05}
\begin{tabular}{||c|ccccc|c|ccccc||}
\hline
\multicolumn{6}{||c|}{\textbf{State error}} &
\multicolumn{6}{c||}{\textbf{Parameter error}} \\
\hline
\multicolumn{12}{||c||}{\textbf{Darcy flow}} \\
\hline
\diagbox[height=0.5cm]{$h_\mathrm{max}$}{noise} & 0\% & 5\% & 10\% & 20\% & 50\% & \diagbox[height=0.5cm]{$h_\mathrm{max}$}{noise} & 0\% & 5\% & 10\% & 20\% & 50\% \\
\hline
0.04 & 0.02 & 0.07 & 0.17 & 0.33 & 0.77 & 0.04 & 0.24 & 1.08 & 1.39 & 1.48 & 1.76 \\
0.1 & 0.02 & 0.05 & 0.14 & 0.30 & 0.64 & 0.1 & 0.23 & 0.89 & 1.36 & 1.51 & 1.78 \\
0.2 & 0.02 & 0.03 & 0.12 & 0.29 & 0.64 & 0.2 & 0.34 & 0.55 & 1.25 & 1.79 & 2.54 \\
\hline\hline
\multicolumn{12}{||c||}{\textbf{Fisher--KPP}} \\
\hline
\diagbox[height=0.5cm]{$h_\mathrm{max}$}{noise} & 0\% & 5\% & 10\% & 20\% & 50\% & \diagbox[height=0.5cm]{$h_\mathrm{max}$}{noise} & 0\% & 5\% & 10\% & 20\% & 50\% \\
\hline
0.1 & 0.05 & 0.07 & 0.12 & 0.21 & 0.58 & 0.1 & 0.52 & 0.66 & 0.79 & 0.87 & 0.92 \\
0.25 & 0.04 & 0.09 & 0.14 & 0.24 & 0.83 & 0.25 & 0.55 & 0.72 & 0.82 & 0.89 & 0.92 \\
0.5 & 0.07 & 0.08 & 0.13 & 0.24 & 0.82 & 0.5 & 0.66 & 0.62 & 0.72 & 0.85 & 0.93 \\
\hline\hline
\multicolumn{12}{||c||}{\textbf{Nonlinear potential}} \\
\hline
\diagbox[height=0.5cm]{$h_\mathrm{max}$}{noise} & 0\% & 5\% & 10\% & 20\% & 50\% & \diagbox[height=0.5cm]{$h_\mathrm{max}$}{noise} & 0\% & 5\% & 10\% & 20\% & 50\% \\
\hline
0.2 & 0.08 & 0.02 & 0.03 & 0.06 & 0.16 & 0.2 & 0.04 & 0.02 & 0.04 & 0.07 & 0.28 \\
0.4 & 0.09 & 0.09 & 0.18 & 0.36 & 0.91 & 0.4 & 0.05 & 0.10 & 0.18 & 0.35 & 1.06 \\
0.8 & 0.10 & 0.06 & 0.12 & 0.24 & 0.59 & 0.8 & 0.09 & 0.10 & 0.13 & 0.20 & 0.55 \\
\hline\hline
\multicolumn{12}{||c||}{\textbf{Allen--Cahn}} \\
\hline
\diagbox[height=0.5cm]{$h_\mathrm{max}$}{noise} & 0\% & 5\% & 10\% & 20\% & 50\% & \diagbox[height=0.5cm]{$h_\mathrm{max}$}{noise} & 0\% & 5\% & 10\% & 20\% & 50\% \\
\hline
0.1 & 0.39 & 0.42 & 0.44 & 0.50 & - & 0.1 & 0.28 & 0.29 & 0.29 & 0.33 & - \\
0.25 & 0.41 & 0.42 & 0.44 & 0.51 & 0.62 & 0.25 & 0.30 & 0.30 & 0.30 & 0.32 & 0.68 \\
0.5 & 0.41 & 0.42 & 0.44 & 0.50 & 0.74 & 0.5 & 0.31 & 0.31 & 0.31 & 0.32 & 0.39 \\
\hline\hline
\end{tabular}
\caption{Space discretization and noise studies for all examples. Relative $L^2$-error for the state and parameter are shown for varying noise levels and mesh size $h_\text{max}$. A hyphen \enquote{-} indicates that the algorithm did not converge. The results indicate that moderate discretizations often suffice.
}
\label{tab:meshsize_error}
\end{center}
\end{minipage}
\end{table}

\begin{table}[!ht]\hspace{-0cm}
\begin{minipage}[b]{\linewidth}
\begin{center}
\scriptsize
\setlength{\tabcolsep}{4pt}
\renewcommand{\arraystretch}{1.05}
\begin{tabular}{||c|ccccc|c|ccccc||}
\hline
\multicolumn{6}{||c|}{\textbf{State error}} &
\multicolumn{6}{c||}{\textbf{Parameter error}} \\
\hline
\multicolumn{12}{||c||}{\textbf{Darcy flow}} \\
\hline
\diagbox[height=0.5cm]{$\Delta t$}{noise} & 0\% & 5\% & 10\% & 20\% & 50\% & \diagbox[height=0.5cm]{$\Delta t$}{noise} & 0\% & 5\% & 10\% & 20\% & 50\% \\
\hline
0.001 & 0.02 & 0.07 & 0.17 & 0.33 & 0.77 & 0.001 & 0.24 & 1.08 & 1.39 & 1.48 & 1.76 \\
0.01 & 0.02 & 0.06 & 0.11 & 0.23 & 0.70 & 0.01 & 0.25 & 0.90 & 1.05 & 1.09 & 1.32 \\
0.1 & - & - & - & - & - & 0.1 & - & - & - & - & - \\
\hline\hline
\multicolumn{12}{||c||}{\textbf{Fisher--KPP}} \\
\hline
\diagbox[height=0.5cm]{$\Delta t$}{noise} & 0\% & 5\% & 10\% & 20\% & 50\% & \diagbox[height=0.5cm]{$\Delta t$}{noise} & 0\% & 5\% & 10\% & 20\% & 50\% \\
\hline
0.001 & 0.05 & 0.07 & 0.12 & 0.21 & 0.58 & 0.001 & 0.52 & 0.66 & 0.79 & 0.87 & 0.92 \\
0.01 & 0.05 & 0.08 & 0.13 & 0.23 & - & 0.01 & 0.53 & 0.69 & 0.80 & 0.88 & - \\
0.1 & 0.05 & 0.08 & 0.13 & 0.22 & 0.82 & 0.1 & 0.52 & 0.69 & 0.80 & 0.88 & 0.92 \\
\hline\hline
\multicolumn{12}{||c||}{\textbf{Nonlinear potential}} \\
\hline
\diagbox[height=0.5cm]{$\Delta t$}{noise} & 0\% & 5\% & 10\% & 20\% & 50\% & \diagbox[height=0.5cm]{$\Delta t$}{noise} & 0\% & 5\% & 10\% & 20\% & 50\% \\
\hline
0.001 & 0.08 & 0.02 & 0.03 & 0.06 & 0.16 & 0.001 & 0.04 & 0.02 & 0.04 & 0.07 & 0.28 \\
0.01 & 0.08 & 0.02 & 0.03 & 0.07 & 0.17 & 0.01 & 0.04 & 0.03 & 0.04 & 0.07 & 0.29 \\
0.1 & 0.08 & 0.02 & 0.03 & 0.06 & 0.16 & 0.1 & 0.04 & 0.02 & 0.04 & 0.07 & 0.28 \\
\hline\hline
\multicolumn{12}{||c||}{\textbf{Allen--Cahn}} \\
\hline
\diagbox[height=0.5cm]{$\Delta t$}{noise} & 0\% & 5\% & 10\% & 20\% & 50\% & \diagbox[height=0.5cm]{$\Delta t$}{noise} & 0\% & 5\% & 10\% & 20\% & 50\% \\
\hline
0.001 & 0.41 & 0.42 & 0.44 & 0.50 & - & 0.001 & 0.28 & 0.29 & 0.29 & 0.33 & - \\
0.01 & 0.39 & 0.42 & 0.44 & 0.50 & - & 0.01 & 0.28 & 0.29 & 0.29 & 0.33 & - \\
0.1 & 0.35 & 0.39 & 0.41 & 0.46 & - & 0.1 & 0.28 & 0.29 & 0.29 & 0.33 & - \\
\hline\hline
\end{tabular}
\caption{Time discretization and noise studies for all examples. Relative $L^2$-error for the state and parameter are reported for varying noise levels and time step $\Delta t$. A hyphen \enquote{-} indicates that the algorithm did not converge. The results show that MRAS is generally stable for moderate time steps.} 
\label{tab:timestep_errors}
\end{center}
\end{minipage}
\end{table}

\begin{table}[!ht]
\centering
\scriptsize
\setlength{\tabcolsep}{4pt}
\renewcommand{\arraystretch}{1.05}
\begin{tabular}{||c|cccc|c|cccc||}
\hline
\multicolumn{5}{||c|}{\textbf{Fisher--KPP: State error}} &
\multicolumn{5}{c|}{\textbf{Fisher--KPP: Parameter error}} \\
\hline
\diagbox[height=0.5cm]{noise}{$a_0$} & 1 & 2 & 3 & 4 & \diagbox[height=0.5cm]{noise}{$a_0$} & 1 & 2 & 3 & 4 \\
\hline
0\% & 0.05 & 0.17 & 0.29 & 0.41 & 0\% & 0.56 & 1.23 & 2.02 & 2.83 \\
10\% & 0.14 & 0.17 & 0.23 & 0.30 & 10\% & 0.83 & 1.17 & 1.65 & 2.18 \\
\hline\hline
\end{tabular}
\bigskip
\begin{tabular}{||c|cccc|c|cccc||}
\hline
\multicolumn{5}{||c|}{\textbf{Allen--Cahn: State error}} &
\multicolumn{5}{c|}{\textbf{Allen--Cahn: Parameter error}} \\
\hline
\diagbox[height=0.5cm]{noise}{$c_0$} & 0 & 5 & 10 & 15 & \diagbox[height=0.5cm]{noise}{$c_0$} & 0 & 5 & 10 & 15 \\
\hline
0\% & 0.08 & 0.27 & 0.47 & 0.68 & 0\% & 0.12 & 0.25 & 0.42 & 0.59 \\
10\% & 0.23 & 0.28 & 0.38 & 0.49 & 10\% & 0.17 & 0.24 & 0.34 & 0.45 \\
\hline
\end{tabular}
\caption{Sensitivity to initial parameter $q_0$ for Fisher-KPP and Allen-Cahn problems.
Relative $L^2$-errors for state and parameter reconstruction are shown for varying amplitudes of $q_0$ in $\{0,1,2,3,4,5,10, 15\}$, differing from the ground truth. Larger initial errors lead to worse reconstruction, especially without coercivity as in Fisher-KPP problem.
}
\label{tab:init_sensitivity_combined}
\end{table}

To test the performance of our method under challenging conditions,  we conducted a large number of numerical experiments. In particular, we studied the effect of high noise levels, space-time  discretization, and sensitivity of the reconstruction to the initial data.  
The comprehensive results are reported in Tables \ref{tab:meshsize_error}-\ref{tab:init_sensitivity_combined}. 

\paragraph{Noise effect} 

From Table \ref{tab:meshsize_error}-\ref{tab:timestep_errors}, the relative error in the state reconstruction is smaller than that of the parameter reconstruction for the Darcy and Fisher-KPP problems; whereas the opposite holds for the other two examples. 
In all cases, state and parameter errors gradually grow according to noise. The method remains stable up to a breaking point; 
see the $50\%$ noise column in Table \ref{tab:timestep_errors}.

\paragraph{Mesh refinement} Table \ref{tab:meshsize_error} shows that mesh refinement beyond a moderate level gives diminishing returns. For several cases considered here, decreasing mesh size below $h_\text{max}=0.25$ does not noticeably improve reconstruction.
Interestingly, under extreme noise, a finer mesh may even perform worse than a coarser one; this is visible, for instance, in the 
Allen-Cahn example.

\paragraph{Time step refinement} Just as with mesh refinement, under extreme noise, a finer time step may perform worse than a coarser one; see Table \ref{tab:timestep_errors}. This is visible, for instance, in the Fisher and Allen-Cahn examples. For many rows, reducing $\Delta t$ from $0.1$ to $0.01$ or $0.001$ changes little, although for the Darcy problem, $\Delta t=0.1$ is clearly too large. In our showcased experiments, we used $\Delta t= 0.001$ throughout, although $\Delta t=0.01$ may offer a better trade-off between accuracy and efficiency in practice.

Overall, Table \ref{tab:meshsize_error}-\ref{tab:timestep_errors} indicates that MRAS can provide reliable state reconstruction over a fairly wide range of discretization choices. The parameter reconstruction remains stable  and tends to deteriorate as the noise increases. 
Our method is generally robust for relative  noise level up to (and possibly higher than) 20\% for moderate discretizations. In most examples, instability appears only at the extreme noise level of 50\%, or for very coarse time discretization, such as $\Delta t=0.1$.
Under extreme noise, a finer discretization can sometimes perform worse than a coarser one.

\paragraph{Initial guess} 

\begin{figure}
    \centering
    \begin{minipage}[t]{0.98\textwidth}
        {\small
        \hspace{65pt}
        Initial  $q_0$\hspace{45pt}
        Reconstructed $q$\hspace{65pt}
        Truth $q^\dagger$}
        \vspace{8pt}
    \end{minipage}
    \vspace{10pt}
    \begin{minipage}[t]{1\textwidth}
        \centering
        \begin{subfigure}[b]{0.08\textwidth}
             \raisebox{10pt}{\pgfplotscolorbardrawstandalone[
                colormap/jet,
                point meta min=0,
                point meta max=4,
                colorbar style={
                    height = 1.6cm,
                    width=0.15\textwidth,
                    /pgf/number format/fixed,
                    /pgf/number format/precision=2,
                    tick style={font=\tiny},
                    ytick={0, 2,4},
                    yticklabels={0, 2,4},
                }
            ]
            }
        \end{subfigure}\hfill
        \begin{subfigure}[b]{0.25\textwidth}
            \includegraphics[width=\linewidth,trim={14cm 0cm 14cm 14cm},clip]{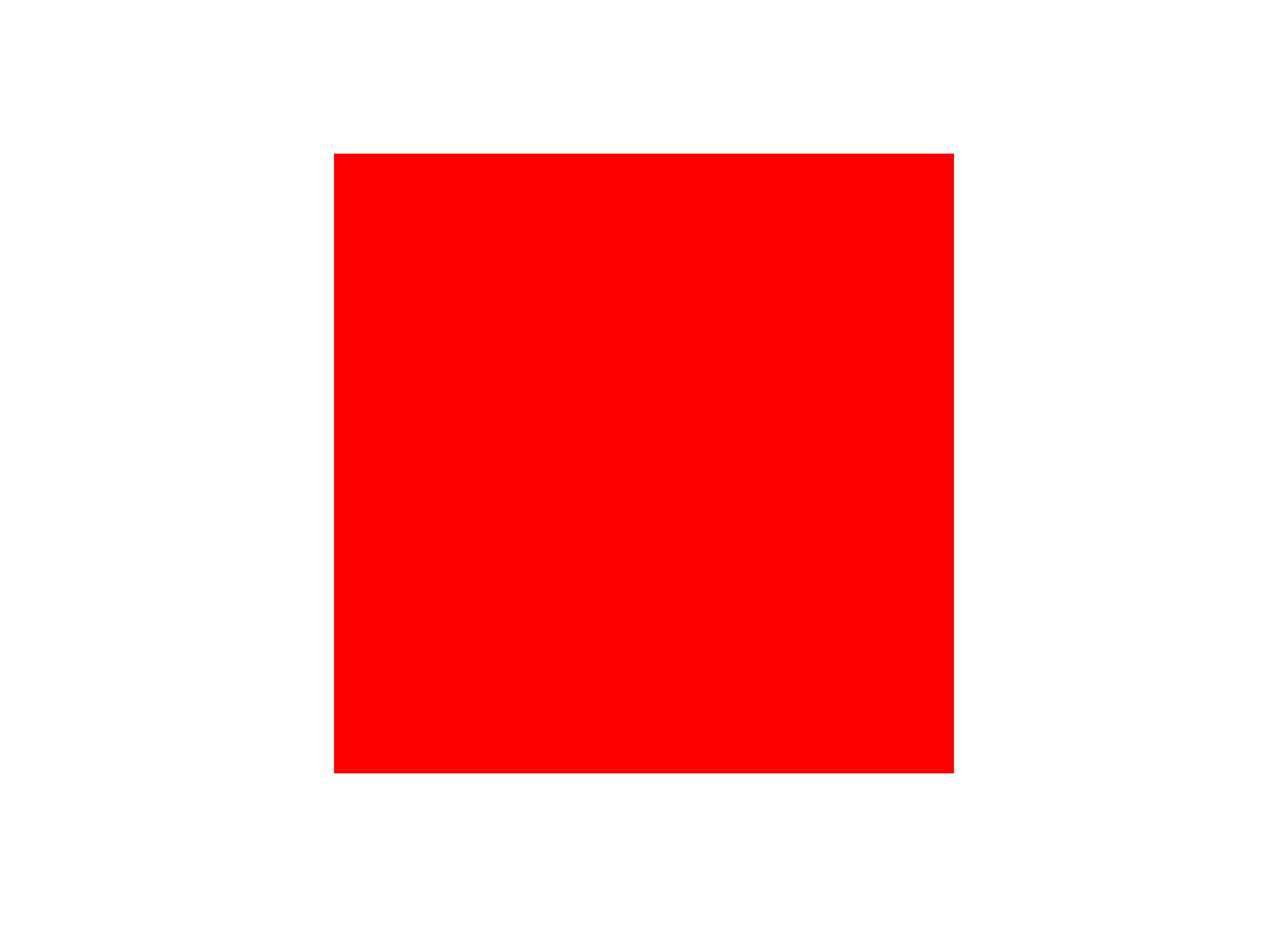}
        \end{subfigure}\hfill
        \begin{subfigure}[b]{0.25\textwidth}
            \includegraphics[width=\linewidth,trim={14cm 0cm 14cm 14cm},clip]{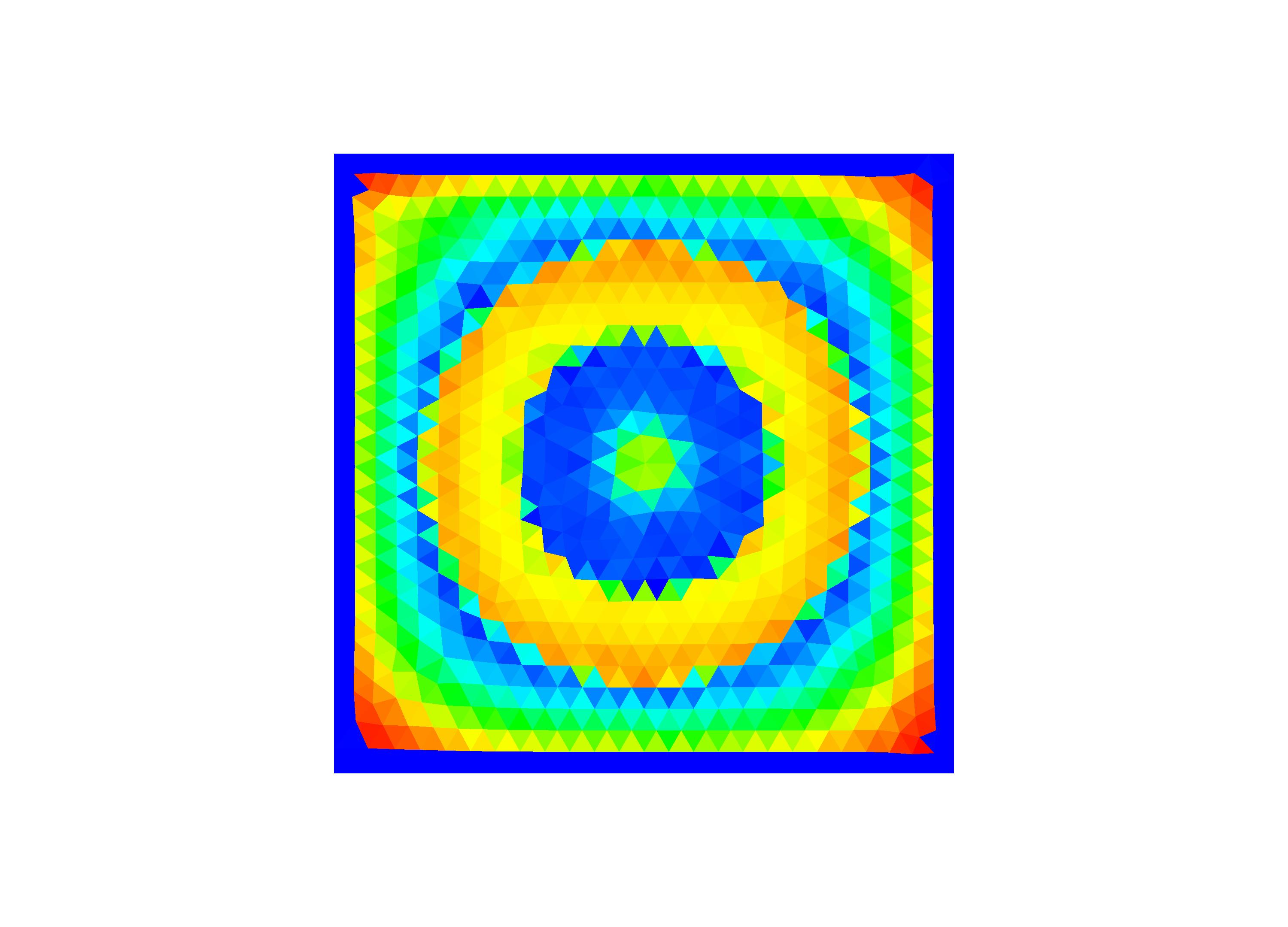}
        \end{subfigure}\hfill
        \raisebox{1pt}[0pt][0pt]{\rule{0.5pt}{0.15\textheight}}\hfill
        \begin{subfigure}[b]{0.25\textwidth}
            \includegraphics[width=\linewidth,trim={15cm 0cm 15cm 15cm},clip]{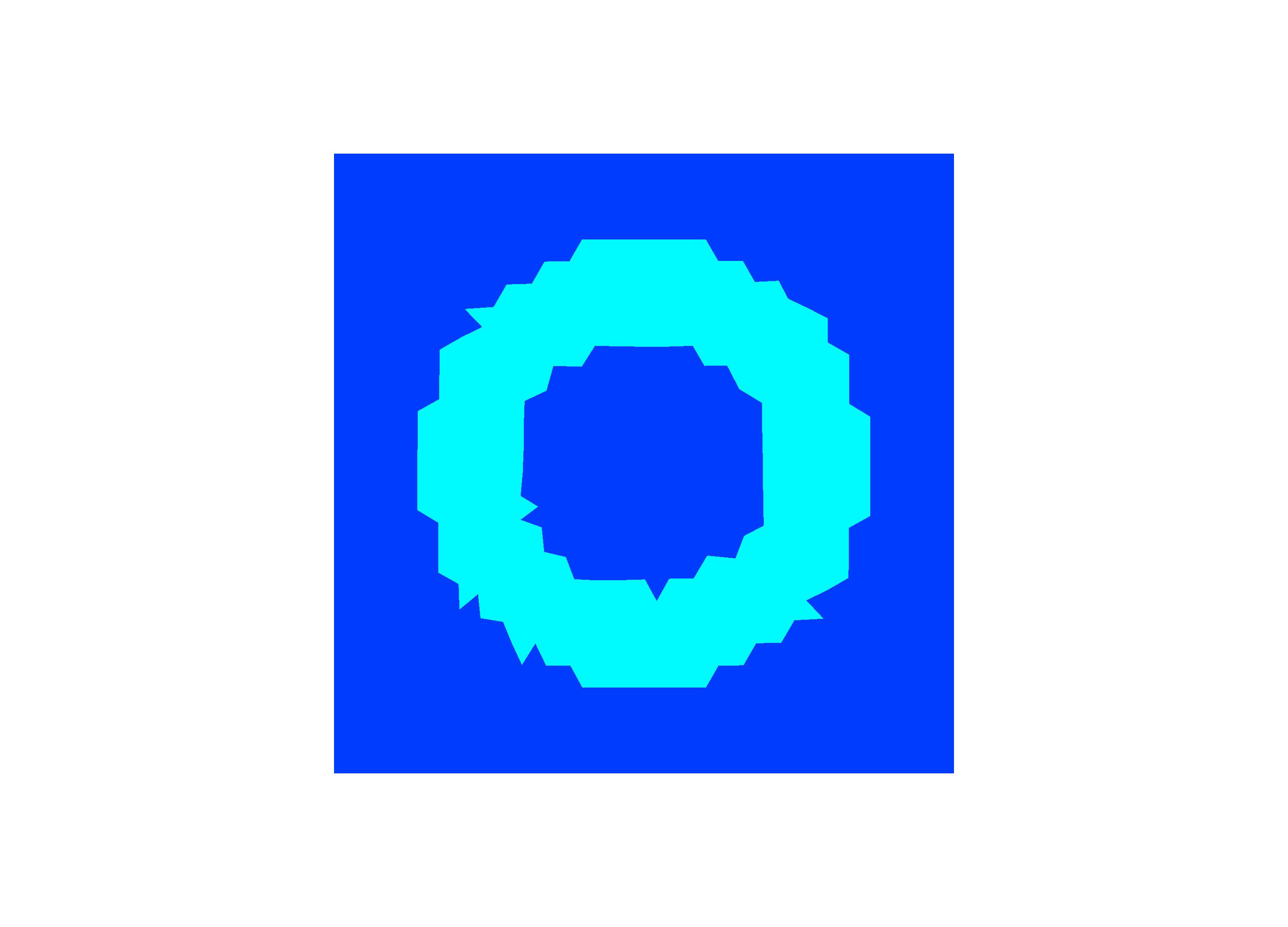}
        \end{subfigure}
        \\
        \vspace{0.3cm}
        \begin{subfigure}[b]{0.08\textwidth}
             \raisebox{12pt}{\pgfplotscolorbardrawstandalone[
                colormap/jet,
                point meta min=0,
                point meta max=15,
                colorbar style={
                    height = 1.6cm,
                    width=0.15\textwidth,
                    /pgf/number format/fixed,
                    /pgf/number format/precision=1,
                    tick style={font=\tiny},
                    ytick={0, 7.5, 15},
                    yticklabels={0, 7.5, 15},
                }
            ]
            }
        \end{subfigure}\hfill
        \begin{subfigure}[b]{0.25\textwidth}
            \includegraphics[width=\linewidth,trim={15cm 0cm 15cm 15cm},clip]{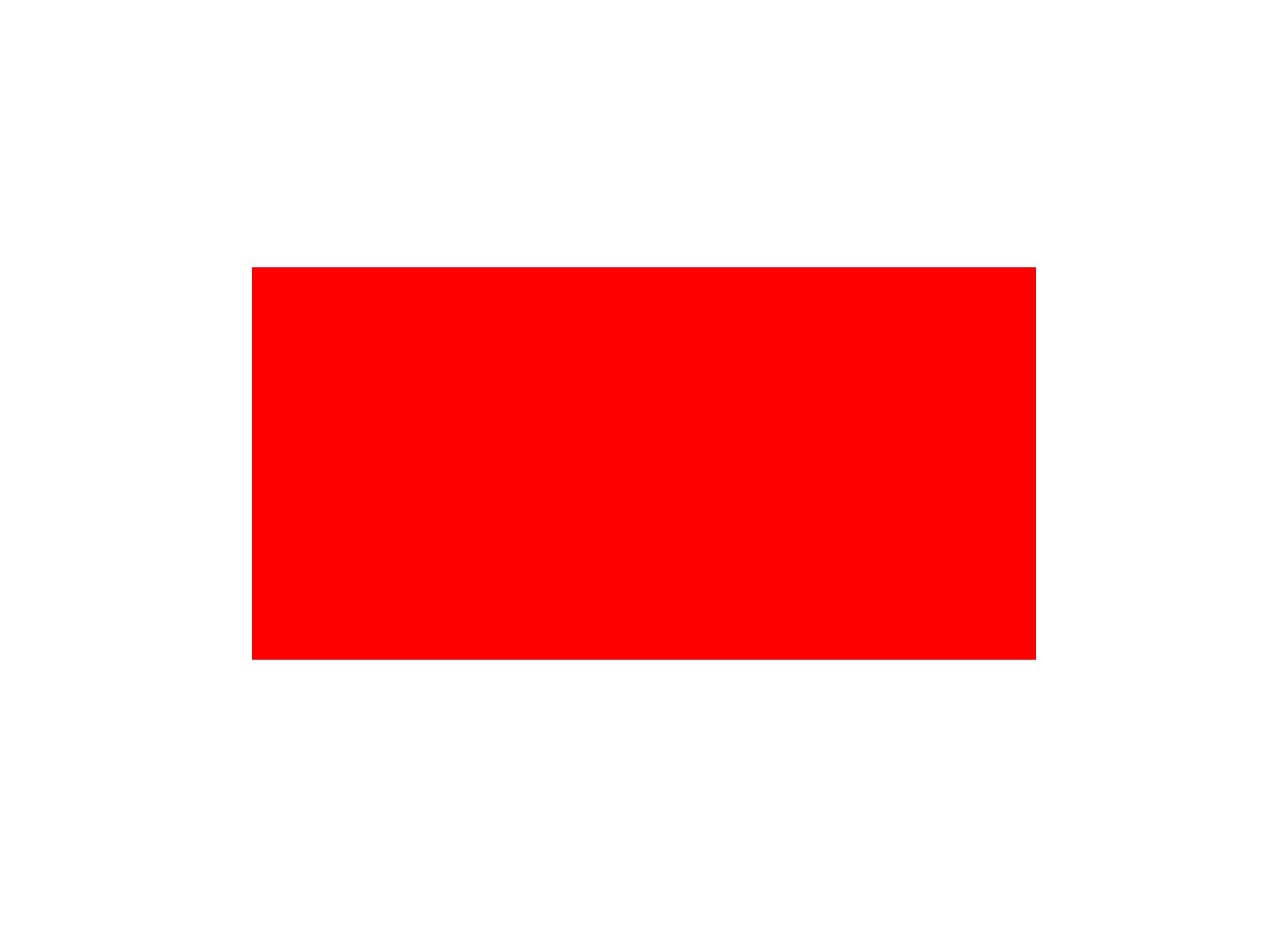}
        \end{subfigure}\hfill
        \begin{subfigure}[b]{0.25\textwidth}
            \includegraphics[width=\linewidth,trim={15cm 0cm 15cm 15cm},clip]{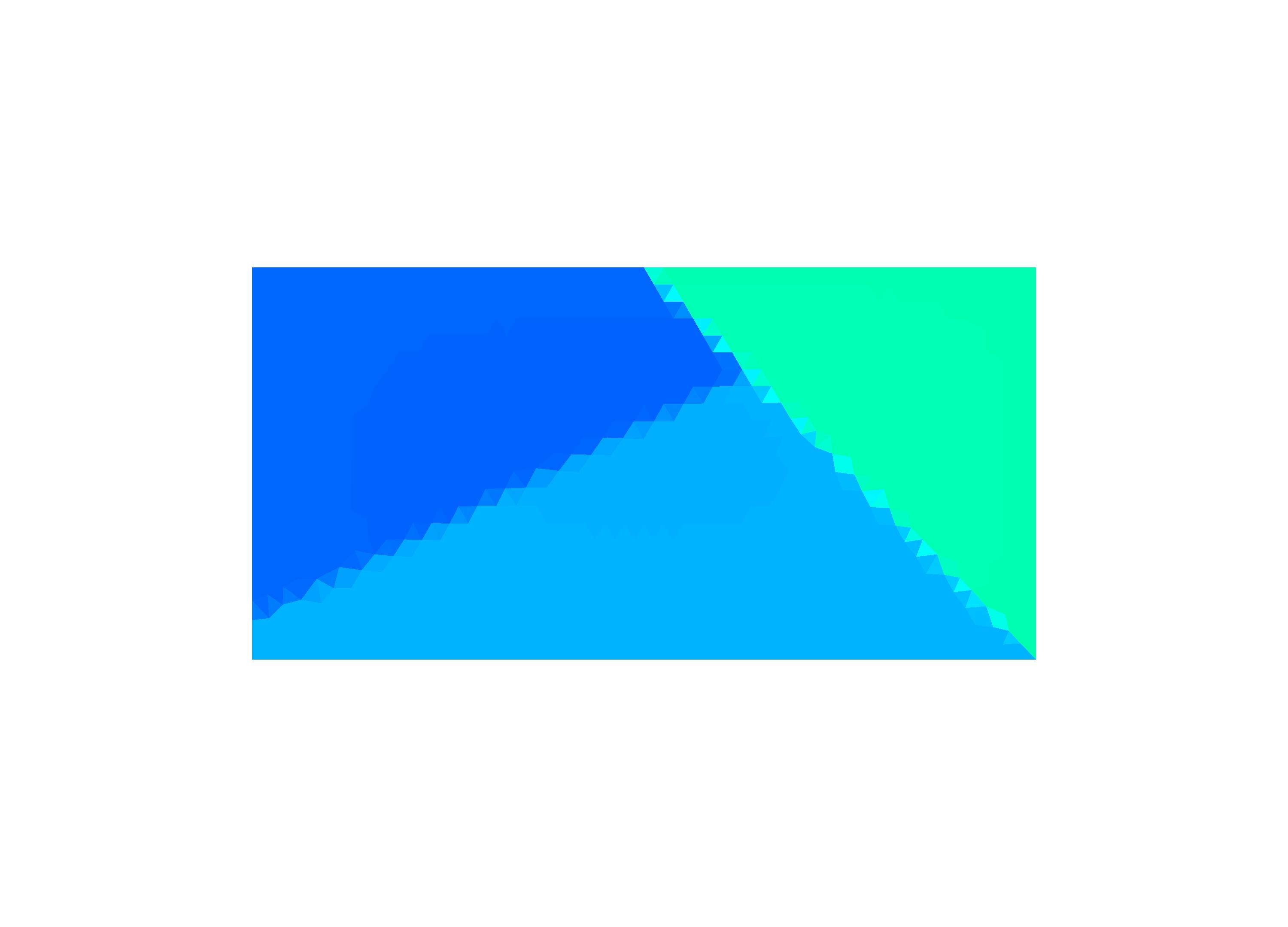}
        \end{subfigure}\hfill
        \raisebox{12pt}[0pt][0pt]{\rule{0.5pt}{0.12\textheight}}\hfill
        \begin{subfigure}[b]{0.25\textwidth}
            \includegraphics[width=\linewidth,trim={15cm 0cm 15cm 15cm},clip]{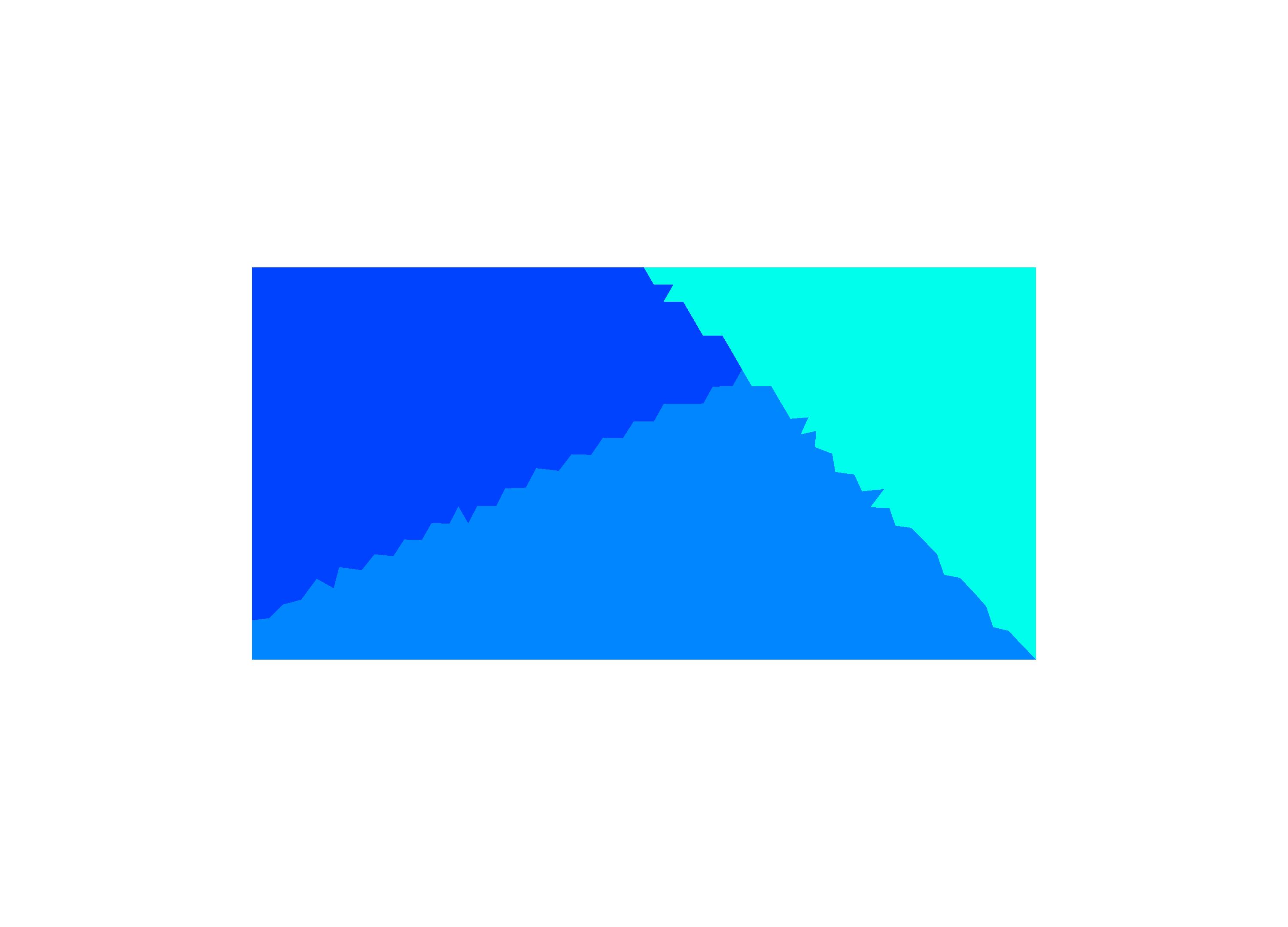}
        \end{subfigure}
    \end{minipage}
    \caption{Sensitivity to the initial parameter $q_0$ with clean data and very poor initial guess. Top: Fisher--KPP problem, $a_0=4$. Bottom: Allen--Cahn problem, $c_0=15$. The Allen-Cahn results are highly robust, whereas for Fisher-KPP, the parameter shape is recovered well but not its amplitude.}
    \label{fig::init_vs_final}
\end{figure}

We now examine the sensitivity of the initial guess $q(0)$ on the Fisher--KPP and Allen--Cahn problems. In Table~\ref{tab:init_sensitivity_combined}, we use constant initial parameters over the entire domain with amplitudes $0,5,10,15$ (Allen-Cahn) and $1,2,3,4$ (Fisher--KPP), chosen in proportion to the amplitudes of the two different ground truths for these examples.

For both examples, the reconstruction error generally increases as the initial amplitude moves farther from the truth, indicating that the problem becomes harder when $q(0)$ is farther from $q^\dagger$. This reflects the dependence on the initial error $e(0)$ in the error bounds of Proposition \ref{prop-noisefree}.

In Table \ref{tab:init_sensitivity_combined}, the Allen-Cahn entries show smoother dependence on the initial parameter $c(0)$, both state and parameter errors increasing gradually with initial error magnitude. This points to a robust reconstruction, with both shape and amplitude well recovered; see Figure \ref{fig::init_vs_final}.

The Fisher-KPP case appears more delicate than Allen-Cahn, especially for parameter recovery. In Table \ref{tab:init_sensitivity_combined},  while the state error grows gradually in both noise regimes, the parameter error increases substantially. This is likely due to the lack of coercivity in Fisher-KPP (and Darcy, see Remark~\ref{rem:coercive}). Indeed in Figure \ref{fig::init_vs_final}, the shape of the unknown potential $c$ is recovered well, but not its exact amplitude.

\section{Conclusion and outlook}

In this work, we have presented an implementation-oriented numerical framework for the model reference adaptive system (MRAS) first proposed in \cite{Tram_Kaltenbacher_2021}, accompanied by four increasingly complex case studies -- including PDE and space analysis, explicit evaluation of the terms appearing in the MRAS and an in-depth numerical examination of the MRAS's ability to reconstruct unknown parameters. Overall, these findings support the idea that the MRAS is broadly and easily applicable to a variety of time-dependent physical problems, and can consistently return high-quality reconstructions even in the presence of -- possibly extreme -- nonlinearity.

There exist various highly promising extensions that we would like to extend to, both on the theoretical and implementation level. 
Significant generality would be gained by considering more types of observation operators, such as restricted measurement 
or {trace measurement} \cite{BoigerKaltenbacher}. 
By building on our work \cite{nguyen-passive}, it is our intent to expand the scope of the MRAS to {nonlinear measurement} operators, such as correlation measurement. 

In terms of the MRAS implementation itself, it is interesting to consider also other time-stepping schemes. One highly promising family of schemes is the {Runge-Kutta} methods, allowing the entire history of the state to be employed in the MRAS. This can be viewed as expanding the assimilation time window. 
A future extension is to incorporate lower-quality data -- such as nonmatching, discrete or restricted data -- through measurement operators in \emph{optimal experimental design} \cite{Aarset_2025, AarsetNguyenOED25}.  Combining this with MRAS in a unified framework is an interesting direction for future work.
Finally, with the extensive numerical examples presented in the current work demonstrating the validity of the MRAS, it is natural to extend further to real-world datasets, such as those used in {weather forecasting} and other atmospheric applications.

\paragraph{Data availability statement}
The codes and data that support this article are publicly accessible at \cite{Code}.

\paragraph{Acknowledgments} The authors wish to thank Barbara Kaltenbacher for many valuable comments. They also thank the reviewers for their insightful suggestions, leading to an improvement of the manuscript. Part of CA's work was carried out at the Institute for Numerical and Applied Mathematics, University of Göttingen, Germany. Part of TN's work was carried out at the Max-Planck Institute for Solar System Research, Germany. CA and TN acknowledge support from the DFG through Grant 432680300 -- SFB 1456 (C04). CA acknowledges support from the Excellence Initiative of Aix–Marseille University -- A*MIDEX, a French \enquote{Investissements d’Avenir} programme. TN acknowledges support from CHIMiRA: Cambridge Hub for Innovative Mathematics in Research and Applications.

\FloatBarrier
\printbibliography
\end{document}